\documentclass[11pt,reqno]{amsproc}
\usepackage{lipsum}
\usepackage{titletoc}
\usepackage{amsmath,amscd} 
\usepackage{cite}
\usepackage{color}
\usepackage{graphicx}
\usepackage{psfrag}
\usepackage{multirow}
\usepackage{rotating}
\usepackage{fullpage}
\usepackage{subfigure}
\usepackage{enumerate}
\usepackage{algorithmic}
\usepackage{algorithm}
\usepackage[small]{caption}
\newtheorem{theorem}{Theorem}[section]
\newtheorem{lemma}[theorem]{Lemma}

\newtheorem{definition}[theorem]{Definition}

\newtheorem{remark}[theorem]{Remark}

\numberwithin{equation}{section}
\newcommand*{\bigchi}{\mbox{\Large$\chi$}}
\newcommand*{\bigrho}{\mbox{\Large$\rho$}}
\newcommand*{\bigsigma}{\mbox{\Large$\sigma$}}
 
\usepackage{morefloats}

\newlength{\drop}
\definecolor{amethyst}{rgb}{0.6, 0.4, 0.8}
\definecolor{burgundy}{rgb}{0.5, 0.0, 0.13}
  
\title{On mesh restrictions to satisfy comparison principles, 
  maximum principles, and the non-negative constraint: 
  Recent developments and new results}
\author{\textbf{M.~K.~Mudunuru} and \textbf{K.~B.~Nakshatrala} \\
{\small Department of Civil and Environmental Engineering, 
University of Houston}}
\date{\today{}}

\begin{document}

%===========================;
%  Title page of the paper  ;
%===========================;
\begin{titlepage}
    \drop=0.1\textheight
    \centering
    \vspace*{\baselineskip}
    \rule{\textwidth}{1.6pt}\vspace*{-\baselineskip}\vspace*{2pt}
    \rule{\textwidth}{0.4pt}\\[\baselineskip]
    {\LARGE \textbf{\color{burgundy} 
    On mesh restrictions to satisfy comparison principles, \\
    maximum principles, and the non-negative constraint: \\ 
    Recent developments and new results}}\\[0.3\baselineskip]
    \rule{\textwidth}{0.4pt}\vspace*{-\baselineskip}\vspace{3.2pt}
    \rule{\textwidth}{1.6pt}\\[\baselineskip]
    \scshape
%    An e-print of the paper will soon be made available on arXiv. \par
    
    \vspace*{2\baselineskip}
    Authored by \\[\baselineskip]
    
    {\Large M.~K.~Mudunuru\par}
    {\itshape Graduate Student, University of Houston}\\[\baselineskip]
    
    {\Large K.~B.~Nakshatrala\par}
    {\itshape Department of Civil \& Environmental Engineering \\
    University of Houston, Houston, Texas 77204--4003.\\ 
    \textbf{phone:} +1-713-743-4418, \textbf{e-mail:} knakshatrala@uh.edu \\
    \textbf{website:} http://www.cive.uh.edu/faculty/nakshatrala\par}
    
    %-------------------------------------------------------------------;
    %  Figure-6: Non-negativity, monotone, and monotonicity preserving  ;
    %-------------------------------------------------------------------;
    \begin{figure}[h]
      \psfrag{A1}{{\small Non-negativity}}
      \psfrag{A2}{{\small Monotone}}
      \psfrag{A3}{{\small Monotonicity}}
      \psfrag{A4}{{\small preserving}}
      \includegraphics[scale=0.5]{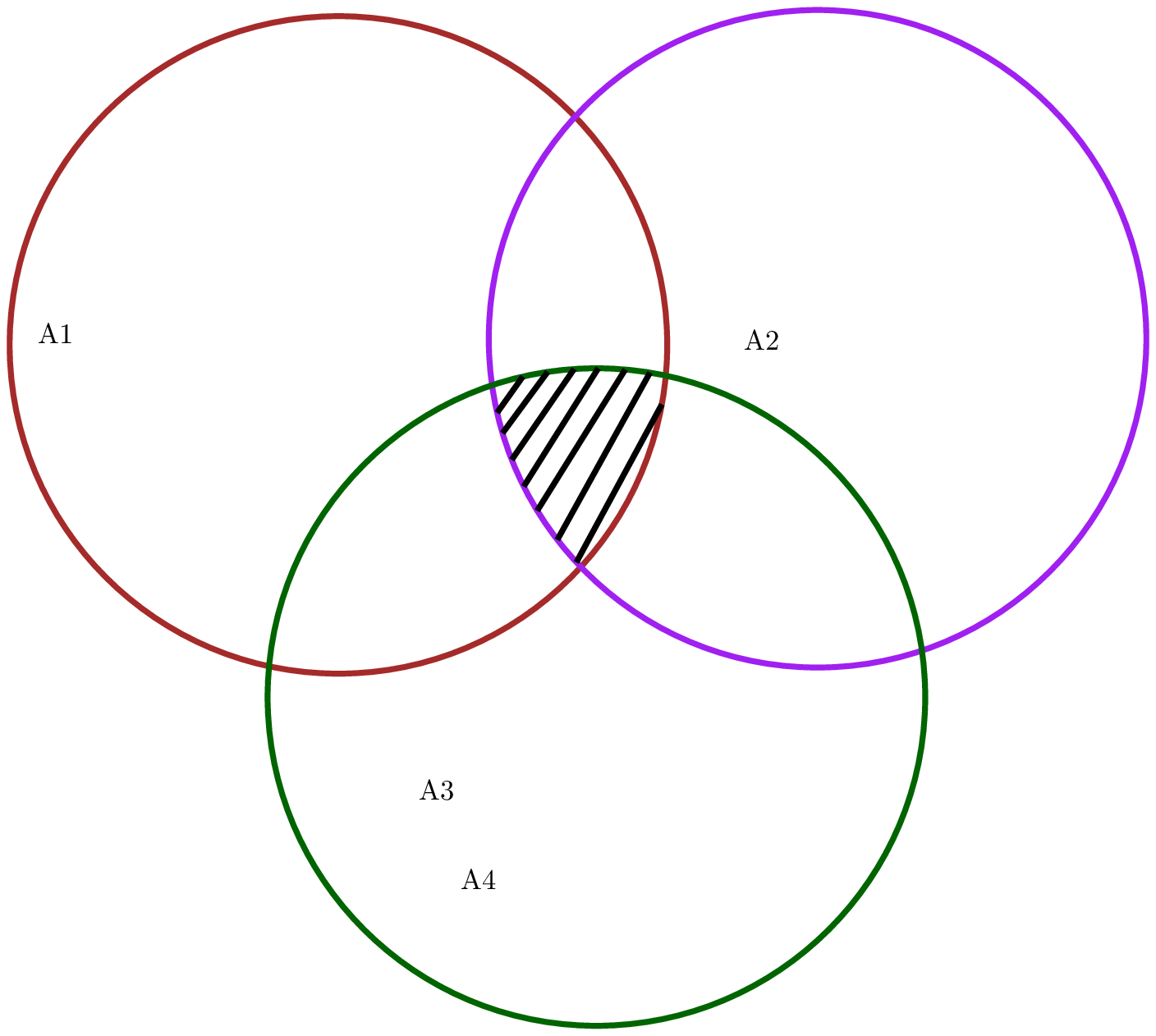}
    \end{figure}
    \emph{Designing a numerical methodology that possesses 
    non-negativity, monotone, and monotonicity preserving 
    properties is still an open problem.}
    \vfill
    {\scshape 2015} \\
    {\small Computational \& Applied Mechanics Laboratory} \par
\end{titlepage}

%=========================;
%  Abstract and keywords  ;
%=========================;
\begin{abstract}
  This paper concerns with mesh restrictions that are 
  needed to satisfy several important mathematical 
  properties -- maximum principles, comparison 
  principles, and the non-negative constraint -- 
  for a general linear second-order elliptic partial 
  differential equation. We critically review some 
  recent developments in the field of discrete maximum 
  principles, derive new results, and discuss some 
  possible future research directions in this area. 
  In particular, we derive restrictions for a three-node 
  triangular (T3) element and a four-node quadrilateral 
  (Q4) element to satisfy comparison principles, maximum 
  principles, and the non-negative constraint under the 
  standard single-field Galerkin formulation. Analysis 
  is restricted to uniformly elliptic linear differential 
  operators in divergence form with Dirichlet boundary 
  conditions specified on the entire boundary of the 
  domain. Various versions of maximum principles and 
  comparison principles are discussed in both continuous 
  and discrete settings. 
  In the literature, it is well-known that an acute-angled 
  triangle is sufficient to satisfy the discrete weak maximum 
  principle for pure \emph{isotropic} diffusion. Herein, we 
  show that this condition can be either too restrictive or 
  not sufficient to satisfy various discrete principles when 
  one considers anisotropic diffusivity, advection velocity 
  field, or linear reaction coefficient. Subsequently, we 
  derive appropriate restrictions on the mesh for simplicial 
  (e.g., T3 element) and non-simplicial (e.g., Q4 element) 
  elements. Based on these conditions, an iterative algorithm 
  is developed to construct simplicial meshes that preserves 
  discrete maximum principles using existing open source mesh 
  generators. Various numerical examples based on different 
  types of triangulations are presented to show the pros and 
  cons of placing restrictions on a computational mesh. We 
  also quantify local and global mass conservation errors 
  using representative numerical examples, and illustrate 
  the performance of metric-based meshes with respect to 
  mass conservation.  
\end{abstract}

\maketitle

%\vspace{-0.3in}
%\thispagestyle{empty}
%\setcounter{tocdepth}{4}
%\tableofcontents
%\addtocontents{toc}{~\hfill\textbf{Page}\par}
%\newpage
%\clearpage
%\setcounter{page}{1}

%======================;
%  Input all sections  ;
%======================;

%****************************************;
%                                        ;
%  NAME                                  ;
%    S1_MP_Introduction.tex              ;
%                                        ;
%  WRITTEN BY                            ;
%    Maruti Kumar Mudunuru               ;
%    Kalyana Babu Nakshatrala            ;
%                                        ;
%****************************************;
\section{INTRODUCTION, MOTIVATION, AND CONTEMPORARY ADVANCEMENTS}
\label{Sec:S1_MP_Introduction}
Diffusion-type equations are commonly encountered in 
various branches of engineering, sciences, and even 
in economics \cite{Crank_Oxford_1975,Mei,Aoki}. These 
equations have been well-studied in Applied Mathematics, 
and several properties and \emph{a priori} estimates have 
been derived \cite{Pao}. Numerous numerical formulations 
have been proposed and their performance has been 
analyzed both theoretically and numerically \cite{Gresho_Sani_v1}. 
Several sophisticated software packages, such as \textsf{ABAQUS} 
\cite{Abaqus_6.8-3}, \textsf{ANSYS} \cite{ANSYS_v14.5}, \textsf{COMSOL} 
\cite{COMSOL_v4.3a}, and \textsf{MATLAB's PDE Toolbox} \cite{MATLAB_2012a}, 
have been developed to solve these types of equations. Special solvers for 
solving the resulting discrete equations have also been proposed and studied 
adequately \cite{Hackbusch}. 

This paper is concerned with numerical solutions for anisotropic 
advection-diffusion-reaction equations. Despite the aforementioned
advances, it should be noted that a numerical solution always loses 
some mathematical properties that the exact solution possesses. 
In particular, the aforementioned software packages and popular 
numerical formulations do not satisfy the so-called discrete 
comparison principles (DCPs), discrete maximum principles (DMPs), 
and the non-negative constraint (NC). For example, consider the 
numerical simulation for pure anisotropic diffusion equation in 
an L-shaped domain with multiple holes using the commercial software
package \textsf{ABAQUS}.

Numerical simulations are performed based on various unstructured 
finite element meshes (see Figure \ref{Fig:ABAQUS_T3_Q4_Meshes}, 
which is to the scale), using the following popular anisotropic 
diffusivity tensor from hydrogeological and subsurface flow 
literature \cite{2013_Nakshatrala_Mudunuru_Valocchi_JCP_v253_p278_p307}:
%----------------------------------------------;
%  Equation: Diffusivity tensor h-convergence  ;
%----------------------------------------------;
\begin{align}
  \label{Eqn:AnisoRotationDiff}
  \mathbf{D}(\mathbf{x}) = \mathbf{R} 
  \mathbf{D}_{\mathrm{eigen}} \mathbf{R}^{\mathrm{T}}
\end{align}
where $\mathbf{D}_{\mathrm{eigen}}$ is a diagonal matrix comprised 
of the eigenvalues of $\mathbf{D}(\mathbf{x})$. The corresponding 
principal eigenvectors are the column entries in the orthogonal 
matrix $\mathbf{R}$. The expressions for $\mathbf{D}_{\mathrm{eigen}}$ 
and $\mathbf{R}$ are assumed as follows:
%----------------------------------------;
%  Equation: Rotated diffusivity tensor  ;
%----------------------------------------;
\begin{align}
  \label{Eqn:NN_rotation_tensor}
  \mathbf{R} &= \left(\begin{array}{cc}
    \cos(\theta) & -\sin(\theta) \\
    \sin(\theta) & \cos(\theta) \\
  \end{array}\right) \\
  \label{Eqn:NN_Eigenvalue_Matrix}
  \mathbf{D}_{\mathrm{eigen}} &= \left(\begin{array}{cc}
    d_{\mathrm{max}} & 0 \\
    0 & d_{\mathrm{min}} \\
  \end{array}\right) 
\end{align}
Herein, $d_{\mathrm{max}}$ and $d_{\mathrm{min}}$ correspond to the 
maximum and minimum eigenvalues. $\theta$ corresponds to the angle 
of orientation of the eigenvector coordinate system. It should be 
noted that these eigenvalues have physical significance and are 
related to the \emph{transverse and longitudinal diffusivities} 
in the eigenvector coordinate system. Diffusion process is simulated 
based on equations \eqref{Eqn:NN_AD_GE_BE}--\eqref{Eqn:NN_AD_GE_DBC}. 
We assume $\mathbf{v}(\mathbf{x}) = \mathbf{0}$, $\alpha(\mathbf{x}) 
= 0$, and $f(\mathbf{x}) = 0$. The prescribed concentration on 
the sides of the L-shaped domain is equal to zero. Correspondingly, 
the concentration on the perimeter of the holes are set to be equal 
to one. The values of $d_{\mathrm{max}}$, $d_{\mathrm{min}}$, and 
$\theta$ for $\mathbf{D}(\mathbf{x})$ given by the equations 
\eqref{Eqn:AnisoRotationDiff}--\eqref{Eqn:NN_Eigenvalue_Matrix} 
are assumed to be equal to 1000, 1, and $\frac{\pi}{3}$. 

Very fine triangular (where the total number of nodes and mesh elements 
are equal to 86326 and 169453) and quadrilateral (where the total number 
of nodes and mesh elements are equal to 91778 and 90625) meshes are used 
to perform \textsf{ABAQUS} numerical simulations. The concentration profile 
obtained is shown in Figure \ref{Fig:ABAQUS_T3_Q4_Contours}. In this figure, 
we have not shown the concentration contour using four-node quadrilateral 
mesh, as it is almost identical to that of the contour obtained by 
employing three-node triangular mesh. The white area within the L-shaped 
domain with multiple holes represents the region in which the \emph{numerical 
value} of concentration is \emph{negative} and also exceeds the \emph{maximum}. 
To be precise, in the case of triangular mesh, 2.22\% and 3.92\% of the 
nodes have violated the non-negative and maximum constraints. Correspondingly, 
the minimum and maximum values for concentration obtained are -0.0238 and 
1.0076. These are considerably far away from the possible values, which 
are between 0 and 1. Similarly, for quadrilateral mesh, these values are 
slightly lower. Quantitatively, these are around 2.17\% and 3.75\%. But 
the minimum and maximum values of concentration (-0.0287 and 1.0086) are 
slightly higher than that of the triangular mesh. 
Additionally, it is evident that more than 6\% of the nodes 
have \emph{unphysical negative values} for the concentration. 
There are nodes for which the concentration exceeded the 
maximum constraint. 
Furthermore, from Figures \ref{Fig:ABAQUS_T3_Q4_MinMaxConc} 
and \ref{Fig:ABAQUS_T3_Q4_MinMaxpercent}, it is evident that 
these values \emph{do not decrease} with mesh refinement. In 
general, there are three possible routes to overcome such 
limitations and satisfy DCPs, DMPs, and NC; which we shall 
describe below. 

%=================================;
%  Subsection: Mesh restrictions  ;
%=================================;  
\subsection{Strategy - I: Mesh restrictions}
The first approach is to place restrictions on the mesh 
to meet maximum principles and the non-negative constraint.
For isotropic homogeneous diffusivity, Ciarlet and Raviart 
\cite{Ciarlet_Raviart_CMAME_1973_v2_p17} have shown that 
numerical solutions based on the single-field Galerkin 
finite element formulation, in general, does not converge 
\emph{uniformly}. It should however be noted that the 
single-field Galerkin formulation is a converging scheme. 
Ciarlet and Raviart have also shown that a sufficient condition 
for single-field Galerkin formulation to converge uniformly for 
\emph{isotropic} diffusion is to employ a well-centered 
three-node triangular element mesh with low-order interpolation.
   
The obvious advantage of this approach is that one can use the 
single-field Galerkin formulation without any modification. The 
drawback is that an appropriate computational mesh may not exist 
because of the required restrictions on the shape and size of the 
finite element. For example, it is not an easy task (sometimes 
it \emph{not} possible) to generate a well-centered triangular mesh for any 
given two-dimensional domain \cite{Vanderzee_Hirani_Guoy_Ramos_SIAMJSC_2010_v31_p4497}. 
Note that requiring a mesh to be well-centered is a more 
stringent than requiring the mesh to be Delaunay. In fact, 
a well-centered mesh is Delaunay but the converse need not 
be true. 

In scientific literature, there are numerous commercial and non-commercial 
mesh generators that produce premium quality structured and unstructured 
meshes for various complicated domains. For instance, the survey paper 
by Owen \cite{1998_Owen_IMR_p239_p267} accounts for more than 70 
unstructured mesh generation software products. But it should be emphasized 
that Owen \cite{1998_Owen_IMR_p239_p267} rarely mentions about non-obtuse, 
acute, and anisotropic $\mathcal{M}$-uniform mesh generators. However, it 
is evident from the above discussion that these types of meshes have a profound 
impact on solving various important physical problems related to diffusion-type 
equations. In recent years, there has been considerable effort in developing 
such types of mesh generators. For example, some open source meshing software packages 
which are relevant to mesh restrictions methodology are as follows:
%------------------------------------------------------;
%  Non-obtuse, acute, and anisotropic mesh generators  ;
%------------------------------------------------------;
\begin{itemize} 
  \item Non-obtuse and acute triangulations in 2D: \textsf{aCute} 
    \cite{2007_Erten_Ungor_CCCG_p205_p208,
    2009_Erten_Ungor_PVD_p192_p201,2009_Erten_Ungor_SIAMJSC_v31_p2103_p2130} 
    (a meshing software, which is based on \textsf{Triangle} 
    \cite{shewchuk96b})
    \item Anisotropic $\mathcal{M}$-uniform triangulations in 2D: \textsf{BAMG} 
      \cite{1998_Hecht_BAMG} in \textsf{FreeFem++} \cite{FreeFempp,2002_Hecht_JNM_v20_p251_p266},
      \textsf{BL2D}  \cite{1996_Laug_Borouchaki}
    \item Anisotropic $\mathcal{M}$-uniform triangulations in 3D: \textsf{Mmg3d} 
      \cite{2012_Dobrzynski_Mmg3d}
    \item Locally uniform anisotropic Delaunay meshes (surface, 2D, and 3D): 
      \textsf{CGALmesh} \cite{2003_Alliez_etal_v22_p485_p493,
      2008_Boissonnat_Wormser_Yvinec_Proc_p270_p277,
      2009_Boissonnat_Pons_Yvinec_LNCS_5416,
      2011_Boissonnat_Wormser_Yvinec}
\end{itemize}
However, the use of these mesh generators in the area of numerical analysis 
and engineering, in particular, to construct mesh restrictions for diffusion-type 
equations to satisfy DCPs, DMPs, and NC is hardly known. Recently, Huang and 
co-workers \cite{2005_Li_Huang_JCP_v229_p8072_p8094,2012_Lu_Huang_Qiu_arXiv_1201_3564,
2012_Huang_arXiv_1306_1987} used \textsf{BAMG} to generate anisotropic simplicial 
meshes to satisfy various discrete properties for linear advection-diffusion-reaction 
equations. But in their research works, the computational domains under consideration 
are \emph{uncomplicated}. 

%======================================;
%  Subsection: Numerical formulations  ;
%======================================;  
\subsection{Strategy - II: Non-negativity, monotone, and monotonicity preserving 
formulations}
The second approach is mainly concerned with developing new innovative 
numerical methodologies based on certain \emph{physical and variational 
principles}, so that they satisfy DCPs, DMPs, and NC. Broadly, these 
methods can be classified into the following three categories:
\begin{itemize}
  \item \emph{Non-negative formulations}: A numerical 
    method belongs to the class of non-negative formulations 
    if the resulting numerical solution satisfies certain 
    DMPs and NC.
  \item \emph{Monotone formulations}: A numerical method is 
    said to be monotone if the resulting numerical solution 
    satisfies certain DMPs, DCPs, and NC.
  \item \emph{Monotonicity preserving formulations}: A 
    numerical formulation is said to be monotonicity preserving 
    if the resulting numerical solution does not exhibit 
    spurious oscillations within itself.
\end{itemize}
It needs to be emphasized that a non-negative formulation need not satisfy 
monotone conditions, a monotone numerical method need not be monotonicity 
preserving, and vice-versa. It is still an open research problem to develop 
a numerical formulation that meets all the aforementioned properties.

Some popular formulations and notable research works in this direction are 
finite difference schemes (FDS) \cite{Ciarlet_AequationesMath_1970_v4_p338,
Varga_SIAMJNA_1966_v3_p355}, mimetic finite difference methods (MFDM) 
\cite{Brezzi_Lipnikov_Shashkov_SIAMJNA_2005_v43_p1872,
2011_Lipnikov_Manzini_Svyatskiy_JCP_v230_p2620_p2642}, 
finite volume methods (FVM) \cite{2009_LePotier_IJFV_v6,
Nordbotten_Aavatsmark_Eigestad_NumerMath_2007_v106_p255,
2011_Droniou_LePotier_SIAMJNA_v49_p459_p490},
and finite element methods (FEM) \cite{Burman_Ern_MathComput_2005_v74_p1637,
2005_Draganescu_etal_MC_v74_p1_p23,Liska_Shashkov_CiCP_2008_v3_p852,
Nakshatrala_Valocchi_JCP_2009_v228_p6726}. It should be noted that most of 
these techniques are inherently non-linear. For example, the optimization-based 
finite element methodologies proposed by Nakshatrala and co-workers 
\cite{Nakshatrala_Valocchi_JCP_2009_v228_p6726,
Nagarajan_Nakshatrala_IJNMF_2011_v67_p820,
Nakshatrala_Nagaraja_Shabouei_arXiv_2012,
2013_Nakshatrala_Mudunuru_Valocchi_JCP_v253_p278_p307} enforce the desirable 
properties as explicit constraints. This is achieved by constructing variationally 
consistent constrained minimization problems for various numerical formulations. 
But one should note that this comes with an additional computational cost. 
However, Nakshatrala \emph{et al.} \cite{2013_Nakshatrala_Mudunuru_Valocchi_JCP_v253_p278_p307}
have shown through numerical experiments that the additional computational cost 
is less than 10\% . 

%=======================================;
%  Subsection: Post-processing methods  ;
%=======================================;  
\subsection{Strategy - III: Post-processing methods}
The third approach is about post-processing (PP) based methods. In literature, 
there are various types of PP methods, which can be used to recover certain 
discrete properties for diffusion-type equations. Some of the notable research 
works in this direction include:
\begin{itemize} 
\item Local and global remapping/repair methods 
  \cite{2003_Kucharik_Shashkov_Wendroff_JCP_v188_p462_p471,
    2007_Garimella_Kucharik_Shashkov_CF_v36_p224_p237}
\item Constrained monotonic regression based methods 
  \cite{2012_Burdakov_etal_JCP_v231_p3126_p3142}
\item Cutoff methods (also known as the clipping 
  methods) \cite{2012_Kreuzer_arXiv_1208_3958,
    2013_Lu_Huang_Vleck_JCP_v242_p24_p36}
\item A combination of remapping/repair methods
  and cutoff methods \cite{2012_Wang_etal_IJNMF_v70_p1188_p1205,
    2013_Zhao_etal_DOI}
\end{itemize}
We shall now give a brief description of the pros and cons of 
these methods. Nevertheless, it is very difficult to apply these 
techniques to recover DCPs, DMPs, and NC for higher-order FEM 
methods, as the shape functions can change their sign within the 
element. In addition, most of the above methods \emph{do not} 
have a variational basis.

The remapping/repair techniques proposed by Shashkov and co-workers 
are designed to improve the quality of the numerical solutions, so that 
they satisfy some discrete properties. Even though these are efficient, 
conservative, linearity and bound preserving interpolation algorithms, 
it should be emphasized that they are mesh dependent. In addition, 
application of such algorithms to anisotropic diffusion equations to 
satisfy DMPs and NC are seldom \cite{2012_Wang_etal_IJNMF_v70_p1188_p1205,
2013_Zhao_etal_DOI}.

The post-processing procedure based on a constrained monotonic regression 
problem proposed by Burdakov \emph{et al.} \cite{2012_Burdakov_etal_JCP_v231_p3126_p3142}
is locally conservative, bound preserving, and monotonicity recovering method.
This is a constrained optimization-based PP method, wherein one needs to 
specify various constraints in order to fulfill certain discrete properties.
It is applicable to FDS, FVM, and FEM. But in the case of FEM, this PP method 
is valid only for linear and multi-linear shape functions. In order to 
construct appropriate constraints for the optimization problem, one needs 
to know a prior information on the lower bounds, upper bounds, and monotonicity 
of the numerical solution for the physical problem. In general, obtaining the
qualitative and quantitative nature of the solution to a given physical 
problem is not always possible. If such information 
on the monotonicity, lower and upper bounds for the numerical solution is 
not available, then this methodology reduces to the standard clipping procedure. 
In addition, one should note that it is not always possible to satisfy DCPs 
using this constrained monotonic regression algorithm. For instance, a 
counterexample similar to the research work by Nakshatrala \emph{et al.} 
(see \cite[Section 4]{2013_Nakshatrala_Mudunuru_Valocchi_JCP_v253_p278_p307}) 
can be constructed to show that it does not satisfy a DCP.

Finally, we would like to emphasize that \emph{a posterior cutoff analysis} 
is a variational crime. In general, this method is neither conservative 
nor satisfies DMPs and DCPs. The primary objective of this method is to 
cut off the values of a numerical solution if it is less than a given number 
(which is the cutoff value). Hence, it is called the cutoff method. In the case 
of highly anisotropic diffusion problems and for distorted meshes, this method 
predicts erroneous numerical results \cite{Nagarajan_Nakshatrala_IJNMF_2011_v67_p820,
2013_Nakshatrala_Mudunuru_Valocchi_JCP_v253_p278_p307}. By specifying the 
cutoff value to be zero, it is always guaranteed to satisfy NC through 
this methodology. In addition, if the nature of the solution is known a prior, 
then one can also prevent undershooting and overshooting of the numerical 
solution by chopping off those values.

%===============================================================;
%  Subsection: Main contributions and an outline of this paper  ;
%===============================================================;  
\subsection{Main contributions and an outline of this paper} 
Herein, we shall focus on the first approach of placing 
restrictions on the mesh to meet desired mathematical 
properties. We shall derive sufficient conditions for 
restrictions on the three-node triangular and four-node 
quadrilateral finite elements to meet comparison principles, 
maximum principles, and the non-negative constraint in the 
case of heterogeneous anisotropic advection-diffusion-reaction 
(ADR) equations. The notable contributions of this paper are 
as follows: 
\begin{enumerate}[(i)]
  \item We provide an in-depth review of various versions 
  of comparison principles, maximum principles, and the 
  non-negative constraint in the continuous setting.
  \item We derive necessary and sufficient conditions 
    on the coefficient (i.e., the ``stiffness") matrix 
    to satisfy discrete weak and strong comparison 
    principles. 
  \item A relationship between various discrete 
    principles within the context of mesh restrictions, 
    numerical formulations, and post-processing methods 
    is presented.
  \item We propose an iterative method to generate simplicial meshes 
    that satisfy discrete properties using open source mesh generators 
    such as \textsf{BAMG}, \textsf{FreeFem++}, and \textsf{Gmsh}.
  \item Different types of non-dimensional quantities are 
    proposed for anisotropic diffusivity, which are variants 
    of the standard P\'{e}clet and Damk\"{o}hler numbers. 
    These quantities are extremely useful in numerical 
    simulations and have not been discussed in the 
    literature.
  \item Lastly, several realistic numerical examples are 
    presented to corroborate the theoretical findings as 
    well as to show the importance of preserving discrete 
    principles.
\end{enumerate}

The remainder of this paper is organized as follows. In Section 
\ref{Sec:S2_MP_Diffusion}, we present the governing equations 
for a general linear second-order elliptic equation and discuss 
associated mathematical principles: comparison principles, maximum principles, 
and the non-negative constraint. Section \ref{Sec:S3_MP_GeneralRemarks} 
provides several important remarks on the continuous and discrete 
properties of elliptic equations. In Section \ref{Sec:S4_MP_Restrictions}, 
we shall derive mesh restrictions for the three-node triangular 
element and the rectangular element to meet the discrete 
versions of maximum principles, comparison principles, and 
the non-negative constraint. Finally, conclusions are drawn 
in Section \ref{Sec:S5_MP_Conclusions}.  

We will denote scalars by lower case English alphabet 
or lower case Greek alphabet (e.g., concentration $c$ 
and density $\rho$). We will make a distinction between 
vectors in the continuum and finite element settings. 
Similarly, a distinction will be made between second-order 
tensors in the continuum setting versus matrices in the 
discrete setting. The continuum vectors are denoted by 
lower case boldface normal letters, and the second-order 
tensors will be denoted using upper case boldface normal 
letters (e.g., vector $\mathbf{x}$ and second-order tensor 
$\mathbf{D}$). In the finite element context, we shall 
denote the vectors using lower case boldface italic 
letters, and the matrices are denoted using upper case 
boldface italic letters (e.g., vector $\boldsymbol{f}$ 
and matrix $\boldsymbol{K}$). Other notational conventions 
adopted in this paper are introduced as needed.

%****************************************;
%                                        ;
%  NAME                                  ;
%    S2_MP_Diffusion.tex                 ;
%                                        ;
%  WRITTEN BY                            ;
%    Maruti Kumar Mudunuru               ;
%    Kalyana Babu Nakshatrala            ;
%                                        ;
%****************************************;
\section[Short title]{LINEAR SECOND-ORDER ELLIPTIC EQUATION 
AND ASSOCIATED MATHEMATICAL PRINCIPLES}
\label{Sec:S2_MP_Diffusion}
Let $\Omega \subset \mathbb{R}^{nd}$ be a open bounded domain, 
where ``$nd$'' denotes the number of spatial dimensions. The 
boundary of the domain is denoted by $\partial \Omega$, which 
is assumed to be piecewise smooth. Mathematically, $\partial 
\Omega := \overline{\Omega} - \Omega$, where a superposed 
bar denotes the set closure. A spatial point is denoted by 
$\mathbf{x} \in \overline{\Omega}$. The gradient and divergence 
operators with respect to $\mathbf{x}$ are, respectively, 
denoted by $\mathrm{grad}[\bullet]$ and $\mathrm{div}[\bullet]$. 
Let $c(\mathbf{x})$ denote the concentration field. We assume 
that Dirichlet boundary condition is prescribed (i.e., the 
concentration is prescribed) on the entire boundary.The 
remainder of this paper deals with the following boundary 
value problem, which is written in terms of a general linear 
second-order differential operator in \emph{divergence form}: 
%---------------------------------;
%  Equation: Governing equations  ;
%---------------------------------;
\begin{subequations}
  \label{Eqn:NN_AD_GE}
  \begin{align}
    \label{Eqn:NN_AD_GE_BE}
    \mathcal{L}[c] &:= -\mathrm{div}\left[\mathbf{D}
    (\mathbf{x}) \mathrm{grad}[c(\mathbf{x})]\right] 
    + \mathbf{v}(\mathbf{x}) \cdot \mathrm{grad}[
    c(\mathbf{x})] + \alpha(\mathbf{x}) c(\mathbf{x}) 
    = f(\mathbf{x}) \quad \mathrm{in} \; \Omega \\
    \label{Eqn:NN_AD_GE_DBC}
    c(\mathbf{x}) &= c^{\mathrm{p}}(\mathbf{x}) \quad 
    \mathrm{on} \; \partial \Omega
  \end{align}
\end{subequations}
where $\mathcal{L}$ denotes the second-order linear differential 
operator, $f(\mathbf{x})$ is the prescribed volumetric source, 
$\alpha(\mathbf{x})$ is the linear reaction coefficient, $\mathbf{v}
(\mathbf{x})$ is the velocity vector field, $\mathbf{D}(\mathbf{x})$ 
is the anisotropic diffusivity tensor, and $c^{\mathrm{p}}(\mathbf{x})$ 
is the prescribed concentration. Physics of the problem demands 
that the diffusivity tensor (which is a second-order tensor) be 
symmetric. That is, 
%--------------------------------------------;
%  Equation: Symmetry of diffusivity tensor  ;
%--------------------------------------------;
\begin{align}
  \label{Eqn:Diffusivity_Symmetric}
  \mathbf{D}^{\mathrm{T}}(\mathbf{x}) = \mathbf{D}
  (\mathbf{x}) \quad \forall \mathbf{x} \in \Omega
\end{align}

%---------------------------------------------------------;
% Remark: Non-divergence form of second-order linear PDE  ;
%---------------------------------------------------------;
\begin{remark}
  In mathematical analysis, the divergence form is a suitable 
  setting for energy methods. However, some studies on maximum 
  principles do employ the non-divergence form, which can be 
  written as follows: 
  %------------------------------------------------;
  %  General non-divergence form for the operator  ; 
  %------------------------------------------------;
  \begin{align}
    \mathcal{L}[c] = \sum \limits_{i,j=1}^{nd} (\mathbf{P})_{ij} 
    \frac{\partial^2 c}{\partial x_i \partial x_j} + \sum 
    \limits_{i=1}^{nd} (\mathbf{q})_{i} \frac{\partial c}{\partial x_i} 
    + r(\mathbf{x}) c
  \end{align}
  where the coefficient $(\mathbf{P})_{ij}$, $(\mathbf{q})_{i}$, 
  and $r(\mathbf{x})$, can be related to the physical quantities 
  such as the diffusivity tensor, velocity field, and linear reaction 
  coefficient. It should be, however, noted that the non-divergence 
  form exists irrespective of differentiability of the diffusivity 
  tensor. If $\mathbf{D}(\mathbf{x})$ is continuously differentiable, 
  then there exists a one-to-one correspondence between the divergence 
  form and the non-divergence form. In such cases, the operator $\mathcal{L}$ 
  in the divergence form given by equation \eqref{Eqn:NN_AD_GE_BE} 
  can be put into the following non-divergence form \cite[Chapter 6]{Evans_PDE}: 
  %--------------------------------------------------------------;
  %  Equation: Non-divergence linear second-order PDE operators  ;
  %--------------------------------------------------------------;
  \begin{align}
    \label{Eqn:General_Linear_Elliptic_Operator}
    \mathcal{L}[c] = -\mathbf{D}(\mathbf{x}) \cdot \mathrm{grad}
    \left[\mathrm{grad}[c(\mathbf{x})]\right] + \left( \mathbf{v}
    (\mathbf{x}) - \mathrm{div} \left[\mathbf{D}(\mathbf{x})\right] 
    \right) \cdot \mathrm{grad}[c(\mathbf{x})] + \alpha(\mathbf{x}) 
    c(\mathbf{x})
  \end{align}
  where we have used the following identity in combination with 
  equation \eqref{Eqn:Diffusivity_Symmetric} to obtain equation 
  \eqref{Eqn:General_Linear_Elliptic_Operator}
  %------------------------------------------------------------------;
  %  Equation: One of the divergence identities (gutin book, pg:30)  ;
  %------------------------------------------------------------------;
  \begin{align}
    \label{Eqn:Divergence_Identity}
    \mathrm{div} \left[\mathbf{D}^{\mathrm{T}}(\mathbf{x}) 
    \mathrm{grad}[c(\mathbf{x}) \right] = \mathbf{D}(\mathbf{x}) 
    \cdot \mathrm{grad} \left[\mathrm{grad}[c(\mathbf{x})]\right] 
    + \mathrm{div} \left[\mathbf{D}(\mathbf{x})\right] 
    \cdot \mathrm{grad}[c(\mathbf{x})]
  \end{align}
\end{remark}

Based on the nature of the coefficients and connectedness of the physical 
domain, different versions of maximum and comparison principles exist in 
the mathematical literature \cite{Evans_PDE,Gilbarg_Trudinger,Pucci_Serrin}. 
As stated earlier in this paper, we shall restrict ourselves to the boundary 
value problem given by the equations \eqref{Eqn:NN_AD_GE_BE}--\eqref{Eqn:NN_AD_GE_DBC}. 
Further analysis pertaining to Neumann boundary conditions and mixed boundary 
conditions within the context of maximum principles, comparison principles, 
and the non-negative constraint is beyond the scope of this paper, and one 
can refer to references \cite{Karatson_Korotov_NumerMath_2005_v99_p669,Pucci_Serrin,
Borsuk_Kondratiev,Pao}.

%---------------------------------------------------------------;
%  Ellipticity notions for a general second-order elliptic PDE  ;
%---------------------------------------------------------------;
We shall say that the operator $\mathcal{L}$ is \emph{elliptic} at a point 
$\mathbf{x} \in \Omega$ if
%----------------------;
%  Equation: Elliptic  ;
%----------------------;
\begin{align}
  \label{Eqn:Diffusivity_PositiveDefinite}
  0 < \lambda_{\mathrm{min}} (\mathbf{x}) 
  \boldsymbol{\xi} \cdot \boldsymbol{\xi} 
  \leq \boldsymbol{\xi} \cdot \mathbf{D}
  (\mathbf{x}) \boldsymbol{\xi}\leq \lambda_
  {\mathrm{max}} (\mathbf{x}) \boldsymbol{\xi} 
  \cdot \boldsymbol{\xi} \quad \forall \boldsymbol{\xi} 
  \in \mathbb{R}^{nd} \backslash \{\boldsymbol{0}\}
\end{align}
where $\lambda_{\mathrm{min}}(\mathbf{x})$ and $\lambda_{\mathrm{max}}(\mathbf{x})$ 
are, respectively, the minimum and maximum eigenvalues of $\mathbf{D}(\mathbf{x})$. 
The operator $\mathcal{L}$ is said to be \emph{strictly elliptic} if there exists 
a constant $\lambda_0$, such that
%-------------------------------;
%  Equation: Strictly elliptic  ;
%-------------------------------;
\begin{align}
  \label{Eqn:Diffusivity_StrictlyElliptic}
  0 < \lambda_0 \leq \lambda_{\mathrm{min}} 
  (\mathbf{x}) \quad \forall \; \mathbf{x} 
  \in \Omega
\end{align}
and \emph{uniformly elliptic} if 
%--------------------------------;
%  Equation: Uniformly elliptic  ;
%--------------------------------;
\begin{align}
  \label{Eqn:Diffusivity_UniformlyElliptic}
  0 < \frac{\lambda_{\mathrm{max}}(\mathbf{x})}{
  \lambda_{\mathrm{min}} (\mathbf{x})} < +\infty 
  \quad \forall \; \mathbf{x} \in \Omega
\end{align}
In the studies on maximum principles, it is common to impose the following 
restrictions on the velocity field $\mathbf{v}(\mathbf{x})$ and the linear 
reaction coefficient $\alpha(\mathbf{x})$:
%-----------------------------------;
%  Equation: Condition on velocity  ;
%-----------------------------------;
\begin{subequations}
\begin{align}
  \label{Eqn:alpha_NeccSuff}
  &\alpha(\mathbf{x}) \geq 0 \quad \forall \; \mathbf{x} \in 
  \Omega \\
  \label{Eqn:alpha_velocity}
  &\alpha(\mathbf{x}) - \frac{1}{2} \mathrm{div}\left[ 
  \mathbf{v}(\mathbf{x}) \right] \geq 0 \quad \forall \; 
  \mathbf{x} \in \Omega \\
  \label{Eqn:velocity_DiffEigenValue}
  &0 \leq \frac{\left|\left(\mathbf{v}(\mathbf{x})\right)_i 
  \right|}{\lambda_{\mathrm{min}}(\mathbf{x})} \leq 
  \beta_0 < +\infty \quad \forall \; \mathbf{x} \in \Omega 
  \quad \mathrm{and} \quad \forall \; i = 1, \cdots, nd
\end{align}
\end{subequations}
where $\beta_0$ is a bounded non-negative constant. If $(\mathbf{D})
_{ij}$ and $(\mathbf{v})_{i}$ are continuous in $\Omega$, then the 
operator $\mathcal{L}$ is uniformly elliptic for any bounded subdomain 
$\Omega^{'} \subset \subset \Omega$ (which means that $\Omega^{'}$ is 
\emph{compactly embedded} in $\Omega$) and the condition given in 
equation \eqref{Eqn:velocity_DiffEigenValue} holds. The restrictions 
given in equation \eqref{Eqn:alpha_velocity} can be relaxed in some 
situations (e.g., see references \cite{2012_Lu_Huang_Qiu_arXiv_1201_3564,
2012_Huang_arXiv_1306_1987}). But the constraint on $\alpha(\mathbf{x})$ 
given by equation \eqref{Eqn:alpha_NeccSuff} cannot be relaxed. If 
$\alpha(\mathbf{x}) < 0$, then equation \eqref{Eqn:NN_AD_GE_BE} is 
referred to as an Helmholtz-type equation, which does not possess a 
maximum principle. From the theory of partial differential equations, 
it is well-known that the aforementioned boundary value problem given 
by equations \eqref{Eqn:NN_AD_GE_BE}--\eqref{Eqn:NN_AD_GE_DBC} satisfies 
the so-called (weak and strong) comparison principles, (weak and strong) 
maximum principles, and the non-negative constraint. For future reference 
and for completeness, we shall briefly outline the main results. For a 
more detailed mathematical treatment, one could consult references 
\cite{Evans_PDE,Pao,Gilbarg_Trudinger}.

%---------------------------------------------------------------;
%  Theorem: Continuous weak and strict weak maximum principles  ;
%---------------------------------------------------------------;
\begin{theorem}[\texttt{Continuous weak and strict weak maximum principles}]
  \label{Thm:Continuous_WeakMaximum_Principle}
  Let $\mathcal{L}$ be a uniformly elliptic operator satisfying the 
  conditions given by equations 
  \eqref{Eqn:alpha_NeccSuff}--\eqref{Eqn:velocity_DiffEigenValue}. 
  In addition, let $\mathbf{D}(\mathbf{x})$ be continuously differentiable. 
  Suppose that $c(\mathbf{x}) \in C^{2}(\Omega) \cap C^{0}(\overline{\Omega})$ 
  satisfies the differential inequality $\mathcal{L}[c] \leq 0$ in $\Omega$,
  %----------------------------;
  %  A weak maximum principle  ;
  %----------------------------;
  then the maximum of $c(\mathbf{x})$ in $\overline{\Omega}$ is obtained 
  on $\partial \Omega$. That is, $c(\mathbf{x})$ possesses the weak maximum 
  principle ($\mathrm{wMP}$), which can be written as follows: 
  %-----------------;
  %  Equation: wMP  ;
  %-----------------; 
  \begin{align}
    \label{Eqn:Weak_Maximum_Principle}
    \mathop{\mathrm{max}}_{\mathbf{x} \in \overline{\Omega}} 
    \left [ c(\mathbf{x}) \right ] \leq \mathop{\mathrm{max}} 
    \left [ 0, \mathop{\mathrm{max}}_{\mathbf{x} \in \partial 
    \Omega} \left [ c(\mathbf{x}) \right ] \right ]
  \end{align}
   %%
   %-----------------------------------;
   %  A strict weak maximum principle  ;
   %-----------------------------------;
   Moreover, if $\alpha(\mathbf{x}) = 0$, then we have the strict weak maximum 
   principle ($\mathrm{WMP}$):
   %-----------------;
   %  Equation: WMP  ;
   %-----------------; 
   \begin{align}
     \label{Eqn:StrictWeak_Maximum_Principle}
     \mathop{\mathrm{max}}_{\mathbf{x} \in \overline{\Omega}} 
     \left [ c(\mathbf{x}) \right ] = \mathop{\mathrm{max}}_
     {\mathbf{x} \in \partial \Omega} \left [ c(\mathbf{x}) 
     \right ]
   \end{align}
\end{theorem}
\begin{proof} 
  For a proof, see references \cite{Gilbarg_Trudinger,Evans_PDE}. 
\end{proof}

%-------------------------------------------------------------------;
%  Theorem: Continuous strong and strict strong maximum principles  ;
%-------------------------------------------------------------------;
\begin{theorem}[\texttt{Continuous strong and strict strong maximum principles}]
  \label{Thm:Continuous_StrongMaximum_Principle}
   %------------------------------;
   %  A strong maximum principle  ;
   %------------------------------;
   Let the domain $\Omega$ be simply connected. Given that 
   $c(\mathbf{x})$ satisfies $\mathrm{wMP}$ and the conditions 
   given in Theorem \ref{Thm:Continuous_WeakMaximum_Principle}, 
   then $c(\mathbf{x})$ cannot attain an interior non-negative 
   maximum in $\overline{\Omega}$ unless it is a constant. This 
   means that, $c(\mathbf{x})$ possesses the strong maximum 
   principle ($\mathrm{sMP}$) if the following hold:
   %-----------------;
   %  Equation: sMP  ;
   %-----------------; 
   \begin{align}
     \label{Eqn:Strong_Maximum_Principle}
      \mathop{\mathrm{max}}_{\mathbf{x} \in \Omega} 
      \left [ c(\mathbf{x}) \right ] = \mathop{\mathrm{max}}_
      {\mathbf{x} \in \overline{\Omega}} \left [ c(\mathbf{x}) 
      \right ] = m \geq 0 \qquad \Rightarrow \qquad c(\mathbf{x}) 
      \equiv m \quad \mathrm{in} \; \overline{\Omega}  
   \end{align}
   %%
   %---------------------------------------;
   %  A strictly strong maximum principle  ;
   %---------------------------------------;
   Moreover, if $\alpha(\mathbf{x}) = 0$ and $c(\mathbf{x})$ 
   satisfies $\mathrm{WMP}$, then we have the strict strong 
   maximum principle ($\mathrm{SMP}$) given as follows:
   %-----------------;
   %  Equation: SMP  ;
   %-----------------;
   \begin{align}
     \label{Eqn:StrictStrong_Maximum_Principle}
     \mathop{\mathrm{max}}_{\mathbf{x} \in \Omega} 
     \left [ c(\mathbf{x}) \right ] = \mathop{\mathrm{max}}_
     {\mathbf{x} \in \overline{\Omega}} \left [ c(\mathbf{x}) 
     \right ] = m \qquad \Rightarrow \qquad c(\mathbf{x}) 
     \equiv m \quad \mathrm{in} \; \overline{\Omega}  
   \end{align}        
\end{theorem}
\begin{proof} For a proof, see references 
\cite{Gilbarg_Trudinger,Evans_PDE}. 
\end{proof}

%-------------------------------------------------------------;
%  Theorem: Continuous weak and strong comparison principles  ;
%-------------------------------------------------------------;
\begin{theorem}[\texttt{Continuous weak and strong comparison principles}]
  \label{Thm:Continuous_Comparison_Principle}
  Let $c_1(\mathbf{x}), \; c_2(\mathbf{x}) \in C^{2}(\Omega) \cap 
  C^{0}(\overline{\Omega})$. Suppose $\mathcal{L}$ be a uniformly 
  elliptic operator satisfying the conditions given by the equations 
  \eqref{Eqn:alpha_NeccSuff}--\eqref{Eqn:velocity_DiffEigenValue}. 
  Then $\mathcal{L}$ is said to possess
  \begin{itemize}
  %-------------------------------;
  %  A weak comparison principle  ;
  %-------------------------------;
  \item the weak comparison principle ($\mathrm{wCP}$) 
    if $c_1(\mathbf{x})$ and $ c_2(\mathbf{x})$ satisfies 
    $\mathrm{wMP}$, $\mathcal{L}[c_1] \leq \mathcal{L}[c_2]$ 
    in $\Omega$, and $c_1(\mathbf{x}) \leq c_2(\mathbf{x})$ 
    on $\partial \Omega$, then the following holds:
    %------------------------;
    %  Equation: wCP result  ;
    %------------------------; 
    \begin{align}
      \label{Eqn:wCP_Result}
      c_1(\mathbf{x}) \leq c_2(\mathbf{x}) \quad 
      \forall \mathbf{x} \in \overline{\Omega}
    \end{align}
    %%
  %---------------------------------;
  %  A strong comparison principle  ;
  %---------------------------------;
  \item the strong comparison principle ($\mathrm{sCP}$) 
    if $c_1(\mathbf{x})$ and $ c_2(\mathbf{x})$ satisfies 
    $\mathrm{sMP}$, $\mathcal{L}[c_1] < \mathcal{L}[c_2]$ 
    in $\Omega$, and $c_1(\mathbf{x}) \leq c_2(\mathbf{x})$ 
    on $\partial \Omega$, then the following holds:
    %-----------------------;
    %  Equation: sCP result ;
    %-----------------------; 
    \begin{align}
      \label{Eqn:sCP_Result}
      c_1(\mathbf{x}) < c_2(\mathbf{x}) \quad 
      \forall \mathbf{x} \in \Omega
    \end{align}
  \end{itemize}
\end{theorem}
\begin{proof}
  For proof, see reference \cite{Gilbarg_Trudinger}.
\end{proof}

Numerical formulations based on the finite element method, finite volume 
method, and finite difference method exist to solve the boundary value 
problem given by the equations \eqref{Eqn:NN_AD_GE_BE}--\eqref{Eqn:NN_AD_GE_DBC}.
It is well-known that the framework offered by the finite element method 
is particularly attractive in obtaining accurate numerical results for 
elliptic partial differential equations. In particular, the single-field 
Galerkin formulation is a very popular finite element formulation. In 
this paper, we shall use the single-field Galerkin formulation to derive 
mesh restrictions. It should be, however, noted that restrictions imposed 
on a mesh may alter if an alternate numerical formulation is employed. But 
the overall procedure presented in this paper can be employed to derive mesh 
restrictions for other numerical formulations.

%=================================================;
%  Subsection: Single-field Galerkin formulation  ;
%=================================================;
\subsection{Single-field Galerkin formulation}
Let us define the following function spaces: 
%-----------------------------;
%  Equation: Function spaces  ;
%-----------------------------;
\begin{subequations}
  \begin{align}
    \mathcal{C} &:= \left\{c(\mathbf{x}) \in H^{1}(\Omega) \; 
    \big| \; c(\mathbf{x}) = c^{\mathrm{p}}(\mathbf{x}) \; 
    \mathrm{on} \; \partial \Omega \right\} \\
    \mathcal{W} &:= \left\{w(\mathbf{x}) \in H^{1}(\Omega) \; 
    \big| \; w(\mathbf{x}) = 0 \; \mathrm{on} \; \partial \Omega 
    \right\} 
  \end{align}
\end{subequations}
where $H^{1}(\Omega)$ is a standard Sobolev space 
\cite{Evans_PDE}. For weak solutions, the regularity 
of the diffusivity tensor can be relaxed as follows:
%------------------------------------;
%  Equation: D is square integrable  ;
%------------------------------------;
\begin{align}
  \int \limits_{\Omega} \mathrm{tr}\left[
    \mathbf{D}(\mathbf{x})^{\mathrm{T}} 
    \mathbf{D}(\mathbf{x})\right] \; 
  \mathrm{d} \Omega < +\infty
\end{align}
where $\mathrm{tr}[\bullet]$ is the standard trace operator 
used in tensor algebra and continuum mechanics \cite{Gurtin}.
Given two fields $a(\mathbf{x})$ and $b(\mathbf{x})$ on a 
set $\mathcal{D}$, the standard $L_2$ inner-product over 
$\mathcal{D}$ will be denoted as follows:
%------------------------------;
%  Equation: L2 inner product  ;
%------------------------------;
\begin{align}
  \left(a;b\right)_{\mathcal{D}} = \int \limits_{\mathcal{D}} 
  a(\mathbf{x}) \cdot b(\mathbf{x}) \; \mathrm{d} \mathcal{D}
\end{align}
The subscript on the inner-product will be dropped if 
$\mathcal{D} = \Omega$. The single-field Galerkin 
formulation for the boundary value problem given by 
equations \eqref{Eqn:NN_AD_GE_BE}--\eqref{Eqn:NN_AD_GE_DBC}
can be written as follows: Find $c(\mathbf{x}) \in \mathcal{C}$, 
such that we have 
%--------------------------------------;
%  Equation: Single field formulation  ;
%--------------------------------------;
\begin{align}
  \label{Eqn:single-field_Galerkin_Formulation}
  \mathcal{B}(w;c) = L(w) \quad \forall \; 
  w(\mathbf{x}) \in \mathcal{W}
\end{align}
where the bilinear form and the linear functional 
are, respectively, defined as follows:
%---------------------------------------------;
%  Equation: Bilinear and linear functionals  ;
%---------------------------------------------;
\begin{subequations}
  \begin{align}
    \label{Eqn:Bilinear_Form}
    \mathcal{B}(w;c) &:= \left(w; \alpha(\mathbf{x}) c \right) 
    + \left(w;\mathbf{v}(\mathbf{x}) \cdot \mathrm{grad}[c]\right) 
    + \left(\mathrm{grad}[w] ; \mathbf{D}(\mathbf{x}) \mathrm{grad}[c] 
    \right) \\
    \label{Eqn:Linear_Form}
    L(w) &:= \left(w;f(\mathbf{x})\right)
  \end{align}
\end{subequations}

%==========================================================;
%  Subsection: Discrete single-field Galerkin formulation  ;
%==========================================================;
\subsection{Discrete single-field Galerkin formulation}
\label{SubSec:Discrete_Formulation}
Let the computational domain $\Omega$ be decomposed into 
``$Nele$'' non-overlapping open sub-domains, which in the 
finite element context will be elements. That is,
%-----------------------------------------;
%  Equation: Decomposition into elements  ;
%-----------------------------------------;
\begin{align}
  \label{Eqn:Omega_To_Omega_e}
  \overline{\Omega} = \bigcup_{e = 1}^{Nele} \overline{\Omega}^{e}
\end{align}
The boundary of $\Omega^e$ is denoted as $\partial \Omega^{e} 
:= \overline{\Omega}^{e} - \Omega^{e}$. Let $\mathbb{P}^{1}
(\Omega^{e})$ denote the vector space spanned by linear 
polynomials on the sub-domain $\Omega^{e}$. We shall define 
the following finite dimensional subsets of $\mathcal{C}$ and 
$\mathcal{W}$: 
%--------------------------------------------------------;
%  Equation: Finite dimensional subspaces for diffusion  ;
%--------------------------------------------------------;
\begin{subequations}
  \begin{align}
    \label{Eqn:Ch}
    \mathcal{C}^{h} &:= \left\{c^{h}(\mathbf{x}) \in \mathcal{C} \; 
    \big| \; c^{h}(\mathbf{x})  \in C^{0}(\bar{\Omega}); \, c^{h}
    (\mathbf{x}) \big|_{\Omega^e} \in \mathbb{P}^{1}(\Omega^{e}); \, 
    e = 1, \cdots, Nele \right\} \\
    \label{Eqn:Wh} 
    \mathcal{W}^{h} &:= \left\{w^{h}(\mathbf{x}) \in \mathcal{W} 
    \; \big| \; w^{h}(\mathbf{x}) \in C^{0}(\bar{\Omega}); \, w^{h}
    (\mathbf{x}) \big|_{\Omega^e} \in \mathbb{P}^{1}(\Omega^{e});  \, 
    e = 1, \cdots, Nele \right\}       
  \end{align}
\end{subequations}
A corresponding finite element formulation can be written 
as follows: Find $c^{h}(\mathbf{x}) \in \mathcal{C}^{h}$, 
such that we have
%--------------------------------------------------------;
%  Equation: Discrete single-field Galerkin formulation  ;
%--------------------------------------------------------;
\begin{align}
  \label{Eqn:Discrete_Formulation}
  \mathcal{B}(w^{h};c^{h}) = L(w^{h}) \quad \forall \; 
  w^{h}(\mathbf{x}) \in \mathcal{W}^{h}
\end{align}
where $\mathcal{B}(w^{h};c^{h})$ and $L(w^{h})$ are, 
respectively, given as follows:
%------------------------------------------------------;
%  Equation: Discrete bilinear and linear functionals  ;
%------------------------------------------------------;
\begin{subequations}
\begin{align}
  \label{Eqn:Discrete_Bilinear_Form}
  \mathcal{B}(w^{h};c^{h}) &:= \left(w^{h}; \alpha(\mathbf{x}) 
  c^{h} \right) + \left(w^{h};\mathbf{v}(\mathbf{x}) \cdot 
  \mathrm{grad}[c^{h}]\right) + \left(\mathrm{grad}[w^{h}] ; 
  \mathbf{D}(\mathbf{x}) \mathrm{grad}[c^{h}] \right) \\
  \label{Eqn:Discrete_Linear_Form}
  L(w^{h}) &:= \left(w^{h};f(\mathbf{x})\right)
\end{align}
\end{subequations}

Let ``$n_t$'' denote the total number of degrees-of-freedom, 
``$n_f$'' denote the free degrees-of-freedom, and ``$n_p$'' 
be the prescribed degrees-of-freedom for the concentration 
vector. Obviously, we have $n_t = n_f + n_p$. We assume 
that $n_t,n_p \geq 2$. After finite element discretization, 
the discrete equations for the boundary value problem take 
the following form:
%--------------------------------------------------------;
%  Equation: Discretization equation for Dirichilet BVP  ;
%--------------------------------------------------------;
\begin{align}
  \label{Eqn:Discrete_Diffusion}
  \boldsymbol{K}\boldsymbol{c} = \boldsymbol{r}
\end{align}
where $\boldsymbol{K} \equiv \left [ \boldsymbol{K}_{ff} 
| \boldsymbol{K}_{fp} \right]$ is the stiffness matrix, 
$\boldsymbol{c} \equiv \left [\boldsymbol{c}_{f}^{\mathrm{T}} 
\; | \; \boldsymbol{c}_{p}^{\mathrm{T}} \right]^{\mathrm{T}}$ is the vector 
containing nodal concentration, and  $\boldsymbol{r} = \left [ 
\boldsymbol{r}_{f} \right]^{\mathrm{T}}$ is the corresponding nodal 
volumetric source vector.  The stiffness matrices $\boldsymbol{K}$, 
$\boldsymbol{K}_{ff}$, and $\boldsymbol{K}_{fp}$ are, respectively, 
of size $n_f \times n_t$, $n_f \times n_f$, and $n_f \times n_p$. 
Correspondingly, the nodal 
concentration vectors $\boldsymbol{c}$, $\boldsymbol{c}_f$, 
and $\boldsymbol{c}_{p}$ are of sizes $n_t \times 1$, $n_f 
\times 1$, and $n_p \times 1$. Similar inference is 
applicable to the load vector $\boldsymbol{r}$.

Before we state a discrete version of (weak and strong) 
maximum and comparison principles, we introduce the 
required notation. The symbols $\preceq$ and $\succeq$ 
shall denote component-wise inequalities for vectors 
and matrices. That is, given two (finite dimensional) 
vectors $\boldsymbol{a}$ and $\boldsymbol{b}$
%-----------------------------------------;
%  Equation: Definition of preceq symbol  ;
%-----------------------------------------;
\begin{align}
  \boldsymbol{a} \preceq \boldsymbol{b} \quad \mbox{means that} 
  \quad a_i \leq b_i \; \forall \; i
\end{align}
Correspondingly, given two matrices 
$\boldsymbol{A}$ and $\boldsymbol{B}$
%-----------------------------------------;
%  Equation: Definition of preceq symbol  ;
%-----------------------------------------;
\begin{align}
  \boldsymbol{A} \preceq \boldsymbol{B} \quad \mbox{means that} 
  \quad \left(\boldsymbol{A} \right)_{ij} \leq \left(\boldsymbol{B} 
  \right)_{ij} \; \forall \; i,j
\end{align}
Similarly, one can define the symbol $\succeq$, 
$\prec$, and $\succ$. In the remainder of this 
paper, we will be frequently using the symbols 
$\boldsymbol{0}$ and $\boldsymbol{O}$, which, 
respectively, denote a zero vector and a zero 
matrix.

We shall now briefly outline the main results corresponding 
to the discrete weak and strong maximum principles in the 
form of definitions and theorems. Using these results, 
we shall discuss in detail about discrete comparison 
principles. However, it should be noted that Theorem 
\ref{Thm:DCP_Conditions_Galerkin} and its proof are 
new and have not been discussed elsewhere. We shall 
also present the necessary and sufficient conditions on the 
stiffness matrices $\boldsymbol{K}_{ff}$ and $\boldsymbol{K}_{fp}$ 
to satisfy different versions of discrete maximum principles 
and the non-negative constraint. For more details, see references 
\cite{1987_Ishihara_LAA_v88_p431_p448,2012_Mincsovics_Hovarth_LSSC_p614_p621}.

%----------------------------------------------------------------------------------;
%  Theorem: Discrete weak/strict weak and strong/strict strong maximum principles  ;
%----------------------------------------------------------------------------------;
\begin{definition}[\texttt{Discrete maximum principles}
  \cite{2012_Mincsovics_Hovarth_LSSC_p614_p621}]
  \label{Thm:Discrete_WeakStrongMaximum_Principle}
  A numerical formulation is said to possess 
  \begin{itemize}
    \item the discrete weak maximum principle 
      ($\mathrm{DwMP}$) if
      %------------------;
      %  Equation: DwMP  ;
      %------------------; 
      \begin{align}
        \label{Eqn:DwMP}
        \boldsymbol{r} \preceq \boldsymbol{0} 
        \qquad \mathrm{implies} \qquad \mathop{\mathrm{max}} 
        \left[ \boldsymbol{c} \right] \leq \mathop{\mathrm{max}} 
        \left[ 0, \mathop{\mathrm{max}} \left [ \boldsymbol{c}_{p} 
        \right ] \right]
      \end{align}
    \item the discrete strict weak maximum principle 
      ($\mathrm{DWMP}$) if
      %------------------;
      %  Equation: DWMP  ;
      %------------------; 
      \begin{align}
      \label{Eqn:DWMP}
        \boldsymbol{r} \preceq \boldsymbol{0} 
        \qquad \mathrm{implies} \qquad \mathop{\mathrm{max}} 
        \left[ \boldsymbol{c} \right] = \mathop{\mathrm{max}} 
        \left[ \boldsymbol{c}_{p} \right]
      \end{align}
    \item the discrete strong maximum principle ($\mathrm{DsMP}$) 
      if it possesses $\mathrm{DwMP}$ and satisfies the following 
      condition: 
      %------------------;
      %  Equation: DsMP  ;
      %------------------; 
      \begin{align}
      \label{Eqn:DsMP}
        \boldsymbol{r} \preceq \boldsymbol{0}, 
        \quad \mbox{and} \quad \mathop{\mathrm{max}} 
        \left[ \boldsymbol{c} \right] = \mathop{\mathrm{max}} 
        \left[ \boldsymbol{c}_{f} \right] = m \geq 0 \qquad 
        \mathrm{implies} \qquad \boldsymbol{c} = m \boldsymbol{1}
      \end{align}
    \item the discrete strict strong maximum principle 
      ($\mathrm{DSMP}$) if it possesses $\mathrm{DWMP}$ 
      and satisfies the following condition: 
      %------------------;
      %  Equation: DSMP  ;
      %------------------; 
      \begin{align}
      \label{Eqn:DSMP}
        \boldsymbol{r} \preceq \boldsymbol{0}, 
        \quad \mbox{and} \quad \mathop{\mathrm{max}} 
        \left[ \boldsymbol{c} \right] = \mathop{\mathrm{max}} 
        \left[ \boldsymbol{c}_{f} \right] = m \qquad 
        \mathrm{implies} \qquad \boldsymbol{c} = m \boldsymbol{1}
      \end{align}
  \end{itemize}
  where $\mathop{\mathrm{max}} \left[ \bullet \right]$ 
  denotes the maximal element of a vector and the symbol 
  $\boldsymbol{1}$ is the vector whose components are all 
  equal to 1. 
\end{definition}

%------------------------------------------------------------;
%  Theorem: Necessary and sufficient conditions for the DMP  ;
%------------------------------------------------------------;
\begin{theorem}[\texttt{Necessary and sufficient conditions to satisfy DMPs}]
  \label{Thm:NeccSuff_Conditions_DMP}
  The stiffness matrix $\boldsymbol{K}$ given by equation 
  \eqref{Eqn:Discrete_Diffusion} is said to possess
  \begin{itemize}
    \item the discrete weak maximum principle 
      ($\mathrm{DwMP}_{\boldsymbol{K}}$) if and 
      only if all of the following conditions 
      are satisfied:
      %------------------;
      %  DwMP conditions ;
      %------------------; 
      \begin{align}
      \label{Eqn:StiffMat_DwMP}
        (\mathrm{a})\boldsymbol{K}_{ff}^{-1} \succeq \boldsymbol{O} 
        \qquad (\mathrm{b})-\boldsymbol{K}_{ff}^{-1} \boldsymbol{K}_{fp} 
        \succeq \boldsymbol{O} \qquad (\mathrm{c})-\boldsymbol{K}_{ff}^{-1} 
        \boldsymbol{K}_{fp} \boldsymbol{1} \preceq \boldsymbol{1}
      \end{align}
    \item the discrete strict weak maximum principle 
      ($\mathrm{DWMP}_{\boldsymbol{K}}$) if and only 
      if all of the following conditions are 
      satisfied:
      %------------------;
      %  DWMP conditions ;
      %------------------; 
      \begin{align}
      \label{Eqn:StiffMat_DWMP}
        (\mathrm{a})\boldsymbol{K}_{ff}^{-1} \succeq \boldsymbol{O} 
        \qquad (\mathrm{b})-\boldsymbol{K}_{ff}^{-1} \boldsymbol{K}_{fp} 
        \succeq \boldsymbol{O} \qquad (\mathrm{c})-\boldsymbol{K}_{ff}^{-1} 
        \boldsymbol{K}_{fp} \boldsymbol{1} = \boldsymbol{1}
      \end{align}
    \item the discrete strong maximum principle 
      ($\mathrm{DsMP}_{\boldsymbol{K}}$) if and 
      only if all of the following conditions 
      are satisfied:
      %------------------;
      %  DsMP conditions ;
      %------------------; 
      \begin{align}
      \label{Eqn:StiffMat_DsMP}
        (\mathrm{a})\boldsymbol{K}_{ff}^{-1} \succ \boldsymbol{O} 
        \qquad (\mathrm{b})-\boldsymbol{K}_{ff}^{-1} \boldsymbol{K}_{fp} 
        \succ \boldsymbol{O} \qquad (\mathrm{c})-\boldsymbol{K}_{ff}^{-1} 
        \boldsymbol{K}_{fp} \boldsymbol{1} \prec \boldsymbol{1} \quad 
        \mathrm{or} \; \; -\boldsymbol{K}_{ff}^{-1} \boldsymbol{K}_{fp} 
        \boldsymbol{1} = \boldsymbol{1}
      \end{align}
    \item the discrete strict strong maximum principle 
      ($\mathrm{DSMP}_{\boldsymbol{K}}$) if and only if 
      all of the following conditions are satisfied: 
      %------------------;
      %  DsMP conditions ;
      %------------------; 
      \begin{align}
      \label{Eqn:StiffMat_DSMP}
        (\mathrm{a})\boldsymbol{K}_{ff}^{-1} \succ \boldsymbol{O} 
        \qquad (\mathrm{b})-\boldsymbol{K}_{ff}^{-1} \boldsymbol{K}_{fp} 
        \succ \boldsymbol{O} \qquad (\mathrm{c})-\boldsymbol{K}_{ff}^{-1} 
        \boldsymbol{K}_{fp} \boldsymbol{1} = \boldsymbol{1}
      \end{align}
    \end{itemize}  
\end{theorem}
\begin{proof}
  For a proof, see reference 
  \cite{2012_Mincsovics_Hovarth_LSSC_p614_p621}.
\end{proof}

%-----------------------------------------------------------;
%  Theorem: Discrete weak and strong comparison principles  ;
%-----------------------------------------------------------;
\begin{definition}[\texttt{Discrete weak and strong comparison principles}]
  \label{Def:Discrete_Comparison_Principle}
  A numerical formulation is said to possess
  \begin{itemize}
    \item the \emph{discrete weak comparison principle} 
      ($\mathrm{DwCP}$) if it satisfies $\mathrm{DwMP}$, 
      and 
      %------------------;
      %  Equation: DwCP  ;
      %------------------; 
      \begin{align}
        \boldsymbol{c}_{1} \preceq \boldsymbol{c}_{2} \; 
        \mathrm{on} \; \partial \Omega 
        \quad \mathrm{and} \quad
        \boldsymbol{r}_1 \preceq \boldsymbol{r}_2 
        \; \; \mathrm{in} \; \Omega 
        \qquad \mathrm{implies} \qquad 
        \boldsymbol{c}_1 \preceq \boldsymbol{c}_2 
        \; \; \mathrm{in} \; \overline{\Omega}
      \end{align}
    \item the \emph{discrete strong comparison principle} 
      ($\mathrm{DsCP}$) if it satisfies $\mathrm{DsMP}$, 
      and 
      %------------------;
      %  Equation: DsCP  ;
      %------------------; 
      \begin{align}
        \boldsymbol{c}_{1} \preceq \boldsymbol{c}_{2} \; 
        \mathrm{on} \; \partial \Omega 
        \quad \mathrm{and} \quad
        \boldsymbol{r}_1 \prec \boldsymbol{r}_2 
        \; \; \mathrm{in} \; \Omega 
        \qquad \mathrm{implies} \qquad 
        \boldsymbol{c}_1 \prec \boldsymbol{c}_2 
        \; \; \mathrm{in} \; \Omega
      \end{align}
  \end{itemize}   
\end{definition}

%----------------------------------------------------------;
%  Theorem: Necessary and sufficient conditions to satisfy ;
%           discrete weak and strong comparison principle  ;
%----------------------------------------------------------;
\begin{theorem}[\texttt{Necessary and sufficient conditions to satisfy DCPs}]
  \label{Thm:DCP_Conditions_Galerkin}
  Let $\boldsymbol{c}_1$ and $\boldsymbol{c}_2$ be 
  two nodal concentration vectors corresponding to 
  the volumetric source vectors $\boldsymbol{r}_1$ 
  and $\boldsymbol{r}_2$ based on the equation 
  \eqref{Eqn:Discrete_Diffusion}. 
  If $\boldsymbol{c}_1$ and $\boldsymbol{c}_2$ satisfy 
  $\mathrm{DwMP}$ and the hypothesis of $\mathrm{DwCP}$ 
  (i.e., $\boldsymbol{c}_{1} \preceq \boldsymbol{c}_{2}$ 
  on $\partial \Omega$ and $\boldsymbol{r}_1 \preceq 
  \boldsymbol{r}_2$ in $\Omega$), then a necessary and 
  sufficient condition to satisfy the discrete weak 
  comparison principle ($\mathrm{DwCP}_{\boldsymbol{K}}$) 
  (which means that $\boldsymbol{c}_1 \preceq \boldsymbol{c}_2 
  \; \mathrm{in} \; \overline{\Omega}$)
   is that the stiffness matrix $\boldsymbol{K}$ possess 
   $\mathrm{DwMP}_{\boldsymbol{K}}$ (which is given by 
   equation \eqref{Eqn:StiffMat_DwMP} in Theorem 
   \ref{Thm:NeccSuff_Conditions_DMP}). 

   If $\boldsymbol{c}_1$ and $\boldsymbol{c}_2$ satisfy 
   $\mathrm{DsMP}$ and the hypothesis of $\mathrm{DsCP}$, 
   (i.e., $\boldsymbol{c}_{1} \preceq \boldsymbol{c}_{2}$ 
   on $\partial \Omega$ and $\boldsymbol{r}_1 \prec 
   \boldsymbol{r}_2$ in $\Omega$), then a necessary 
   and sufficient condition to satisfy the discrete 
   strong comparison principle ($\mathrm{DsCP}_{\boldsymbol{K}}$) 
   (which means that $\boldsymbol{c}_1 \prec \boldsymbol{c}_2 
     \; \mathrm{in} \; \Omega$)
   is that the stiffness matrix $\boldsymbol{K}$ possess 
   $\mathrm{DsMP}_{\boldsymbol{K}}$ (which is given by 
   equation \eqref{Eqn:StiffMat_DsMP} in Theorem 
   \ref{Thm:NeccSuff_Conditions_DMP}).
\end{theorem}
\begin{proof} 
  For convenience, let us define the following: 
  %------------------------------------------;
  %  Equation: Difference vectors r3 and c3  ;
  %------------------------------------------;
  \begin{subequations}
  \begin{align}
    \label{Eqn:c3}
    \boldsymbol{c}_3 &:= \boldsymbol{c}_1 - \boldsymbol{c}_2 \\
    \label{Eqn:r3}
    \boldsymbol{r}_3 &:= \boldsymbol{r}_1 - \boldsymbol{r}_2
  \end{align}
  \end{subequations}
  Clearly, $\boldsymbol{c}_3$ and $\boldsymbol{r}_3$ 
  satisfy the following:
  %--------------------------;
  %  Equation: K(c_3) = r_3  ;
  %--------------------------; 
  \begin{align}
    \label{Eqn:Kc3r3}
    \boldsymbol{K} \boldsymbol{c}_3 = \boldsymbol{r}_3
  \end{align}
    
  %%**********************************************************%%
  %%          Necessary condition proof for DwCP_K            %%
  %%**********************************************************%%
  \emph{Necessary condition to satisfy $\mathrm{DwCP}_{\boldsymbol{K}}:$}
  Let $\boldsymbol{c}_1 \preceq \boldsymbol{c}_2$ in 
  $\overline{\Omega}$, which implies that $\boldsymbol{c}_3 
  \preceq \boldsymbol{0}$ in $\overline{\Omega}$. The 
  hypothesis of $\mathrm{DwCP}_{\boldsymbol{K}}$ and the fact 
  that $\boldsymbol{c}_3 \preceq \boldsymbol{0}$ in 
  $\overline{\Omega}$ imply the following:
  %------------------------------------;
  %  Equation: Result-0 wrt to DwCP_K  ;
  %------------------------------------; 
  \begin{subequations}
    \begin{align}
      \label{Eqn:DwCP_Proof_Result01}
      \boldsymbol{r}_3 &\preceq \boldsymbol{0} 
      \quad \mathrm{in} \; \Omega \\
      \label{Eqn:DwCP_Proof_Result02}
      \boldsymbol{c}_3 &\preceq \boldsymbol{0} 
      \quad \mathrm{on} \; \partial 
      \Omega\\
      \label{Eqn:DwCP_Proof_Result03}
      \mathop{\mathrm{max}} \left[ \boldsymbol{c}_3 
      \right] &\leq 0 \quad \mathrm{on} \; 
      \partial \Omega \\
      \label{Eqn:DwCP_Proof_Result03}
      \mathop{\mathrm{max}}_{\overline{\Omega}} 
      \left[ \boldsymbol{c}_3 \right] &\leq 
      \mathop{\mathrm{max}} \left[ 0, \mathop{\mathrm{max}}
      _{\partial \Omega} \left[ \boldsymbol{c}_{3} \right] 
      \right] = 0
    \end{align}
  \end{subequations}
  which implies that $\boldsymbol{c}_3$ satisfies $\mathrm{DwMP}$ 
  (based on equation \eqref{Eqn:DwMP} in Definition 
  \ref{Thm:Discrete_WeakStrongMaximum_Principle}). But vector 
  $\boldsymbol{c}_3$ also satisfies equation \eqref{Eqn:Kc3r3}. 
  Hence, according to equation \eqref{Eqn:StiffMat_DwMP} 
  and the hypothesis of Theorem \ref{Thm:NeccSuff_Conditions_DMP}, 
  it is evident that $\boldsymbol{K}$ must possess $\mathrm{DwMP}_
  {\boldsymbol{K}}$. This completes the proof for the necessary 
  condition to satisfy $\mathrm{DwCP}_{\boldsymbol{K}}$.

  %%***********************************************************%%
  %%          Sufficient condition proof for DwCP_K            %%
  %%***********************************************************%%
  \emph{Sufficient condition to satisfy $\mathrm{DwCP}_{\boldsymbol{K}}:$}
  It is given that $\boldsymbol{c}_1$ and $\boldsymbol{c}_2$ 
  satisfy $\mathrm{DwMP}$. Equations \eqref{Eqn:c3}--\eqref{Eqn:r3}
  and $\mathrm{DwCP}_{\boldsymbol{K}}$ imply that 
  %------------------------------------------;
  %  Equation: Result-1 and 2 wrt to DwCP_K  ;
  %------------------------------------------; 
  \begin{subequations}
    \begin{align}
      \label{Eqn:DwCP_Proof_Result1}
      \boldsymbol{c}_3 &\preceq \boldsymbol{0} 
      \quad \mathrm{on} \; \partial 
      \Omega \\
        \label{Eqn:DwCP_Proof_Result2}
        \boldsymbol{r}_3 &\preceq \boldsymbol{0} 
        \quad \mathrm{in} \; \Omega
    \end{align}
  \end{subequations} 
  If the stiffness matrix $\boldsymbol{K}$ possess $\mathrm{DwMP}_
  {\boldsymbol{K}}$, it is evident from Theorem \ref{Thm:NeccSuff_Conditions_DMP} 
  and equations \eqref{Eqn:Kc3r3}, 
    \eqref{Eqn:DwCP_Proof_Result1}--\eqref{Eqn:DwCP_Proof_Result2}       
    that vector $\boldsymbol{c}_3$ satisfies $\mathrm{DwMP}$. Hence, 
    from Definition \ref{Thm:Discrete_WeakStrongMaximum_Principle} 
    and equations \eqref{Eqn:DwMP}, 
    \eqref{Eqn:DwCP_Proof_Result1}--\eqref{Eqn:DwCP_Proof_Result2},  
    we have the following result:
    %------------------------------------;
    %  Equation: Result-3 wrt to DwCP_K  ;
    %------------------------------------; 
    \begin{align}
      \label{Eqn:DwCP_Proof_Result3}
      \mathop{\mathrm{max}}_{\overline{\Omega}} \left[ 
      \boldsymbol{c}_3 \right] &\leq \mathop{\mathrm{max}} 
      \left[ 0, \mathop{\mathrm{max}}_{\partial \Omega} 
      \left[ \boldsymbol{c}_{3} \right] \right] = 0
    \end{align} 
    From equation \eqref{Eqn:DwCP_Proof_Result3}, it is evident 
    that the least upper bound for any component of vector 
    $\boldsymbol{c}_3$ is equal to zero. Hence, we have 
    $\boldsymbol{c}_3 \preceq \boldsymbol{0}$ in $\overline{\Omega}$. 
    This implies that $\boldsymbol{c}_1 \preceq \boldsymbol{c}_2$ 
    on $\overline{\Omega}$, which completes the proof for the 
    sufficient condition to satisfy $\mathrm{DwCP}_{\boldsymbol{K}}$.
    
    %%**********************************************************%%
    %%          Necessary condition proof for DsCP_K            %%
    %%**********************************************************%%
    \emph{Necessary condition to satisfy $\mathrm{DsCP}_{\boldsymbol{K}}:$}  
      Following the arguments about the proof for the necessary 
      condition to satisfy the $\mathrm{DwCP}_{\boldsymbol{K}}$ property, 
      it 
      is evident that $\boldsymbol{c}_3$ satisfies $\mathrm{DwMP}$. 
      In addition, we are given that $\boldsymbol{c}_1 \prec \boldsymbol{c}_2$ 
      in $\Omega$. This implies $\boldsymbol{c}_3 \prec \boldsymbol{0}$ 
      in $\Omega$. Based on the hypothesis of $\mathrm{DsCP}_{\boldsymbol{K}}$ 
      and utilizing the fact that $\boldsymbol{c}_3 \prec \boldsymbol{0}$ in 
      $\Omega$ yields the following:
      %---------------------------------------;
      %  Equation: Result-(-1) wrt to DsCP_K  ;
      %---------------------------------------;
      \begin{subequations} 
      \begin{align}
        \label{Eqn:DwCP_Proof_Result(-2)3}
        \mathop{\mathrm{max}} \left[ \boldsymbol{c}_3 
        \right] &\leq 0 \quad \mathrm{on} \; \partial 
        \Omega \\
        \label{Eqn:DwCP_Proof_Result(-1)}
        \mathop{\mathrm{max}} \left[ \boldsymbol{c}_3 
        \right] &< 0 \quad \mathrm{in} \; \Omega 
      \end{align}
      \end{subequations}
      This means that vector $\boldsymbol{0}$ is the least upper 
      bound for $\boldsymbol{c}_3$ in $\overline{\Omega}$, and any 
      component of $\boldsymbol{c}_3$ is \emph{strictly} less than 
      zero in the interior of the domain $\Omega$. From equation 
      \eqref{Eqn:DwCP_Proof_Result(-2)3} and \eqref{Eqn:DwCP_Proof_Result(-1)}, 
      it is clear that the \emph{non-negative} maximum value for vector 
      $\boldsymbol{c}_3$ occurs on the boundary $\partial \Omega$. 
      From equation \eqref{Eqn:DsMP} in 
      Definition \ref{Thm:Discrete_WeakStrongMaximum_Principle}, it 
      follows that vector $\boldsymbol{c}_3$ satisfies $\mathrm{DsMP}$.
      Hence, according to conditions specified by equation 
      \eqref{Eqn:StiffMat_DsMP} and the hypothesis of Theorem 
      \ref{Thm:NeccSuff_Conditions_DMP}, it is evident that $\boldsymbol{K}$ 
      must possess the $\mathrm{DsMP}_{\boldsymbol{K}}$ property. This completes the 
      proof for the necessary condition to satisfy $\mathrm{DsCP}_{\boldsymbol{K}}$.
       
    %%***********************************************************%%
    %%          Sufficient condition proof for DsCP_K            %%
    %%***********************************************************%%
   \emph{Sufficient condition to satisfy $\mathrm{DsCP}_{\boldsymbol{K}}:$}
      Given that $\boldsymbol{c}_1$ and $\boldsymbol{c}_2$ 
      satisfy $\mathrm{DsMP}$. Under the assumptions of $\mathrm{DsCP}
      _{\boldsymbol{K}}$ and from equations \eqref{Eqn:c3}--\eqref{Eqn:r3}, 
      we have the following relations:
      %------------------------------------;
      %  Equation: Result-0 wrt to DsCP_K  ;
      %------------------------------------;
      \begin{subequations}
      \begin{align}
        \label{Eqn:DsCP_Proof_Result01}
        \boldsymbol{r}_3 &\prec \boldsymbol{0} 
        \quad \mathrm{in} \; \Omega \\
        \label{Eqn:DsCP_Proof_Result02}
        \boldsymbol{c}_3 &\preceq \boldsymbol{0} 
        \quad \mathrm{on} \; \partial \Omega
      \end{align}
      \end{subequations}
      If the stiffness matrix $\boldsymbol{K}$ possess $\mathrm{DsMP}_
      {\boldsymbol{K}}$, it is evident from Theorem \ref{Thm:NeccSuff_Conditions_DMP} 
      and equations \eqref{Eqn:Kc3r3}, 
      \eqref{Eqn:DsCP_Proof_Result01}--\eqref{Eqn:DsCP_Proof_Result02}       
      that vector $\boldsymbol{c}_3$ satisfies $\mathrm{DsMP}$. Hence,
      by appealing to Definition \ref{Thm:Discrete_WeakStrongMaximum_Principle} 
      and equation \eqref{Eqn:DsMP}, if vector $\boldsymbol{c}_3$ 
      \emph{does not} attain a non-negative maximum value at an interior 
      point of $\Omega$, then we have the following result:
      %------------------------------------;
      %  Equation: Result-1 wrt to DsCP_K  ;
      %------------------------------------; 
      \begin{align}
        \label{Eqn:DsCP_Proof_Result1}
        \mathop{\mathrm{max}}_{\Omega} \left[ \boldsymbol{c}_3 
        \right] < \mathop{\mathrm{max}} \left[ 0, \mathop{\mathrm{max}}
        _{\partial \Omega} \left[ \boldsymbol{c}_{3} \right] 
        \right] = 0
      \end{align}
      which implies that each component of vector $\boldsymbol{c}_3$
      is less than zero. Hence, we have $\boldsymbol{c}_1 \prec 
      \boldsymbol{c}_2$ in $\Omega$.
      Suppose, if vector $\boldsymbol{c}_3$ attains a 
      non-negative maximum value at an interior point of $\Omega$, 
      then according to the surmise of $\mathrm{DsMP}$, we first 
      need to satisfy $\mathrm{DwMP}$. So from equation \eqref{Eqn:DwMP}, 
      we have the following relation:
      %---------------------------------------;
      %  Equation: Result-2(1) wrt to DsCP_K  ;
      %---------------------------------------; 
      \begin{align}
        \label{Eqn:DsCP_Proof_Result21}
        \mathop{\mathrm{max}}_{\overline{\Omega}} \left[ 
        \boldsymbol{c}_3 \right] &\leq \mathop{\mathrm{max}} 
        \left[ 0, \mathop{\mathrm{max}}_{\partial \Omega} 
        \left[ \boldsymbol{c}_{3} \right] \right] = 0 
      \end{align}
      Secondly, according to $\mathrm{DsMP}$, we also need to 
      satisfy the equation \eqref{Eqn:DsMP}. These conditions 
      in terms of vector $\boldsymbol{c}_3$ are given as follows: 
      %---------------------------------------;
      %  Equation: Result-2(2) wrt to DsCP_K  ;
      %---------------------------------------; 
      \begin{subequations}
      \begin{align}
        \label{Eqn:DsCP_Proof_Result22}
        \mathop{\mathrm{max}}_{\overline{\Omega}} \left[ 
        \boldsymbol{c}_3 \right] &= \mathop{\mathrm{max}}_{ 
        \Omega} \left[ \boldsymbol{c}_3 \right] = m \geq 0 \\
        \label{Eqn:DsCP_Proof_Result23}
        \boldsymbol{c}_3 &= m \boldsymbol{1} \quad 
        \mathrm{in} \; \overline{\Omega}
      \end{align}
      \end{subequations}
      From equations \eqref{Eqn:DsCP_Proof_Result21} and 
      \eqref{Eqn:DsCP_Proof_Result22}--\eqref{Eqn:DsCP_Proof_Result23}, 
      it is evident that $m = 0$; which implies that $\boldsymbol{c}_3 
      = \boldsymbol{0}$. Thus, we have $\boldsymbol{c}_1 = \boldsymbol{c}_2$ 
      in $\overline{\Omega}$. But from equation \eqref{Eqn:Kc3r3}, it is 
      obvious that $\boldsymbol{r}_3 = \boldsymbol{0}$ in $\Omega$, which 
      contradicts the hypothesis of $\mathrm{DsCP}_{\boldsymbol{K}}$ given 
      by the equation \eqref{Eqn:DsCP_Proof_Result01}. Hence, we have the 
      final result $\boldsymbol{c}_1 \prec \boldsymbol{c}_2$ in $\Omega$,
      which completes the proof for the sufficient condition to satisfy 
      $\mathrm{DsCP}_{\boldsymbol{K}}$.
\end{proof}

In the next section, we shall discuss the various factors (i.e., mesh restrictions, 
numerical formulations, and post-processing methods) that influence the satisfaction 
of discrete versions of maximum principles, comparison principles, and the non-negative 
constraint.

%****************************************;
%                                        ;
%  NAME                                  ;
%    S3_MP_GeneralRemarks.tex            ;
%                                        ;
%  WRITTEN BY                            ;
%    Maruti Kumar Mudunuru               ;
%    Kalyana Babu Nakshatrala            ;
%                                        ;
%****************************************;
\section{YET ANOTHER LOOK AT CONTINUOUS AND DISCRETE PRINCIPLES}
\label{Sec:S3_MP_GeneralRemarks}
Based on the finite element methodology outlined in subsection 
\ref{SubSec:Discrete_Formulation}, we shall analyze the properties 
that the stiffness matrix $\boldsymbol{K}$ inherits from the 
continuous problem. An important attribute that the discrete 
system needs to have in order to mimic the mathematical properties 
that the continuous system possesses is that the stiffness matrix 
$\boldsymbol{K}_{ff}$ has to be a \emph{(reducible or irreducible) 
monotone matrix}. The part (a) in all the equations 
\eqref{Eqn:StiffMat_DwMP}--\eqref{Eqn:StiffMat_DSMP} of Theorem 
\ref{Thm:NeccSuff_Conditions_DMP} corresponds to reducibility or 
irreducibility of $\boldsymbol{K}_{ff}$.
 
On general computational grids, it is well-known that the stiffness 
matrix $\boldsymbol{K}_{ff}$ obtained via low-order finite element 
discretization \emph{might not} be a monotone matrix 
\cite{Ciarlet_Raviart_CMAME_1973_v2_p17,2005_Draganescu_etal_MC_v74_p1_p23,
2012_Mincsovics_Hovarth_LSSC_p614_p621}. So, the discrete single-field
Galerkin formulation might (or shall) violate the non-negative constraint, 
discrete maximum principles, and discrete comparison principles on 
unstructured computational meshes \cite{Ciarlet_Raviart_CMAME_1973_v2_p17,
1987_Ishihara_LAA_v88_p431_p448,Liska_Shashkov_CiCP_2008_v3_p852,
Nakshatrala_Valocchi_JCP_2009_v228_p6726,Nagarajan_Nakshatrala_IJNMF_2011_v67_p820,
2012_Mincsovics_Hovarth_LSSC_p614_p621}. The violation is more severe 
if the diffusion tensor is anisotropic. One of the ways to overcome 
such unphysical values for concentration and preserve the discrete 
properties is to restrict the element shape and size in a computational 
mesh. This can be achieved by developing sufficient mesh conditions 
under which $\boldsymbol{K}_{ff}$ is ensured to be a reducibly or 
irreducibly diagonally dominant matrix \cite{Berman_Plemmons}. Before 
we discuss such a class of monotone matrices, which are easily amenable 
for deriving mesh restrictions, some important remarks on various DMPs and 
their relationship to DCPs and NC are in order. We would like to emphasize 
that such a comprehensive discussion is not reported elsewhere in the 
literature.

%==============================================================;
%  Subsection: Simply connected vs. multiple connected domains ;
%==============================================================;
\subsection{Simply connected vs. multiple connected domains}
\label{SubSec:Simple_Multiple_ConnDomains}
For many applications in mathematics, sciences, and engineering, 
it is necessary to at least satisfy the weak or strict weak maximum 
principle. But there are numerous cases where in it is required 
to satisfy a strong version of the maximum principle 
\cite{1987_Ishihara_LAA_v88_p431_p448,2005_Draganescu_etal_MC_v74_p1_p23}. 
In such scenarios, geometry and topology of the domain play a vital 
role. According to the hypothesis of Theorem \ref{Thm:Continuous_StrongMaximum_Principle},
it is evident that a strong maximum principle exists if the domain 
is simply connected (see reference \cite[Chapter 3]{Gilbarg_Trudinger}). 
However, one should not immediately conclude that if a domain is 
not simply connected, then a strong maximum principle will not exist 
\cite{Pucci_Serrin,2012_Mincsovics_Hovarth_LSSC_p614_p621}. 

In a discrete setting, Ishihara \cite{1987_Ishihara_LAA_v88_p431_p448},
Dr\v{a}g\v{a}nescu et.al. \cite{2005_Draganescu_etal_MC_v74_p1_p23}, 
and Mincsovics and Hov\'{a}rth \cite{2012_Mincsovics_Hovarth_LSSC_p614_p621} 
have conducted various numerical experiments related to discrete strong 
maximum principles for multiple connected domains. They performed analysis 
related to satisfaction of $\mathrm{DsMP}_{\boldsymbol{K}}$ and $\mathrm{DSMP}
_{\boldsymbol{K}}$ for various non-obtuse and acute triangulations for 
multiple connected domains. In particular, Mincsovics and Hov\'{a}rth 
discuss various interesting examples related to the irreducibility 
property of the stiffness matrix $\boldsymbol{K}_{ff}$ when the domain 
is not simply connected. In all of their examples, they solve the following 
equations: 
%----------------------------------------------------------------;
%  Equation: Laplace equation or isotropic diffusion with decay  ;
%----------------------------------------------------------------;
\begin{subequations}
  \begin{align}
    & \alpha c - \Delta c = 0 \quad \mathrm{in} \; \Omega \\
    &c(\mathbf{x}) = c^{\mathrm{p}}(\mathbf{x}) \quad 
    \mathrm{on} \; \partial \Omega
  \end{align}
\end{subequations}
where the linear decay $\alpha = 0$ or $\alpha = 128$. Through numerical 
experiments, the authors demonstrate that even though the triangulation 
satisfies the non-obtuse or acute angled mesh condition (proposed by 
Ciarlet and Raviart \cite{Ciarlet_Raviart_CMAME_1973_v2_p17}), it is 
not guaranteed to fulfill either DsMP or DSMP. This means that non-obtuse 
\cite{2007_Erten_Ungor_CCCG_p205_p208} and well-centered triangulation 
\cite{Vanderzee_Hirani_Guoy_Ramos_SIAMJSC_2010_v31_p4497} of any given 
domain will always satisfy the weak DMPs, but need not satisfy the strong 
DMPs.
  
Within the context of directed graphs \cite{Berman_Plemmons, 
2005_Draganescu_etal_MC_v74_p1_p23}, there is a one-to-one 
correspondence between irreducibility of the stiffness matrix 
$\boldsymbol{K}_{ff}$ and the interior vertices of the computational 
mesh \cite{2012_Huang_arXiv_1306_1987}. In order to satisfy 
the discrete (strong and strictly strong) maximum principle, 
the mesh has to be \emph{interiorly connected}, which in turn 
implies that $\boldsymbol{K}_{ff}$ has to be irreducible 
\cite{Varga_MatrixIterativeAnalysis_SecondEdition}. By interiorly 
connected mesh, we mean that any pair of interior vertices of 
the mesh are connected at least by an interior edge path 
\cite{2005_Draganescu_etal_MC_v74_p1_p23}. Hence,  $\boldsymbol{K}
_{ff}^{-1} \succ \boldsymbol{0}$ and $-\boldsymbol{K}_{ff}^{-1} 
\boldsymbol{K}_{fp} \succ \boldsymbol{0}$ in Theorem 
\ref{Thm:NeccSuff_Conditions_DMP} correspond to this discrete 
connectedness property of the computational mesh 
\cite{2005_Draganescu_etal_MC_v74_p1_p23,2012_Mincsovics_Hovarth_LSSC_p614_p621}.
However, it should be noted that irreducibility is a necessary 
condition, but not sufficient. For other details on numerical 
aspects related to mesh connectivity, see references \cite[Section 4, 
Figures 1--4]{2012_Mincsovics_Hovarth_LSSC_p614_p621}, 
\cite{2005_Draganescu_etal_MC_v74_p1_p23}, and \cite{2012_Huang_arXiv_1306_1987}.
  
%===========================================================================;
%  Subsection: Minimum principles, non-negativity, and min-max constraints  ;
%===========================================================================;
\subsection{Minimum principles, and non-negative and min-max constraints}
\label{SubSec:MinMax_Principles}
Due to linearity of the operator $\mathcal{L}$, similar theorems 
corresponding to minimum principles and the non-negative constraint 
for equations \eqref{Eqn:NN_AD_GE_BE}--\eqref{Eqn:NN_AD_GE_DBC} 
can be derived. To obtain the non-negative solution and corresponding 
min-max constraint on $c(\mathbf{x})$, we shall appeal to the 
continuous weak minimum/minimum-maximum principle, which can 
be written as follows \cite{Evans_PDE,Gilbarg_Trudinger}:
%--------------------------------------;
%  Continuous weak minimum principles  ;
%--------------------------------------;
\begin{lemma}[\texttt{Continuous weak minimum/minimum-maximum principle}]
  \label{Lemma:Continuous_WeakMinimum_Principles}
  Let $\mathcal{L}$ be a uniformly elliptic operator satisfying the 
  conditions given by \eqref{Eqn:alpha_NeccSuff}--\eqref{Eqn:velocity_DiffEigenValue} 
  and $\mathbf{D}(\mathbf{x})$ be continuously differentiable. Given 
  that $\mathcal{L}[c] \geq 0$, $c^{\mathrm{p}}(\mathbf{x}) \geq 0$, 
  and $c(\mathbf{x}) \in C^{2}(\Omega) \cap C^{0}(\overline{\Omega})$, 
  then $c(\mathbf{x})$ possess a continuous weak minimum principle, 
  which is given as follows:
  %--------------------------;
  %  Weak minimum principle  ;
  %--------------------------;
  \begin{align}
    \label{Eqn:Weak_Minimum_Principle}
    \mathop{\mathrm{min}}_{\mathbf{x} \in \overline{\Omega}} 
    \left [ c(\mathbf{x}) \right ] \geq \mathop{\mathrm{min}} 
    \left [ 0, \mathop{\mathrm{min}}_{\mathbf{x} \in \partial 
    \Omega} \left [ c(\mathbf{x}) \right ] \right ]
  \end{align}
  Moreover, if $\mathcal{L}[c] = 0$, then we obtain the classical weak 
  minimum-maximum principle for $c(\mathbf{x})$ in $\overline{\Omega}$,
  which is given as follows:
  %----------------------------------;
  %  Weak minimum-maximum principle  ;
  %----------------------------------;
  \begin{align}
    \label{Eqn:Weak_Minimum_Maximum_Principle}
    \mathop{\mathrm{min}} \left [ 0, \mathop{\mathrm{min}}_{\mathbf{x} 
    \in \partial \Omega} \left [ c(\mathbf{x}) \right ] \right ] \leq 
    c(\mathbf{x}) \leq \mathop{\mathrm{max}} \left [ 0, \mathop{\mathrm{max}}_
    {\mathbf{x} \in \partial \Omega} \left [ c(\mathbf{x}) \right ] \right ]
  \end{align}  
\end{lemma}
\begin{proof}
  For a proof, see references \cite{Gilbarg_Trudinger,Evans_PDE}.
\end{proof}
It is evident from equation \eqref{Eqn:Weak_Minimum_Principle} 
that for $f(\mathbf{x}) \geq 0$ and $c^{\mathrm{p}}(\mathbf{x}) 
\geq 0$, we have $c(\mathbf{x}) \geq 0$ for any $\mathbf{x} 
\in \overline{\Omega}$. Correspondingly, a discrete version 
of continuous weak minimum principle and weak minimum-maximum 
principle is given as follows:
%------------------------------------;
%  Discrete weak minimum principles  ;
%------------------------------------;
\begin{definition}[\texttt{Discrete weak minimum/minimum-maximum principle}]
  \label{Def:Discrete_WeakMinimum_Principles}
  A numerical formulation is said to possess 
  \begin{itemize}
    \item the discrete weak minimum principle if
    %-----------------------;
    %  Equation: DweakMinP  ;
    %-----------------------; 
    \begin{align}
      \label{Eqn:DiscreteMinPrinciple}
      \boldsymbol{r} \succeq \boldsymbol{0} 
      \qquad \mathrm{implies} \qquad \mathop{\mathrm{min}} 
      \left[ \boldsymbol{c} \right] \geq \mathop{\mathrm{min}} 
      \left[ 0, \mathop{\mathrm{min}} \left [ \boldsymbol{c}_{p} 
      \right ] \right]
    \end{align}
    \item the discrete weak minimum-maximum principle if
    %--------------------------;
    %  Equation: DweakMinMaxP  ;
    %--------------------------; 
    \begin{align}
      \label{Eqn:DiscreteMinMaxPrinciple}
      \boldsymbol{r} = \boldsymbol{0} \qquad \mathrm{implies} 
      \qquad c_{\mathrm{min}} \boldsymbol{1} \preceq \boldsymbol{c} 
      \preceq c_{\mathrm{max}} \boldsymbol{1}
    \end{align}
  \end{itemize}
  where $c_{\mathrm{min}}:= \mathop{\mathrm{min}} \left[ 
  \boldsymbol{c}_{p} \right]$, $c_{\mathrm{max}}:= 
  \mathop{\mathrm{max}} \left[ \boldsymbol{c}_{p} \right]$, 
  and $\mathop{\mathrm{min}} \left[ \bullet \right]$ denotes 
  the minimal element of a vector.
\end{definition}
      
%================================================;
%  Subsection: High-order finite element methods ;
%================================================;
\subsection{High-order finite element methods}
\label{SubSec:Higher_Order_FEM}
An attribute of low-order finite elements, which plays a 
central role in designing non-negative formulations for 
diffusion-type equations, is that the shape functions for 
these elements are monotonic and do not change their sign 
within the element \cite{Nagarajan_Nakshatrala_IJNMF_2011_v67_p820}. 
Moreover, they are convenient to generate computational meshes 
for complex geometries \cite{Frey_George}, to perform error 
analysis \cite{Babuska_Whiteman_Strouboulis}, and for adaptive 
local mesh refinement \cite{Plewa_Linde_Weirs}. 
High-order finite elements are widely use for solving smooth 
problems, as one can obtain exponential convergence under 
high-order interpolations for these problems. But the shape 
functions of high-order finite elements change the sign within 
an element, which makes them not suitable under most of the 
current non-negative formulations (e.g., see reference 
\cite{Nagarajan_Nakshatrala_IJNMF_2011_v67_p820,
Payette_Nakshatrala_Reddy_IJNME_2012}). The conditions 
presented in this paper will also not be applicable to 
high-order finite elements for the same reason of change 
in sign of interpolation functions within an element. 

As compared to low-order finite element methods, the discrete 
counterparts of continuous weak and strong maximum principles 
for high-order finite element methods is not well understood 
yet. This is because of the complicated task related to the 
test of non-negativity of a multivariate polynomial \cite{2010_Vejchodsky_NMAA}. 
It should be noted that construction of non-negative high-order 
shape functions for finite element methods is still an unsolved 
problem and its roots can be traced back to the famous Hilbert's 
$17^{\mathrm{th}}$ problem \cite{Prestel_Delzell,2000_Reznick_CM_v253_p251_p272}. 
Within the context of variational methods, probably, the research 
works by Ciarlet \cite{1970_Ciarlet_AequationesMath_v4_p74_p82,
1970_Ciarlet_Varga_NM_v16_p115_p128} are the first attempt to develop 
high-order non-negative shape functions to satisfy a discrete maximum 
principle. This study is based on a general theory of discrete Green's 
function (DGF) for uniformly elliptic linear partial differential 
operators. Later, various attempts were made by different researchers 
to develop shape functions and derive mesh restrictions based on the 
DGF approach. Most of them are pertinent to one-dimensional problems 
or particular cases of isotropic diffusion. The conditions to be met 
for high-order elements to satisfy DMPs based on DGF methodology are 
much more stringent and have a less broad scope for general applications. 
Furthermore, one should be aware that the discrete analogues of continuous 
Green's functions are applicable only for linear problems, and cannot be 
extended to to non-linear problems (such as semi-linear and quasi-linear 
elliptic partial differential equations). Hence, such a method will have 
limited scope. For more details, one can consult the following references 
\cite{Hohn_Mittelmann_Computing_1981_v27_p145,Solin_Vejchodsky_JCAM_2007_v209_p54,
2010_Vejchodsky_NMAA}.
  
%=======================================================;
%  Subsection: Relation between various DCPs and DMPs   ;
%=======================================================;
\subsection{Relationship between various DCPs and DMPs}
\label{SubSec:Relationship_DMPs_DCPs}
It is evident from Definitions \ref{Thm:Discrete_WeakStrongMaximum_Principle}, 
\ref{Def:Discrete_Comparison_Principle}, and \ref{Def:Discrete_WeakMinimum_Principles}; 
Lemma \ref{Lemma:Continuous_WeakMinimum_Principles}; and Theorems 
\ref{Thm:NeccSuff_Conditions_DMP} and \ref{Thm:DCP_Conditions_Galerkin} 
that if a numerical formulation satisfies either $\mathrm{DwCP}$/$\mathrm{DwCP}_
{\boldsymbol{K}}$ or $\mathrm{DsCP}$/$\mathrm{DsCP}_{\boldsymbol{K}}$, 
then it automatically obeys $\mathrm{DwMP}$/$\mathrm{DwMP}_{\boldsymbol{K}}$ 
and NC. Figure \ref{Fig:VennDiagram_MeshRestrictions} illustrates a graphical 
representation among various numerical solution spaces that satisfy different 
DMPs and DCPs within the context of mesh restrictions. By a (finite dimensional) 
numerical solution space $\mathcal{V}_{\mathfrak{p}}$, we mean a set of numerical 
solutions $\{\boldsymbol{c}_i\}_{i = 1}^{n}$, which satisfy a given discrete 
property given by $\mathfrak{p}$. For example, if a concentration vector 
$\boldsymbol{c}_i$ corresponding to a given volumetric source vector 
$\boldsymbol{r}_i$ satisfies the discrete property $\mathrm{DwMP}$, then 
$\boldsymbol{c}_i \in \mathcal{V}_{\mathrm{DwMP}}$. It should be noted that 
within the context of DMPs, DCPs, and NC; $\mathcal{V}_{\mathfrak{p}}$ is 
\emph{not} a vector space. This is because if $\boldsymbol{c}_i \in 
\mathcal{V}_{\mathfrak{p}}$, then according to non-negative property 
$-\boldsymbol{c}_i \notin \mathcal{V}_{\mathfrak{p}}$, which is one 
of the properties needed for $\mathcal{V}_{\mathfrak{p}}$ to be a 
vector space.
   
Herein, we would like to emphasize that the route taken to satisfy 
$\mathfrak{p}$ is very important. One can fulfill a discrete property 
$\mathfrak{p}$ in numerous ways. In general, this is achieved by either 
placing mesh restrictions or developing a new (non-negative or monotone 
or monotonicity based) numerical formulation or through various post-processing 
methods. Couple of these techniques are developed based along the lines 
similar to Theorem \ref{Thm:NeccSuff_Conditions_DMP} and others based 
on Definition \ref{Thm:Discrete_WeakStrongMaximum_Principle}. But it 
should be noted that developing numerical formulations accordant to 
Theorem \ref{Thm:NeccSuff_Conditions_DMP} is much more difficult than 
that of Definition \ref{Thm:Discrete_WeakStrongMaximum_Principle}. This 
is because in order to satisfy Theorem \ref{Thm:NeccSuff_Conditions_DMP}, 
we need to place restrictions on the stiffness matrices $\boldsymbol{K}_{ff}$ 
and $\boldsymbol{K}_{fp}$. On the other hand, the hypothesis of Definition 
\ref{Thm:Discrete_WeakStrongMaximum_Principle} does not assume any particular 
constraints on $\boldsymbol{K}$. Hence, we would like to differentiate between 
the set of discrete properties given by $\mathrm{DwMP}_{\boldsymbol{K}}$, 
$\mathrm{DWMP}_{\boldsymbol{K}}$, $\mathrm{DsMP}_{\boldsymbol{K}}$, 
$\mathrm{DSMP}_{\boldsymbol{K}}$, $\mathrm{DwCP}_{\boldsymbol{K}}$, and  
$\mathrm{DsCP}_{\boldsymbol{K}}$ to that of $\mathrm{DwMP}$, $\mathrm{DWMP}$, 
$\mathrm{DsMP}$, $\mathrm{DSMP}$, $\mathrm{DwCP}$, and $\mathrm{DsCP}$.
   
In spite of the fact that there are several numerical methods available 
to satisfy a given discrete property $\mathfrak{p}$, from the characterization 
of $\mathcal{V}_{\mathfrak{p}}$, it is evident that the resulting numerical 
solution spaces will be the same (for example, we have $\mathcal{V}_
{\mathrm{DwCP}_{\boldsymbol{K}}} \equiv \mathcal{V}_{\mathrm{DwCP}}$). 
From Theorems \ref{Thm:NeccSuff_Conditions_DMP} and \ref{Thm:DCP_Conditions_Galerkin}, 
it is evident that among various DMPs and DCPs, we have the following 
set inclusions:
%------------------------------------;
% Equation: DMPs and DCPs inclusion  ;
%------------------------------------;
\begin{subequations}
  \begin{align} 
    &\mathcal{V}_{\mathrm{DSMP}_{\boldsymbol{K}}} \subset 
    \mathcal{V}_{\mathrm{DsMP}_{\boldsymbol{K}}} \subset 
    \mathcal{V}_{\mathrm{DWMP}_{\boldsymbol{K}}} \subset 
    \mathcal{V}_{\mathrm{DwMP}_{\boldsymbol{K}}} \\
    &\mathcal{V}_{\mathrm{DsCP}_{\boldsymbol{K}}} \subset 
    \mathcal{V}_{\mathrm{DwCP}_{\boldsymbol{K}}}
  \end{align}
\end{subequations}
But it should be noted that a similar type of enclosure for 
numerical solution spaces between DMPs and DCPs does not hold:
%---------------------------------------------;
% Equation: DMPs and DCPs non-enclosure sets  ;
%---------------------------------------------;
\begin{subequations}
  \begin{align}
    \label{Eqn:V_DwCP_DWMP}
    \mathcal{V}_{\mathrm{DwCP}_{\boldsymbol{K}}} &\not \subset 
    \mathcal{V}_{\mathrm{DWMP}_{\boldsymbol{K}}} \\
    \label{Eqn:V_DwCP_DsMP}
    \mathcal{V}_{\mathrm{DwCP}_{\boldsymbol{K}}} &\not \subset 
    \mathcal{V}_{\mathrm{DsMP}_{\boldsymbol{K}}} \\
    \label{Eqn:V_DwCP_DSMP}
    \mathcal{V}_{\mathrm{DwCP}_{\boldsymbol{K}}} &\not \subset 
    \mathcal{V}_{\mathrm{DSMP}_{\boldsymbol{K}}} \\
    \label{Eqn:V_DsCP_DSMP}
    \mathcal{V}_{\mathrm{DsCP}_{\boldsymbol{K}}} &\not \subset 
    \mathcal{V}_{\mathrm{DSMP}_{\boldsymbol{K}}}
  \end{align}
\end{subequations}
The reason for such a non-enclosure stems from the hypothesis of 
Theorems \ref{Thm:NeccSuff_Conditions_DMP} and \ref{Thm:DCP_Conditions_Galerkin}, 
wherein we only need to satisfy $\mathrm{DwMP}_{\boldsymbol{K}}$ 
for $\mathrm{DwCP}_{\boldsymbol{K}}$ and $\mathrm{DsMP}_{\boldsymbol{K}}$ 
for $\mathrm{DsCP}_{\boldsymbol{K}}$. In a discrete setting, a 
numerical methodology may inherit one or more than one of these 
discrete principles and in some cases none. Now, we shall discuss 
in detail a class of numerical formulations, which satisfy a certain 
discrete property $\mathfrak{p}$. We shall epitomize our findings 
based on various popular research works in literature, which span 
across different disciplines, such as computational geometry, optimization 
theory, numerical linear algebra, and partial differential equations.
  
%=========================================================;
%  Isotropic and anisotropic non-obtuse angle conditions  ;
%=========================================================;
\begin{enumerate}[(a)]
\item \emph{Isotropic and anisotropic non-obtuse angle conditions}: 
Recently, Huang and co-workers \cite{2005_Li_Huang_JCP_v229_p8072_p8094,
2012_Lu_Huang_Qiu_arXiv_1201_3564,2012_Huang_arXiv_1306_1987} were 
able to satisfy certain discrete properties through mesh restrictions, 
which are based on anisotropic $\mathcal{M}$-uniform mesh generation 
techniques. However, their theoretical investigation is mainly restricted 
to linear simplicial elements, specifically, the three-node triangular 
element and four-node tetrahedral element. But one should note that it 
is difficult to extend the procedure outlined by Huang and co-workers 
to multi-linear elements, such as the four-node quadrilateral element, 
six-node wedge element, and eight-node brick element. This is because 
the partial derivatives of the shape functions for multi-linear finite 
elements are not constant (for more details, see subsection 
\ref{SubSec:MeshRestrict_Rectangular_Element} of this paper, which discusses 
mesh restrictions for a rectangular element). Nevertheless, constructing 
a WCT mesh \cite{Vanderzee_Hirani_Guoy_Ramos_SIAMJSC_2010_v31_p4497} or an 
anisotropic $\mathcal{M}$-uniform triangular mesh \cite{2005_Huang_JCP_v204_p633_p665,
2013_Schneider_LNACM_p133_p152} that satisfies various DMPs for an 
\emph{arbitrary domain} is still an open problem \cite{1995_Krizek_Qun_EWJNM_v3_p59_p69}. 
The key-concept we would like to emphasize is that the numerical solutions 
obtained for isotropic diffusion-type equations using WCT meshes and 
anisotropic diffusion-type equations using the diffusivity tensor based 
anisotropic $\mathcal{M}$-uniform meshes satisfy all versions of discrete 
maximum principles. In addition, if the hypothesis of $\mathrm{DwCP}_
{\boldsymbol{K}}$ and $\mathrm{DsCP}_{\boldsymbol{K}}$ is satisfied, then 
these meshes also satisfy all versions of discrete comparison principles.

%=================================================================;
%  Nonlinear finite volume and mimetic finite difference methods  ;
%=================================================================;
\item \emph{Non-linear finite volume and mimetic finite difference methods}:    
Le Potier's method \cite{2009_LePotier_IJFV_v6} and Lipnikov 
et al. \cite{Lipnikov_Shashkov_Svyatskiy_Vassilevski_JCP_2007_v227_p492} 
are some of the noteworthy works in the direction of FVM that satisfy 
the non-negative constraint, but do not possess a discrete version of 
the comparison principles and the maximum principles. However, it should 
be noted that recently these authors have developed techniques based on 
non-linear finite volume methods \cite{2009_LePotier_IJFV_v6,
2011_Droniou_LePotier_SIAMJNA_v49_p459_p490} and mimetic finite 
difference methods \cite{2011_Lipnikov_Manzini_Svyatskiy_JCP_v230_p2620_p2642} 
to satisfy various versions of DMPs for a certain \emph{specific} 
class of linear \emph{self-adjoint} elliptic operators. But it 
should be noted that there is no discussion on satisfying various
DCPs. 

%=============================================;
%  Optimization-based finite element methods  ;
%=============================================;
\item \emph{Optimization-based finite element methods}:
Based on the works by Liska and Shashkov \cite{Liska_Shashkov_CiCP_2008_v3_p852} 
and Nakshatrala and co-workers \cite{Nakshatrala_Valocchi_JCP_2009_v228_p6726,
Nagarajan_Nakshatrala_IJNMF_2011_v67_p820,Nakshatrala_Nagaraja_Shabouei_arXiv_2012,
2013_Nakshatrala_Mudunuru_Valocchi_JCP_v253_p278_p307}, the optimization-based 
low-order finite element methods, under certain conditions (when $\boldsymbol{K}
_{ff}$ is symmetric and positive definite), can be written as follows:
%-----------------------------------------------;
%  Equation: First-order optimality conditions  ;
%-----------------------------------------------;
\begin{subequations}
  \label{Eqn:Diffusion_FOOC}
  \begin{align}
    \label{Eqn:Diffusion_FOOC_1}
    &\boldsymbol{K}_{ff} \boldsymbol{c}_{f} = \boldsymbol{r}
    - \boldsymbol{K}_{fp} \boldsymbol{c}_{p}
    + \boldsymbol{\lambda}_{\mathrm{min}} - 
    \boldsymbol{\lambda}_{\mathrm{max}} \\
    \label{Eqn:Diffusion_FOOC_2}
    &c_{\mathrm{min}}^{*} \boldsymbol{1} \preceq 
    \boldsymbol{c}_{f} \preceq c_{\mathrm{max}}^{*} 
    \boldsymbol{1} \\
    \label{Eqn:Diffusion_FOOC_3}
    &\boldsymbol{\lambda}_{\mathrm{min}} \succeq \boldsymbol{0} \\
    \label{Eqn:Diffusion_FOOC_4}
    &\boldsymbol{\lambda}_{\mathrm{max}} \succeq \boldsymbol{0} \\
    \label{Eqn:Diffusion_FOOC_5}
    &\left(\boldsymbol{c}_f - c_{\mathrm{min}}^{*} \boldsymbol{1} 
    \right) \cdot \boldsymbol{\lambda}_{\mathrm{min}} = 0 \\
    \label{Eqn:Diffusion_FOOC_6}
    &\left(c_{\mathrm{max}}^{*} \boldsymbol{1} - \boldsymbol{c}_f 
    \right) \cdot \boldsymbol{\lambda}_{\mathrm{max}} = 0 
  \end{align}
\end{subequations}
where $c_{\mathrm{min}}^{*}$ and $c_{\mathrm{max}}^{*}$ are 
the minimum and maximum concentration values possible in 
$\overline{\Omega}$. These values can be obtained based on 
the boundary conditions and a prior knowledge about the solution. 
$\boldsymbol{\lambda}_{\mathrm{min}}$ is the vector of Lagrange 
multipliers corresponding to the constraint $c_{\mathrm{min}}^{*} 
\boldsymbol{1} \preceq \boldsymbol{c}_{f}$ and similarly 
$\boldsymbol{\lambda}_{\mathrm{max}}$ is the vector of Lagrange 
multipliers corresponding to the constraint $\boldsymbol{c}_{f} 
\preceq c_{\mathrm{max}}^{*} \boldsymbol{1}$.
     
Based on the nature of constraints, one can satisfy different discrete 
principles and it should be emphasized that DsMP, DSMP, DwCP, and DsCP 
can be fulfilled only under certain conditions. If either $\boldsymbol{r} 
\succ \boldsymbol{0}$ or $\boldsymbol{r} \prec \boldsymbol{0}$, then the 
non-negative constraint and the weaker versions of discrete minimum/maximum 
principles can be satisfied by specifying either $c_{\mathrm{min}}^{*}$ or 
$c_{\mathrm{max}}^{*}$. But in the case of $\boldsymbol{r} = \boldsymbol{0}$,
both $c_{\mathrm{min}}^{*}$ and $c_{\mathrm{max}}^{*}$ can be prescribed 
based on the Dirichlet boundary conditions; moreover, according to Definition 
\ref{Def:Discrete_WeakMinimum_Principles} and from equation \eqref{Eqn:DiscreteMinMaxPrinciple}, 
we have $c_{\mathrm{min}}^{*} = c_{\mathrm{min}}$ and $c_{\mathrm{max}}^{*} 
= c_{\mathrm{max}}$. However, one should note that if the qualitative and 
quantitative nature of the solution is known a prior, then one can satisfy 
\emph{all} of the discrete versions of (weak and strong) maximum principles 
by specifying $c_{\mathrm{min}}^{*}$ and $c_{\mathrm{max}}^{*}$ in the 
Karush-Kuhn-Tucker conditions given by equations 
\eqref{Eqn:Diffusion_FOOC_1}--\eqref{Eqn:Diffusion_FOOC_6}. In general, 
these methods do not inherit a discrete strong maximum principle and a 
discrete comparison principle. For more details, a counter example is 
shown in the reference \cite[Section 4, Figure 1]{2013_Nakshatrala_Mudunuru_Valocchi_JCP_v253_p278_p307}). Nevertheless, satisfying DsMP, DSMP, DwCP, and DsCP is still an open 
problem and are interesting topics to investigate in future endeavors.   

%===========================;
%  Non-variational methods  ;
%===========================;
\item \emph{Variationally inconsistent methods}:   
In literature, there are various post-processing methods 
\cite{2012_Kreuzer_arXiv_1208_3958,2012_Burdakov_etal_JCP_v231_p3126_p3142,
2013_Lu_Huang_Vleck_JCP_v242_p24_p36} available that can 
recover certain discrete properties if a prior information 
about the numerical solution is known. However, one should 
note that such methods are variationally inconsistent. A 
summary of the above discussion between various discrete 
principles within the context of FDS, MFDM, FVM, FEM and 
PP based methods is pictorially described in Figure 
\ref{Fig:VennDiagram_NumericalFormulations}.
\end{enumerate}

In the next section, we shall derive sufficient conditions on the three-node 
triangular element and four-node quadrilateral element to satisfy discrete 
versions of comparison principles, maximum principles, and the non-negative 
constraint.

%****************************************;
%                                        ;
%  NAME                                  ;
%    S4_MP_Restrictions1.tex             ;
%                                        ;
%  WRITTEN BY                            ;
%    Maruti Kumar Mudunuru               ;
%    Kalyana Babu Nakshatrala            ;
%                                        ;
%****************************************;
\section{MESH RESTRICTIONS TO SATISFY DISCRETE PRINCIPLES}
\label{Sec:S4_MP_Restrictions}
In this section, we shall utilize and build upon the research 
works of Huang and co-workers \cite{2005_Huang_JCP_v204_p633_p665,
2005_Li_Huang_JCP_v229_p8072_p8094,2012_Lu_Huang_Qiu_arXiv_1201_3564,
2012_Huang_arXiv_1306_1987} for linear second-order elliptic equations. 
We first present, without proofs, relevant mathematical and geometrical 
results required to obtain mesh restrictions for simplicial elements. 
We will then use these results to construct mesh restriction theorems 
and generate various types of triangulations (see Algorithm \ref{Algo:Aniso_Meshes} 
and the discussion in subsection \ref{SubSec:NumEg_Triangulations}) 
using the open source mesh generators, such as \textsf{Gmsh} \cite{gmsh.www} 
and \textsf{BAMG} \cite{1998_Hecht_BAMG} available in the \textsf{FreeFem++} 
software package \cite{FreeFempp,2002_Hecht_JNM_v20_p251_p266}. It should 
be emphasized that these results cannot be extended to Q4 element, as this 
element is not simplicial. For more details on mesh restrictions for Q4 
element, see subsection \ref{SubSec:MeshRestrict_Rectangular_Element} of 
this paper.

Let $\hat{\boldsymbol{x}}_{1}, \hat{\boldsymbol{x}}_{2}, \dots, 
\hat{\boldsymbol{x}}_{nd+1}$ denote the vertices of an arbitrary 
simplex $\Omega_e \in \mathcal{T}_{h}$, where $\mathcal{T}_{h}$ 
is a simplicial triangulation of the domain $\Omega$. The subscript 
`$h$' in the triangulation $\mathcal{T}_{h}$ corresponds to the maximum 
element size (which will be described later in this section, see equation 
\eqref{Eqn:GlobMeshSize}). Based on various values of $h$, we have an 
affine family of such simplicial meshes denoted by $\{ \mathcal{T}_{h} \}$. 
Designate the total number of vertices and the corresponding interior 
vertices of $\mathcal{T}_{h}$ by `$Nv$' and `$Niv$'. The edge matrix of 
$\Omega_e$, which is denoted by $\boldsymbol{E}_{\Omega_e}$, is defined 
as follows: 
%-------------------------;
%  Equation: Edge matrix  ; 
%-------------------------;
\begin{align}
  \label{Eqn:Edge_Matrix}
  \boldsymbol{E}_{\Omega_e}:= \left[ \hat{\boldsymbol{x}}_{2} - 
  \hat{\boldsymbol{x}}_{1}, \hat{\boldsymbol{x}}_{3} - 
  \hat{\boldsymbol{x}}_{1}, \dots, \hat{\boldsymbol{x}}_{nd+1} - 
  \hat{\boldsymbol{x}}_{1} \right] \quad \forall \Omega_e \in 
  \mathcal{T}_{h}
\end{align}
then an edge connecting vertices $\hat{\boldsymbol{x}}_{p}$ and 
$\hat{\boldsymbol{x}}_{q}$ of $\Omega_e$ is denoted as $e_{pq}$. 
Correspondingly, the edge vector $\boldsymbol{e}_{pq,\Omega_e}$ 
(which can be expressed as a linear combination of the elements 
corresponding to the edge matrix $\boldsymbol{E}_{\Omega_e}$) 
and the element boundary $\partial \Omega_e$ in-terms of $\hat{
\boldsymbol{x}}_{p}$, $\hat{\boldsymbol{x}}_{q}$, and $e_{pq}$ 
are given by:
%-----------------------------------------------;
%  Equation: Edge vector and boundary elements  ;
%-----------------------------------------------;
\begin{subequations}
\begin{align}
  \label{Eqn:Edge_Vector}
  \boldsymbol{e}_{pq,\Omega_e} &= \hat{\boldsymbol{x}}_{q} - 
  \hat{\boldsymbol{x}}_{p} \quad \forall p,q = 1, 2, \dots, nd+1 
  \quad \mathrm{and} \quad p \neq q \\
  \label{Eqn:Element_Boundary}
  \partial \Omega_e &= \bigcup_{\substack{p,q = 1 \\ p \neq q}}^{nd+1} e_{pq}
\end{align}
\end{subequations}
Following references \cite{2008_Brandts_etal_LAA_v429_p2344_p2357,
2005_Li_Huang_JCP_v229_p8072_p8094}, a set of $\boldsymbol{q}$-vectors 
corresponding to this edge matrix $\boldsymbol{E}_{\Omega_e}$ are defined 
as follows:
%------------------------------;
%  Equation: q-vectors matrix  ;
%------------------------------; 
\begin{subequations}
\begin{align}
  \label{Eqn:q_vectors_1}
  &\boldsymbol{E}^{-\mathrm{T}}_{\Omega_e}:= \left[ \boldsymbol{q}_2, 
  \boldsymbol{q}_3, \dots, \boldsymbol{q}_{nd+1} \right] \\
  \label{Eqn:q_vectors_2}
  &\boldsymbol{q}_1 + \displaystyle \sum \limits_{p=2}^{nd+1} \boldsymbol{q}_p 
  = \boldsymbol{0}    \quad \forall \Omega_e \in \mathcal{T}_{h}
\end{align}
\end{subequations}
Let $\varphi_{p_g}$ denote the linear basis function associated with 
the $p_g$-th global vertex in the triangulation $\mathcal{T}_{h}$. 
Then, $\{ \varphi_{p_g} \}^{Nv}_{p_g=1}$ and $\{ \varphi_{p_g} \}^{Niv}
_{p_g=1}$ span the respective finite dimensional subsets $\mathcal{C}^{h}$ 
and $\mathcal{W}^{h}$ given by equations \eqref{Eqn:Ch}--\eqref{Eqn:Wh}. 
Denote the face opposite to vertex $\hat{\boldsymbol{x}}_p$ by $F_p$ and 
the corresponding unit inward normal pointing towards the vertex 
$\hat{\boldsymbol{x}}_p$ by $\mathbf{n}_p$. The perpendicular distance 
(or the height) from vertex $\hat{\boldsymbol{x}}_p$ to face $F_p$ is 
denoted by $h_p$. In the case of 2D, the above set of geometrical properties 
are pictorially described in Figure \ref{Fig:T3_Simplex_GeoProp}.
 
%------------------------------------;
%  Definition: Positive linear maps  ;
%------------------------------------;
\begin{definition}[\texttt{Positive linear maps}]
  \label{Def:Positive_Linear_Maps}
  Let $\mathbb{M}_n := \mathbb{M}_{n \times n} (\mathbb{R})$ 
  be the set of all real matrices of size $n \times n$, 
  which forms a vector space over the field $\mathbb{R}$. 
  A linear map $\Phi: \mathbb{M}_n \rightarrow \mathbb{M}_k$ 
  is called positive if $\Phi(\boldsymbol{A})$ is positive 
  semi-definite whenever $\boldsymbol{A}$ is positive 
  semi-definite. Similarly, $\Phi$ is called strictly 
  positive if $\Phi(\boldsymbol{A})$ is positive 
  definite whenever $\boldsymbol{A}$ is positive 
  definite.
\end{definition}

%------------------------------------------------------------------------;
%  Theorem: Strictly positive linear mapping of anisotropic diffusivity  ;
%------------------------------------------------------------------------;
\begin{theorem}[\texttt{Strictly positive linear mapping of anisotropic diffusivity}]
  \label{Thm:Positive_LinearMap_AnisoDiff}
  Let $\Phi[\bullet]:= \int_{\Omega_e} [\bullet] \mathrm{d} \Omega$. Show 
  that $\Phi[\bullet]$ is a linear map and $\Phi[\mathbf{D} (\mathbf{x})]$ 
  is symmetric, uniformly elliptic, and bounded above.
\begin{proof}
  From Definition \ref{Def:Positive_Linear_Maps}, it is evident that 
  $\Phi[\bullet]$ is a linear map and $\Phi[\mathbf{D} (\mathbf{x})]$ is 
  symmetric. From equations \eqref{Eqn:Diffusivity_PositiveDefinite} and 
  \eqref{Eqn:Omega_To_Omega_e}, we have $\boldsymbol{\xi} \cdot \mathbf{D}
  (\mathbf{x}) \boldsymbol{\xi} > 0$ and $\mathrm{meas}(\Omega_e) > 0$. It 
  is well known that Lebesgue integration of a scalar for a strictly positive 
  measure is always greater than zero. Hence, $\int_{\Omega_e} 
  \boldsymbol{\xi} \cdot \mathbf{D}(\mathbf{x}) \boldsymbol{\xi} \mathrm{d} 
  \Omega > 0 \quad \forall \mathbf{x} \in \Omega_e$. Now, integrating equation 
  \eqref{Eqn:Diffusivity_PositiveDefinite} over $\Omega_e$ results in the   
  following relation:
%---------------------------------------------------;
%  Equation: Integrated Diffusivity tensor is SPD  ;
%---------------------------------------------------;
\begin{align}
  \label{Eqn:IntDiff_SPD}
  0 < \gamma_{\mathrm{min}} \boldsymbol{\xi} \cdot \boldsymbol{\xi} 
  \leq \frac{1}{\mathrm{meas}(\Omega_e)} \displaystyle \int \limits_
  {\Omega_e} \boldsymbol{\xi} \cdot \mathbf{D}(\mathbf{x}) \boldsymbol{\xi} 
  \mathrm{d} \Omega \leq \gamma_{\mathrm{max}} \boldsymbol{\xi} \cdot 
  \boldsymbol{\xi} \quad \forall \boldsymbol{\xi} \in \mathbb{R}^{nd} 
  \backslash \{\boldsymbol{0}\} \; \mbox{and} \; \forall \mathbf{x} \in 
  \Omega_e
\end{align}
where the positive constants $\gamma_{\mathrm{min}}$ and $\gamma_{\mathrm{max}}$ 
are respectively the integral average of minimum and maximum eigenvalues of 
$\mathbf{D}(\mathbf{x})$ over $\Omega_e$. These are given as follows:
%------------------------------------;
%  Equation: Integrated eigenvalues  ;
%------------------------------------;
\begin{subequations}
\begin{align}
  \gamma_{\mathrm{min}} &:= \frac{1}{\mathrm{meas}(\Omega_e)} \displaystyle 
  \int \limits_{\Omega_e} \lambda_{\mathrm{min}} (\mathbf{x}) 
  \mathrm{d} \Omega \\
  \gamma_{\mathrm{max}} &:= \frac{1}{\mathrm{meas}(\Omega_e)} \displaystyle 
  \int \limits_{\Omega_e} \lambda_{\mathrm{max}} (\mathbf{x}) 
  \mathrm{d} \Omega
\end{align}
\end{subequations}
Since the vector $\boldsymbol{\xi}$ is independent of $\mathbf{x}$ 
and $\Omega_e$, we can interchange the order of integration. This gives 
us the following equation:
%--------------------------------------------------;
%  Equation: Integrated Diffusivity tensor is SPD  ;
%--------------------------------------------------;
\begin{align}
  \label{Eqn:IntDiff_SPD}
  0 < \gamma_{\mathrm{min}} \boldsymbol{\xi} \cdot \boldsymbol{\xi} 
  \leq \frac{1}{\mathrm{meas}(\Omega_e)} \boldsymbol{\xi} \cdot 
  \Phi[\mathbf{D} (\mathbf{x})] \boldsymbol{\xi} \leq \gamma_
  {\mathrm{max}} \boldsymbol{\xi} \cdot \boldsymbol{\xi} \quad 
  \forall \boldsymbol{\xi} \in \mathbb{R}^{nd} \backslash 
  \{\boldsymbol{0}\} \; \mbox{and} \; \forall \mathbf{x} \in 
  \Omega_e
\end{align}
which shows that $\Phi[\mathbf{D} (\mathbf{x})]$ is a strictly 
positive linear map of $\mathbf{D} (\mathbf{x})$ and indeed 
preserves its properties.  
\end{proof} 
\end{theorem}

%==================================================================;
%  Subsection: Geometrical properties and finite element analysis  ; 
%              of simplicial elements                              ;
%==================================================================;
\subsection{Geometrical properties and finite element analysis of 
  simplicial elements}
\label{SubSec:GeoProp_Simplicial_Elements}
Based on the above notation, we have the following important mathematical 
results relating the linear basis functions, finite element matrices, and 
geometrical properties of simplicial elements. 
%-------------------------------------------------------------------;
%  Enumerate: Summary of geometrical and finite element properties  ;
%-------------------------------------------------------------------;
\begin{enumerate}
  \item The integral element average anisotropic diffusivity 
   $\widetilde{\boldsymbol{D}}_{\Omega_e}$ is given by:
    %---------------------------------------------------------------;
    %  Equation: Geometrical property no # 1 (average diffusivity)  ; 
    %---------------------------------------------------------------;
    \begin{align}
      \label{Eqn:EleAvg_Aniso_Diff}
      \widetilde{\boldsymbol{D}}_{\Omega_e} := \frac{1}{\mathrm{meas}(\Omega_e)} 
      \int \limits_{\Omega_e} \mathbf{D} (\mathbf{x}) \mathrm{d} \Omega \quad 
      \forall \Omega_e \in \mathcal{T}_{h}
    \end{align}
    is a strictly positive linear map (see Definition \ref{Def:Positive_Linear_Maps} 
    and Theorem \ref{Thm:Positive_LinearMap_AnisoDiff}).
  \item For any arbitrary simplicial element $\Omega_e \in \mathcal{T}_{h}$, 
    the vector $\boldsymbol{q}_p$ associated with the face $F_p$, the gradient 
    of the linear basis function $\mathrm{grad}[\varphi_{p_g}]$, the unit inward 
    normal $\mathbf{n}_p$, and the height $h_p$ are related as 
    follows \cite{1995_Krizek_Qun_EWJNM_v3_p59_p69,2008_Brandts_etal_LAA_v429_p2344_p2357}:
    %-----------------------------------------------------;
    %  Equation: Geometrical property no # 2 (q-vectors)  ;
    %-----------------------------------------------------;
    \begin{align}
      \label{Eqn:q_Vectors_Inward_Normals}
      \boldsymbol{q}_p = \mathrm{grad}[\varphi_{p_g}] \Big |_{\Omega_e} = 
      \frac{\mathbf{n}_p}{h_p} \quad \forall p = 1,2, \dots, nd+1
    \end{align}
  \item The dihedral angle $\beta_{pq}$ (measured in the Euclidean metric) 
    between any two faces $F_p$ and $F_q$ is related to $\boldsymbol{q}$-vectors 
    ($\mathbf{q}_p$ and $\mathbf{q}_q$) and unit inward normals ($\mathbf{n}_p$ 
    and $\mathbf{n}_q$) by the following equation
    \cite{2005_Li_Huang_JCP_v229_p8072_p8094,2012_Lu_Huang_Qiu_arXiv_1201_3564}:
    %---------------------------------------------------;
    %  Equation: Geometrical property no # 3a           ;
    %            (Dihedral angles in Eucledian metric)  ;
    %---------------------------------------------------; 
    \begin{align}
      \label{Eqn:Dihedral_Angles_EucMetric}
      \cos(\beta_{pq}) = - \mathbf{n}_p \cdot \mathbf{n}_q = 
      - \frac{\mathbf{q}_p \cdot \mathbf{q}_q}{\|\mathbf{q}_p\| \,
      \|\mathbf{q}_q\|} \quad \forall p,q = 1, 2, \dots, nd+1 \quad 
      \mathrm{and} \quad p \neq q
    \end{align}
    where $\| \bullet \|$ is the standard Euclidean norm. Similarly, 
    the dihedral angle $\beta_{pq,\widetilde{\boldsymbol{D}}^{-1}_{\Omega_e}}$ 
    measured in $\widetilde{\boldsymbol{D}}^{-1}_{\Omega_e}$ metric is given as 
    follows:
    %-------------------------------------------------------------;
    %  Equation: Geometrical property no # 3b                     ;
    %            (Dihedral angles in inverse diffusivity metric)  ;
    %-------------------------------------------------------------;
    \begin{align}
      \label{Eqn:Dihedral_Angles_RieMetric}
      \cos(\beta_{pq,\widetilde{\boldsymbol{D}}^{-1}_{\Omega_e}}) =  
      - \frac{\mathbf{q}_p \cdot \widetilde{\boldsymbol{D}}_{\Omega_e} 
      \mathbf{q}_q}{\|\mathbf{q}_p\|_{\widetilde{\boldsymbol{D}}_{\Omega_e}} 
      \, \|\mathbf{q}_q\|_{\widetilde{\boldsymbol{D}}_{\Omega_e}}} \quad 
      \forall p,q = 1, 2, \dots, nd+1 \quad \mathrm{and} \quad p \neq q
    \end{align}
    where $\| \bullet \|_{\widetilde{\boldsymbol{D}}_{\Omega_e}}$ denotes the norm 
    in $\widetilde{\boldsymbol{D}}_{\Omega_e}$ metric. For example, $\| \mathbf{q}_p \|
    _{\widetilde{\boldsymbol{D}}_{\Omega_e}} = \sqrt{\mathbf{q}_p \cdot \widetilde{
    \boldsymbol{D}}_{\Omega_e} \mathbf{q}_p}$.
   \item For any simplex $\Omega_e \in \mathcal{T}_{h}$, the gradient 
     of the linear basis functions ($\mathrm{grad}[\varphi_{p_g}]$ and 
     $\mathrm{grad}[\varphi_{q_g}]$), the dihedral angle $\beta_{pq}$, 
     and heights ($h_p$ and $h_q$) are related as follows 
     \cite{2005_Li_Huang_JCP_v229_p8072_p8094,2012_Lu_Huang_Qiu_arXiv_1201_3564}:
     %------------------------------------------------------------;
     %  Equation: Geometrical property no # 4a (pq-components of  ; 
     %            stiffness matrix of Laplacian)                  ;
     %------------------------------------------------------------;
     \begin{align}
       \label{Eqn:pq_Comp_LaplacianMatrix_1}
       \mathrm{meas}(\Omega_e) \left( \mathrm{grad}[\varphi_{p_g}] 
       \Big |_{\Omega_e} \cdot \mathrm{grad}[\varphi_{q_g}] \Big |_
       {\Omega_e} \right) = &- \frac{\mathrm{meas}(\Omega_e) 
       \cos(\beta_{pq})}{h_p h_q} \nonumber \\
       &\forall p,q = 1, 2, \dots, nd+1 \quad \mathrm{and} \quad p 
       \neq q
     \end{align}
     and in 2D, equation \eqref{Eqn:pq_Comp_LaplacianMatrix_1} reduces to:
     %------------------------------------------------------------;
     %  Equation: Geometrical property no # 4b (pq-components of  ; 
     %            stiffness matrix of Laplacian in 2D)            ;
     %------------------------------------------------------------;
     \begin{align}
       \label{Eqn:pq_Comp_LaplacianMatrix_2}
       \mathrm{meas}(\Omega_e) \left( \mathrm{grad}[\varphi_{p_g}] 
       \Big |_{\Omega_e} \cdot \mathrm{grad}[\varphi_{q_g}] \Big |_
       {\Omega_e} \right) = &- \frac{\cot(\beta_{pq})}{2} \nonumber \\
       &\forall p,q = 1, 2, \dots, nd+1 \quad \mathrm{and} \quad p 
       \neq q
     \end{align}
    \item For any arbitrary simplicial element $\Omega_e \in \mathcal{T}_{h}$, 
      the integral element average anisotropic diffusivity $\widetilde{\boldsymbol{D}}_
      {\Omega_e}$, the gradient of the linear basis functions ($\mathrm{grad}[
      \varphi_{p_g}]$ and $\mathrm{grad}[\varphi_{q_g}]$), the dihedral angle 
      $\beta_{pq,\widetilde{\boldsymbol{D}}^{-1}_{\Omega_e}}$ measured in 
      $\widetilde{\boldsymbol{D}}^{-1}_{\Omega_e}$ metric, and heights ($h_p$ and 
      $h_q$) are related as follows \cite{2012_Lu_Huang_Qiu_arXiv_1201_3564}:
      %-------------------------------------------------------------;
      %  Equation: Geometrical property no # 5a (pq-components of   ; 
      %            stiffness matrix of anisotropic diffusivity)     ;
      %-------------------------------------------------------------;
      \begin{align}
        \label{Eqn:pq_Comp_AnisoDiffMatrix_1}
        \mathrm{meas}(\Omega_e) \left( \mathrm{grad}[\varphi_{p_g}] 
        \Big |_{\Omega_e} \cdot \widetilde{\boldsymbol{D}}_{\Omega_e} \, 
        \mathrm{grad}[\varphi_{q_g}] \Big |_{\Omega_e} \right) = &- 
        \frac{\mathrm{meas}(\Omega_e) \cos(\beta_{pq,\widetilde{\boldsymbol{D}}
        ^{-1}_{\Omega_e}})}{\| \mathbf{q}_p \|^{-1}_{\widetilde{\boldsymbol{D}}_
        {\Omega_e}} \| \mathbf{q}_q \|^{-1}_{\widetilde{\boldsymbol{D}}_{\Omega_e}}} 
        \nonumber \\
        &\forall p,q = 1, 2, \dots, nd+1 \quad \mathrm{and} \quad p 
        \neq q
      \end{align}
      and in 2D, equation \eqref{Eqn:pq_Comp_AnisoDiffMatrix_1} reduces to:
      %----------------------------------------------------------------;
      %  Equation: Geometrical property no # 5b (pq-components of      ; 
      %            stiffness matrix of anisotropic diffusivity) in 2D  ;
      %----------------------------------------------------------------;
      \begin{align}
        \label{Eqn:pq_Comp_AnisoDiffMatrix_2}
        \mathrm{meas}(\Omega_e) \left( \mathrm{grad}[\varphi_{p_g}] 
        \Big |_{\Omega_e} \cdot \widetilde{\boldsymbol{D}}_{\Omega_e} \, 
        \mathrm{grad}[\varphi_{q_g}] \Big |_{\Omega_e} \right) = &- 
        \frac{\mathrm{det}[ \widetilde{\boldsymbol{D}}_{\Omega_e} ]^
        {\frac{1}{2}}}{2} \cot(\beta_{pq,\widetilde{\boldsymbol{D}}^{-1}
        _{\Omega_e}}) \nonumber \\
        &\forall p,q = 1, 2, \dots, nd+1 \quad \mathrm{and} \quad p 
        \neq q
      \end{align} 
\end{enumerate}

%=========================================================================;
%  Subsection: Sufficient conditions for a three-node triangular element  ;
%=========================================================================;
\subsection{Sufficient conditions for a three-node triangular element}
Using the mathematical results outlined in subsection \ref{SubSec:GeoProp_Simplicial_Elements}, 
we shall present various sufficient conditions on the T3 element to satisfy 
different types of discrete properties. In general, there are \emph{two 
different approaches} to obtain sufficient conditions. The first approach, 
which shall be called \emph{global stiffness restriction method}, 
involves manipulating the entries of the global stiffness matrix, so that 
it is either weakly or strictly diagonally dominant (see Theorems 
\ref{Thm:Aniso_Weak_Condition} and \ref{Thm:Delaunay_Weak_Condition}) 
based on the nature of $\alpha(\mathbf{x})$. This means that $\boldsymbol{K}
_{ff}^{-1}$ exists and $\boldsymbol{K}_{ff}^{-1} \succeq \boldsymbol{0}$.
The component wise entries of $\boldsymbol{K}_{ff}$ satisfy the following 
conditions:
%---------------------------------------------------------;
%  Properties of strictly diagonally dominate matrices-1  ;
%---------------------------------------------------------;
\begin{enumerate}[(a)]
  \item Positive diagonal entries: $\left( \boldsymbol{K}_{ff} \right)_{ii} > 0$,
  \item Non-positive off-diagonal entries: $\left( \boldsymbol{K}_{ff} \right)_{ij} 
    \leq 0 \quad \forall i \neq j$, and
\end{enumerate}
one of the following two conditions:
%---------------------------------------------------------;
%  Properties of strictly diagonally dominate matrices-2  ;
%---------------------------------------------------------;
\begin{enumerate}[(c)]
  \item Strict diagonal dominance of rows: $|\left( \boldsymbol{K}_{ff} \right)_{ii}| 
    > \displaystyle \sum \limits_{i \neq j} |\left( \boldsymbol{K}_{ff} \right)_{ij}| 
    \quad \forall i,j$ 
  \item Weak diagonal dominance of rows: $|\left( \boldsymbol{K}_{ff} \right)_{ii}| 
    \geq \displaystyle \sum \limits_{i \neq j} |\left( \boldsymbol{K}_{ff} \right)_{ij}| 
    \quad \forall i,j$
\end{enumerate}
The second method, which shall be called \emph{local stiffness restriction method}, 
engineers on stiffness matrices at the local level, so that they are weakly diagonally 
dominant. Once we ascertain that all of the local stiffness matrices are weakly 
diagonally dominant, then the standard finite element assembly process 
\cite{1989_Wathen_CMAME_v74_p271_p287} guarantees that the global stiffness matrix 
$\boldsymbol{K}_{ff}$ is monotone and weakly diagonally dominant. In particular, 
if $\alpha(\mathbf{x}) > 0$, then $\boldsymbol{K}_{ff}$ is strictly diagonally 
dominant. It should be noted that the above three conditions are sufficient, but 
not necessary, for the global stiffness matrix $\boldsymbol{K}_{ff}$ to be monotone. 

%=================================================================;
%  Subsubsection: Global and local stiffness restriction methods  ;
%=================================================================;
\subsubsection{\textbf{Global and local stiffness restriction methods}}
\label{Subsubsec:Global_Local_Stiffness_Restriction_Methods}
In general, to get an explicit analytical formula for $\boldsymbol{K}_{ff}^{-1}$ 
is extremely difficult and not practically viable. Hence, it is not feasible 
to find mesh restrictions based on the condition that $\boldsymbol{K}_{ff}^{-1} 
\succeq \boldsymbol{0}$. So an expedient route to obtain monotone stiffness matrices 
through mesh restrictions is by means of weakly or strictly diagonally dominant 
matrices, which form a subset to the class of monotone matrices \cite[Section 3, 
Corollary 3.20 and Corollary 3.21]{Varga_MatrixIterativeAnalysis_SecondEdition}. 
The obvious edge being that there is no need to compute $( \boldsymbol{K}_{ff}
^{-1} )_{ij}$ explicitly. Based on the global stiffness restriction method, we 
shall now present \emph{stronger and weaker mesh restriction theorems}. A brief 
and concise proof to the theorems shall be put forth, which is along similar lines 
to the one presented in references \cite{2012_Lu_Huang_Qiu_arXiv_1201_3564,
2012_Huang_arXiv_1306_1987}. These mesh restriction theorems shall be used in 
constructing triangular meshes to satisfy different discrete principles (see 
subsection \ref{SubSec:NumEg_Triangulations}).

%--------------------------------------------------;
%  Theorem: Anisotropic non-obtuse angle condition  ;
%--------------------------------------------------;
\begin{theorem}[\texttt{Anisotropic non-obtuse angle condition}]
  \label{Thm:Aniso_Weak_Condition}
  If any $nd$-simplicial mesh satisfies the following anisotropic non-obtuse 
  angle condition: 
  %----------------------------------------------------------------;
  %  Equation: Anisotropic non-obtuse angle condition (DwMP/DWMP)  ;
  %----------------------------------------------------------------;
  \begin{align}
    \label{Eqn:Aniso_Weak_Condition}
    0 < \frac{h_{p} \, \|\boldsymbol{v}\|_{\infty,\overline{\Omega}_e}}{(nd+1) 
    \, \Lambda_{min,\widetilde{\boldsymbol{D}}_{\Omega_e}}} &+ \frac{h_{p} \, 
    h_{q} \, \|\alpha\|_{\infty,\overline{\Omega}_e}}{(nd+1) \, (nd+2) \, 
    \Lambda_{min,\widetilde{\boldsymbol{D}}_{\Omega_e}}} \leq \cos(\beta_{pq,
    \widetilde{\boldsymbol{D}}^{-1}_{\Omega_e}}) \nonumber \\
    &\forall p,q = 1,2,\cdots,nd+1, 
    \; p \neq q, \; \Omega_e \in \mathcal{T}_h
  \end{align}
  then we have the following three results:
  %------------------------------;
  %  Itemization: Results-1,2,3  ; 
  %------------------------------;
  \begin{itemize}
    \item The global stiffness matrix $\boldsymbol{K}_{ff}$ is (reducibly/irreducibly) 
      weakly diagonally dominant if $\alpha(\mathbf{x}) \geq 0$ and is (reducibly/irreducibly) 
      strictly diagonally dominant if $\alpha(\mathbf{x}) > 0$ for all $\mathbf{x} \in 
      \overline{\Omega}$.
    \item The discrete single-field Galerkin formulation given by equations 
      \eqref{Eqn:Discrete_Bilinear_Form}--\eqref{Eqn:Discrete_Linear_Form} 
      in combination with equations 
      \eqref{Eqn:EleAvg_Aniso_Diff}--\eqref{Eqn:pq_Comp_AnisoDiffMatrix_2} 
       satisfies $\mathrm{DwMP}_{\boldsymbol{K}}$/$\mathrm{DWMP}_{\boldsymbol{K}}$,
       where $\|\boldsymbol{v}\|_{\infty,\overline{\Omega}_e}$ and $\|\alpha\|_{\infty,
       \overline{\Omega}_e}$ 
       are defined as: 
       %---------------------------------------------------------------; 
       %  Equation: Elemental maximum quantities (velocity and decay)  ;
       %---------------------------------------------------------------;
       \begin{subequations}
       \begin{align}
         \label{Eqn:COP_Vel_11}
         \|\boldsymbol{v}\|_{\infty,\overline{\Omega}_e} &:= \displaystyle 
         \mathop{\mbox{maximize}}_{\boldsymbol{x} \in \overline{
         \Omega}_{e}} \; \; \|\boldsymbol{v}(\boldsymbol{x})\| \\
         \label{Eqn:COP_Alpha_11}
         \| \alpha \|_{\infty,\overline{\Omega}_e} &:= \displaystyle \mathop{
         \mbox{maximize}}_{\boldsymbol{x} \in \overline{\Omega}_{e}} 
         \; \; \alpha(\boldsymbol{x})
       \end{align}
       \end{subequations}
       and $\Lambda_{min,\widetilde{\boldsymbol{D}}_{\Omega_e}}$ denotes the minimum 
       eigenvalue of $\widetilde{\boldsymbol{D}}_{\Omega_e}$; $h_p$, $h_q$, and $\beta
       _{pq,\widetilde{\boldsymbol{D}}^{-1}_{\Omega_e}}$ are respectively the heights 
       and metric based dihedral angle opposite to the face $F_r$ of element $\Omega_e$.
    \item Moreover, if the triangulation $\mathcal{T}_{h}$ is interiorly connected, 
      then the global stiffness matrix $\boldsymbol{K}_{ff}$ is irreducibly weakly 
      or strictly diagonally dominant based on the nature of $\alpha(\mathbf{x})$ 
      and the discrete single-field Galerkin formulation given by equations 
      \eqref{Eqn:Discrete_Bilinear_Form}--\eqref{Eqn:Discrete_Linear_Form} 
      satisfies $\mathrm{DsMP}_{\boldsymbol{K}}$/$\mathrm{DSMP}_{\boldsymbol{K}}$.
  \end{itemize}
\end{theorem}
\begin{proof}
  For proof, see References \cite{2012_Lu_Huang_Qiu_arXiv_1201_3564,
  2012_Huang_arXiv_1306_1987}.
\end{proof}

%------------------------------------------------------;
%  Theorem: Generalized Delaunay-type angle condition  ;
%------------------------------------------------------;
\begin{theorem}[\texttt{Generalized Delaunay-type angle condition}]
  \label{Thm:Delaunay_Weak_Condition}
  In 2D, if a simplicial mesh satisfies the following generalized Delaunay-type 
  angle condition (which is much weaker than that of equation 
  \eqref{Eqn:Aniso_Weak_Condition}): 
  %-------------------------------------------------------------------;
  %  Equation: Generalized Delaunay-type angle condition (DwMP/DWMP)  ;
  %-------------------------------------------------------------------;
  \begin{align}
    \label{Eqn:Delaunay_Weak_Condition}
    0 &< \frac{1}{2} \left[ \beta_{pq,\widetilde{\boldsymbol{D}}^{-1}_{\Omega_e}} + 
    \beta_{pq,\widetilde{\boldsymbol{D}}^{-1}_{\Omega^{'}_e}} \right] \nonumber \\
    &+ \frac{1}{2} \mathrm{arccot} 
    \left( \sqrt{\frac{\mathrm{det} [ \widetilde{\boldsymbol{D}}_{\Omega^{'}_e} ]}
    {\mathrm{det} [ \widetilde{\boldsymbol{D}}_{\Omega_e} ]}} \cot(\beta_{pq,
    \widetilde{\boldsymbol{D}}^{-1}_{\Omega^{'}_e}}) - \frac{2 \, \mathfrak{C}_{q,
    \Omega_e,\Omega^{'}_e}}{\sqrt{\mathrm{det} [ \widetilde{\boldsymbol{D}}_{\Omega_e} ]}} 
    \right) \nonumber \\
    &+ \frac{1}{2} \mathrm{arccot} 
    \left( \sqrt{\frac{\mathrm{det} [ \widetilde{\boldsymbol{D}}_{\Omega_e} ]}
    {\mathrm{det} [ \widetilde{\boldsymbol{D}}_{\Omega^{'}_e} ]}} \cot(\beta_{pq,
    \widetilde{\boldsymbol{D}}^{-1}_{\Omega_e}}) - \frac{2 \, \mathfrak{C}_{q,
    \Omega_e,\Omega^{'}_e}}{\sqrt{\mathrm{det} [ \widetilde{\boldsymbol{D}}_{\Omega^{'}_e} ]}} 
    \right) \leq \pi
  \end{align}
  for every internal edge $e_{pq}$ connecting the $p$-th and $q$-th vertices 
  of the elements $\Omega_e$ and $\Omega^{'}_e$ that share this common edge
  (see Figure \ref{Fig:T3_Simplex_GeoProp}), then we have the following three 
  results:
  %------------------------------;
  %  Itemization: Results-4,5,6  ;
  %------------------------------;
  \begin{itemize}
    \item The global stiffness matrix $\boldsymbol{K}_{ff}$ is (reducibly/irreducibly) 
      weakly diagonally dominant if $\alpha(\mathbf{x}) \geq 0$ and is (reducibly/irreducibly) 
      strictly diagonally dominant if $\alpha(\mathbf{x}) > 0$ for all $\mathbf{x} \in 
      \overline{\Omega}$.
    \item The discrete single-field Galerkin formulation given by equations 
      \eqref{Eqn:Discrete_Bilinear_Form}--\eqref{Eqn:Discrete_Linear_Form} 
      in association with equations 
      \eqref{Eqn:EleAvg_Aniso_Diff}--\eqref{Eqn:pq_Comp_AnisoDiffMatrix_2} 
      satisfies $\mathrm{DwMP}_{\boldsymbol{K}}$/$\mathrm{DWMP}_{\boldsymbol{K}}$. 
      The quantities $h_{p,\Omega_e}$ and $h_{q,\Omega_e}$ are the heights of 
      $\Omega_e$, $h_{p,\Omega^{'}_e}$ and $h_{q,\Omega^{'}_e}$ are the heights 
      of $\Omega^{'}_e$, $\beta_{pq,\widetilde{\boldsymbol{D}}^{-1}_{\Omega_e}}$ and 
      $\beta_{pq,\widetilde{\boldsymbol{D}}^{-1}_{\Omega^{'}_e}}$ are the relevant 
      metric based dihedral angles in the elements $\Omega_e$ and $\Omega^{'}_e$ 
      that face the edge $e_{pq}$, and the parameters $\|\boldsymbol{v}\|_{\infty,
      \overline{\Omega}_e}$ and $\|\alpha\|_{\infty,\overline{\Omega}_e}$ are 
      evaluated in $\Omega_e$ based on the equations \eqref{Eqn:COP_Vel_11}--\eqref{Eqn:COP_Alpha_11}. 
      Similarly, $\|\boldsymbol{v}\|_{\infty,\overline{\Omega}^{'}_e}$ and $\|\alpha\|_
      {\infty,\overline{\Omega}^{'}_e}$ are evaluated in $\Omega^{'}_e$. The constant 
      $\mathfrak{C}_{q,\Omega_e,\Omega^{'}_e}$ in equation \eqref{Eqn:Delaunay_Weak_Condition} 
      is given as follows:
      %-------------------------------------------------------------------;
      %  Equation: Constant in generalized Delaunay-type angle condition  ;
      %-------------------------------------------------------------------;
      \begin{align}
        \mathfrak{C}_{q,\Omega_e,\Omega^{'}_e} &:= \mathrm{meas}(\Omega_e) \left( 
        \frac{\|\boldsymbol{v}\|_{\infty,\overline{\Omega}_e}}{3 h_{q,\Omega_e}} +
        \frac{\|\alpha\|_{\infty,\overline{\Omega}_e}}{12} \right) + \mathrm{meas}(
        \Omega^{'}_e) \left ( \frac{\|\boldsymbol{v}\|_{\infty,\overline{\Omega}^
        {'}_e}}{3 h_{q,\Omega^{'}_e}} + \frac{\|\alpha\|_{\infty,\overline{\Omega}
        ^{'}_e}}{12} \right)
      \end{align}
    \item Additionally, if the triangulation $\mathcal{T}_{h}$ is interiorly connected, 
      then the global stiffness matrix $\boldsymbol{K}_{ff}$ is irreducibly weakly 
      or strictly diagonally dominant based on the nature of $\alpha(\mathbf{x})$ 
      and the discrete single-field Galerkin formulation given by equations 
      \eqref{Eqn:Discrete_Bilinear_Form}--\eqref{Eqn:Discrete_Linear_Form} 
      satisfies $\mathrm{DsMP}_{\boldsymbol{K}}$/$\mathrm{DSMP}_{\boldsymbol{K}}$.
  \end{itemize}
\end{theorem}
\begin{proof}
  For proof, see References \cite{2012_Lu_Huang_Qiu_arXiv_1201_3564,
  2012_Huang_arXiv_1306_1987}.
\end{proof}

Each method (global stiffness restriction method or local stiffness restriction 
method) has it own advantages and disadvantages. The advantage of the global 
stiffness restriction method is that we can operate at a global level. This gives 
us different types of relationships between various mesh parameters, which can 
be used in generating different types of triangulations, such as Delaunay-Voronoi, 
non-obtuse, well-centered, and anisotropic $\mathcal{M}$-uniform finite element 
meshes. For instance, Taylor series expansion of equations \eqref{Eqn:Aniso_Weak_Condition} 
and \eqref{Eqn:Delaunay_Weak_Condition} gives the following restrictions on the 
metric based dihedral angles for all simplicial elements in a triangulation:
%---------------------------------------------------------------;
%  Equations: Metric based well-centeredness and Delaunay-Mesh  ;
%---------------------------------------------------------------;
\begin{subequations}
\begin{align}
  \label{Eqn:Metric_WellCentered_Angles}
  0 &< \beta_{pq,\widetilde{\boldsymbol{D}}^{-1}_{\Omega_e}} \leq 
  \frac{\pi}{2} - \mathcal{O} \left( h \|\boldsymbol{v}\|_{\infty,
  \mathcal{T}_{h}} + h^2 \| \alpha \|_{\infty,\mathcal{T}_{h}} \right) \\
  \label{Eqn:Metric_Delaunay_Angles}
  0 &< \frac{1}{2} \left[ \beta_{pq,\widetilde{\boldsymbol{D}}^{-1}_{\Omega_e}} + 
  \beta_{pq,\widetilde{\boldsymbol{D}}^{-1}_{\Omega^{'}_e}} \right] 
  %\\ \nonumber
  %
  + \frac{1}{2} \mathrm{arccot} 
  \left( \sqrt{\frac{\mathrm{det} [ \widetilde{\boldsymbol{D}}_{\Omega^{'}_e} ]}
  {\mathrm{det} [ \widetilde{\boldsymbol{D}}_{\Omega_e} ]}} \cot(\beta_{pq,
  \widetilde{\boldsymbol{D}}^{-1}_{\Omega^{'}_e}}) \right) \nonumber \\
  &+ \frac{1}{2} \mathrm{arccot} 
  \left( \sqrt{\frac{\mathrm{det} [ \widetilde{\boldsymbol{D}}_{\Omega_e} ]}
  {\mathrm{det} [ \widetilde{\boldsymbol{D}}_{\Omega^{'}_e} ]}} \cot(\beta_{pq,
  \widetilde{\boldsymbol{D}}^{-1}_{\Omega_e}}) \right) \leq \pi - \mathcal{O} 
  \left( h \|\boldsymbol{v}\|_{\infty,\mathcal{T}_{h}} + h^2 \| \alpha \|_{\infty,
  \mathcal{T}_{h}} \right)
\end{align}
\end{subequations}
where the maximum element size $h$, maximum element normed velocity $\| \boldsymbol{v} \|
_{\infty,\mathcal{T}_{h}}$, and maximum element normed linear reaction coefficient 
$\| \alpha \|_{\infty,\mathcal{T}_{h}}$ are given as follows: 
%---------------------------------------------------;
%  Equation: Global mesh size, velocity, and decay  ;
%---------------------------------------------------;
\begin{subequations}
\begin{align}
  \label{Eqn:GlobMeshSize}
  h &:= \mathop{\mathrm{max}}_{\Omega_e \in \mathcal{T}_{h}} \left [ h_{\mathrm{max}, 
  \Omega_e} \right] \\ 
  \label{Eqn:GlobVel}
  \| \boldsymbol{v} \|_{\infty,\mathcal{T}_{h}} &:= \mathop{\mathrm{max}}_{\Omega_e 
  \in \mathcal{T}_{h}} \left [ \|\boldsymbol{v}\|_{\infty,\overline{\Omega}_e} \right] 
  \\ 
  \label{Eqn:GlobAlpha}
  \| \alpha \|_{\infty,\mathcal{T}_{h}} &:= \mathop{\mathrm{max}}_{\Omega_e \in 
  \mathcal{T}_{h}} \left [ \|\alpha\|_{\infty,\overline{\Omega}_e} \right]
\end{align}
\end{subequations}
where $h_{\mathrm{max}, \Omega_e}$ is the maximum possible height in a given simplicial 
element $\Omega_e$, and $\mathcal{O}(\bullet)$ is the standard ``big-oh'' notation 
\cite{Heath_Numerical}. Specifically, on a $h$-refined triangular mesh, which conforms 
to Theorems \ref{Thm:Aniso_Weak_Condition} and \ref{Thm:Delaunay_Weak_Condition}, then
equation \eqref{Eqn:Metric_WellCentered_Angles} implies that all the dihedral angles 
when measured in the metric of $\widetilde{\boldsymbol{D}}^{-1}_{\Omega_e}$ have to 
be $\mathcal{O} \left( h \|\boldsymbol{v}\|_{\infty,\mathcal{T}_{h}} + h^2 \| \alpha 
\|_{\infty,\mathcal{T}_{h}} \right)$ acute/non-obtuse and equation \eqref{Eqn:Metric_Delaunay_Angles} indicates that the triangulation needs to be $\mathcal{O} \left( h \|\boldsymbol{v}\|_
{\infty,\mathcal{T}_{h}} + h^2 \| \alpha \|_{\infty,\mathcal{T}_{h}} \right)$ Delaunay. 

One downside of the global stiffness restriction approach is that 
obtaining mesh conditions is mathematically cumbersome. Moreover, 
extending it to non-simplicial low-order finite elements is extremely 
hard and is not straightforward. This is because the basis functions 
$\varphi_{p_g}$ spanning the finite dimensional subsets $\mathcal{C}^{h}$ 
and $\mathcal{W}^{h}$ are multi-linear, which makes $\mathrm{grad}
[\varphi_{p_g}]$ on any arbitrary element $\Omega_e$ to be non-constant. 
So most of properties given by equations 
\eqref{Eqn:EleAvg_Aniso_Diff}--\eqref{Eqn:pq_Comp_AnisoDiffMatrix_2} 
are not valid for low-order finite elements such as the Q4 element and its 
corresponding elements in higher dimensions. 

Now we shall describe the local stiffness restriction method and highlight the pros 
and cons of using this methodology. For the sake of illustration, we shall 
consider a pure anisotropic diffusion equation and assume that $\hat{\boldsymbol{x}}_p 
= (0,0)$, $\hat{\boldsymbol{x}}_q = (1,0)$, and $\hat{\boldsymbol{x}}_r = (a,b)$. Our 
objective is to find the coordinates $(a,b)$, such that the local stiffness matrix is 
weakly diagonally dominant for any given type of diffusivity tensor. The local stiffness 
matrix for an anisotropic diffusion equation based on discrete single-field Galerkin 
formulation is given by:
%--------------------------------------------;
%  Equation: Local element stiffness matrix  ;
%--------------------------------------------;
\begin{align}
\label{Eqn:Local_Element_Stiffness_Matrix}
  \boldsymbol{K}_{e} = \displaystyle \int \limits_{\Omega_e} 
  \boldsymbol{B} \boldsymbol{D} (\boldsymbol{x}) 
  \boldsymbol{B}^{\mathrm{t}} \; \mathrm{d} \Omega
\end{align}
where the matrix $\boldsymbol{B}$ in terms of the coordinates 
$(a,b)$ is given as follows: 
%------------------------------------;
%  Equation: $boldsymbol{B}$ matrix  ;
%------------------------------------;
\begin{align}
\label{Eqn:B_Matrix}
  \boldsymbol{B} = \frac{1}{b} \left(\begin{array}{cc}
    -b & (a-1) \\
    b & -a \\
    0 & 1 \\
  \end{array}\right)
\end{align}
Since the entries in the matrix $\boldsymbol{B}$ 
are constants, we have the local stiffness matrix 
$\boldsymbol{K}_{e}$ given as follows:
%-----------------------------------------------------;
%  Equation: Modified local element stiffness matrix  ;
%-----------------------------------------------------;
\begin{align}
  \boldsymbol{K}_{e} = \boldsymbol{B} \left( \; \displaystyle 
  \int \limits_{\Omega_e} \boldsymbol{D} (\boldsymbol{x}) \; 
  \mathrm{d} \Omega \; \right ) \boldsymbol{B}^{\mathrm{t}}
\end{align}
In the subsequent subsections, we present various 
sufficient conditions through which we can find 
these coordinates $(a,b)$ and glean information 
on the possible angles and corresponding shape and 
size of the triangle PQR.

%==================================================================;
%  Subsection: T3 element for isotropic heterogeneous diffusivity  ;
%==================================================================;
\subsubsection{\textbf{T3 element for heterogeneous isotropic diffusivity}}
\label{SubSec:Heterogeneous_Isotropic_Diffusivity}
In this subsection, we consider the case where the diffusivity is 
isotropic and heterogeneous in the total domain. For this case, we
show that the diffusivity \emph{does not} have any influence on 
determining the coordinates $(a,b)$. This means that the restrictions 
we obtain on the coordinates and the angles of the triangle PQR is 
independent of how the diffusivity is varying across the domain. 
The following is the local stiffness matrix for scalar heterogeneous 
isotropic diffusion:
%--------------------------------------------------------------;
%  Equation: Local stiffness matrix for isotropic diffusivity  ;
%--------------------------------------------------------------;
\begin{align}
  \boldsymbol{K}_{e} = \frac{1}{b^2} \displaystyle \int 
  \limits_{\Omega_e} D (\boldsymbol{x}) \; \mathrm{d} \Omega  
  \left(\begin{array}{ccc}
    b^2 + (a-1)^2 & a - a^2 - b^2 & (a-1) \\
    a - a^2 - b^2 & a^2 + b^2 & -a \\
    (a-1) & -a & 1 \\
  \end{array}\right)
\end{align}
where the integral of the diffusivity $D (\boldsymbol{x})$ over 
the actual T3 element $\Omega_e$ (triangle PQR) is given as 
follows:
%-----------------------------------------;
%  Equation: Integral of the diffusivity  ;
%-----------------------------------------; 
\begin{align}
  \widetilde{D} = \frac{1}{\mathrm{meas}(\Omega_e)} \displaystyle 
  \int \limits_{\Omega_e} D (\boldsymbol{x}) \; \mathrm{d} 
  \Omega
\end{align}
where $\mathrm{meas}(\Omega_e) = \frac{b}{2}$ is the area of the T3 
element PQR, which is always positive. The simplified expression for 
the local stiffness matrix is given as follows:
%-------------------------------------------------------------------------;
%  Equation: Simplified local stiffness matrix for isotropic diffusivity  ;
%-------------------------------------------------------------------------;
\begin{align}
  \boldsymbol{K}_{e} = \frac{\widetilde{D}}{2b}  
  \left(\begin{array}{ccc}
    b^2 + (a-1)^2 & a - a^2 - b^2 & (a-1) \\
    a - a^2 - b^2 & a^2 + b^2 & -a \\
    (a-1) & -a & 1 \\
  \end{array}\right)
\end{align}
We shall now present the sufficient conditions so that the 
matrix $\boldsymbol{K}_{e}$ is weakly diagonally dominant:

%==============================;
%  Subsection: Condition No-1  ;
%==============================;
\subsection*{Condition No-1}
Positive diagonal entries: $\left( \boldsymbol{K}_e \right)_{ii}
> 0 \quad \forall i = 1,2,3$. This restriction gives us the 
following inequalities:
%---------------------------------------;
%  Equation: Sufficient Condition No-1  ;
%---------------------------------------; 
\begin{subequations}
\begin{align}
  \label{Eqn:Diagonal_Entry_11}
  \left( \boldsymbol{K}_e \right)_{11} &= \frac{\widetilde{D}}{2b} 
  \left( b^2 + (a-1)^2 \right) > 0\\
  \label{Eqn:Diagonal_Entry_22}
  \left( \boldsymbol{K}_e \right)_{22} &= \frac{\widetilde{D}}{2b}
  \left( a^2 + b^2 \right) > 0\\
  \label{Eqn:Diagonal_Entry_33}
  \left( \boldsymbol{K}_e \right)_{33} &= \frac{\widetilde{D}}{2b}
  > 0
\end{align}
\end{subequations}
As $\widetilde{D} > 0$ and $b > 0$, it is evident that all of the 
inequalities given by equations 
\eqref{Eqn:Diagonal_Entry_11}--\eqref{Eqn:Diagonal_Entry_33} 
are trivially satisfied. Hence, this condition has no effect on 
obtaining restrictions on coordinates $(a,b)$. 

%==============================;
%  Subsection: Condition No-2  ;
%==============================;
\subsection*{Condition No-2}
Weak diagonal dominance of rows: $|\left( \boldsymbol{K}_e 
\right)_{ii}| \geq \displaystyle \sum \limits_{i \neq j} |\left( 
\boldsymbol{K}_e \right)_{ij}| \quad \forall i,j$, where $i = 1,2,3$ 
and $j = 1,2,3$. This restriction gives the following inequalities:
%---------------------------------------;
%  Equation: Sufficient Condition No-2  ;
%---------------------------------------;
\begin{subequations}
\begin{align}
  \label{Eqn:Diagonal_Entry_11_12_13}
  \left( b^2 + (a-1)^2 \right) &\geq \left( a^2 + b^2 - a\right) + 
  \left( 1 - a \right) \\
  \label{Eqn:Diagonal_Entry_22_21_23} 
  \left( a^2 + b^2 \right) &\geq a + \left( a^2 + b^2 - a\right) \\
  \label{Eqn:Diagonal_Entry_33_31_32}
  1 &\geq  \left( 1 - a \right) + a
\end{align}
\end{subequations}
Note that these inequalities
\eqref{Eqn:Diagonal_Entry_11_12_13}--\eqref{Eqn:Diagonal_Entry_33_31_32}
are trivially satisfied. Hence, this condition has no influence on
obtaining restrictions on triangle PQR.

%==============================;
%  Subsection: Condition No-3  ;
%==============================;
\subsection*{Condition No-3}
Non-positive off-diagonal entries: $\left( \boldsymbol{K}_e 
\right)_{ij} \leq 0 \quad \forall i \neq j$, where $i = 1,2,3$ 
and $j = 1,2,3$. This restriction gives the following inequalities:
%--------------------------------------------------;
%  Equation: Sufficient Condition No-3 unformatted  ;
%--------------------------------------------------;
\begin{align*} 
  \left( \boldsymbol{K}_e \right)_{ij} = 
  \begin{cases}
    \left( \boldsymbol{K}_e \right)_{12} &= \frac{\widetilde{D}}{2b} 
    \left( a - a^2 - b^2 \right) \leq 0 \\
    \left( \boldsymbol{K}_e \right)_{13} &= 
    \frac{\widetilde{D}}{2b} \left( a-1 \right) \leq 0 \\
    \left( \boldsymbol{K}_e \right)_{23} &= \frac{\widetilde{D}}{2b} 
    \left( -a \right) \leq 0
  \end{cases} \quad \forall i \neq j
\end{align*}
Using the fact that $\widetilde{D} > 0$ and $b > 0$, and rearranging 
the above relations, we get the following inequalities:
%------------------------------------------------;
%  Equation: Sufficient Condition No-3 formate  ;
%------------------------------------------------;
\begin{subequations}
  \begin{align}
    \label{Eqn:Diagonal_Entry_12}
    \left( a - \frac{1}{2} \right )^2 + b^2 &\geq \left( \frac{1}{2} 
    \right )^2 \\
    \label{Eqn:Diagonal_Entry_13}
    a &\leq 1 \\
    \label{Eqn:Diagonal_Entry_23}
    a &\geq 0 
\end{align}
\end{subequations}
The region in which the coordinates $(a,b)$ satisfy the above 
inequalities given by the equations 
\eqref{Eqn:Diagonal_Entry_12}--\eqref{Eqn:Diagonal_Entry_23} 
is shown in Figure \ref{Fig:T3_Isotropic_Actual_Element}. 
According to these inequalities 
\eqref{Eqn:Diagonal_Entry_12}--\eqref{Eqn:Diagonal_Entry_23}, 
heterogeneity of the scalar diffusivity has \emph{no role} 
in obtaining the feasible region for the coordinates $(a,b)$.
It is evident from Figure \ref{Fig:T3_Isotropic_Actual_Element} 
that the interior angles of the triangle PQR are either acute 
or at most right-angle. Based on the sufficient conditions, 
one can also notice that an obtuse-angled triangle is not possible. 
So in order to satisfy discrete comparison principles, discrete 
maximum principles, and non-negative constraint, the triangulation 
of a given computational domain must contain acute-angled triangles 
or right-angled triangles. These three sufficient conditions show 
that non-obtuse or well-centered triangulations inherit all the 
three discrete versions of continuous properties of scalar 
heterogeneous isotropic diffusion equations.

%====================================================================;
%  Subsection: T3 element for heterogeneous anisotropic diffusivity  ;
%====================================================================;
\subsubsection{\textbf{T3 element for heterogeneous anisotropic diffusivity}}
In this subsection, we consider the case where the diffusivity $
\mathbf{D} (\mathbf{x})= \left(\begin{smallmatrix}
D_{xx} (\mathbf{x}) & D_{xy} (\mathbf{x}) \\ 
D_{xy} (\mathbf{x}) & D_{yy} (\mathbf{x})
\end{smallmatrix} \right)$ is anisotropic and heterogeneous across the 
domain. For the sake of brevity and ease of manipulations, we shall drop the 
symbol $(\mathbf{x})$ in the components of the diffusivity tensor.
Note that the symbol `$\mathbf{x}$' in the components of $\mathbf{D} 
(\mathbf{x})$ is dropped for the sake of convenience and \emph{should not} be 
interpreted as though the diffusivity tensor is constant. As discussed in 
Section \ref{Sec:S2_MP_Diffusion}, the diffusivity tensor needs to satisfy 
certain properties. Based on equations \eqref{Eqn:Diffusivity_Symmetric} and \eqref{Eqn:Diffusivity_PositiveDefinite}, we derive various results related 
to $\mathbf{D} (\mathbf{x})$ that will be used in deriving mesh restrictions.    

%-----------------------------------------------------------------------------;
%  Remark: Necessary conditions on the components of anisotropic diffusivity  ;
%-----------------------------------------------------------------------------;
\begin{remark}
  In 2D, it is trivial to show that, if the matrix $\boldsymbol{D} (\boldsymbol{x})$ 
  is symmetric, uniformly elliptic, and bounded above, then its components satisfy 
  the following relations:
  %------------------------------------------------------------;
  %  Equations: Restrictions on diffusivity tensor components  ;
  %------------------------------------------------------------;
  \begin{subequations}
    \begin{align}
      \label{Eqn:D_xx}
      & D_{xx} > 0 \\
      \label{Eqn:D_yy}
      & D_{yy} > 0 \\
      \label{Eqn:D_xxD_yy}
      & D_{xx} D_{yy} > D_{xy}^2
    \end{align}
  \end{subequations}
\end{remark}

Let us denote $\epsilon := \frac{\widetilde{D}_{yy}}{\widetilde{D}_{xx}}$ and 
$\eta := \frac{\widetilde{D}_{xy}}{\widetilde{D}_{xx}}$, where $\widetilde{D}_{xx}$,
$\widetilde{D}_{xy}$, and $\widetilde{D}_{yy}$ are the components of matrix 
$\widetilde{\boldsymbol{D}}_{\Omega_e}$ given by equation \eqref{Eqn:EleAvg_Aniso_Diff}. 
From Theorem \ref{Thm:Positive_LinearMap_AnisoDiff}, it is evident that $\widetilde{D}_{xx} 
> 0$, $\widetilde{D}_{yy} > 0$, and $\widetilde{D}_{xx} 
\widetilde{D}_{yy} > \widetilde{D}_{xy}^2$. So from equation \eqref{Eqn:D_xxD_yy}, we have 
$\eta \in \left( -\sqrt{\epsilon}, \sqrt{\epsilon} \right)$. These two non-dimensional 
quantities $\epsilon$ and $\eta$ govern the mesh restrictions that we impose on the 
coordinates $(a,b)$. From equations \eqref{Eqn:Local_Element_Stiffness_Matrix} 
and \eqref{Eqn:B_Matrix}, the stiffness matrix for any given anisotropic diffusivity 
tensor is given as follows:
%----------------------------------------------------------------;
%  Equation: Local stiffness matrix for anisotropic diffusivity  ;
%----------------------------------------------------------------;
\begin{align}
  \boldsymbol{K}_{e} = 
  \left( \begin{array}{ccc}
    \frac{\widetilde{D}_{xx} b^2 - 2 \widetilde{D}_{xy} b(a-1) + 
    \widetilde{D}_{yy} (a-1)^2}{2b} &
    -\frac{\widetilde{D}_{xx} b^2 + \widetilde{D}_{xy} (b - 2ab) + 
    \widetilde{D}_{yy} a(a-1)}{2b} &
    \frac{-\widetilde{D}_{xy} b + \widetilde{D}_{yy} (a-1)}{2b} \\
    -\frac{\widetilde{D}_{xx} b^2 + \widetilde{D}_{xy} (b - 2ab) + 
    \widetilde{D}_{yy} a(a-1)}{2b} & 
    \frac{\widetilde{D}_{xx} b^2 - 2 \widetilde{D}_{xy} ab + 
    \widetilde{D}_{yy} a^2}{2b} & 
    \frac{\widetilde{D}_{xy} b - \widetilde{D}_{yy} a}{2b} \\
    \frac{-\widetilde{D}_{xy} b + \widetilde{D}_{yy} (a-1)}{2b} &
    \frac{\widetilde{D}_{xy} b - \widetilde{D}_{yy} a}{2b} & 
    \frac{\widetilde{D}_{yy}}{2b} \\
  \end{array} \right)
\end{align}
We now present sufficient conditions so that the matrix $\boldsymbol{K}_{e}$ 
is weakly diagonally dominant.

%==============================;
%  Subsection: Condition No-4  ;
%==============================;
\subsection*{Condition No-4}
Positive diagonal entries: $\left( \boldsymbol{K}_e \right)_{ii} > 0 \quad 
\forall i = 1,2,3$, gives the following relations:
%------------------------------------------;
%  Equation: Condition No - 4 unformatted  ;
%------------------------------------------;
\begin{align*} 
  \left( \boldsymbol{K}_e \right)_{ii} = 
  \begin{cases}
    \left( \boldsymbol{K}_e \right)_{11} &= 
    \frac{\widetilde{D}_{xx} b^2 - 2 \widetilde{D}_{xy} b(a-1) + 
    \widetilde{D}_{yy} (a-1)^2}{2b} > 0 \\
    \left( \boldsymbol{K}_e \right)_{22} &= 
    \frac{\widetilde{D}_{xx} b^2 - 2 \widetilde{D}_{xy} ab + 
    \widetilde{D}_{yy} a^2}{2b} > 0\\
    \left( \boldsymbol{K}_e \right)_{33} &= 
    \frac{\widetilde{D}_{yy}}{2b} > 0
  \end{cases}
\end{align*} 
As $\widetilde{D}_{yy} > 0$ and $b > 0$, rearranging the above relations, we 
have the following restrictions:
%----------------------------------------;
%  Equation: Condition No - 4 formatted  ;
%----------------------------------------;
\begin{subequations}
\begin{align}
  \label{Eqn:Aniso_Diagonal_Entry_11}
  \left( \boldsymbol{K}_e \right)_{11} &= 
  \left( b \sqrt{\widetilde{D}_{xx}} - |a-1| \sqrt{\widetilde{D}_{yy}} 
  \right)^2 + 2 b |a-1| \left( \sqrt{\widetilde{D}_{xx}} \sqrt{\widetilde{D}_{yy}} 
  - \mathrm{Sgn} \left[ |a-1| \right] \widetilde{D}_{xy}\right) > 0 \\
  \label{Eqn:Aniso_Diagonal_Entry_22}
  \left( \boldsymbol{K}_e \right)_{22} &= 
  \left( b \sqrt{\widetilde{D}_{xx}} - |a| \sqrt{\widetilde{D}_{yy}} 
  \right)^2 + 2 b |a| \left( \sqrt{\widetilde{D}_{xx}} \sqrt{\widetilde{D}_{yy}} 
  -\mathrm{Sgn} \left[ |a| \right] \widetilde{D}_{xy}\right) > 0 \\
  \label{Eqn:Aniso_Diagonal_Entry_33}
  \left( \boldsymbol{K}_e \right)_{33} &= 
  \frac{\widetilde{D}_{yy}}{2b} > 0
\end{align}
\end{subequations}
where $\mathrm{Sgn}[\bullet]$ is the standard lignum function (which provides 
the sign of the real number). It is evident that $\sqrt{\widetilde{D}_{xx}} \sqrt{\widetilde{D}_{yy}} > \widetilde{D}
_{xy}$. Hence, equations \eqref{Eqn:Aniso_Diagonal_Entry_11}-\eqref{Eqn:Aniso_Diagonal_Entry_33} 
are trivially satisfied for any abscissa $a$.

%==============================;
%  Subsection: Condition No-5  ;
%==============================;
\subsection*{Condition No-5}
Non-positive off-diagonal entries: $\left( \boldsymbol{K}_e \right)_{ij} 
\leq 0 \quad \forall i \neq j$, where $i = 1,2,3$, and $j = 1,2,3$. This 
restriction gives the following relations:
%-----------------------------------------;
%  Equation: Condition No - 5 unformatted  ;
%-----------------------------------------;
\begin{align*} 
  \left( \boldsymbol{K}_e \right)_{ij} = 
  \begin{cases}
    \left( \boldsymbol{K}_e \right)_{12} &= 
    -\frac{\widetilde{D}_{xx} b^2 + \widetilde{D}_{xy} (b - 2ab) + 
    \widetilde{D}_{yy} a(a-1)}{2b} \leq 0 \\
    \left( \boldsymbol{K}_e \right)_{13} &= 
    \frac{-\widetilde{D}_{xy} b + \widetilde{D}_{yy} (a-1)}{2b} 
    \leq 0\\
    \left( \boldsymbol{K}_e \right)_{23} &= 
    \frac{\widetilde{D}_{xy} b - \widetilde{D}_{yy} a}{2b} 
    \leq 0
  \end{cases}
\end{align*}
Using the parameters $\epsilon$, $\eta$, and the fact that 
ordinate $b > 0$, we have the following inequalities: 
%---------------------------------------;
%  Equation: Condition No - 5 formated  ;
%---------------------------------------;
\begin{subequations}
\begin{align}
  \label{Eqn:Aniso_Diagonal_Entry_12}
  \left( a - \frac{1}{2} \right )^2 + \left( \frac{b}{\sqrt{\epsilon}} 
  \right)^2 - 2 b \left( \frac{\eta}{\epsilon} \right) \left( a - 
  \frac{1}{2} \right )&\geq \left( \frac{1}{2} \right )^2 \\
  \label{Eqn:Aniso_Diagonal_Entry_13}
   \frac{a-1}{b} &\leq \frac{\eta}{\epsilon} \\
   \label{Eqn:Aniso_Diagonal_Entry_23}
   \frac{a}{b} &\geq \frac{\eta}{\epsilon}
\end{align}
\end{subequations}
which dictate the feasible region for coordinates 
$(a,b)$. For a given $\epsilon$ and by varying $\eta$, which 
lies between $-\sqrt{\epsilon}$ and $\sqrt{\epsilon}$, we get 
different feasible regions for $(a,b)$. Herein, we have chosen 
$\epsilon = 10$ and $\eta \in \{-1,0,1\}$. For these values, 
we have plotted the feasible region based on the inequalities 
\eqref{Eqn:Aniso_Diagonal_Entry_12}-\eqref{Eqn:Aniso_Diagonal_Entry_23}.
From Figures \ref{Fig:T3_AnisoDiff_Dxy_EqualToZero}--\ref{Fig:T3_ConstEpsilon_VaryEta}, 
the following can be inferred based on the feasible region:
\begin{itemize}
  \item If $\eta = 0$, the possible T3 elements are either acute-angled 
    or right-angled triangles.
   \item If either $\eta < 0$ or $\eta > 0$, then obtuse-angled triangles 
   are also possible. Moreover, the resulting triangles can be \emph{skinny 
   or skewed}.
\end{itemize} 

%==============================;
%  Subsection: Condition No-6  ;
%==============================;   
\subsection*{Condition No-6}
Weak diagonal dominance of rows: $|\left( \boldsymbol{K}_e \right)_{ii}| \geq 
\displaystyle \sum \limits_{i \neq j} |\left( \boldsymbol{K}_e \right)_{ij}| 
\quad \forall i,j$, where $i = 1,2,3$, and $j = 1,2,3$. This gives the following 
relations:
%------------------------------;
%  Equation: Condition No - 6  ;
%------------------------------;
\begin{subequations}
\begin{align}
  \label{Eqn:AnisoDiff_Diagonal_Entry_11_12_13}
  \left( b^2 - 2 \eta b(a-1) + \epsilon (a-1)^2 \right) 
  &\geq \left( b^2 + \eta (b - 2ab) + \epsilon a(a-1) 
  \right) + \left( \eta b - \epsilon (a-1) \right) \\
  \label{Eqn:AnisoDiff_Diagonal_Entry_22_21_23}
  \left( b^2 - 2 \eta a b + \epsilon a^2 \right) &\geq 
  \left( b^2 + \eta (b - 2ab) + \epsilon a(a-1) \right)
  \left( \epsilon a - \eta b \right) \\
  \label{Eqn:AnisoDiff_Diagonal_Entry_33_31_32}
  \epsilon &\geq \left( \eta b - \epsilon (a-1) \right) 
  + \left( \epsilon a - \eta b \right)
\end{align}
\end{subequations}
if Condition-4 and Condition-5 are satisfied, then this condition is trivially 
satisfied. In a similar fashion, for a general case, where $\hat{\boldsymbol{x}}
_p = (x_1,y_1)$, $\hat{\boldsymbol{x}}_q = (x_2,y_2)$, and $\hat{\boldsymbol{x}}_r 
= (x_3,y_3)$, we have the following conditions based on the local stiffness restriction 
method:
%----------------------------------------------------------------------;
%  General case: Local stiffness restriction method using Mathematica  ;  
%----------------------------------------------------------------------;
\begin{subequations}
\begin{align}
  \label{Eqn:Aniso_LSRM_1}
  (y_1 - y_3)(y_3 - y_2) - \eta (x_1 - x_3)(y_3 - y_2) - \eta (x_3 - x_2)
  (y_1 - y_3) + \epsilon (x_1 - x_3)(x_3 - x_2) &\leq 0 \\
  \label{Eqn:Aniso_LSRM_2}
  (y_2 - y_1)(y_3 - y_2) - \eta (x_3 - x_2)(y_2 - y_1) - \eta (x_2 - x_1)
  (y_3 - y_2) + \epsilon (x_2 - x_1)(x_3 - x_2) &\leq 0 \\
  \label{Eqn:Aniso_LSRM_3}
  (y_1 - y_3)(y_2 - y_1) - \eta (x_1 - x_3)(y_2 - y_1) - \eta (x_2 - x_1)
  (y_1 - y_3) + \epsilon (x_1 - x_3)(x_2 - x_1) &\leq 0 \\
  \label{Eqn:Aniso_LSRM_4}
  (x_1 - x_3)(y_2 - y_1) - (x_1 - x_2)(y_3 - y_1) &> 0
\end{align}
\end{subequations}  
where the first three inequalities given by equations 
\eqref{Eqn:Aniso_LSRM_1}--\eqref{Eqn:Aniso_LSRM_3} are 
obtained based on the condition that $\left( \boldsymbol{K}_e 
\right)_{ij} \leq 0$. The last inequality given by equation 
\eqref{Eqn:Aniso_LSRM_4} is the result of the condition that 
$\mathrm{meas}(\Omega_e) > 0$.
  
The benefit (attractive feature) of the local stiffness restriction method 
is that the local stiffness matrix for the discrete Galerkin formulation 
given by equations \eqref{Eqn:Discrete_Bilinear_Form}--\eqref{Eqn:Discrete_Linear_Form} 
can be calculated quite easily and could be extended to even non-simplicial 
elements (see subsection \ref{SubSec:MeshRestrict_Rectangular_Element}). 
Using this approach, we can obtain general restrictions and analytical 
expressions relating various coordinates of an arbitrary mesh element, 
using popular symbolic packages like \textsf{Mathematica} \cite{Mathematica}. 
But a flip-side of this procedure is that incorporating the inequalities given 
by equations \eqref{Eqn:Aniso_LSRM_1}-\eqref{Eqn:Aniso_LSRM_4} in a mesh generator 
is very difficult and needs further detailed investigation. Additionally, the 
conditions obtained using this method are stringent and similar to that of Theorem 
\ref{Thm:Aniso_Weak_Condition} given by the global stiffness restriction method (which
will be evident based on the numerical examples discussed in the following subsection). 
Finally, it should be noted that extending the local stiffness restriction method to 
include advection and linear reaction is straightforward and shall not be dealt with
to save space. We shall now present various numerical examples and respective triangular 
meshes corresponding to different types of $\mathbf{D}(\mathbf{x})$. Using these meshes, 
we shall analyze and study in detail, which kind of DMPs and DCPs are preserved.

%=============================================================================;
%  Subsection: Numerical examples based on different types of triangulations  ;
%=============================================================================;
\subsection{Numerical examples based on different types of triangulations}
\label{SubSec:NumEg_Triangulations}
In this subsection, we shall first briefly discuss on a metric tensor $\mathcal{M}
(\boldsymbol{x})$ to satisfy DCPs, DMPs, and NC. Based on this metric tensor, we 
shall describe an algorithm to generate various types of DMP-based triangulations 
(mainly utilizing open source mesh generators, such as \textsf{Gmsh} and \textsf{BAMG}). 
Simplicial meshes constructed based on $\mathcal{M}(\boldsymbol{x})$ (where the metric 
tensor $\mathcal{M}(\boldsymbol{x})$ is \emph{not equal} to a scalar multiple of identity
tensor) are called anisotropic $\mathcal{M}$-uniform simplicial meshes. They are uniform 
in the metric specified by $\mathcal{M}(\boldsymbol{x})$ \cite{2005_Huang_JCP_v204_p633_p665,
Frey_George} and are of primal importance in satisfying various important discrete properties 
in the areas of transport of chemical species, fluid mechanics, and porous media applications 
\cite{1997_Castro_IJNMF_v25_p475_p491,2010_Huang_Li_FEAD_v46_p61_p73,
2005_Frey_Alauzet_CMAME_v198_p5068_p5082}. Given a $\mathcal{M}(\boldsymbol{x})$, 
there are different approaches to generate anisotropic $\mathcal{M}$-uniform 
simplicial meshes. Some of the notable research works in this direction include 
blue refinement \cite{1990_Kornhuber_Roitzsch_ICSE_v2_p40_p72}, 
bubble packing \cite{2003_Yamakawa_Shimada_IJNME_v57_p1923_p1942}, 
Delaunay-type triangulation \cite{1997_Castro_IJNMF_v25_p475_p491}, directional refinement 
\cite{1993_Rachowicz_CMAME_v109_p169_p181}, front advancing 
\cite{2003_Yamakawa_Shimada_IJNME_v57_p1923_p1942}, local refinement and modification 
\cite{1996_Bossen_Heckbert_IMR_p63_p74}, and variational mesh generation 
\cite{2001_Huang_JCP_v174_p903_p924}. In general, the metric tensor $\mathcal{M}(\boldsymbol{x})$ 
is symmetric and positive definite, and gives relevant information on the shape, size, 
and orientation of mesh elements in the computational domain.

Let $\bigchi$ be an affine mapping from the reference element $\Omega_{\mathrm{\tiny ref}}$ 
to background mesh element $\Omega_e$. Denote its Jacobian by $\boldsymbol{J}_{\Omega_e}$. 
The affine mapping $\bigchi$ and its Jacobian $\boldsymbol{J}_{\Omega_e}$ are given as follows:
%----------------------------------------------;
%  Equations: Affine mapping and its jacobian  ;
%----------------------------------------------;
\begin{subequations}
\begin{align}
  \label{Eqn:Affine_Mapping}
    \boldsymbol{x} &= \bigchi ( \hat{\boldsymbol{X}} , 
    \boldsymbol{N} ) = \hat{\boldsymbol{X}}^{\mathrm{T}} \boldsymbol{N}
    ^{\mathrm{T}} \\
  \label{Eqn:Jacobian}
    \boldsymbol{J}_{\Omega_e} &= \hat{\boldsymbol{X}}^{\mathrm{T}} \boldsymbol{DN}
\end{align}
\end{subequations}
where $\hat{\boldsymbol{X}}$ is the nodal matrix, which comprises of nodal vertices 
($\hat{\boldsymbol{x}}_{1}, \hat{\boldsymbol{x}}_{2}, \dots, \hat{\boldsymbol{x}}_{nd+1}$) 
of an arbitrary simplex $\Omega_e$. The elements of vector $\boldsymbol{N}$ consists of 
shape functions corresponding to $\Omega_{\mathrm{\tiny ref}}$. The entries of matrix 
$\boldsymbol{DN}$ (which are constants for simplicial elements) correspond to the 
derivatives of shape functions. Straightforward manipulations on equations \eqref{Eqn:q_Vectors_Inward_Normals} and \eqref{Eqn:Jacobian} give the following result:
%-------------------------------------------------------------;
%  Equations: q-vectors on reference and background elements  ;
%-------------------------------------------------------------;
\begin{align}
  \label{Eqn:qpref_qp}
  \boldsymbol{q}_{p,{\mathrm{\tiny ref}}} = \boldsymbol{J}_{\Omega_e} 
  \boldsymbol{q}_p
\end{align} 
where $\boldsymbol{q}_{p,{\mathrm{\tiny ref}}}$ is the corresponding $\boldsymbol{q}$-vector 
of the reference element $\Omega_{\mathrm{\tiny ref}}$. Huang \cite{2005_Huang_JCP_v204_p633_p665} 
has shown that an anisotropic $\mathcal{M}$-uniform simplicial mesh satisfies 
the following two conditions:
%-------------------------------------------------------;
%  Equation: Equidistribution and alignment conditions  ;
%-------------------------------------------------------;
\begin{subequations} 
\begin{align}
  \label{Eqn:Equidistribution_Condition}
  \bigrho_{\Omega_e} \mathrm{meas}(\Omega_e) &= \frac{\bigsigma_h}{Nele} 
  \quad \forall \Omega_e \in \mathcal{T}_{h} \\
  \label{Eqn:Alignment_Condition}
  \frac{1}{nd}\mathrm{tr} \left[\boldsymbol{J}^{\mathrm{T}}_{\Omega_e} 
  \mathcal{M}_{\Omega_e} \boldsymbol{J}_{\Omega_e} \right] &=  \mathrm{det} 
  \left[\boldsymbol{J}^{\mathrm{T}}_{\Omega_e} \mathcal{M}_{\Omega_e} 
  \boldsymbol{J}_{\Omega_e} \right]^{\frac{1}{nd}} \quad \forall \Omega_e 
  \in \mathcal{T}_{h}
\end{align}
\end{subequations}
The quantities $\mathcal{M}_{\Omega_e}$, $\bigrho_{\Omega_e}$, and $\bigsigma_h$ 
are given as follows:
%------------------------------------------------------;
%  Equation: Equidistribution and alignment variables  ;
%------------------------------------------------------;
\begin{subequations} 
\begin{align}
  \label{Eqn:Metric_Element_General}
  \mathcal{M}_{\Omega_e} &:= \frac{1}{\mathrm{meas}(\Omega_e)} 
  \int \limits_{\Omega_e} \mathcal{M} (\boldsymbol{x}) \mathrm{d} 
  \Omega \quad \forall \Omega_e \in \mathcal{T}_{h} \\
  \label{Eqn:Rho_Element_General}
  \bigrho_{\Omega_e} &:= \sqrt{\mathrm{det} \left[\mathcal{M}_{\Omega_e} 
  \right]} \quad \forall \Omega_e \in \mathcal{T}_{h} \\
  \label{Eqn:Sigma_h_General}
  \bigsigma_h &:= \sum \limits_{\Omega_e \in \mathcal{T}_{h}} \bigrho_{\Omega_e} 
  \mathrm{meas}(\Omega_e)
\end{align}
\end{subequations}
The condition given by the equation \eqref{Eqn:Equidistribution_Condition} is called
\emph{equidistribution condition} and that by equation \eqref{Eqn:Alignment_Condition} 
is called \emph{alignment condition}. Equidistribution condition decides the size of 
$\Omega_e$, while alignment condition characterizes the shape and orientation of $\Omega_e$. 
From AM-GM inequality \cite{Cloud_Drachman}, equation \eqref{Eqn:Alignment_Condition} 
implies the following:
%------------------------------------------;
%  Equation: Metric tensor relationship-1  ;
%------------------------------------------;
\begin{align}
  \label{Eqn:Scalar_Metric_Tensor}
  \boldsymbol{J}^{\mathrm{T}}_{\Omega_e} \mathcal{M}_{\Omega_e} 
  \boldsymbol{J}_{\Omega_e} = \left( \frac{\bigsigma_h}{Nele} 
  \right)^{\frac{2}{nd}} \boldsymbol{I} \quad \forall \Omega_e \in 
  \mathcal{T}_{h}
\end{align} 
Now, through trivial manipulations on equations \eqref{Eqn:qpref_qp}, 
\eqref{Eqn:Dihedral_Angles_RieMetric}, and \eqref{Eqn:Scalar_Metric_Tensor}, 
the metric tensor $\mathcal{M}_{\Omega_e}$ has to satisfy the following equation 
in order to meet various DMPs:
%-------------------------------------------;
%  Equation: Metric tensor to satisfy DMPs  ;
%-------------------------------------------;
\begin{align}
  \label{Eqn:Element_Metric_DMP}
  \mathcal{M}_{\Omega_e} = \Theta_{\Omega_e} 
  \widetilde{\boldsymbol{D}}^{-1}_{\Omega_e}
\end{align}
where $\Theta_{\Omega_e}$ is an arbitrary piecewise positive scalar constant, which 
in general is a user-define parameters. Loosely speaking, an anisotropic $\mathcal{M}$-uniform 
simplicial mesh satisfying \eqref{Eqn:Element_Metric_DMP} satisfies weak DMPs. Furthermore, 
if the resulting simplicial mesh is \emph{interiorly connected}, then it satisfies strong 
DMPs.

For a given type of metric tensor (for example, $\mathcal{M}(\boldsymbol{x})$ given by 
equation \eqref{Eqn:Element_Metric_DMP}), open source mesh generators (such as \textsf{BAMG}, 
\textsf{BL2D}, and \textsf{Mmg3d}) take this as an input and operate on a background mesh 
to produce an anisotropic $\mathcal{M}$-uniform simplicial mesh. Algorithm \ref{Algo:Aniso_Meshes} 
provides a methodology to develop an anisotropic $\mathcal{M}$-uniform simplicial mesh 
based on a background mesh, and Figure \ref{Fig:Isotropic_Anisotropic_T3_Element} highlights 
the salient aspects of this algorithm. Nevertheless, it should be noted that Algorithm \ref{Algo:Aniso_Meshes} is a general algorithm to generate DMP-based triangulations (is 
not limited to either \textsf{Gmsh} or \textsf{BAMG}) for any type of mesh generator, which 
operates on a background mesh. On the other hand, there are certain open source and commercially 
available software packages, such as \textsf{CGAL} \cite{cgal} and \textsf{Simmetrix} \cite{MeshSim_Simmetrix}, which create an anisotropic $\mathcal{M}$-uniform simplicial mesh 
directly based on the metric tensor \cite{2009_Nguyen_etal_CMAME_v198_p2964_p2981,
2013_Huang_etal_SCM_v56_p2615_p2630} without the need of background meshes. Investigation 
of such mesh generators is beyond the scope of this research paper and is neither critical 
nor central to the ideas discussed here. Before we discuss various numerical examples based 
on different types of triangulations, we shall now present certain important (mesh-based) 
non-dimensional numbers relevant to our numerical study. Such a discussion is first of its 
kind and is not discussed elsewhere.

%-------------------------------------------------------------------------;
%  Algorithm: Anisotropic $\mathcal{M}$-uniform triangulation satisfying  ;
%             discrete principles using BAMG                              ;
%-------------------------------------------------------------------------;
\begin{algorithm} 
  \caption{An iterative method to generate an anisotropic $\mathcal{M}$-uniform 
    mesh satisfying discrete principles}
  \label{Algo:Aniso_Meshes} 
  \begin{algorithmic}[1]
    %----------------------;
    %  STEP-0: Input data  ;
    %----------------------; 
    \STATE INPUT: Background mesh ($\mathcal{T}_{h,0}$, $Nele_0$, $Nv_0$, and 
      $Nbv_0$); anisotropic diffusivity tensor ($\mathbf{D}(\mathbf{x})$); velocity 
      vector field ($\mathbf{v}(\mathbf{x})$); linear reaction coefficient ($\alpha
      (\mathbf{x})$); maximum number of iterations (\texttt{MaxIters}); piecewise 
      positive scalar element metric constants ($\{ \Theta_e \}^{Nele_0}_{e=1}$); and a 
      stopping criteria (\texttt{StopCrit})
      \begin{itemize}
         \item $\mathcal{T}_{h,0}$ is the initial background triangulation on 
           which an anisotropic mesh generator operates
         \item $Nv_0$ and $Nbv_0$ are correspondingly the total number of vertices 
           and boundary vertices  
      \end{itemize}
    %-----------------------------------------------;
    %  STEP-1: Set the iteration number to be zero  ;
    %-----------------------------------------------;
    \STATE Set the iteration number: $i = 0$
    %----------------------------------;
    %  STEP-2: Loop until convergence  ;
    %----------------------------------;
    \WHILE{(True)}
      %-------------------------------------------------------;
      %  STEP-2a: Compute average element diffusivity tensor  ;
      %-------------------------------------------------------;
      \STATE Compute the element average anisotropic diffusivity tensor 
        using a quadrature rule (for example, see References \cite{Reddy})
        \begin{itemize}
          \item $\widetilde{\boldsymbol{D}}_{e,i} := \frac{1}{\mathrm{meas}(\Omega_e)} 
            \int_{\Omega_e} \mathbf{D} (\mathbf{x}) \mathrm{d} \Omega \quad \forall 
            e = 1, 2, \cdots, Nele_i$
        \end{itemize}
      %--------------------------------------;
      %  STEP-2b: Compute the metric tensor  ;
      %--------------------------------------;
      \STATE Compute the element metric tensor by explicitly inverting 
        $\widetilde{\boldsymbol{D}}_{e,i}$
        \begin{itemize}
          \item $\mathcal{M}_{e,i}:= \Theta_e (\widetilde{\boldsymbol{D}}_{e,i})^{-1} 
            \quad \forall e = 1, 2, \cdots, Nele_i$
        \end{itemize}
      %-----------------------------------------------------;
      %  STEP-2c: Calculate a new triangulation using BAMG  ;
      %-----------------------------------------------------;
      \STATE Based on the set of metric tensors $\{ \mathcal{M}_{e,i} \}^{Nele_i}_{e=1}$, 
        compute a new triangulation $\widetilde{\mathcal{T}}_{h,i}$
        \begin{itemize}
          \item Output the new triangulation $\widetilde{\mathcal{T}}_{h,i}$. 
            Corresponding to this $\widetilde{\mathcal{T}}_{h,i}$, we have 
            $\widetilde{Nele}_i$, $\widetilde{Nv}_i$, and $\widetilde{Nbv}_i$
        \end{itemize}
      %--------------------------------------------------------;
      %  STEP-2d: Compute the new element material quantities  ;
      %--------------------------------------------------------;
      \STATE Compute the following quantities: $\forall e = 1, 2, \cdots, 
        \widetilde{Nele}_i$
        \begin{itemize}
          \item $\widetilde{\boldsymbol{D}}_{e,i}$ ; $\Lambda_{min,
            \widetilde{\boldsymbol{D}}_{e,i}}$ ; $\|\boldsymbol{v}\|
            _{e,i}:= \|\boldsymbol{v}(\boldsymbol{x})\|_{\infty,
            \overline{\Omega}_{e,i}}$; and $\|\alpha\|_{e,i}:= 
            \|\alpha(\boldsymbol{x})\|_{\infty,\overline{\Omega}_{e,i}}$
          \item Need to use a constrained optimization methodology to 
            calculate $\|\boldsymbol{v}\|_{e,i}$ and $\|\alpha\|_{e,i}$ (see 
            Remark \ref{Rmrk:Max_vel_Alpha} for more details)
        \end{itemize}
      %---------------------------------;
      %  STEP-2e: Stopping criterion-1  ;
      %---------------------------------;
      \IF {(\texttt{StopCrit = Anisotropic non-obtuse angle condition})}
        \STATE Check the inequality given by equation \eqref{Eqn:Aniso_Weak_Condition} 
          in Theorem \ref{Thm:Aniso_Weak_Condition} $\forall e = 1, 2, \cdots, 
          \widetilde{Nele}_i$
          \IF{(true)}
            \STATE OUTPUT: The triangulation $\widetilde{\mathcal{T}}_{h,i}$ and 
              corresponding $\{ \mathcal{M}_{e,i} \}^{Nele_i}_{e=1}$. EXIT
          \ELSE
            \STATE Update $\mathcal{T}_{h,i} \gets \widetilde{\mathcal{T}}_{h,i}$, 
              $Nele_i \gets \widetilde{Nele}_i$, $Nv_i \gets \widetilde{Nv}_i$, 
              $Nbv_i \gets \widetilde{Nbv}_i$, and $i \gets (i + 1)$
          \ENDIF
      \ENDIF
      %---------------------------------;
      %  STEP-2f: Stopping criterion-2  ;
      %---------------------------------;
      \IF {(\texttt{StopCrit = Generalized Delaunay-type angle condition})}
        \STATE Check the inequality given by equation \eqref{Eqn:Delaunay_Weak_Condition} 
          in Theorem \ref{Thm:Delaunay_Weak_Condition} $\forall e = 1, 2, \cdots, 
          \widetilde{Nele}_i$
          \IF{(true)}
            \STATE OUTPUT: The triangulation $\widetilde{\mathcal{T}}_{h,i}$ and 
              corresponding $\{ \mathcal{M}_{e,i} \}^{Nele_i}_{e=1}$. EXIT
          \ELSE
            \STATE Update $\mathcal{T}_{h,i} \gets \widetilde{\mathcal{T}}_{h,i}$, 
              $Nele_i \gets \widetilde{Nele}_i$, $Nv_i \gets \widetilde{Nv}_i$, 
              $Nbv_i \gets \widetilde{Nbv}_i$, and $i \gets (i + 1)$
          \ENDIF
      \ENDIF
      %---------------------------------;
      %  STEP-2g: Stopping criterion-3  ;
      %---------------------------------;
      \IF {($i > \mathtt{MaxIters}$)}
        \STATE OUTPUT: The \emph{existing} triangulation $\widetilde{\mathcal{T}}
          _{h,i}$ and corresponding $\{ \mathcal{M}_{e,i} \}^{Nele_i}_{e=1}$.
        \STATE Anisotropic $\mathcal{M}$-uniform triangulation not found in 
        \texttt{MaxIters}. EXIT
      \ENDIF
    \ENDWHILE  
  \end{algorithmic}
\end{algorithm}

%--------------------------------------------------;
%  Remark: Complexities involved in the Algorithm  ;
%--------------------------------------------------;
\begin{remark}
  \label{Rmrk:Max_vel_Alpha}
  It should be noted that \textsc{{\small STEP-7}} in Algorithm 
  \ref{Algo:Aniso_Meshes} might be computationally intensive, as we 
  need to solve a series of small constrained optimization problem 
  for each `$\Omega_{e,i} \in \widetilde{\mathcal{T}}_{h,i}$' and 
  at every iteration level `$i$'. The corresponding constrained 
  optimization problems are given as follows:
  %-------------------------------------------;
  %  Constrained optimization problems set-1  ;
  %-------------------------------------------;
  \begin{subequations}
  \begin{align}
    \label{Eqn:COP_Vel}
    \|\boldsymbol{v}\|_{e,i} &:= \displaystyle \mathop{\mbox{maximize}}_
    {\boldsymbol{x} \in \overline{\Omega}_{e,i}} \; \; \|\boldsymbol{v}(
    \boldsymbol{x})\| \qquad \forall 1 \leq i \leq \mathtt{MaxIters}\\
    \label{Eqn:COP_Alpha}
    \| \alpha \|_{e,i} &:= \displaystyle \mathop{\mbox{maximize}}_
    {\boldsymbol{x} \in \overline{\Omega}_{e,i}} \; \; \alpha(\boldsymbol{x}) 
    \qquad \; \; \forall 1 \leq i \leq \mathtt{MaxIters}
  \end{align}
  \end{subequations}
  As $\overline{\Omega}_{e,i}$ is convex, triangular, and 
  closed, by half-space representation theorem for convex 
  polytopes \cite[Section-2.2.4]{Boyd_convex_optimization} 
  equation \eqref{Eqn:COP_Vel} can be written as follows:
  %-------------------------------------------;
  %  Constrained optimization problems set-1  ;
  %-------------------------------------------;
  \begin{subequations}
  \begin{align}
    \label{Eqn:COP_Reform_11}
    \displaystyle \mathop{\mbox{maximize}}_{\boldsymbol{x} \in \mathbb{R}^{2}} 
    & \quad \|\boldsymbol{v}(\boldsymbol{x})\|  \\
    \label{Eqn:COP_Reform_12}
    \mbox{subject to} & \quad \boldsymbol{A}_{e,i} \boldsymbol{x} \preceq \boldsymbol{b}_{e,i} 
    \qquad \forall 1 \leq i \leq \mathtt{MaxIters}
  \end{align}
  \end{subequations} 
  where $\boldsymbol{A}_{e,i}$ is a $3 \times 2$ matrix and $\boldsymbol{b}
  _{e,i}$ is a $3 \times 1$ vector, whose coefficients correspond to the 
  linear inequalities defining the relevant half-spaces and supporting 
  hyper-planes of the triangle $\overline{\Omega}_{e,i}$. If the element 
  maximum value `$\|\boldsymbol{v}\|_{e,i}$' is known a priori (through 
  analytically or by means of a rigorous mathematical analysis), then one 
  can use such information in \textsc{{\small STEP-7}} of Algorithm 
  \ref{Algo:Aniso_Meshes}. Otherwise, we need to solve equations 
  \eqref{Eqn:COP_Reform_11}--\eqref{Eqn:COP_Reform_12} using the standard 
  constrained optimization algorithms for small-scale problems
  \cite{Boyd_convex_optimization}. Similarly, equation \eqref{Eqn:COP_Alpha} 
  can be reformulated based on the lines of equations 
  \eqref{Eqn:COP_Reform_11}--\eqref{Eqn:COP_Reform_12}.
\end{remark}

%==================================================================;
%  Subsection: Peclet and Damkohler numbers for simplicial meshes  ;
%==================================================================;
\subsubsection{\textbf{P\'{e}clet and Damk\"{o}hler numbers for simplicial meshes}}
Herein, we shall describe three types of P\'{e}clet and Damk\"{o}hler 
numbers for simplicial meshes. They are devised based on Theorems 
\ref{Thm:Aniso_Weak_Condition} and \ref{Thm:Delaunay_Weak_Condition}. 
However, it should be noted that extending it to non-simplicial elements, 
such as Q4, is not straightforward. This is because in order to construct 
P\'{e}clet and Damk\"{o}hler numbers for non-simplicial meshes, one needs 
to obtain mesh restrictions using the global stiffness restriction method. 
This is beyond the scope of the current paper. 

%--------------------------------------------------------;
%  Itemize: Three types of Peclet and Damkohler numbers  ;
%--------------------------------------------------------;
\begin{itemize}
  \item \emph{Element P\'{e}clet and Damk\"{o}hler numbers}: Based on 
    Theorem \ref{Thm:Aniso_Weak_Condition}, one can define the following 
    mesh-based non-dimensional element P\'{e}clet $(\mathbb{P}\mathrm{e}_
    {\Omega_e})$ and Damk\"{o}hler $(\mathbb{D}\mathrm{a}_{\Omega_e})$ numbers:
    %--------------------------------------------------;
    %  Equation: Element Peclet and Damkohler numbers  ;
    %--------------------------------------------------;
    \begin{subequations}
    \begin{align}
      \label{Eqn:Element_Peclet_Number}
      \mathbb{P}\mathrm{e}_{\Omega_e} &:= \frac{h_{\mathrm{max},\Omega_e} 
      \, \|\boldsymbol{v}\|_{\infty,
      \overline{\Omega}_e}}{\Lambda_{min,\widetilde{\boldsymbol{D}}_
      {\Omega_e}}} \\
      \label{Eqn:Element_Damkohler_Number2}
      \mathbb{D}\mathrm{a}_{\Omega_e} &:= \frac{h_{\mathrm{max},\Omega_e} 
      \, h_{\mathrm{pumax},\Omega_e} \, \|\alpha\|_{\infty,\overline{\Omega}
      _e}}{\Lambda_{min,\widetilde{\boldsymbol{D}}_{\Omega_e}}}
    \end{align}
    \end{subequations}
    where the height $h_{\mathrm{pumax},\Omega_e}$ is given as follows:
    %-------------------------------;
    % Equation: Penultimate height  ;
    %-------------------------------;
    \begin{align}
      0 < h_1 \leq h_2 \leq \cdots \leq h_i \leq \cdots \leq h_{\mathrm{pumax},
      \Omega_e} \leq h_{\mathrm{max},\Omega_e} \quad \forall i = 1,2,\cdots,nd+1, 
      \; \; \Omega_e \in \mathcal{T}_h
    \end{align}
    Correspondingly, using equations 
    \eqref{Eqn:Element_Peclet_Number}--\eqref{Eqn:Element_Damkohler_Number2} in 
    equation \eqref{Eqn:Aniso_Weak_Condition} gives the following (stronger) mesh 
    restriction condition based on Theorem \ref{Thm:Aniso_Weak_Condition}:
    %------------------------------------------------------------------------;
    %  Equation: Anisotropic non-obtuse angle condition (Master equation-1)  ;
    %------------------------------------------------------------------------;
    \begin{align}
      \label{Eqn:AnisoNonObtuseMasterEquation1}
      0 < \frac{\mathbb{P}\mathrm{e}_{\Omega_e}}{(nd+1) \, \cos(\beta_{ij,
      \widetilde{\boldsymbol{D}}^{-1}_{\Omega_e}})} &+ \frac{\mathbb{D}\mathrm{a}
      _{\Omega_e}}{(nd+1) \, (nd+2) \, \cos(\beta_{ij,\widetilde{\boldsymbol{D}}^
      {-1}_{\Omega_e}})} \leq 1 \nonumber \\
      &i = \mathrm{max} \; \; \mathrm{and} \; \; j = \mathrm{pumax}, 
      \; i \neq j, \; \forall \Omega_e \in \mathcal{T}_h
    \end{align}
  \item \emph{Edge P\'{e}clet and Damk\"{o}hler numbers}: In a similar 
    fashion, utilizing Theorem \ref{Thm:Delaunay_Weak_Condition}, one can 
    define the following mesh-based non-dimensional edge P\'{e}clet $(
    \mathbb{P}\mathrm{e}_{\Omega_e,e_{pq}})$ and Damk\"{o}hler $(\mathbb{D}
    \mathrm{a}_{\Omega_e,e_{pq}})$ numbers:
    %-----------------------------------------------;
    %  Equation: Edge Peclet and Damkohler numbers  ;
    %-----------------------------------------------;
    \begin{subequations}
    \begin{align}
      \label{Eqn:Edge_Peclet_Number}
      \mathbb{P}\mathrm{e}_{\Omega_e,e_{pq}} &:=  \frac{\mathrm{meas}
      (\Omega_e) \|\boldsymbol{v}\|_{\infty,\overline{\Omega}_e}}{h_
      {q,\Omega_e} \sqrt{\mathrm{det} [ \widetilde{\boldsymbol{D}}_
      {\Omega_e} ]}} \\
      \label{Eqn:Edge_Damkohler_Number2}
      \mathbb{D}\mathrm{a}_{\Omega_e,e_{pq}} &:= \frac{\mathrm{meas}
      (\Omega_e) \| \alpha \|_{\infty,\overline{\Omega}_e}}{ \sqrt{
      \mathrm{det} [ \widetilde{\boldsymbol{D}}_{\Omega_e} ]}}
    \end{align}
    \end{subequations}
    Correspondingly, using equations 
    \eqref{Eqn:Edge_Peclet_Number}--\eqref{Eqn:Edge_Damkohler_Number2} in 
    equation \eqref{Eqn:Delaunay_Weak_Condition} gives the following (weaker) mesh 
    restriction condition based on Theorem \ref{Thm:Delaunay_Weak_Condition}:
    %-------------------------------------------------------------------------;
    %  Equation: Generalized Delaunay-type angle condition (Master equation)  ;
    %-------------------------------------------------------------------------;
    \begin{align}
      \label{Eqn:DelaunayWeakConditionMasterEquation}
      0 &< \frac{1}{2 \pi} \left[ \beta_{pq,\widetilde{\boldsymbol{D}}^{-1}_{\Omega_e}} + 
      \beta_{pq,\widetilde{\boldsymbol{D}}^{-1}_{\Omega^{'}_e}} \right] \nonumber \\
      &+ \frac{1}{2 \pi} \mathrm{arccot} 
      \left( \sqrt{\frac{\mathrm{det} [ \widetilde{\boldsymbol{D}}_{\Omega^{'}_e} ]}
      {\mathrm{det} [ \widetilde{\boldsymbol{D}}_{\Omega_e} ]}} \left( \cot(\beta_{pq,
      \widetilde{\boldsymbol{D}}^{-1}_{\Omega^{'}_e}}) - \frac{2\mathbb{P}\mathrm{e}
      _{\Omega^{'}_e,e_{pq}}}{3} - \frac{\mathbb{D}\mathrm{a}_{\Omega^{'}_e,e_{pq}}}{6} 
      \right) - \frac{2\mathbb{P}\mathrm{e}_{\Omega_e,e_{pq}}}{3} - \frac{\mathbb{D}
      \mathrm{a}_{\Omega_e,e_{pq}}}{6} \right) \nonumber \\
      &+ \frac{1}{2 \pi} \mathrm{arccot} 
      \left( \sqrt{\frac{\mathrm{det} [ \widetilde{\boldsymbol{D}}_{\Omega_e} ]}
      {\mathrm{det} [ \widetilde{\boldsymbol{D}}_{\Omega^{'}_e} ]}} \left( \cot(\beta
      _{pq,\widetilde{\boldsymbol{D}}^{-1}_{\Omega_e}}) - \frac{2\mathbb{P}\mathrm{e}
      _{\Omega_e,e_{pq}}}{3} - \frac{\mathbb{D}\mathrm{a}_{\Omega_e,e_{pq}}}{6} \right) 
      - \frac{2\mathbb{P}\mathrm{e}_{\Omega^{'}_e,e_{pq}}}{3} - \frac{\mathbb{D} 
      \mathrm{a}_{\Omega^{'}_e,e_{pq}}}{6} \right) \nonumber \\
      &\leq 1
    \end{align}
  \item \emph{Global mesh P\'{e}clet and Damk\"{o}hler numbers}: On sufficiently fine 
    $h$-refined simplicial meshes (which confirm to Theorem \ref{Thm:Aniso_Weak_Condition}), 
    one can define a global mesh P\'{e}clet $(\mathbb{P}\mathrm{e}_{h})$ and Damk\"{o}hler 
    $(\mathbb{D}\mathrm{a}_{h})$ numbers by modifying equations 
    \eqref{Eqn:Element_Peclet_Number}--\eqref{Eqn:Element_Damkohler_Number2} as follows:
    %------------------------------------------------------;
    %  Equation: Global mesh Peclet and Damkohler numbers  ;
    %------------------------------------------------------;
    \begin{subequations}
    \begin{align}
      \label{Eqn:Global_Peclet_Number}
      \mathbb{P}\mathrm{e}_{h} &:= \frac{h \displaystyle \mathop{\mbox{max}}_
      {\Omega_e \in \mathcal{T}_h} \left[ \|\boldsymbol{v}\|_{\infty,\overline{
      \Omega}_e} \right]}{\displaystyle \mathop{\mbox{min}}_{\Omega_e \in \mathcal{T}
      _h} \left[\Lambda_{min,\widetilde{\boldsymbol{D}}_{\Omega_e}} \right]} \\
      \label{Eqn:Global_Damkohler_Number2}
      \mathbb{D}\mathrm{a}_{h} &:= \frac{h^2 \displaystyle \mathop{\mbox{max}}_
      {\Omega_e \in \mathcal{T}_h} \left[ \|\alpha\|_{\infty,\overline{\Omega}_e} 
      \right]}{\displaystyle \mathop{\mbox{min}}_{\Omega_e \in \mathcal{T}_h} \left[
      \Lambda_{min,\widetilde{\boldsymbol{D}}_{\Omega_e}} \right]}
    \end{align}
    \end{subequations}
    Conservatively, equation \eqref{Eqn:AnisoNonObtuseMasterEquation1} can be modified 
    to give a (stronger) global mesh restriction condition based on Theorem 
    \ref{Thm:Aniso_Weak_Condition}:
    %-----------------------------------------------------------------------;
    %  Equation: Anisotropic non-obtuse angle condition (h-refined meshes)  ;
    %-----------------------------------------------------------------------;
    \begin{align}
      \label{Eqn:AnisoNonObtuse_hRefinedMeshes}
      0 < \frac{\mathbb{P}\mathrm{e}_{h}}{(nd+1) \, \displaystyle \mathop{\mbox{min}}
      _{\Omega_e \in \mathcal{T}_h} \left[ \cos(\beta_{ij,\widetilde{\boldsymbol{D}}
      ^{-1}_{\Omega_e}}) \right]} &+ \frac{\mathbb{D}\mathrm{a}_{h}}{(nd+1) \, 
      (nd+2) \, \displaystyle \mathop{\mbox{min}}_{\Omega_e \in \mathcal{T}_h} \left[ 
      \cos(\beta_{ij,\widetilde{\boldsymbol{D}}^{-1}_{\Omega_e}}) \right]} \leq 1 
      \nonumber \\
      &i = \mathrm{max} \; \; \mathrm{and} \; \; j = \mathrm{pumax}, 
      \; i \neq j
    \end{align}
    One should note that equation \eqref{Eqn:AnisoNonObtuse_hRefinedMeshes} can provide 
    a useful a priori (conservative) estimate on $h$ for constructing highly refined 
    simplicial meshes. 
\end{itemize}
  
%==========================================================;
%  Subsection: Physics-based Peclet and Damkohler numbers  ;
%==========================================================;
\subsubsection{\textbf{Physics-based P\'{e}clet and Damk\"{o}hler numbers}}
For isotropic diffusivity, it is well-known that the following three 
\emph{physics-based} non-dimensional numbers can be used to understand 
the qualitative nature of the solutions \cite{Gresho_Sani_v1,Donea_Huerta,
Taylor_Krishna}:
%--------------------------------------------------------------------------;
%  Equation: Peclet and various Damkohler Numbers (Isotropic diffusivity)  ;
%--------------------------------------------------------------------------;
\begin{subequations}
  \begin{align}
    \label{SubEqn:Peclet_Numbers_IsoDiff}
    &\mathbb{P}\mathrm{e}_D := \frac{\displaystyle \mathop{\mbox{max}}_
    {\boldsymbol{x} \in \overline{\Omega}} \left[ \|\mathbf{v}(\mathbf{x})\|_
    {\infty} \right] L}{\displaystyle \mathop{\mbox{min}}_{\boldsymbol{x} \in 
    \overline{\Omega}} \left[ D(\mathbf{x}) \right] } \qquad \; \; \, \quad \, 
    \mathrm{P\acute{e}clet \; number} \\
    \label{SubEqn:Damkohler_Number1_IsoDiff}
    &\mathbb{D}\mathrm{a_I} := \frac{\displaystyle \mathop{\mbox{max}}_
    {\boldsymbol{x} \in \overline{\Omega}} \left[\| \alpha(\mathbf{x}) 
    \|_{\infty} \right] L}{\displaystyle \mathop{\mbox{max}}_{\boldsymbol{x} 
    \in \overline{\Omega}} \left[ \|\mathbf{v}(\mathbf{x})\|_{\infty} \right]} 
    \qquad \qquad \mbox{Damk\"{o}hler number of first kind} \\
    \label{SubEqn:Damkohler_Number2_IsoDiff}
    &\mathbb{D}\mathrm{a_{II}}_{,D} := \frac{\displaystyle \mathop{\mbox{max}}_
    {\boldsymbol{x} \in \overline{\Omega}} \left[\| \alpha(\mathbf{x}) \|_{\infty} 
    \right] L^2}{\displaystyle \mathop{\mbox{min}}_{\boldsymbol{x} \in \overline{
    \Omega}} \left[ D(\mathbf{x}) \right] } \qquad \; \mbox{Damk\"{o}hler number 
    of second kind}
  \end{align}
\end{subequations}
where $L$ is the characteristic length of the domain and $\| \bullet \|_{\infty}$ 
is the standard max-norm/infinity-norm for vectors \cite{Gurtin}. However, 
it should be noted that the above three non-dimensional numbers are not independent, 
as they satisfy the relation $\mathbb{D} \mathrm{a_{II}}_{,D} = \mathbb{P}\mathrm{e}_D \, 
\mathbb{D}\mathrm{a_I}$. Physically, the non-dimensional P\'{e}clet number characterizes 
the relative dominance of advection as compared to diffusion processes. For larger 
P\'{e}clet numbers, the advection process dominates and for smaller P\'{e}clet numbers, 
the diffusion process dominates. 
    
In the literature on chemically reacting systems (e.g., see 
\cite[Table 22.2.1]{Chung}), 
there exists various Damk\"{o}hler numbers that relate progress of chemical reactions 
with respect to mixing, diffusivity, reaction coefficient, thermal effects, and advection. 
The Damk\"{o}hler number of first-kind, $\mathbb{D}\mathrm{a_I}$, gives information about 
the relative influence of a linear reaction coefficient to that of advection. For small values 
of $\mathbb{D}\mathrm{a_I}$, advection progresses much faster than decay of the chemical 
species and has the opposite effect for large values of the number. The Damk\"{o}hler number 
of second kind, $\mathbb{D}\mathrm{a_{II}}_{,D}$, gives information related to the progress 
of chemical reaction with respect to diffusion. Based on the chemical system under consideration, 
these three non-dimensional numbers dictate how the reaction, advection, and diffusion interact 
with each other. On the other hand, one should note that extending it to anisotropic diffusivity 
is not straightforward. Motivated by equations \eqref{Eqn:Global_Peclet_Number}--\eqref{Eqn:Global_Damkohler_Number2}, a way to define 
\emph{physics-based} non-dimensional numbers for anisotropic diffusivity tensor $\mathbf{D}
(\mathbf{x})$ is as follows:
%----------------------------------------------------------------------------------;
%  Equation: Peclet and various Damkohler Numbers (Anisodiff: Stronger condition)  ;
%----------------------------------------------------------------------------------;
\begin{subequations}
  \begin{align}
    \label{SubEqn:Peclet_Numbers_AnIsoDiff1}
    &\mathbb{P}\mathrm{e}_{\mathbf{D}} := \frac{\displaystyle \mathop{\mbox{max}}_
    {\boldsymbol{x} \in \overline{\Omega}} \left[ \|\mathbf{v}(\mathbf{x})\|_{\infty} 
    \right] L}{\displaystyle \mathop{\mbox{min}}_{\boldsymbol{x} \in \overline{\Omega}} 
    \left[ \Lambda_{min,\mathbf{D}(\mathbf{x})} \right] } \qquad \; \; \, \, \, \, \, \,
    \mathrm{P\acute{e}clet \; number} \\
    \label{SubEqn:Damkohler_Number2_AnIsoDiff1}
    &\mathbb{D}\mathrm{a_{II}}_{,\mathbf{D}} := \frac{\displaystyle \mathop{\mbox{max}}_
    {\boldsymbol{x} \in \overline{\Omega}} \left[ \| \alpha(\mathbf{x}) \|_{\infty} \right] 
    L^2}{\displaystyle \mathop{\mbox{min}}_{\boldsymbol{x} \in \overline{\Omega}} \left[ 
    \Lambda_{min,\mathbf{D}(\mathbf{x})} \right] } \qquad \mbox{Damk\"{o}hler number 
    of second kind}
  \end{align}
\end{subequations}
where $\Lambda_{min,\mathbf{D}(\mathbf{x})}$ is the minimum eigenvalue of $\mathbf{D}(
\mathbf{x})$ at a given point $\mathbf{x}$. The Damk\"{o}hler number of first kind, 
$\mathbb{D}\mathrm{a_I}$, given by equation \eqref{SubEqn:Damkohler_Number1_IsoDiff}
remains the same as it does not depend on diffusivity tensor $\mathbf{D}(\mathbf{x})$.
Alternatively, inspired by equations \eqref{Eqn:Edge_Peclet_Number}--\eqref{Eqn:Edge_Damkohler_Number2}, 
one can define a different set of physics-based P\'{e}clet and Damk\"{o}hler numbers. 
These are given as follows:
%--------------------------------------------------------------------------------;
%  Equation: Peclet and various Damkohler Numbers (Anisodiff: Weaker condition)  ;
%--------------------------------------------------------------------------------;
\begin{subequations}
  \begin{align}
    \label{SubEqn:Peclet_Numbers_AnIsoDiff2}
    &\mathbb{P}\mathrm{e}_{\mathbf{D}} := \frac{\displaystyle \mathop{\mbox{max}}_
    {\boldsymbol{x} \in \overline{\Omega}} \left[ \|\mathbf{v}(\mathbf{x})\|_{\infty} 
    \right] L}{ \sqrt{\displaystyle \mathop{\mbox{min}}_{\boldsymbol{x} \in \overline{
    \Omega}} \left[ \Lambda_{min,\mathbf{D}(\mathbf{x})} \right]  \displaystyle \mathop{
    \mbox{max}}_{\boldsymbol{x} \in \overline{\Omega}} \left[ \Lambda_{max,\mathbf{D}
    (\mathbf{x})} \right]} } \qquad \; \; \, \, \, \mathrm{P\acute{e}clet \; number} \\
    \label{SubEqn:Damkohler_Number2_AnIsoDiff2}
    &\mathbb{D}\mathrm{a_{II}}_{,\mathbf{D}} := \frac{\displaystyle \mathop{\mbox{max}}_
    {\boldsymbol{x} \in \overline{\Omega}} \left[ \| \alpha(\mathbf{x}) \|_{\infty} \right] 
    L^2}{ \sqrt{\displaystyle \mathop{\mbox{min}}_{\boldsymbol{x} \in \overline{\Omega}} 
    \left[ \Lambda_{min,\mathbf{D}(\mathbf{x})} \right]  \displaystyle \mathop{\mbox{max}}
    _{\boldsymbol{x} \in \overline{\Omega}} \left[ \Lambda_{max,\mathbf{D}(\mathbf{x})} 
    \right]} } \qquad \mbox{Damk\"{o}hler number of second kind}
  \end{align}
\end{subequations}
where $\Lambda_{max,\mathbf{D}(\mathbf{x})}$ is the maximum eigenvalue of $\mathbf{D}(
\mathbf{x})$ at a given point $\mathbf{x}$. One should note that both these sets of 
non-dimensional numbers (given by equations 
\eqref{SubEqn:Peclet_Numbers_AnIsoDiff1}--\eqref{SubEqn:Damkohler_Number2_AnIsoDiff1} 
and \eqref{SubEqn:Peclet_Numbers_AnIsoDiff2}--\eqref{SubEqn:Damkohler_Number2_AnIsoDiff2}) 
are perfectly valid, as anisotropy in diffusivity tensor can introduce multiple ways 
of defining physics-based P\'{e}clet and Damk\"{o}hler numbers. Now, we shall present 
various numerical examples to demonstrate the \emph{pros and cons} of the mesh restrictions 
approach.

%****************************************;
%                                        ;
%  NAME                                  ;
%    S4_MP_Restrictions2_TP.tex          ;
%                                        ;
%  WRITTEN BY                            ;
%    Maruti Kumar Mudunuru               ;
%    Kalyana Babu Nakshatrala            ;
%                                        ;
%****************************************;

%==================================================;
%  Test problem # 1: Transport in fractured media  ;
%==================================================;
\subsubsection{\textbf{Test problem \#1: Transport in fractured media}}
\label{SubSubSec:Test_Prob_1}
This test problem has profound impact in simulating the transport of chemical 
species in fractured media \cite{Therrien_Sudicky_JCH_1996_v23_p1}. The numerical 
simulations, using the Delaunay-Voronoi triangulation (with \texttt{MaxIters} 
= 50) for different cases of diffusivity, velocity field, and linear reaction 
coefficient, are presented in Figure \ref{Fig:Prob1_Delaunay_DMP_Meshes_Conc}.
Homogeneous Dirichlet boundary conditions are prescribed on the sides of the 
fractured domain; $c^{\mathrm{p}}(\mathbf{x}) = 1.0$ on the left set of fracture 
lines and $c^{\mathrm{p}}(\mathbf{x}) = 2.5$ on the right set of fracture lines. 
The volumetric source $f(\mathbf{x})$ is zero inside the fractured domain. 
Diffusivity is assumed to isotropic and scalar, whose value is given by 
$D(\mathbf{x}) = 10^{-3}$. We perform numerical simulations for three different 
cases of velocity field and linear reaction coefficient, which are given by 
$\mathbf{v}(\mathbf{x}) = (0.0,0.0)$ and $\alpha(\mathbf{x}) = 0.0$, $\mathbf{v}
(\mathbf{x}) = (0.1,1.0)$ and $\alpha(\mathbf{x}) = 0.0$, and $\mathbf{v}(\mathbf{x}) 
= (0.1,1.0)$ and $\alpha(\mathbf{x}) = 1.0$.

From Figure \ref{Fig:Prob1_Delaunay_DMP_Meshes_Conc}, it can be inferred that 
we need highly refined DMP-based triangular meshes to obtain physically meaningful 
values for concentration for large values of edge P\'{e}clet and Damk\"{o}hler 
numbers. The white region in the figures indicates the area in which the value 
of concentration is negative and also violated the maximum constraint. The coarse 
Delaunay-Voronoi mesh obtained using the open source mesh generator \textsf{Gmsh} 
satisfies NC and DMPs in the case of pure diffusion. But, this is not true for AD 
and ADR cases. In such scenarios, it produces unphysical values for the concentration 
field. Moreover, the percentage of nodes that have violated NC and maximum 
constraint is also very high. 

Quantitatively, Tables \ref{Tab:Fractured_Domain_AD_MinMaxPercent_Conc} and 
\ref{Tab:Fractured_Domain_ADR_MinMaxPercent_Conc} provide more details pertinent 
to the violations in NC and DMPs for AD and ADR cases. It should be noted that 
the decrease in unphysical values of concentration is not monotonic. This is 
because the \texttt{generalized Delaunay-type angle condition}, in general, 
\emph{does not ensure uniform convergence} for diffusion-type equations in 
$L^{\infty}$ norm (for example, in case of pure isotropic diffusion, see the 
mesh restriction result by Ciarlet and Raviart \cite{Ciarlet_Raviart_CMAME_1973_v2_p17}). 
Figure \ref{Fig:Prob1_Delaunay_Condition} shows that the weak DMP-based condition 
is satisfied only for pure diffusion equation. But in all other cases, 
\texttt{generalized Delaunay-type condition} is violated. In the case of pure 
diffusion, it should be noted that DMP-based mesh given in Figure \ref{Fig:Prob1_Delaunay_Condition} 
is \emph{not} interiorly connected. Hence, it only satisfies $\mathrm{DWMP}_
{\boldsymbol{K}}$, but not $\mathrm{DSMP}_{\boldsymbol{K}}$.

%=============================================================;
%  Test problem \# 2: Species dispersion in subsurface flows  ;
%=============================================================;
\subsubsection{\textbf{Test problem \#2: Species dispersion in subsurface flows}}
\label{SubSubSec:Test_Prob_2}
A pictorial description of the boundary value problem with various parameters 
is shown in Figure \ref{Fig:Prob2_Description}. Relevant (coarse) background 
mesh used and the corresponding DMP-based mesh (with \texttt{MaxIters} = 50) 
obtained using \textsf{BAMG} are shown in Figure \ref{Fig:Prob2_Background_DMP_Meshes}. 
The diffusivity tensor for this problem is taken from the subsurface hydrology 
literature \cite{Pinder_Celia} and is given as follows:
%-------------------------------------------;
%  Equation: Subsurface diffusivity tensor  ;
%-------------------------------------------;
\begin{align}
  \label{Eqn:Subsurface_Diffusivity}
  \mathbf{D}_{\mathrm{subsurface}}(\mathbf{x}) = 
  \alpha_{T} \|\mathbf{v}\| \mathbf{I} + \frac{
  \alpha_L - \alpha_T}{\|\mathbf{v}\|} \mathbf{v} 
  \otimes \mathbf{v}
\end{align}
where $\otimes$ is the tensor product, $\mathbf{I}$ is the identity tensor, 
$\mathbf{v}$ is velocity vector field of the subsurface flow, and $\alpha_T$ 
and $\alpha_L$ are, respectively, transverse and longitudinal diffusivity 
coefficients with  $\alpha_T = 0.01$ and $\alpha_L = 0.1$. It should be 
emphasized that we have neglected advection. Correspondingly, the numerical 
values for the velocity vector field used to define the diffusion tensor is 
given by $\mathbf{v}(\mathbf{x}) = (1.0,1.0)$. This test problem has importance 
in simulating diffusion of chemical species in subsurface flows of hydrogeological 
systems \cite{Zheng_Bennett_Transport,McCarty_Criddle,Dentz_Borgne_Englert_Bijeljic_JHC_2011_v120_p1}. 

Numerical simulations using these meshes are shown in Figure \ref{Fig:Prob2_Background_DMP_Meshes_Conc}.
The white region in this figure depicts the area in which the value of concentration 
is negative. From Figure \ref{Fig:Prob2_Background_DMP_Meshes_Conc}, it is apparent 
that the coarse anisotropic triangulation violates the DMPs and NC. This is because
the Algorithm \ref{Algo:Aniso_Meshes} did not converge in \texttt{MaxIters}. However, 
quantitatively, this violation in NC is low as compared to background mesh. Specifically, 
the minimum concentration and the percentage of nodes that have violated the non-negative 
constraint on the background mesh is about $-4.8 \times 10^{-5}$ and $2.13\%$, while 
these values on the anisotropic triangulation are around $-1.35 \times 10^{-8}$ and 
$0.34\%$. Additionally, from Figure \ref{Fig:Prob2_Background_DMP_Meshes_Conc}, it is 
evident that we need a highly refined DMP-based anisotropic mesh to avoid negative values 
for concentration. Nevertheless, it should be noted that there is a considerable decrease 
in the negative values for concentration if a traditionally $h$-refined anisotropic 
triangulation is used (see Figure \ref{Fig:Prob2_TradRefine_Meshes} and subsubsection 
\ref{SubSubSec:Trad_NonTrad_hRefinement} for more details).

%==========================================================;
%  Test problem \# 3: Contaminant transport in leaky wells  ;
%==========================================================;
\subsubsection{\textbf{Test problem \#3: Contaminant transport in leaky wells}}
\label{SubSubSec:Test_Prob_3}
The computational domain is a circle with a hole centered at origin $(0,0)$. 
The radius of the circular hole and the circular domain are $0.1$ and $1.0$. 
Numerical simulations are performed for four different cases of the velocity 
vector field and linear reaction coefficient, which are given by $\mathbf{v}
(\mathbf{x}) = (0.0,0.0)$ and $\alpha(\mathbf{x}) = 0.0$, $\mathbf{v}(\mathbf{x}) 
= (1.5,1.0)$ and $\alpha(\mathbf{x}) = 1.0$, $\mathbf{v}(\mathbf{x}) = (5.0,0.5)$ 
and $\alpha(\mathbf{x}) = 1.0$, and $\mathbf{v}(\mathbf{x}) = (0.0,0.0)$ and 
$\alpha(\mathbf{x}) = 1000$. Each case is designed to test a particular aspect. 
For example, $\mathbf{v}(\mathbf{x}) = (5.0,0.5)$ and $\alpha(\mathbf{x}) = 1.0$ 
corresponds to advection-dominated ADR problems, while $\mathbf{v}(\mathbf{x}) = 
(0.0,0.0)$ and $\alpha(\mathbf{x}) = 1000$ corresponds to reaction-dominated 
diffusion-reaction problems. The diffusivity tensor for this problem is given 
by equations \eqref{Eqn:AnisoRotationDiff}--\eqref{Eqn:NN_Eigenvalue_Matrix}.
Correspondingly, the values for the parameters $d_{\mathrm{max}}$, $d_{\mathrm{min}}$, 
and $\theta$ are equal to 1, 0.001, and $\frac{\pi}{3}$. 

The computational domain in this test problem has various practical applications 
related to well design in multi-aquifers and groundwater distribution systems 
\cite{2007_Zinn_etal_HJ_v15_p633_p643,2006_Konikow_Hornberger_GW_v44_p648_p660,
2012_Mejia_etal_MG_v44_p227_p238}, understanding  emanation of gaseous hydrocarbons 
through bore-holes in oil and gas reservoirs \cite{2012_Myers_GW_v50_p872_p882,
2012_Saiers_Barth_GW_v50_p826_p828,2013_Cohen_etal_GW_v51_p317_p319}, and to study 
leakage of contaminants (such as $\mathrm{CO}_2$, salts, and nitrates) through 
abandoned wells \cite{1994_Avci_WRR_v30_p2565_p2578,1995_Lacombe_etal_WRR_v31_p1871_p1882,
Ebigbo_Class_Helmig_ComputGeoscience_2007_v11_p103,2013_Nakshatrala_Turner_IJCMESM_v14_p524_p541}.
The background mesh used and the corresponding DMP-based (coarse) mesh obtained 
using \textsf{BAMG} are shown in Figure \ref{Fig:Prob3_Background_DMP_Meshes} 
(\texttt{MaxIters} = 50). For the cases when $\mathbf{v}(\mathbf{x}) = (0.0,0.0)$ 
and $\alpha(\mathbf{x}) = 0.0$, $\mathbf{v}(\mathbf{x}) = (1.5,1.0)$ and 
$\alpha(\mathbf{x}) = 1.0$, Algorithm \ref{Algo:Aniso_Meshes} converged in 
\texttt{MaxIters}. But for $\mathbf{v}(\mathbf{x}) = (5.0,0.5)$ and $\alpha(
\mathbf{x}) = 1.0$, $\mathbf{v}(\mathbf{x}) = (0.0,0.0)$ and $\alpha(\mathbf{x}) 
= 1000$, Algorithm \ref{Algo:Aniso_Meshes} did not converge in \texttt{MaxIters}. 
Herein, the DMP-based coarse mesh is composed of \emph{needle-type} triangles. 
This is because the ratio of the minimum eigenvalue of $\mathbf{D}(\mathbf{x})$ 
to its maximum is 0.001 (which is related to the aspect ratio of the sides of 
the triangle in the DMP-based anisotropic mesh). Moreover, it is evident that 
the triangles in the mesh are aligned and oriented along the principal axis of 
the eigenvectors of the diffusivity tensor.

Numerical simulations using these meshes are shown in Figure \ref{Fig:Prob3_Background_DMP_Meshes_Conc}. 
The white region in the figures (circular annulus) shows the area in which the 
value of concentration is negative. The minimum concentration and the percentage 
of nodes that have violated the non-negative constraint for the background mesh 
are about $-1.67 \times 10^{-2}$ and $30.28\%$. As the anisotropic mesh is coarse 
and the Algorithm \ref{Algo:Aniso_Meshes} did not converge in \texttt{MaxIters} 
for advection-dominated ADR problems and reaction-dominated diffusion-reaction 
problems, the resulting mesh not only violates the non-negative constraint, but 
also produces spurious oscillations. The minimum concentration and the percentage 
of nodes that have violated the non-negative constraint for the case when $\mathbf{v} 
= (5.0,0.5)$ and $\alpha = 1.0$ are about $-1.78 \times 10^{-1}$ and $13.47\%$, 
whereas for the case when $\mathbf{v} = (0.0,0.0)$ and $\alpha = 1000$ are around 
$-2.79 \times 10^{-1}$ and $20.54\%$. Hence in both these cases, we need a highly 
refined DMP-based mesh to avoid negative values for concentration (see subsubsection 
\ref{SubSubSec:Trad_NonTrad_hRefinement} for more details). For scenarios when 
$\mathbf{v}(\mathbf{x}) = (0.0,0.0)$ and $\alpha(\mathbf{x}) = 0.0$, $\mathbf{v}
(\mathbf{x}) = (1.5,1.0)$ and $\alpha(\mathbf{x}) = 1.0$, the coarse DMP-based mesh 
is interiorly connected and satisfies $\mathrm{DsMP}_{\boldsymbol{K}}$ (or 
$\mathrm{DSMP}_{\boldsymbol{K}}$, when $\alpha(\mathbf{x}) = 0$).

 %****************************************;
%                                        ;
%  NAME                                  ;
%    S3_MP_Restrictions2_hLGM1.tex       ;
%                                        ;
%  WRITTEN BY                            ;
%    Maruti Kumar Mudunuru               ;
%    Kalyana Babu Nakshatrala            ;
%                                        ;
%****************************************;

%-------------------------------------------------------;
%  Subsubsection: Issues with DMP-based $h$-refinement  ;
%-------------------------------------------------------; 
\subsubsection{\textbf{Issues with DMP-based $h$-refinement}}
\label{SubSubSec:Trad_NonTrad_hRefinement}
From the above set of test problems, it is apparent that mesh refinement 
is needed to avoid spurious node-to-node oscillations and satisfaction 
of various discrete principles (mainly for the cases when the values of 
the velocity vector field and linear reaction coefficient are predominant 
as compared to minimum eigenvalue of diffusivity tensor). In general, 
within the context of computational geometry and mesh generation literature,  
there are various methods to generate different types of $h$-refined 
meshes \cite{Frey_George,Plewa_Linde_Weirs,2013_Schneider_LNACM_p133_p152}. 
However, it is evident from these test problems that we are interested 
in $h$-refined meshes that conform to either \texttt{anisotropic non-obtuse 
angle condition} or \texttt{generalized Delaunay-type angle condition}. 
Herein, we shall study two different and important methodologies that 
are popular in mesh generation literature and implemented in the open 
source mesh generators, such as \textsf{Gmsh} and \textsf{BAMG} in 
\textsf{FreeFem++}. Our objective is to understand whether the $h$-refined 
meshes generated using these mesh generators satisfy DMPs, DCPs, and NC 
or not.

%-------------------------------------------------------;
%  Subsection (Method - 1): Traditional $h$-refinement  ;
%-------------------------------------------------------;
\subsection*{Traditional $h$-refinement}
In this method, an $h$-refined mesh is obtained by splitting each triangle 
in the coarse mesh into multiple triangles. Based on this methodology,
numerical simulations are performed using the traditional $h$-refined meshes
(which are derived based on the DMP-based coarse meshes presented in Figures 
\ref{Fig:Prob1_Delaunay_DMP_Meshes_Conc}, \ref{Fig:Prob2_Background_DMP_Meshes},
and \ref{Fig:Prob3_Background_DMP_Meshes}). The resulting concentration 
profiles for test problems \#1, \#2, and \#3, are shown in Figures 
\ref{Fig:Prob1_TradRefine_Meshes}--\ref{Fig:Prob3_TradRefine_Meshes}.
Qualitatively and quantitatively, the following inferences can be drawn 
from these numerical results:
%--------------------------------------------------;
%  Inferences: Test problems \# 1, \# 2, and \# 3  ;
%--------------------------------------------------;
\begin{itemize}
  \item Test problem \#1: For isotropic diffusivity, from Figure 
    \ref{Fig:Prob1_TradRefine_Meshes}, Table \ref{Tab:Fractured_Domain_AD_MinMaxPercent_Conc}, 
    and Table \ref{Tab:Fractured_Domain_ADR_MinMaxPercent_Conc}, 
    it is apparent that there is a decrease in negative values 
    for concentration and reduction in spurious oscillations.
  \item Test problem \#2: Qualitatively, based on Figure \ref{Fig:Prob2_TradRefine_Meshes}, 
    it can be concluded that there is a considerable decrease in the 
    violation of non-negative constraint. Quantitatively, for this 
    $h$-refined anisotropic triangulation, the minimum concentration 
    and the percentage of nodes that have violated the non-negative 
    constraint is about $-1.15 \times 10^{-11}$ and $0.02\%$ (which 
    is significantly close to machine epsilon $\epsilon_{\mathrm{mach}} 
    \approx 2.22 \times 10^{-16}$).
  \item Test problem \#3: For the pure anisotropic diffusion case, the 
    \emph{coarse} anisotropic DMP-based mesh given in Figure 
    \ref{Fig:Prob3_Background_DMP_Meshes} satisfies NC and all of the DMPs. 
    However, contrary to this, on traditional $h$-refinement, the resulting 
    anisotropic triangulation produces unphysical negative values and violates 
    the DMPs. Quantitatively, this violation is far from machine epsilon, 
    $\epsilon_{\mathrm{mach}}$. Correspondingly, the minimum concentration 
    and the percentage of nodes that have violated the non-negative constraint 
    is about $-2.7 \times 10^{-1}$ and $28.51\%$, which is way higher as compared 
    the numerical simulations based on the background mesh. 
\end{itemize}
To summarize, \emph{it is clear that traditional $h$-refinement does not always 
reduce the unphysical numerical values}. Certainly, there is a decrease in negative 
values for concentration for test problem \#1 and \#2. But, this is not the case 
for test problem \#3. This is because the methodology to obtain traditional $h$-refinement 
meshes from a given coarse mesh \emph{need not conform} to the conditions outlined 
in either Theorem \ref{Thm:Aniso_Weak_Condition} or Theorem \ref{Thm:Delaunay_Weak_Condition}. 

%-----------------------------------------------------------;
%  Subsection (Method - 2): Non-traditional $h$-refinement  ;
%-----------------------------------------------------------;
\subsection*{Non-traditional $h$-refinement}
In this method, a $h$-refined mesh is obtained \emph{directly} from the 
background mesh (using Algorithm \ref{Algo:Aniso_Meshes}) by change certain 
parameters related to metric tensor, geometry of the domain, and number of 
nodes on the boundary of the domain \cite[Chapter-5]{FreeFempp}. It should 
be noted that this methodology is completely different from the traditional 
$h$-refinement procedure, as we never generate a coarse DMP-based triangulation 
and then split the respective mesh elements. However, even this non-traditional 
$h$-refined mesh is not guaranteed to produce physics-compatible numerical 
values for concentration (as the mesh generation procedure need not converge 
in \texttt{MaxIters}). This is evident from the numerical simulations performed 
on the non-traditional $h$-refined mesh for test problem \#3. The corresponding 
concentration profile is shown in Figure \ref{Fig:Prob3_NonTradRefine_Meshes}. 
Quantitatively, the minimum concentration and the percentage of nodes that have 
violated the non-negative constraint are about $-9.69 \times 10^{-2}$ and $34.11
\%$, which is comparatively similar to that of the traditional $h$-refinement 
methodology.

%--------------------------------------------------------------------;
%  Table: Fractured domain with isotropic diffusivity (AD equation)  ; 
%--------------------------------------------------------------------;
\begin{table}
  \centering
	\caption{Fractured domain with isotropic diffusivity: For AD equation
	\label{Tab:Fractured_Domain_AD_MinMaxPercent_Conc}}
	\begin{tabular}{|c|c|c|c|c|} \hline
	  \multirow{2}{*}{Delaunay triangulation \vspace{0.15in}} & 
	  \multicolumn{2}{|c|}{Concentration} & 
      \multicolumn{2}{|c|}{\% of nodes violated} \\ 
      \cline{2-3} \cline{4-5}
	  $\left(Nv, Nele, h \right)$ & Minimum & Maximum & (Non-negative 
	  constraint) & (Maximum constraint) \\ \hline
	  $(539,906,0.1)$ & -10.51 & 9.85 & 17.44 & 13.73  \\
	  $(1564,2826,0.05)$ & -18.26 & 5.86 & 18.22 & 14.19  \\
	  $(5090,9620,0.025)$ & -4.65 & 5.71 & 16.09 & 13.77  \\
	  $(18372,35665,0.0125)$ & -2.97 & 5.17 & 12.82 & 12.70  \\
	  $(69995,137881,0.00625)$ & -1.62 & 3.89 & 8.08 & 8.47  \\
	  \hline
	\end{tabular}
\end{table}

%---------------------------------------------------------------------;
%  Table: Fractured domain with isotropic diffusivity (ADR equation)  ; 
%---------------------------------------------------------------------;
\begin{table}
  \centering
	\caption{Fractured domain with isotropic diffusivity: For ADR equation
	\label{Tab:Fractured_Domain_ADR_MinMaxPercent_Conc}}
	\begin{tabular}{|c|c|c|c|c|} \hline
	  \multirow{2}{*}{Delaunay triangulation \vspace{0.15in}} & 
	  \multicolumn{2}{|c|}{Concentration} & 
      \multicolumn{2}{|c|}{\% of nodes violated} \\ 
      \cline{2-3} \cline{4-5}
	  $\left(Nv, Nele, h \right)$ & Minimum & Maximum & (Non-negative 
	  constraint) & (Maximum constraint) \\ \hline
	  $(539,906,0.1)$ & -9.64 & 4.53 & 14.29 & 8.16  \\
	  $(1564,2826,0.05)$ & -1.54 & 3.57 & 18.03 & 1.53  \\
	  $(5090,9620,0.025)$ & -4.47 & 3.30 & 16.01 & 0.29  \\
	  $(18372,35665,0.0125)$ & -2.95 & 2.56 & 12.89 & 0.04  \\
	  $(69995,137881,0.00625)$ & -1.62 & 2.50 & 8.06 & 0.00  \\
	  \hline
	\end{tabular}
\end{table}

%****************************************;
%                                        ;
%  NAME                                  ;
%    S3_MP_Restrictions2_hLGM2.tex       ;
%                                        ;
%  WRITTEN BY                            ;
%    Maruti Kumar Mudunuru               ;
%    Kalyana Babu Nakshatrala            ;
%                                        ;
%****************************************;

%-------------------------------------------------------------;
%  Subsubsection: Errors in local and global species balance  ;
%-------------------------------------------------------------;
\subsubsection{\textbf{Errors in local and global species balance}}
\label{SubSubSec:Errors_LSB_GSB}
It is well-known that without using a post-processing method, the discrete 
standard single-field Galerkin formulation \emph{does not possess} local 
and global conservation properties \cite{2000_Hughes_etal_JCP_v163_p467_p488,
2012_Burdakov_etal_JCP_v231_p3126_p3142}. In finite element literature, 
there are various post-processing methods to quantify the errors incurred 
in satisfying local and global species balance \cite{Cook,Gresho_Sani_v1,
Donea_Huerta,Zienkiewicz_etal_SeventhEdition_v1}. Herein, we shall use 
Herrmann's method of optimal sampling \cite[Subsection 15.2.2]{Zienkiewicz_etal_SeventhEdition_v1},
which is a popular post-processing technique to obtain derived quantities 
from primary variables (such as the concentration variable in single-field 
finite element formulations). In the context of recovery-based error estimators 
\cite{Babuska_Whiteman_Strouboulis}, this post-processing method is also 
known as traditional global smoothing method \cite[Subsection 9.9]{Cook}.
To summarize, this method involves minimizing a unconstrained least-squares 
flux functional to obtain nodal flux vectors. Then, using these flux vectors, local and global species balance errors are computed.
 
Qualitatively, the contours corresponding to local species balance errors 
for test problems \#1, \#2, and \#3 on \emph{coarse} DMP-based meshes 
are shown in Figure \ref{Fig:Prob123_Local_Mass_Errors}. From this figure, 
it is clear that test problem \#3 exhibits considerable errors in satisfying 
local species balance. Quantitatively, Table \ref{Tab:Errors_Local_Global_Species_Balance} 
provides local and global species balance errors on traditional $h$-refined 
meshes for test problems \#1 and \#2. From this table, it can be concluded 
that the decrease in these errors is slow. Hence, there is need for locally 
conservative DMP-preserving finite element methods.

%---------------------------------------------------------;
%  Remark: Locally conservative global smoothing methods  ;
%---------------------------------------------------------;
\begin{remark}
  It should be noted that there are various ways in which one can develop 
  locally conservative global smoothing methods. For example, this can be 
  achieved by constructing a constrained monotonic regression based method 
  (following the procedure outlined by Burdakov \emph{et al.}
  \cite{2012_Burdakov_etal_JCP_v231_p3126_p3142}) to obtain a conservative 
  flux vector. However, as discussed in Section \ref{Sec:S1_MP_Introduction} 
  and Section \ref{Sec:S3_MP_GeneralRemarks}, it should be note that such 
  type of post-processing methods are not variationally consistent and refutes 
  the purpose of developing physics-compatible DMP-based meshes.
\end{remark}

%-----------------------------------------------------;
%  Table: Errors in local and global species balance  ; 
%-----------------------------------------------------;
\begin{table}
  \centering
	\caption{Errors in local and global species balance: For pure 
	  isotropic and anisotropic diffusion equation.
	\label{Tab:Errors_Local_Global_Species_Balance}}
	\begin{tabular}{|c|c|c|c|c|c|} \hline
	  \multirow{2}{*}{$h$} & \multicolumn{2}{|c|}{Test problem \#1} & 
	  \multirow{2}{*}{$h$} & \multicolumn{2}{|c|}{Test problem \#2} \\ 
      \cline{2-3} \cline{5-6}
	  & Local error (abs. max.) & Global error & 
	  & Local error (abs. max.) & Global error \\ \hline
	  0.1 & $1.80 \times 10^{-3}$ & $1.53 \times 10^{-2}$ & 
	  0.08 & $5.75 \times 10^{-4}$ & $-1.44 \times 10^{-4}$ \\
	  0.05 & $4.66 \times 10^{-4}$ & $1.23 \times 10^{-2}$ & 
	  0.042 & $1.47 \times 10^{-4}$ & $-1.38 \times 10^{-4}$ \\
	  0.025 & $3.48 \times 10^{-4}$ & $8.48 \times 10^{-3}$ & 
	  0.028 & $7.34 \times 10^{-5}$ & $-1.12 \times 10^{-4}$ \\
	  0.0125 & $7.14 \times 10^{-4}$ & $7.82 \times 10^{-3}$ & 
	  0.0135 & $1.59 \times 10^{-5}$ & $-6.83 \times 10^{-5}$ \\
	  0.00625 & $6.75 \times 10^{-4}$ & $5.38 \times 10^{-3}$ & 
	  0.0075 & $5.46 \times 10^{-6}$ & $-4.36 \times 10^{-5}$ \\
	  \hline
	\end{tabular} 
\end{table}

%****************************************;
%                                        ;
%  NAME                                  ;
%    S3_MP_Restrictions3.tex             ;
%                                        ;
%  WRITTEN BY                            ;
%    Maruti Kumar Mudunuru               ;
%    Kalyana Babu Nakshatrala            ;
%                                        ;
%****************************************;

%===============================================================;
%  Subsection: Sufficient conditions for a rectangular element  ;
%===============================================================;
\subsection{Sufficient conditions for a rectangular element}
\label{SubSec:MeshRestrict_Rectangular_Element} 
For the sake of illustration, consider a rectangular element whose 
vertices are located at $(0,0)$, $(a,0)$, $(a,b)$, and $(0,b)$. Our 
objective is to derive conditions on $a$ and $b$, such that the local 
stiffness matrix is weakly diagonally dominant. Herein, we consider a 
pure anisotropic diffusion equation and derive restrictions on the 
coordinates of the rectangular element. Based on the lines of the 
\emph{local stiffness restriction method} outlined in subsubsection 
\ref{Subsubsec:Global_Local_Stiffness_Restriction_Methods}, the local 
stiffness matrix for the case when diffusivity $\mathbf{D} (\mathbf{x}) 
= \left(\begin{smallmatrix}
D_{xx} & D_{xy} \\ 
D_{xy} & D_{yy} \end{smallmatrix} \right)$ is anisotropic and constant 
in $\Omega$ is given as follows:
%----------------------------------------------------------------;
%  Equation: Local stiffness matrix for anisotropic diffusivity  ;
%----------------------------------------------------------------;
\begin{align}
  \label{Eqn:Stiffness_Matrix_Q4}
  \boldsymbol{K}_{e} = 
  \left( \begin{array}{cccc}
    \frac{b D_{xx}}{3a} + \frac{D_{xy}}{2} + \frac{a D_{yy}}{3b} &
    -\frac{b D_{xx}}{3a} + \frac{a D_{yy}}{6b} &
    \frac{b D_{xx}}{6a} - \frac{a D_{yy}}{3b} &
    -\frac{b D_{xx}}{6a} - \frac{D_{xy}}{2} - \frac{a D_{yy}}{6b} \\
    -\frac{b D_{xx}}{3a} + \frac{a D_{yy}}{6b} & 
    \frac{b D_{xx}}{3a} - \frac{D_{xy}}{2} + \frac{a D_{yy}}{3b} &
    -\frac{b D_{xx}}{6a} + \frac{D_{xy}}{2} - \frac{a D_{yy}}{6b} &
    \frac{b D_{xx}}{6a} - \frac{a D_{yy}}{3b} \\
    \frac{b D_{xx}}{6a} - \frac{a D_{yy}}{3b} &
    -\frac{b D_{xx}}{6a} + \frac{D_{xy}}{2} - \frac{a D_{yy}}{6b} &
    \frac{b D_{xx}}{3a} - \frac{D_{xy}}{2} + \frac{a D_{yy}}{3b} &
    -\frac{b D_{xx}}{3a} + \frac{a D_{yy}}{6b} \\
    -\frac{b D_{xx}}{6a} - \frac{D_{xy}}{2} - \frac{a D_{yy}}{6b} &
    \frac{b D_{xx}}{6a} - \frac{a D_{yy}}{3b} &
    -\frac{b D_{xx}}{3a} + \frac{a D_{yy}}{6b} &
    \frac{b D_{xx}}{3a} + \frac{D_{xy}}{2} + \frac{a D_{yy}}{3b}
  \end{array} \right)
\end{align}

%==============================;
%  Subsection: Condition No-7  ;
%==============================;
\subsection*{Condition No-7}
From AM-GM inequality, we have the following relation:
%--------------------------------------------;
%  Equation: AM-GM inequality (Condition-7)  ;
%--------------------------------------------;
\begin{align}
  \label{Eqn:Condition_7_AMGM_Ineq}
  \frac{b D_{xx}}{3a} + \frac{a D_{yy}}{3b} \geq \frac{2}{3} \sqrt{D_{xx} 
  D_{yy}} > \frac{2}{3} |D_{xy}| \geq \frac{1}{2} |D_{xy}|
\end{align}
which implies that the positive diagonal entries: $\left( \boldsymbol{K}_e 
\right)_{ii} > 0 \quad \forall i = 1,2,3,4$, is trivially satisfied. 

%==============================;
%  Subsection: Condition No-8  ;
%==============================;
\subsection*{Condition No-8}
Non-positive off-diagonal entries: $\left( \boldsymbol{K}_e \right)_{ij} 
\leq 0 \quad \forall i \neq j$ where $i = 1,2,3,4$, and $j = 1,2,3,4,$ needs 
to be satisfied. For instance, when $i = 1$ and $j = 2,3,4$, this restriction 
gives the following relations: 
%----------------------------------------;
%  Equation: Inequalities (Condition-8)  ;
%----------------------------------------;
\begin{subequations}
\begin{align}
  &\sqrt{\frac{D_{xx}}{2 D_{yy}}} \leq \frac{a}{b} 
  \leq \sqrt{\frac{2 D_{xx}}{D_{yy}}} \\
  &-\frac{b D_{xx}}{6a} - \frac{D_{xy}}{2} - 
  \frac{a D_{yy}}{6b} \leq 0
\end{align}
\end{subequations}
Additionally, we should also satisfy the restrictions imposed by equations 
\eqref{Eqn:D_xx}--\eqref{Eqn:D_xxD_yy} on diffusivity tensor. In a similar 
fashion, one can derive restrictions for other combinations of $i$ and $j$.

%==============================;
%  Subsection: Condition No-9  ;
%==============================;   
\subsection*{Condition No-9}
Weak diagonal dominance of rows: $|\left( \boldsymbol{K}_e \right)_{ii}| \geq 
\displaystyle \sum \limits_{i \neq j} |\left( \boldsymbol{K}_e \right)_{ij}| 
\quad \forall i,j$ where $i = 1,2,3,4$, and $j = 1,2,3,4,$ is trivially satisfied 
if Condition-7 and Condition-8 are met. This is because from equation 
\eqref{Eqn:Stiffness_Matrix_Q4}, it is evident that $ \left( \boldsymbol{K}_e 
\right)_{ii} + \displaystyle \sum \limits_{i \neq j} \left( \boldsymbol{K}_e 
\right)_{ij} = 0 \quad \forall i,j$ where $i = 1,2,3,4$, and $j = 1,2,3,4$.

Finally, as explained in subsubsection \ref{Subsubsec:Global_Local_Stiffness_Restriction_Methods},
extending the local stiffness restriction approach to incorporate advection and 
linear reaction is straightforward and shall not be dealt with to save space. 
Moreover, it is easy to construct mesh restrictions for any set of arbitrary 
coordinates of a quadrilateral element using symbolic packages like \textsf{Mathematica}. 
But the resulting inequalities will be more complex to analyze mathematically 
and visualize graphically.

%****************************************;
%                                        ;
%  NAME                                  ;
%    S5_MP_Conclusions.tex               ;
%                                        ;
%  WRITTEN BY                            ;
%    Maruti Kumar Mudunuru               ;
%    Kalyana Babu Nakshatrala            ;
%                                        ;
%****************************************;
\section{CONCLUDING REMARKS AND OPEN QUESTIONS}
\label{Sec:S5_MP_Conclusions}
We outlined a general procedure to obtain the restrictions that are needed for 
a computational grid to satisfy various mathematical principles -- comparison 
principles, maximum principles, and the non-negative constraint. We illustrated 
the workings of this procedure by obtaining the mesh restrictions for T3 and Q4 
finite elements. The procedure is, however, equally applicable to other low-order 
finite elements. 

First, we critiqued three different approaches to satisfy maximum principles, 
comparison principles, and non-negative constraint for a general linear 
second-order elliptic equation. A pictorial description of a generic 
relationship between DMPs, DCPs, and NC based on a Venn diagram is proposed. 
This sketch helps to easily relate the space of solutions obtained using mesh 
restrictions, non-negative numerical formulations, and post-processing methods. 
We then presented necessary and sufficient conditions on the stiffness matrix 
$\boldsymbol{K}$ to meet the mathematical properties. Using these conditions, 
we derived stronger and weaker mesh restrictions for T3 element. The stronger 
mesh restriction corresponds to the \texttt{anisotropic non-obtuse angle condition} 
while the weaker one corresponds to the \texttt{generalized Delaunay-type angle 
condition}. Motivated by these mesh restriction conditions, different kinds of 
P\'{e}clet and Damk\"{o}hler numbers are proposed for advective-diffusive-reactive 
systems when diffusivity is anisotropic.

For isotropic diffusivity, we established that acute-angled or right-angled 
triangle is sufficient to satisfy discrete principles. However, for anisotropic 
diffusivity, we showed that in order to satisfy DMPs, DCPs, and NC, all the 
dihedral angles of a simplex measured in the metric of $\widetilde{\boldsymbol{D}}
^{-1}_{\Omega_e}$ have to be either $\mathcal{O} \left( h \|\boldsymbol{v}\|_{\infty,
\mathcal{T}_{h}} + h^2 \| \alpha \|_{\infty,\mathcal{T}_{h}} \right)$ acute/non-obtuse 
or $\mathcal{O} \left( h \|\boldsymbol{v}\|_{\infty,\mathcal{T}_{h}} + h^2 \| 
\alpha \|_{\infty,\mathcal{T}_{h}} \right)$ Delaunay. Pictorially, this means 
that the feasible region for T3 and Q4 elements to satisfy various discrete 
principles is based on a metric tensor whose components are a function of 
anisotropic diffusivity with respect to a suitable coordinate system. Then, 
an anisotropic metric tensor and an iterative algorithm to generate various 
types of DMP-based triangulations are described. Different numerical examples 
and respective DMP-based triangular meshes are presented for different types 
of $\mathbf{D}(\mathbf{x})$ to demonstrate the pros and cons of imposing mesh 
restrictions. Furthermore, the errors incurred in satisfying local and global 
species balance are documented. Based on these numerical experiments, the 
following inferences can be drawn:
%-----------------------------------------;
%  Inferences based on numerical results  ;
%-----------------------------------------;
\begin{enumerate}[(C1)]
  \item For pure isotropic or anisotropic diffusion equation, a coarse DMP-based 
    triangulation is sufficient to satisfy various discrete principles. However, 
    for advection-dominated and reaction-dominated scenarios, we need a highly 
    refined DMP-preserving computational mesh to obtain non-negative solutions. 
  \item Existing traditional and non-traditional methods of $h$-refinement may not 
    guarantee the satisfaction of DMPs, DCPs, and NC always. 
  \item On coarse DMP-based meshes, errors incurred in satisfying local and global 
    species balance for highly anisotropic diffusion-type problems is considerable 
    due to the skewed nature of the mesh elements. Moreover, the decrease in local 
    and global species balance errors upon $h$-refinement is slow.
  \item DMP-based meshes change as one alters the underlying anisotropic diffusivity 
    tensor.
\end{enumerate}

In the light of the recent developments and motivated by the above discussions, we 
have chosen to emphasize on the following \emph{four open problems} that we consider
particularly interesting in view of their mathematical richness, numerical challenges, 
and potential applications:
%---------------------------------------------;
%  Open problems: OP-1, OP-2, OP-3, and OP-4  ;
%---------------------------------------------;
\begin{enumerate}[(OP1)]
  \item In this paper, all the meshes used in the numerical examples are of 
    Delaunay-type. This is because most of the existing open source mesh 
    generators such as \textsf{BAMG}, \textsf{Gmsh}, \textsf{Triangle}, 
    \textsf{BL2D}, \textsf{Mmg3d}, and \textsf{CGALmesh} are Delaunay. 
    Recently, Erten and \"{U}ng\"{o}r \cite{2007_Erten_Ungor_CCCG_p205_p208,
    2009_Erten_Ungor_PVD_p192_p201,2009_Erten_Ungor_SIAMJSC_v31_p2103_p2130} 
    have developed a non-obtuse/acute angled mesh generator called \textsf{aCute} 
    by modifying \textsf{Triangle}. However, \textsf{aCute} is restricted to 
    2D and Eucledian metric tensors. Hence, such a software can only be used 
    to satisfy discrete principles for problems involving heterogeneous isotropic 
    diffusivity. Having an anisotropic non-obtuse/acute $\mathcal{M}$-uniform 
    meshing software would be of great importance, as the numerical solutions 
    obtained from these meshes not only satisfy DMPs, DCPs, and NC, but also 
    converge uniformly (an attractive aspect in finite element analysis 
    \cite{Ciarlet_Raviart_CMAME_1973_v2_p17}). \emph{To date, there is no such 
    mesh generator}. Developing such a software will have a profound impact on 
    obtaining physically meaningful numerical solutions for diffusion-type equations. 
  \item For advection-dominated and reaction-dominated advection-diffusion-reaction 
    problems, mesh refinement that adheres to DMPs is needed to obtain stable and 
    sufficiently accurate numerical solutions. As described in (C2), not every method 
    of $h$-refinement is DMP-preserving. Hence, a consistent way of generating a 
    DMP-based $h$-refined mesh (that satisfies \texttt{Generalized Delaunay-type 
    angle condition}) is still unresolved.
  \item From (C3), it is apparent that local and global mass conservation property 
    is needed. An approach to preserve such a property without violating DMPs, DCPs, 
    and NC, is to obtain mesh restrictions for mixed Galerkin formulation based on 
    lowest-order Raviart-Thomas spaces. Recently, Huang and Wang \cite{2014_Huang_Wang_arXiv_1401_6232} 
    have developed a methodology to satisfy DMPs for a class of locally conservative 
    weak Galerkin methods using lowest-order Raviart-Thomas spaces. However, this 
    methodology is limited to pure anisotropic steady-state diffusion equation in 
    two-dimensions. Hence, a mesh restriction based method to satisfy different 
    discrete principles, local species balance, and global species balance for 
    anisotropic advection-diffusion-reaction equations thus far is unsolved.
  \item In order to construct DMP-based meshes for low-order non-simplicial finite 
    elements such as Q4, from subsection \ref{SubSec:MeshRestrict_Rectangular_Element}, 
    it is evident that stronger and weaker mesh conditions are needed. \emph{So 
    far, there are no mesh restriction theorems analogous to simplicial meshes 
    that can provide a general framework to construct non-simplicial meshes for 
    anisotropic diffusivity tensors}. Theoretically and numerically, it would 
    be very interesting and informative to have a comparative study on the 
    performance of simplicial vs. non-simplicial DMP-based meshes for various 
    benchmark problems discussed in Section \ref{Sec:S4_MP_Restrictions}.
\end{enumerate}
Nevertheless, due to enormous research activity in the field of advection-diffusion-reaction
equations, it is impossible to list every open question on preserving DCPs, DMPs, 
and NC. To conclude, the research findings in this paper will be invaluable to the 
research community and finite element practitioners in two respects. \emph{First}, 
it will guide the existing users on the restrictions to be placed on the computational 
mesh to meet important mathematical properties like maximum principles, comparison 
principles, and the non-negative constraint. \emph{Second}, for complex geometries 
and highly anisotropic media, this study has clearly shown that placing restrictions 
on computational grids may not always be a viable approach to achieve physically 
meaningful non-negative solutions. \emph{We hope that this research work will motivate 
researchers to develop new methodologies for advective-diffusive-reactive systems 
that satisfy local and global species balance, comparison principles, maximum principles, 
and the non-negative constraint on coarse general computational grids}.

%===================;
%  Acknowledgments  ;
%===================;
\section*{ACKNOWLEDGMENTS}
The authors acknowledge the support from the DOE Nuclear 
Energy University Programs (NEUP). The opinions expressed 
in this paper are those of the authors and do not necessarily 
reflect that of the sponsors. The authors are grateful to the 
developers of the open source scientific software packages: 
\textsf{BAMG}, \textsf{FreeFem++}, and \textsf{Gmsh}. 

%======================;
%  Include references  ;
%======================; 
\bibliographystyle{unsrt}
\bibliography{Master_References/Master_References,Master_References/Books}

%===================================;
%  Figures from numerical analyses  ;
%  PLEASE DONOT DELETE THIS FILE    ;
%===================================;
%****************************************;
%                                        ;
%  NAME                                  ;
%    MP_Figures.tex                      ;
%                                        ;
%  WRITTEN BY                            ;
%    Maruti Kumar Mudunuru               ;
%    Kalyana Babu Nakshatrala            ;
%                                        ;
%****************************************; 
 
%---------------------------------------------------------------------;
%  Figure-2: ABAQUS heat conduction L-shaped domain T3 and Q4 meshes  ;
%---------------------------------------------------------------------;
\begin{figure}
  \subfigure[T3 mesh]{\includegraphics[scale=0.35]
  {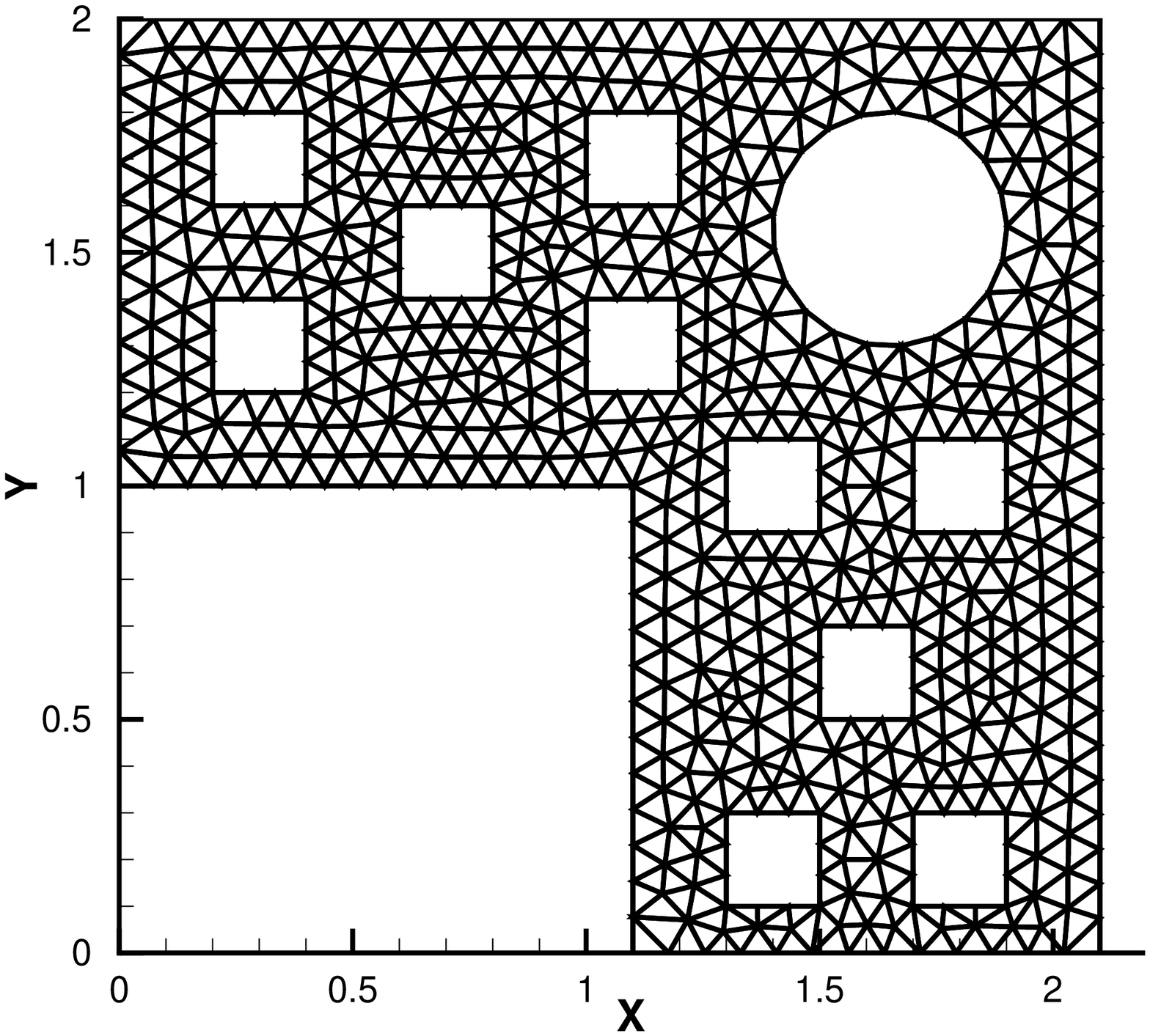}}
  \subfigure[Q4 mesh]{\includegraphics[scale=0.35]
  {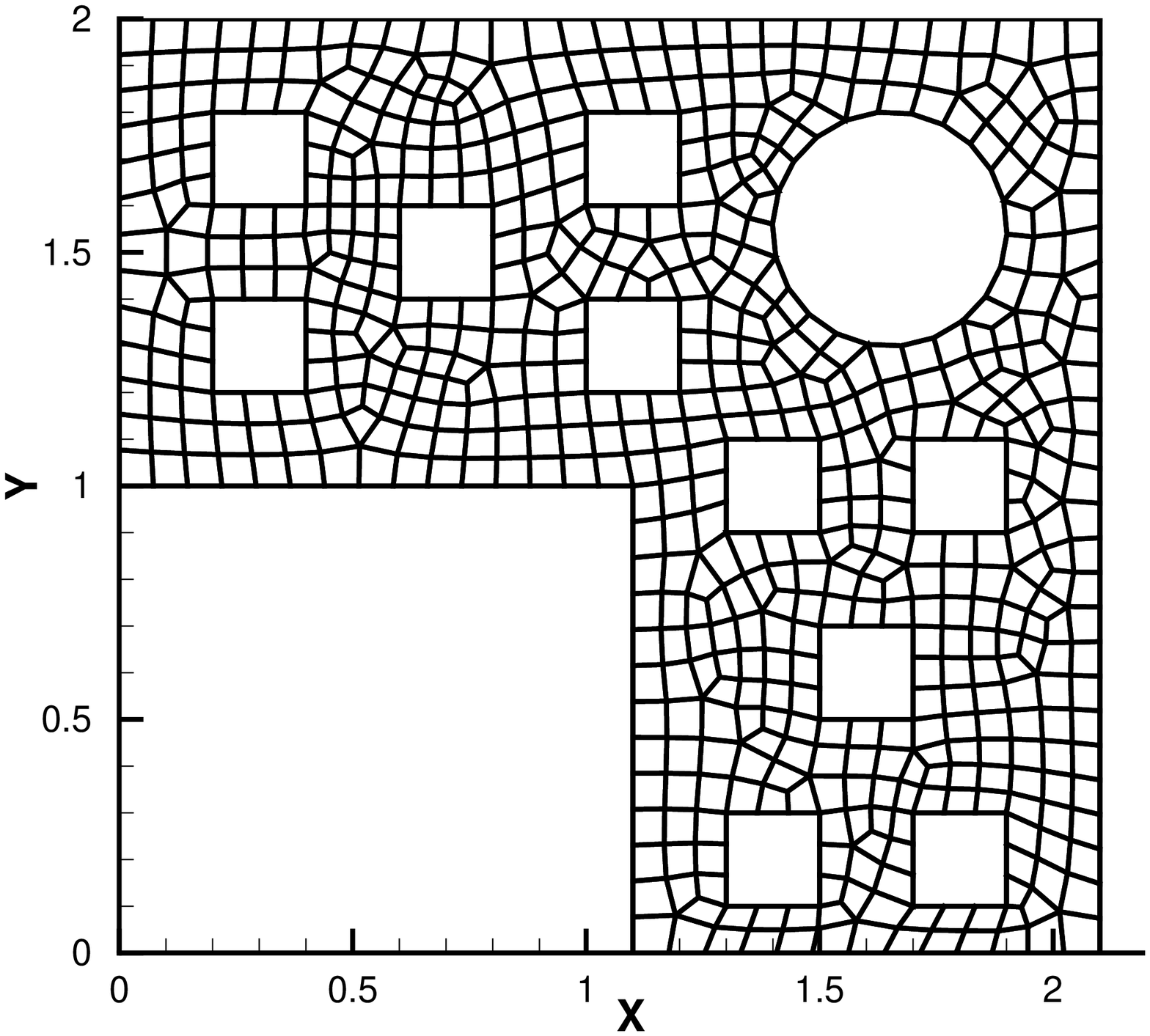}}
  \caption{\textsf{ABAQUS} \textsc{unstructured meshes for an L-shaped domain 
    with multiple holes}: The left and right figures show an instance of three-node 
    triangular and four-node quadrilateral meshes employed in the numerical 
    simulation of a pure anisotropic diffusion problem using \textsf{ABAQUS}.
    \label{Fig:ABAQUS_T3_Q4_Meshes}}
\end{figure}

%------------------------------------------------------------;
%  Figure-3: ABAQUS contour plots based on T3 and Q4 meshes  ;
%------------------------------------------------------------;
\begin{figure}
  \includegraphics[scale=0.35]
  {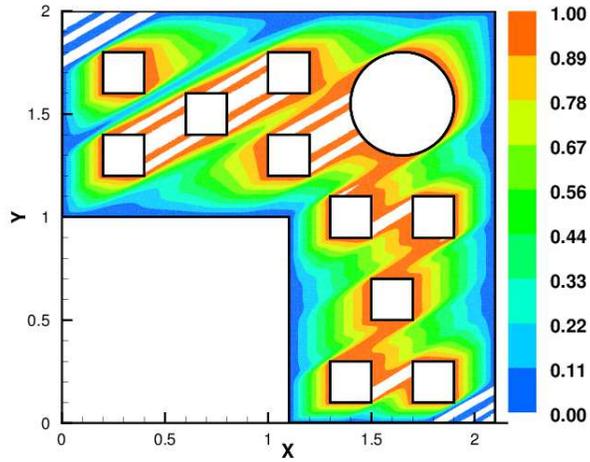}
  \caption{\textsf{ABAQUS} \textsc{numerical simulation for an L-shaped domain 
    with multiple holes}: The contours of concentration obtained using \textsf{ABAQUS} 
    are based on three-node triangular mesh.
    \label{Fig:ABAQUS_T3_Q4_Contours}}
\end{figure}

%------------------------------------------------------------------;
%  Figure-4: Min and Max concentration plots for T3 and Q4 meshes  ;
%------------------------------------------------------------------;
\begin{figure}
  \subfigure[Non-negative constraint violation]{\includegraphics[scale=0.12]
  {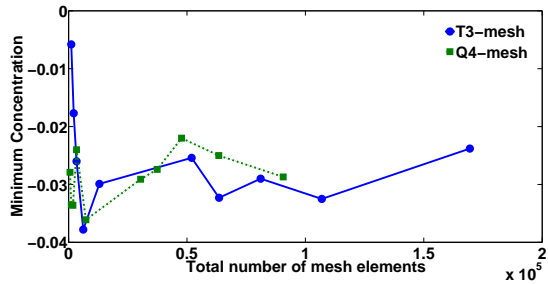}}
  \subfigure[Maximum constraint violation]{\includegraphics[scale=0.12]
  {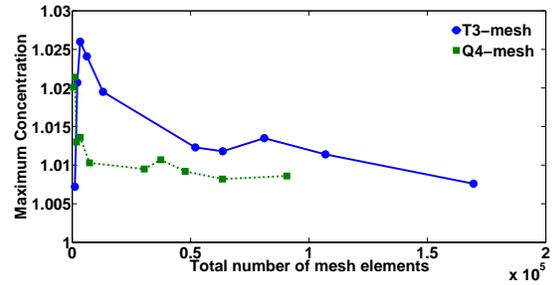}}
  \caption{\textsc{Minimum and maximum values for concentration in an L-shaped 
    domain with multiple holes}: The left and right figures show the minimum 
    and maximum values attained in the computational domain based on the numerical 
    simulations for various three-node triangular and four-node quadrilateral 
    meshes using \textsf{ABAQUS}. From the above figures, it is evident that 
    the violation in the non-negative and maximum constraints do not reduce 
    on mesh refinement.
    \label{Fig:ABAQUS_T3_Q4_MinMaxConc}}
\end{figure}

%----------------------------------------------;
%  Figure-5: % violation for T3 and Q4 meshes  ;
%----------------------------------------------;
\begin{figure}
  \subfigure[\% violation: Non-negative constraint]{\includegraphics[scale=0.12]
  {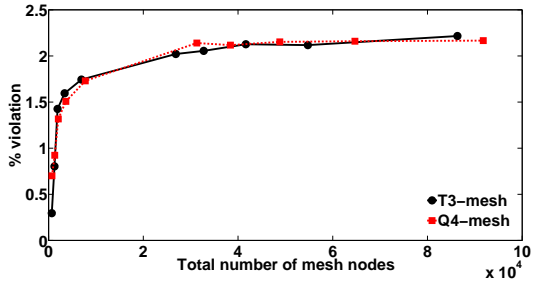}}
  \subfigure[\% violation: Maximum constraint]{\includegraphics[scale=0.12]
  {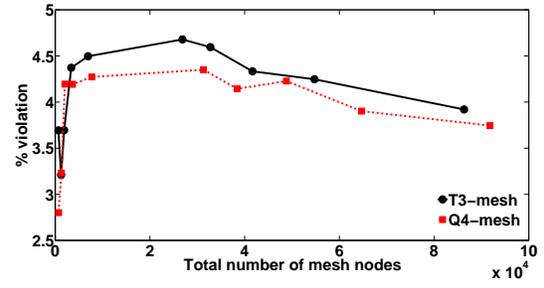}}
  \caption{\textsc{Percentage of violation in minimum and maximum constraints 
    for concentration in an L-shaped domain with multiple holes}: The left and 
    right figures show the percentage of nodes that have violated the non-negative 
    and maximum constraint for concentration obtained using \textsf{ABAQUS}. 
    These numerical results are based on various three-node triangular and 
    four-node quadrilateral meshes. In both of these cases, the percentage 
    of mesh nodes that violated these constraints never decreases to zero. 
    \label{Fig:ABAQUS_T3_Q4_MinMaxpercent}}
\end{figure}

%------------------------------------------------;
%  Figure-7: Venn diagram for mesh restrictions  ;
%------------------------------------------------;
\begin{figure}
  \psfrag{A1}{$\mathcal{V}_{\mathrm{DwMP}_{\boldsymbol{K}}}$}
  \psfrag{A2}{$\mathcal{V}_{\mathrm{DWMP}_{\boldsymbol{K}}}$}
  \psfrag{A3}{$\mathcal{V}_{\mathrm{DsMP}_{\boldsymbol{K}}}$}
  \psfrag{A4}{$\mathcal{V}_{\mathrm{DSMP}_{\boldsymbol{K}}}$}
  \psfrag{A5}{$\mathcal{V}_{\mathrm{DwCP}_{\boldsymbol{K}}}$}
  \psfrag{A6}{$\mathcal{V}_{\mathrm{DsCP}_{\boldsymbol{K}}}$}
  \includegraphics[scale=0.44]{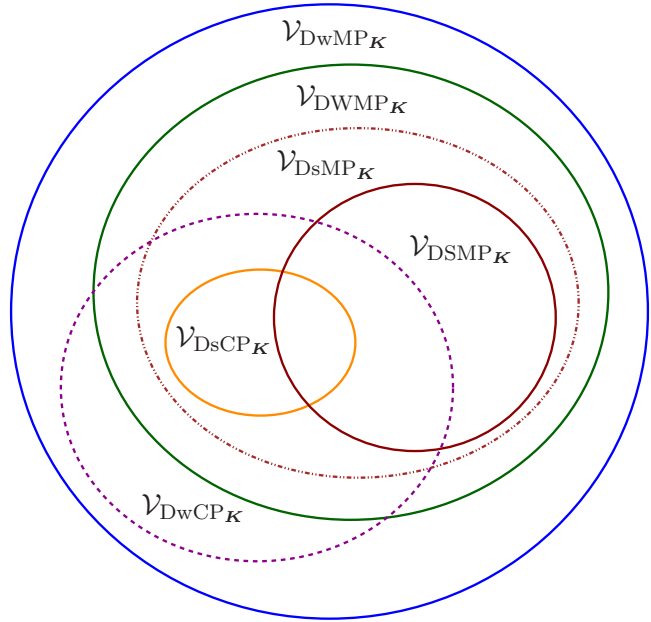}
  \caption{\textsc{Venn diagram for the space of solutions based on 
    mesh restrictions}: A pictorial description of the space of numerical 
    solutions satisfying various DMPs and DCPs based on equation 
    \eqref{Eqn:Discrete_Diffusion} and Theorem \ref{Thm:NeccSuff_Conditions_DMP}.
    It is evident from the above figure that 
    $\mathcal{V}_{\mathrm{DSMP}_{\boldsymbol{K}}} 
    \subset \mathcal{V}_{\mathrm{DsMP}_{\boldsymbol{K}}}
    \subset \mathcal{V}_{\mathrm{DWMP}_{\boldsymbol{K}}} 
    \subset \mathcal{V}_{\mathrm{DwMP}_{\boldsymbol{K}}}$
    and $\mathcal{V}_{\mathrm{DsCP}_{\boldsymbol{K}}} 
    \subset \mathcal{V}_{\mathrm{DwCP}_{\boldsymbol{K}}}$.
    But we would like to emphasize that we \emph{do not} 
    have the following enclosures: 
    $\mathcal{V}_{\mathrm{DwCP}_{\boldsymbol{K}}}
    \subset \mathcal{V}_{\mathrm{DWMP}_{\boldsymbol{K}}}$ 
    and $\mathcal{V}_{\mathrm{DsCP}_{\boldsymbol{K}}}
    \subset \mathcal{V}_{\mathrm{DSMP}_{\boldsymbol{K}}}$.     
    \label{Fig:VennDiagram_MeshRestrictions}}
\end{figure}

%-----------------------------------------------------;
%  Figure-8: Venn diagram for numerical formulations  ;
%-----------------------------------------------------;
\begin{figure}
  \psfrag{A0}{$\mathcal{V}_{\mathrm{NC}}$}
  \psfrag{A1}{$\mathcal{V}_{\mathrm{DwMP}}$}
  \psfrag{A2}{$\mathcal{V}_{\mathrm{DWMP}}$}
  \psfrag{A3}{$\mathcal{V}_{\mathrm{DsMP}}$}
  \psfrag{A4}{$\mathcal{V}_{\mathrm{DSMP}}$}
  \psfrag{A5}{$\mathcal{V}_{\mathrm{DwCP}}$}
  \psfrag{A6}{$\mathcal{V}_{\mathrm{DsCP}}$}
  \psfrag{A7}{{\small Numerical solution space based on}}
  \psfrag{A8}{{\small FDS, MFDM, FVM, FEM, and PP-methods}}
  \includegraphics[scale=0.44]{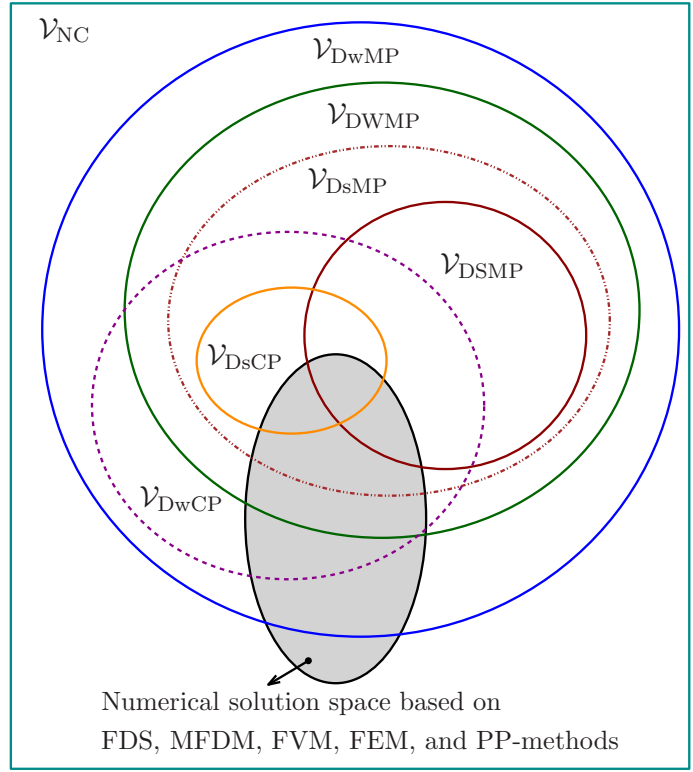}
  \caption{\textsc{Venn diagram for the space of solutions based on 
    various numerical formulations}: A pictorial description of the 
    space of numerical solutions satisfying various DMPs, DCPs, and NC. 
    In numerical literature, most of the numerical methods that exist 
    to satisfy various DMPs are mainly non-linear. In the past decade, 
    considerable advancements have been made to fulfill various version 
    of DMPs for a certain class of linear elliptic and parabolic partial 
    differential equations. But it should be noted that there is seldom 
    research progress related to satisfaction of different DCPs (see 
    \cite[Section 4]{2013_Nakshatrala_Mudunuru_Valocchi_JCP_v253_p278_p307}). 
    \emph{Hence, we would like to highlight that developing a general and 
    variationally consistent numerical technique to encompass all these 
    discrete principles is still an open problem.} 
    \label{Fig:VennDiagram_NumericalFormulations}}
\end{figure}

%------------------------------------------------;
%  Figure-9: Geometric properties of T3-element  ;
%------------------------------------------------;
\begin{figure}
  \psfrag{O}{$(0,0)$}
  \psfrag{x}{$x$}
  \psfrag{y}{$y$}
  \psfrag{P}{P}
  \psfrag{Q}{Q}
  \psfrag{R}{R}
  \psfrag{hP}{$h_p$}
  \psfrag{hQ}{$h_q$}
  \psfrag{hR}{$h_r$}
  \psfrag{aPQ}{$\beta_{pq}$}
  \psfrag{aPR}{$\beta_{pr}$}
  \psfrag{arq}{$\beta_{rq}$}
  \psfrag{nP}{$\mathbf{n}_p$}
  \psfrag{nQ}{$\mathbf{n}_q$}
  \psfrag{nR}{$\mathbf{n}_r$}
  \subfigure{\includegraphics[scale=0.75]{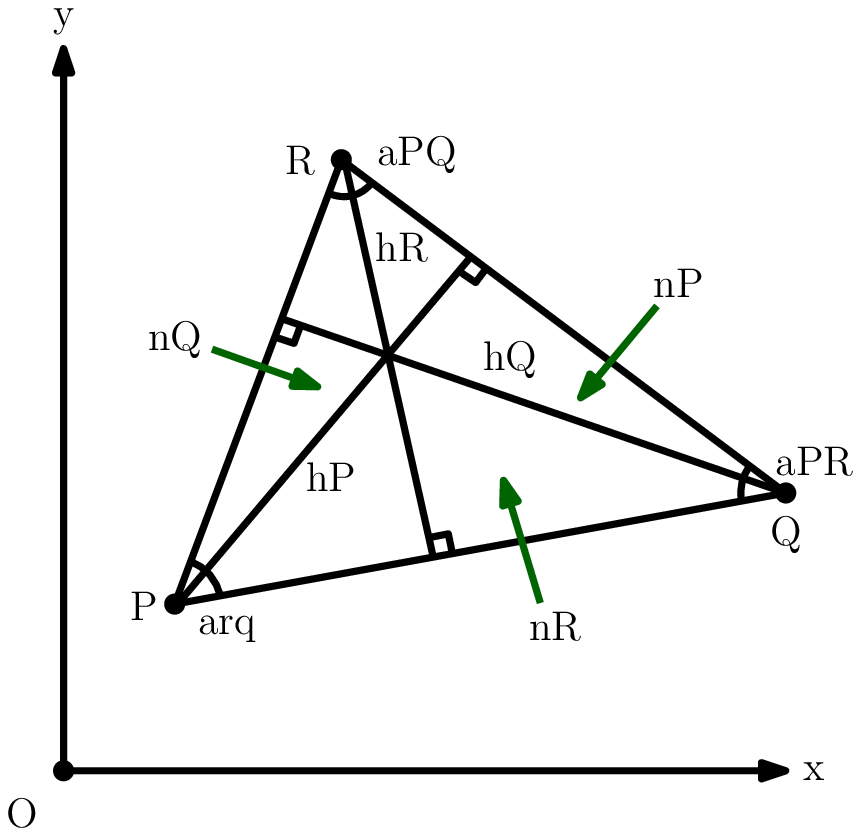}}
  \hspace{0.3in}
  \psfrag{x11}{$\hat{\boldsymbol{x}}_p$}
  \psfrag{x12}{$\hat{\boldsymbol{x}}_{p^{'}}$}
  \psfrag{x21}{$\hat{\boldsymbol{x}}_q$}
  \psfrag{x22}{$\hat{\boldsymbol{x}}_{q^{'}}$}
  \psfrag{x31}{$\hat{\boldsymbol{x}}_r$}
  \psfrag{x23}{$\hat{\boldsymbol{x}}_{r^{'}}$}
  \psfrag{a}{$\beta_{pq,\Omega_e}$}
  \psfrag{b}{$\beta_{pq,\Omega^{'}_e}$}
  \psfrag{e}{$e_{pq}$}
  \psfrag{D1}{$\Omega_e$}
  \psfrag{D2}{$\Omega^{'}_e$}
  \subfigure{\includegraphics[scale=0.81]{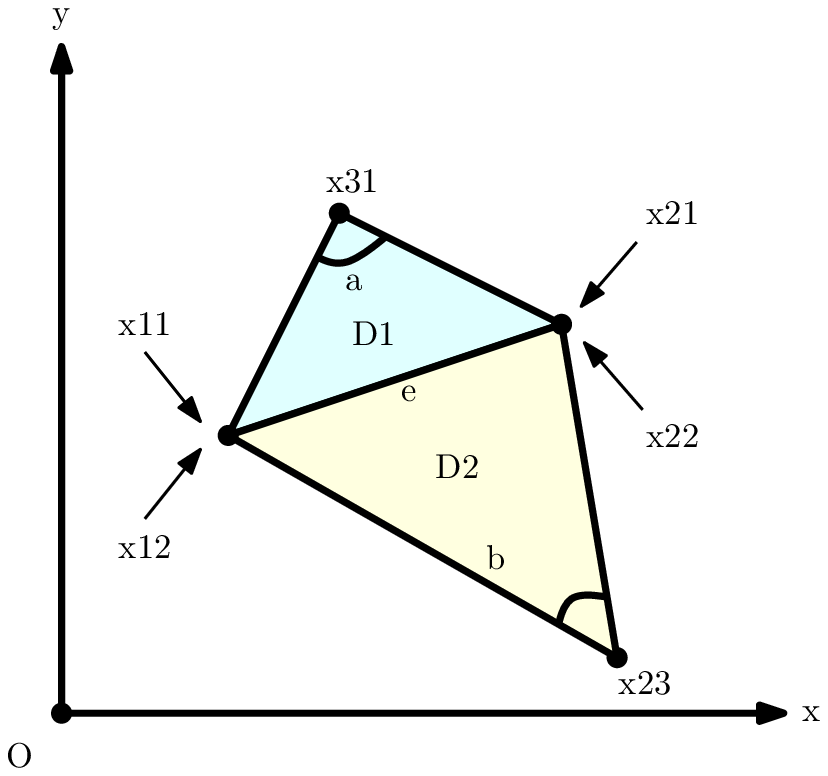}}
  \caption{\textsc{Geometrical properties of an arbitrary simplex in 2D}: The left 
    figure shows a pictorial description of various geometrical properties, such 
    as unit inward normals ($\mathbf{n}_p$, $\mathbf{n}_q$, and $\mathbf{n}_r$), 
    dihedral angles in Euclidean metric ($\beta_{pq}$, $\beta_{pr}$, and $\beta_
    {rq}$), and heights ($h_p$, $h_q$, and $h_r$) of an arbitrary element $\Omega
    _e \in \mathcal{T}_{h}$. Correspondingly, the vertices of this triangle PQR 
    are given by $\hat{\boldsymbol{x}}_p$, $\hat{\boldsymbol{x}}_q$, and $\hat{
    \boldsymbol{x}}_r$. The right figure shows an arbitrary patch of elements 
    $\Omega_e$ and $\Omega^{'}_e$, (which belong to the triangulation $\mathcal{T}
    _{h}$) sharing a common edge $e_{pq}$. The edge $e_{pq}$ connects the coordinates 
    $\hat{\boldsymbol{x}}_p$ (= $\hat{\boldsymbol{x}}_{p^{'}}$) and $\hat{\boldsymbol{x}}
    _q$ (= $\hat{\boldsymbol{x}}_{q^{'}}$). The dihedral angles in Euclidean metric 
    opposite to edge $e_{pq}$ are denoted by $\beta_{pq,\Omega_e}$ and $\beta_{pq,
    \Omega^{'}_e}$.
    \label{Fig:T3_Simplex_GeoProp}}
\end{figure}

%-------------------------------------------------------------------;
%  Figure-10: T3-Element (Isotropic and heterogeneous diffusivity)  ;
%-------------------------------------------------------------------;
\begin{figure}
  \subfigure{\includegraphics[scale=0.44]{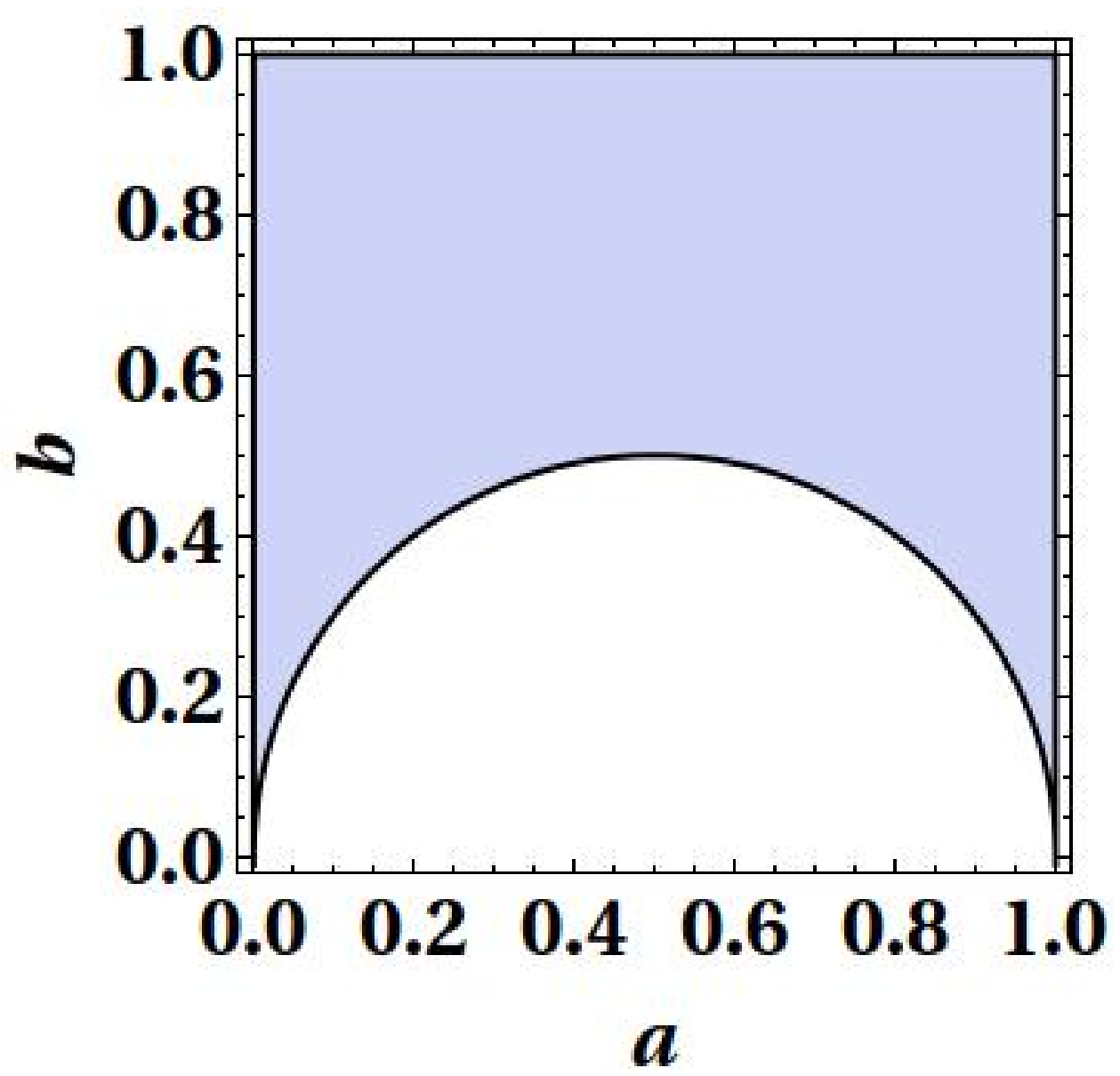}}
  \hspace{0.45in}
  %\vspace{0.1in}
  %%
  \psfrag{x}{$x$}
  \psfrag{y}{$y$}
  \psfrag{a}{$(0,0)$}
  \psfrag{b}{$(\frac{1}{2},0)$}
  \psfrag{c}{$(1,0)$}
  \psfrag{d}{Feasible region}
  \psfrag{e}{Acute-angled triangle}
  \psfrag{f}{Right-angled triangle}
  \subfigure{\includegraphics[scale=0.75]{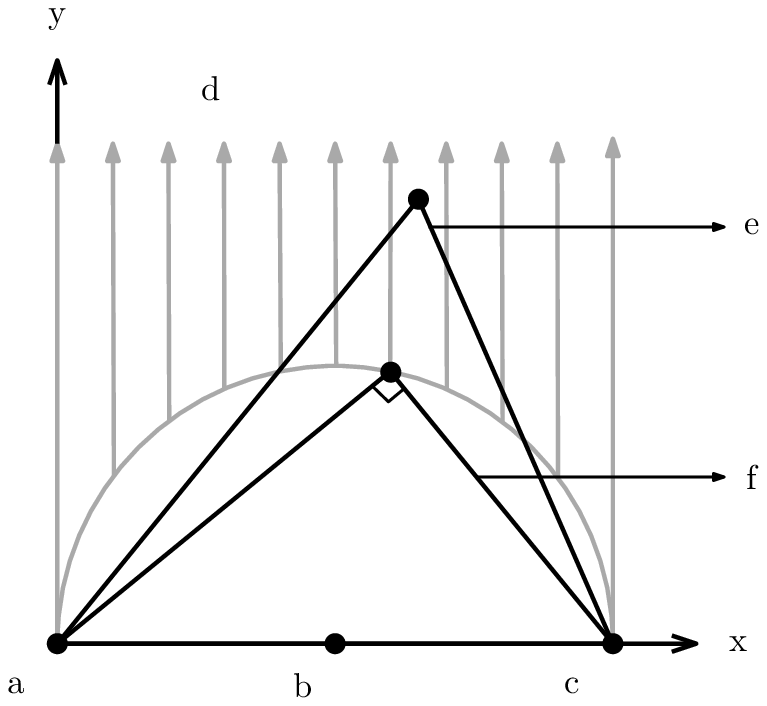}}
  \caption{\textsc{T3 element for heterogeneous isotropic diffusivity}: A 
    pictorial description of the feasible region (left figure) is shown in 
    light blue color. The right figure indicates that the point $(a,b)$ can 
    lie either on the circle with center $(\frac{1}{2},0)$ and radius $\frac{1}{2}$ 
    or outside the circular region. The points within the circular region are 
    infeasible. This results in two possibilities for choosing a T3 element 
    in the realm of the feasible region, which is either a right-angled 
    triangle or an acute-angled triangle.
    \label{Fig:T3_Isotropic_Actual_Element}}
\end{figure} 

%-----------------------------------------------------------;
%  Figure-11: T3-Element D_xy = 0 (Anisotropic Diffusivity) ;
%-----------------------------------------------------------;
\begin{figure}
  \subfigure[Feasible region for $(a,b)$]{\includegraphics[scale=0.44]
  {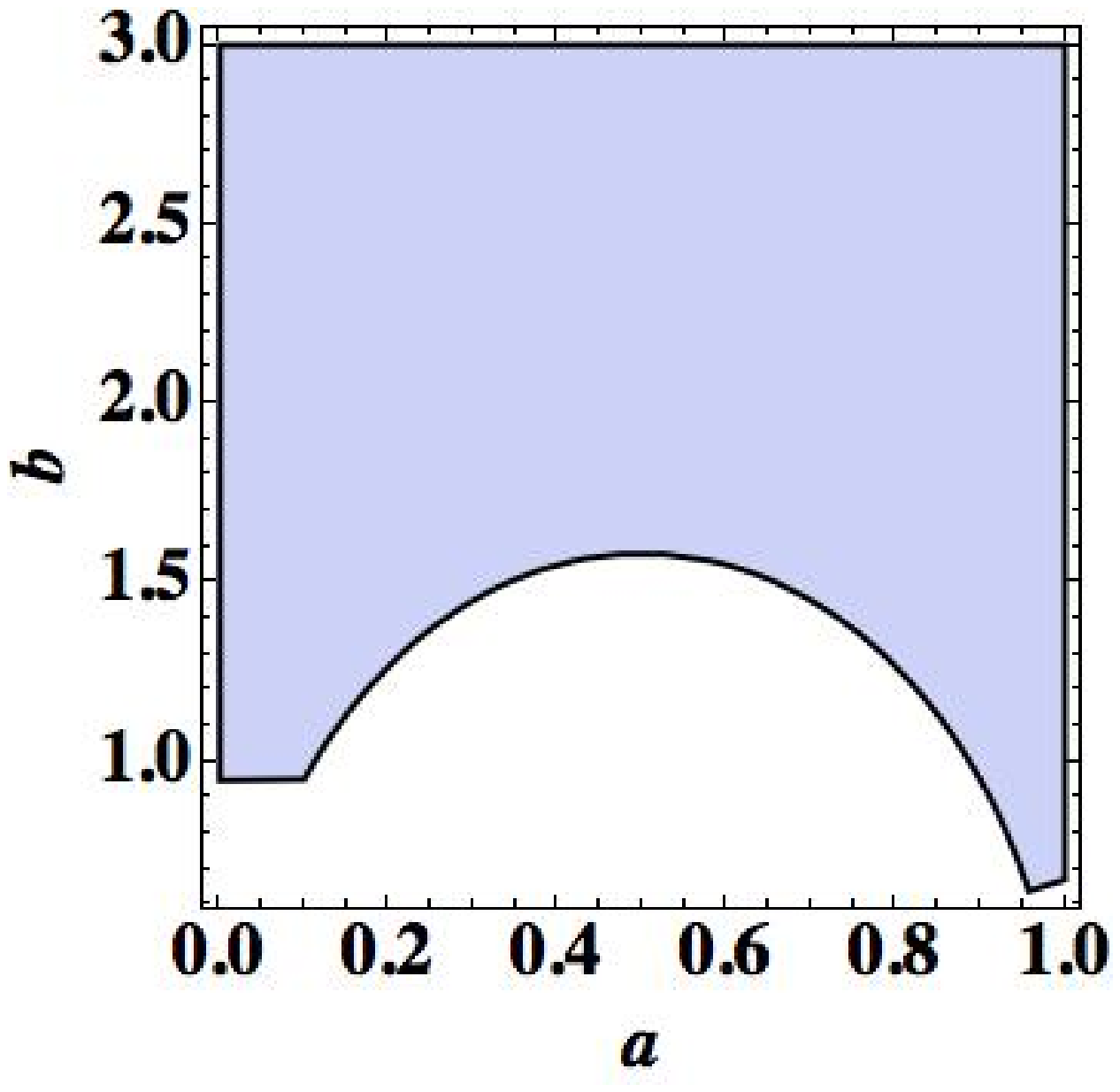}}
  \hspace{0.2in}
  \psfrag{a}{Right-angled triangle}
  \psfrag{b}{Acute-angled triangle}
  \subfigure[Possible T3 elements when $\widetilde{D}_{xy} = 0$]
  {\includegraphics[scale=0.44]{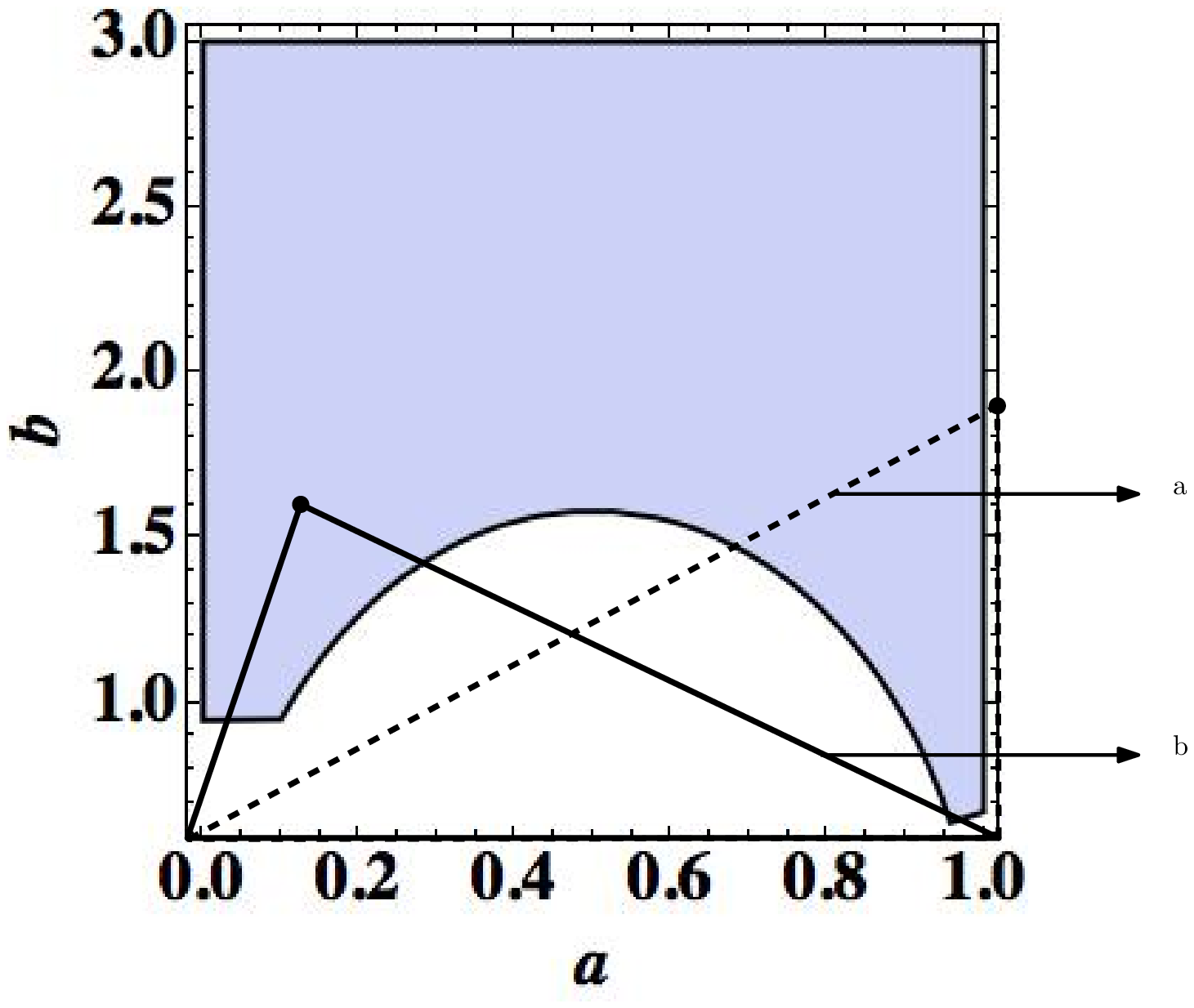}}
  \caption{\textsc{T3 element for anisotropic diffusivity when} 
    $\widetilde{D}_{xy} = 0$: A pictorial description of the feasible 
    region (left figure) for the coordinates $(a,b)$ is indicated in 
    light blue color. The numerical values for the two parameters, which 
    decide the feasible region, are chosen to be $\epsilon = 10$ and 
    $\eta = 0$. In this case, the right figure indicates that acute-angled 
    and right-angled triangles are possible. As $\epsilon$ increases, the 
    coordinate $b$ has to increase proportionally to satisfy the inequality 
    given by the equation \eqref{Eqn:Aniso_Diagonal_Entry_12}.
    \label{Fig:T3_AnisoDiff_Dxy_EqualToZero}}
\end{figure}

%-----------------------------------------------------------------;
%  Figure-16: T3-Element (Isotropic and Anisotropic Diffusivity)  ;
%-----------------------------------------------------------------;
\begin{figure}
  \centering
  \psfrag{M1}{ $(0,0)$}
  \psfrag{M2}{$(1,0)$}
  \psfrag{M3}{$(0,1)$}
  \psfrag{K1}{$\Omega_{\mathrm{\tiny ref}}$}
  \psfrag{X1}{$\xi_1$}
  \psfrag{Y1}{$\xi_2$}
  \psfrag{S1}{Shape functions}
  \psfrag{S2}{$\boldsymbol{x} = \hat{\boldsymbol{X}}^{\mathrm{T}} 
              \boldsymbol{N}^{\mathrm{T}}$}
  \psfrag{Ele1}{{\small (a) Reference element}}
  \psfrag{M1}{$(0,0)$}
  \psfrag{A1}{P}
  \psfrag{A2}{Q}
  \psfrag{A3}{R}
  \psfrag{K2}{$\Omega_{e}$}
  \psfrag{X2}{$x$}
  \psfrag{Y2}{$y$}
  \psfrag{Ele2}{{\small (b) Background mesh element}}
  \psfrag{D11}{{\small Isotropic}}
  \psfrag{D12}{{\small Diffusivity}}
  \psfrag{D21}{{\small Anisotropic}}
  \psfrag{D22}{{\small Diffusivity}}
  \psfrag{ID1}{$\mathrm{P}_1$}
  \psfrag{ID2}{$\mathrm{Q}_1$}
  \psfrag{ID3}{$\mathrm{R}_1$}
  \psfrag{K2}{$\Omega_{e}$}
  \psfrag{Acute11}{(c) {\small Delaunay/non-obtuse/acute}}
  \psfrag{Acute12}{{\small triangle in Euclidean metric}}
  \psfrag{AD1}{$\mathrm{P}_2$}
  \psfrag{AD2}{$\mathrm{Q}_2$}
  \psfrag{AD3}{$\mathrm{R}_2$}
  \psfrag{K2}{$\Omega_{e}$}
  \psfrag{Acute21}{(d) {\small Delaunay-type/non-obtuse/acute}}
  \psfrag{Acute22}{{\small triangle in $\widetilde{\boldsymbol{D}}_{\Omega_{e}}^{-1}$ metric}}
  \includegraphics[scale=0.75]{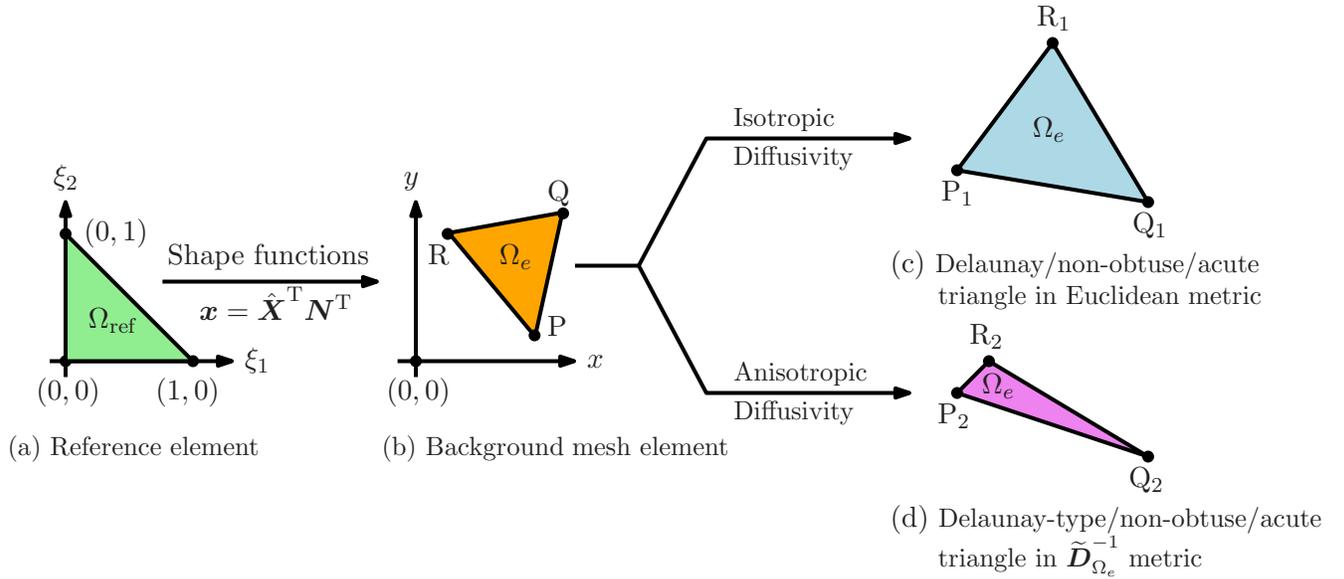}
  \caption{\textsc{DMP-based T3 elements for heterogeneous isotropic and anisotropic 
    diffusivity}: A pictorial description of a mesh generation procedure to
    obtain a new triangulation using a given background mesh. This new simplicial 
    mesh satisfies various discrete properties as contrary to the background mesh.    
    The procedure to obtain such a triangulation is iterative and is based on 
    Theorems \ref{Thm:Aniso_Weak_Condition} and \ref{Thm:Delaunay_Weak_Condition}.    
    \label{Fig:Isotropic_Anisotropic_T3_Element}}
\end{figure}

%-----------------------------------------------------------;
%  Figure-12: T3-Element D_xy < 0 (Anisotropic Diffusivity) ;
%-----------------------------------------------------------;
\begin{figure}
  \subfigure[Feasible region for $(a,b)$]{\includegraphics[scale=0.44]
  {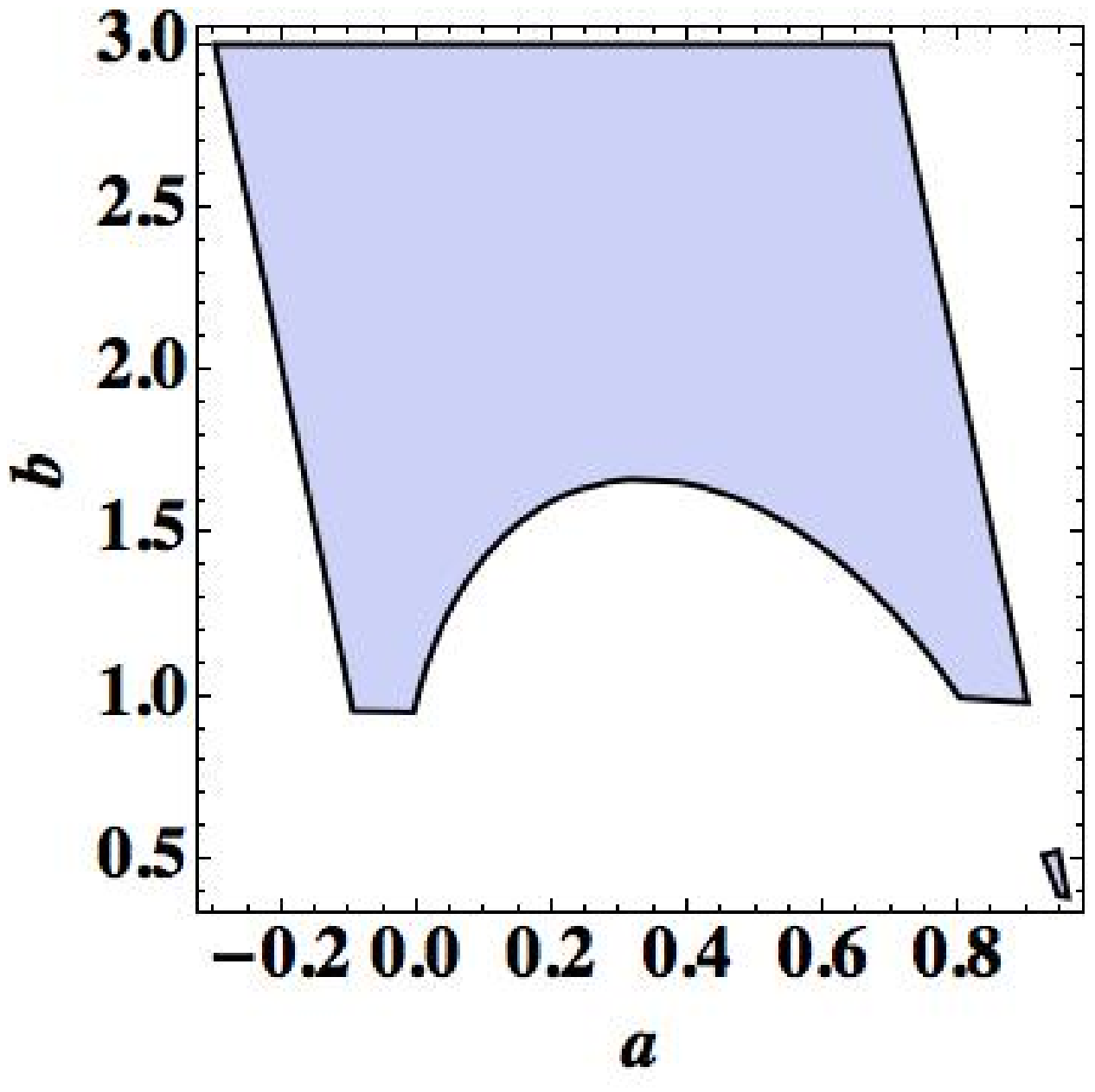}}
  \hspace{0.2in}
  \psfrag{a}{Acute-angled triangle}
  \psfrag{b}{Right-angled triangle}
  \psfrag{c}{Obtuse-angled triangle}
  \subfigure[Possible T3 elements when $\widetilde{D}_{xy} < 0$]
  {\includegraphics[scale=0.44]{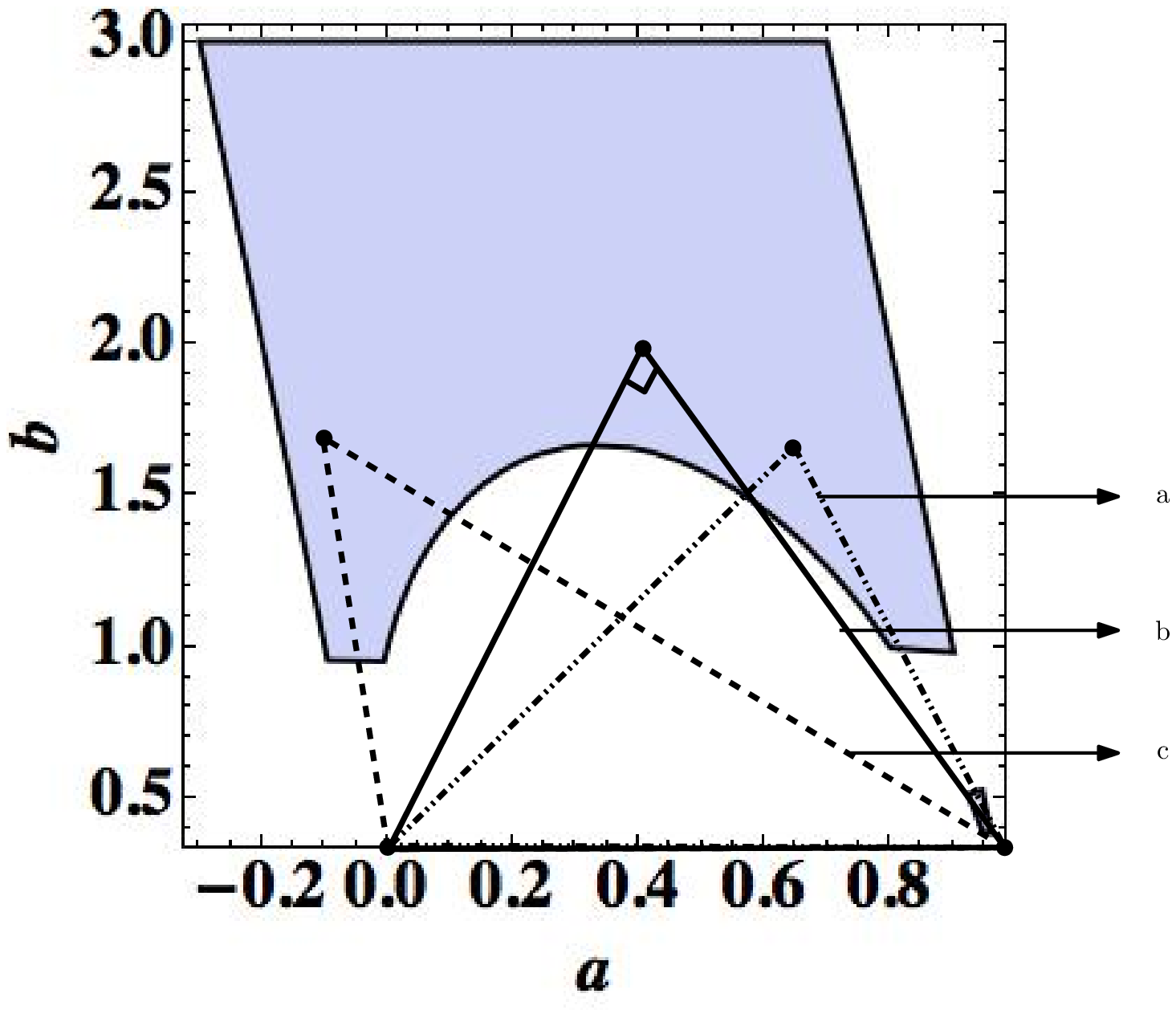}}
  \caption{\textsc{T3 element for anisotropic diffusivity when} 
    $\widetilde{D}_{xy} < 0$: The left figure indicates the 
    feasible region for the coordinates $(a,b)$ in light blue 
    color. The right figure indicates that when $\widetilde{D}_{xy} 
    < 0$, the (Euclidean metric) dihedral angles in the T3 
    element can be acute-angled or right-angled or even obtuse-angled. 
    In this case, we have chosen $\epsilon = 10$ and $\eta = -1$. 
    For a fixed $\eta$ as $\epsilon$ increases, the value of 
    coordinate $b$ also increases. So it is a daunting task 
    to find a viable T3 element. One can also notice that the 
    feasible region is \emph{not} contiguous.      
    \label{Fig:T3_AnisoDiff_Dxy_LessThanZero}}
\end{figure}

%-----------------------------------------------------------;
%  Figure-13: T3-Element D_xy > 0 (Anisotropic Diffusivity) ;
%-----------------------------------------------------------;
\begin{figure}
  \subfigure[Feasible region for $(a,b)$]{\includegraphics[scale=0.44]
  {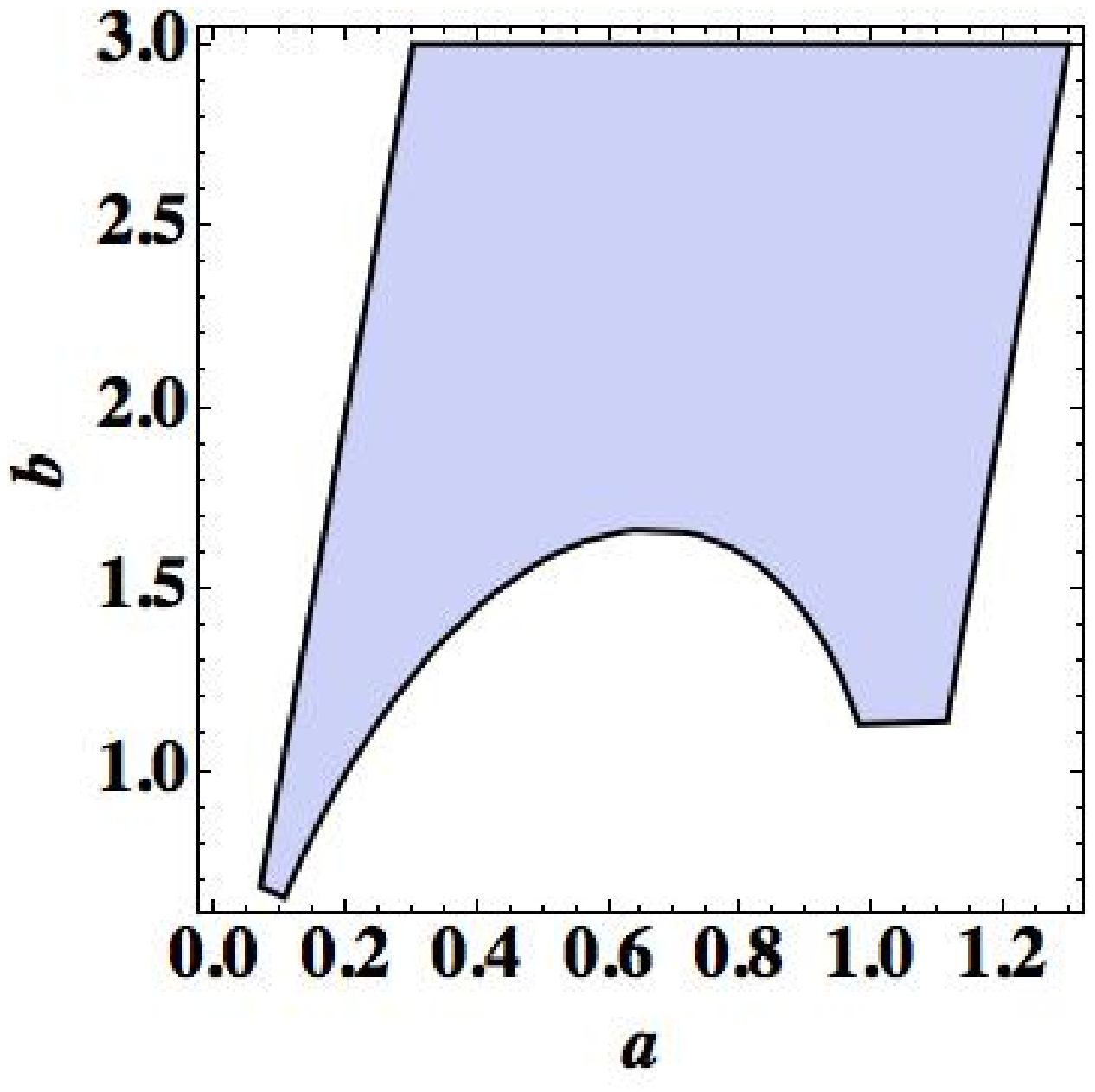}}
  \hspace{0.2in}
  \psfrag{a}{Obtuse angled triangle}
  \psfrag{b}{Right angled triangle}
  \psfrag{c}{Acute angled triangle}
  \subfigure[Possible T3 elements when $\widetilde{D}_{xy} > 0$]
  {\includegraphics[scale=0.44]{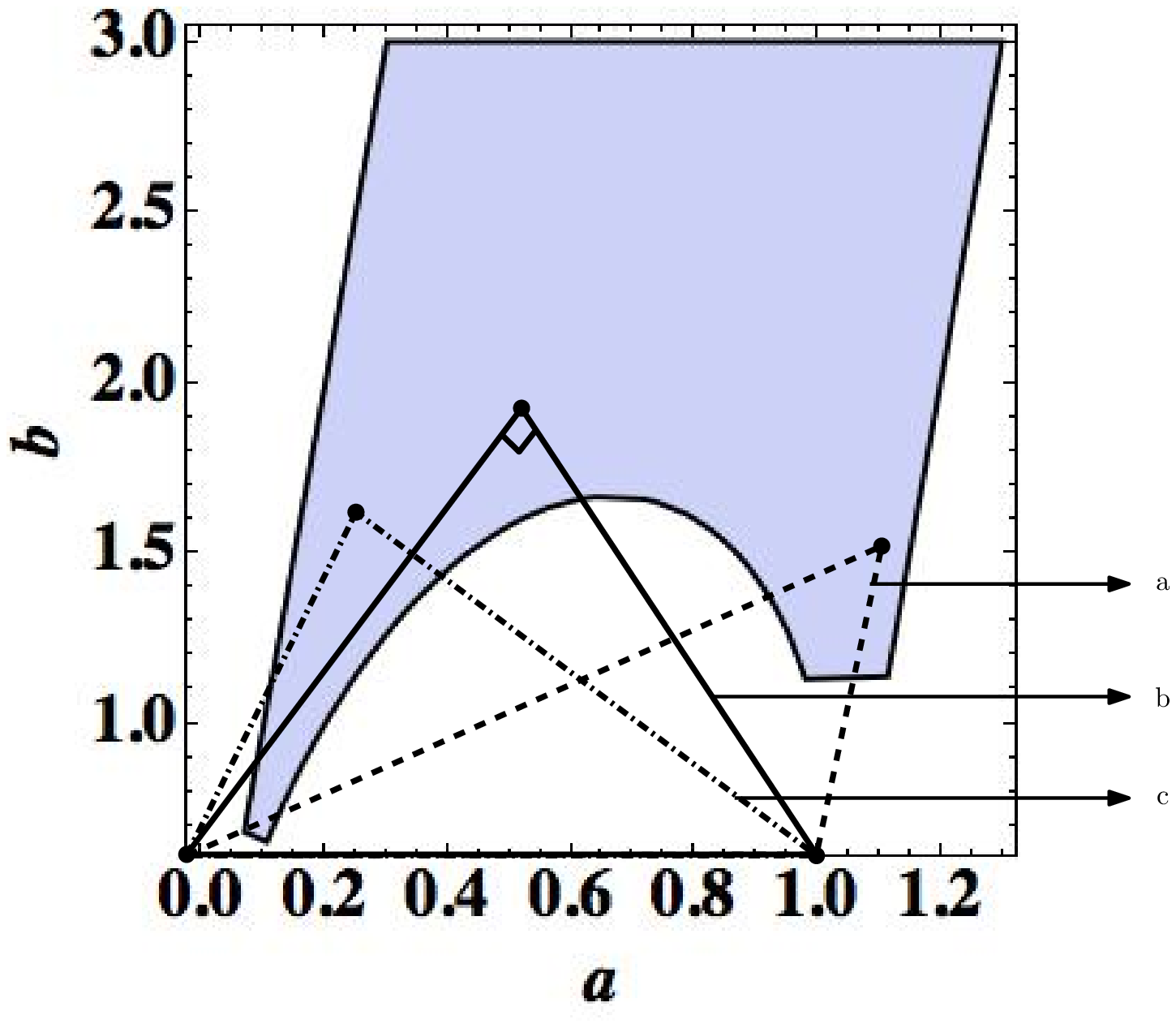}}
  \caption{\textsc{T3 element for anisotropic diffusivity when} 
    $\widetilde{D}_{xy} > 0$: The left figure indicates the 
    feasible region for the coordinates $(a,b)$ in light blue 
    color. The right figure indicates that when $\widetilde{D}_
    {xy} > 0$, the (Euclidean metric) dihedral angles in the T3 
    element can be acute-angled or right-angled or even obtuse-angled. 
    In this case, we have chosen $\epsilon = 10$ and $\eta = 1$. 
    For a fixed $\eta$ as $\epsilon$ increases, the value of 
    coordinate $b$ also increases. For higher values of $\epsilon$, 
    it is very difficult to find a suitable T3 element, which can 
    mesh any given computational domain.
    \label{Fig:T3_AnisoDiff_Dxy_GreaterThanZero}}
\end{figure}

%-------------------------------------------------------------------------------;
%  Figure-14: T3-Element for \eta = -1 and \epsilon = 2, 10, 50, 100, 200, 500  ;
%-------------------------------------------------------------------------------;
\begin{figure}
  \subfigure[$\epsilon = 2 \quad \mathrm{and} \quad \eta = -1$]
  {\includegraphics[scale=0.425]{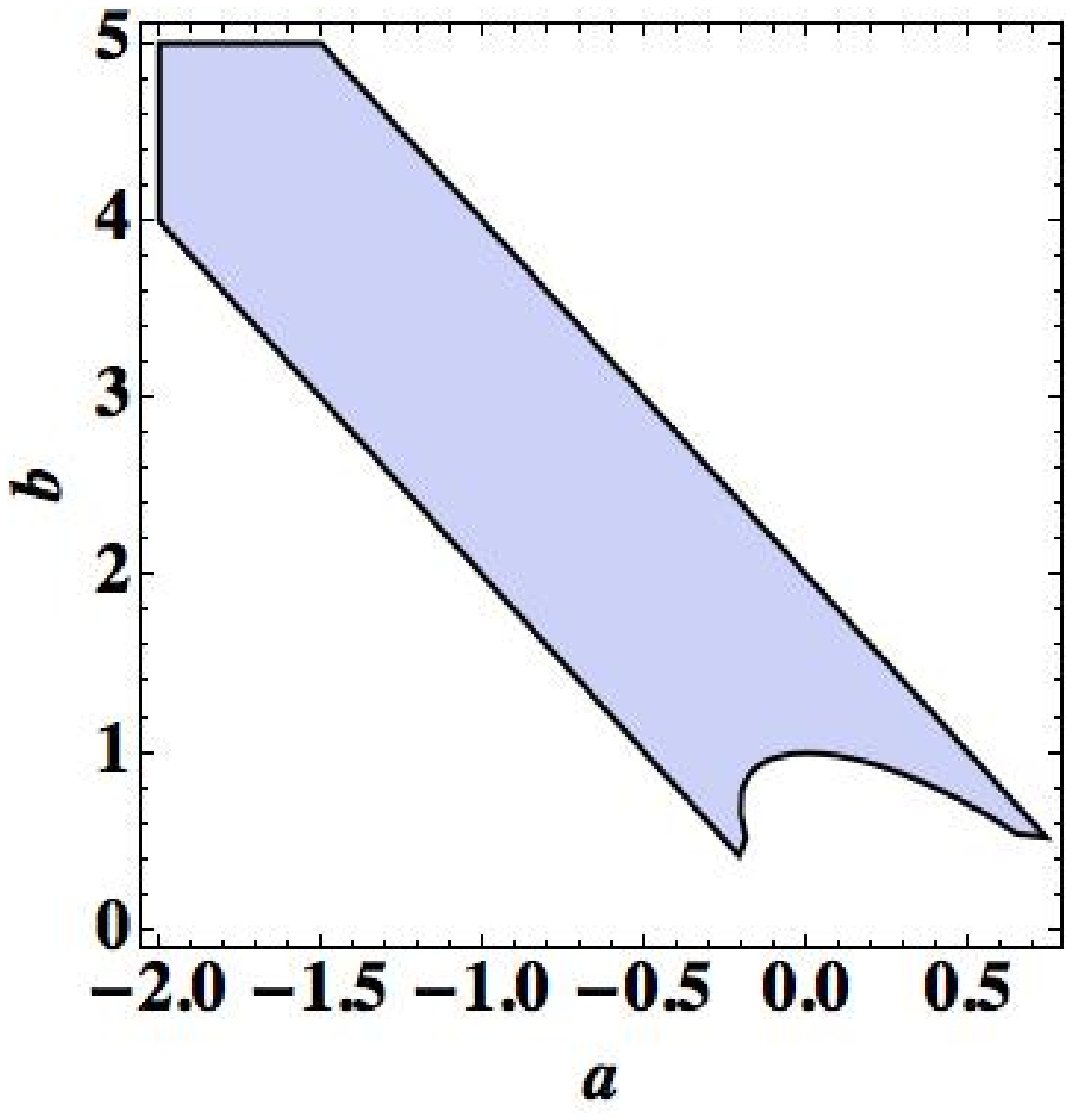}}
  \hspace{0.2in}
  \subfigure[$\epsilon = 10 \quad \mathrm{and} \quad \eta = -1$]
  {\includegraphics[scale=0.425]{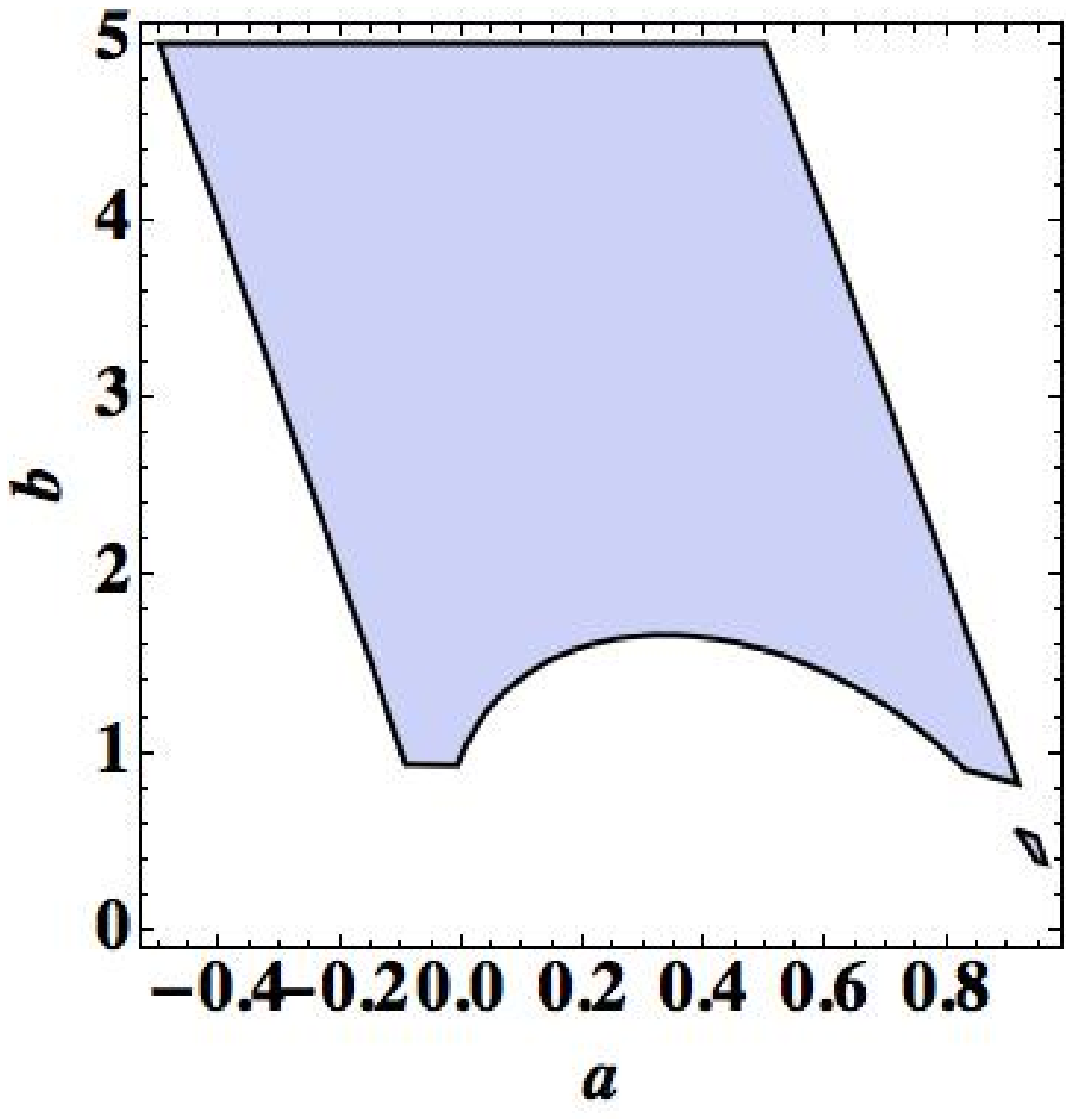}}
  \hspace{0.2in}
  \subfigure[$\epsilon = 50 \quad \mathrm{and} \quad \eta = -1$]
  {\includegraphics[scale=0.425]{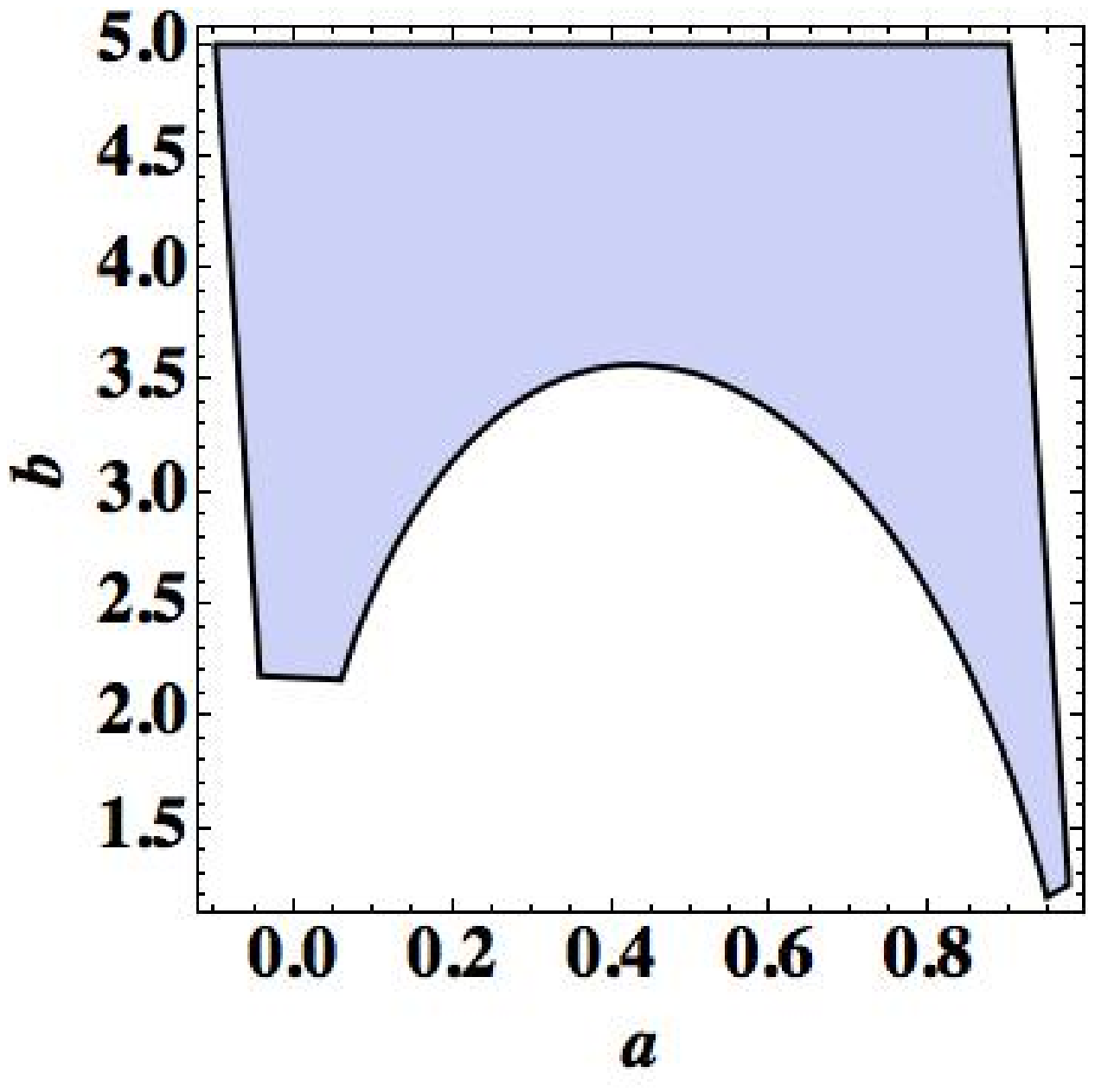}}
  \hspace{0.2in}
  \subfigure[$\epsilon = 100 \quad \mathrm{and} \quad \eta = -1$]
  {\includegraphics[scale=0.425]{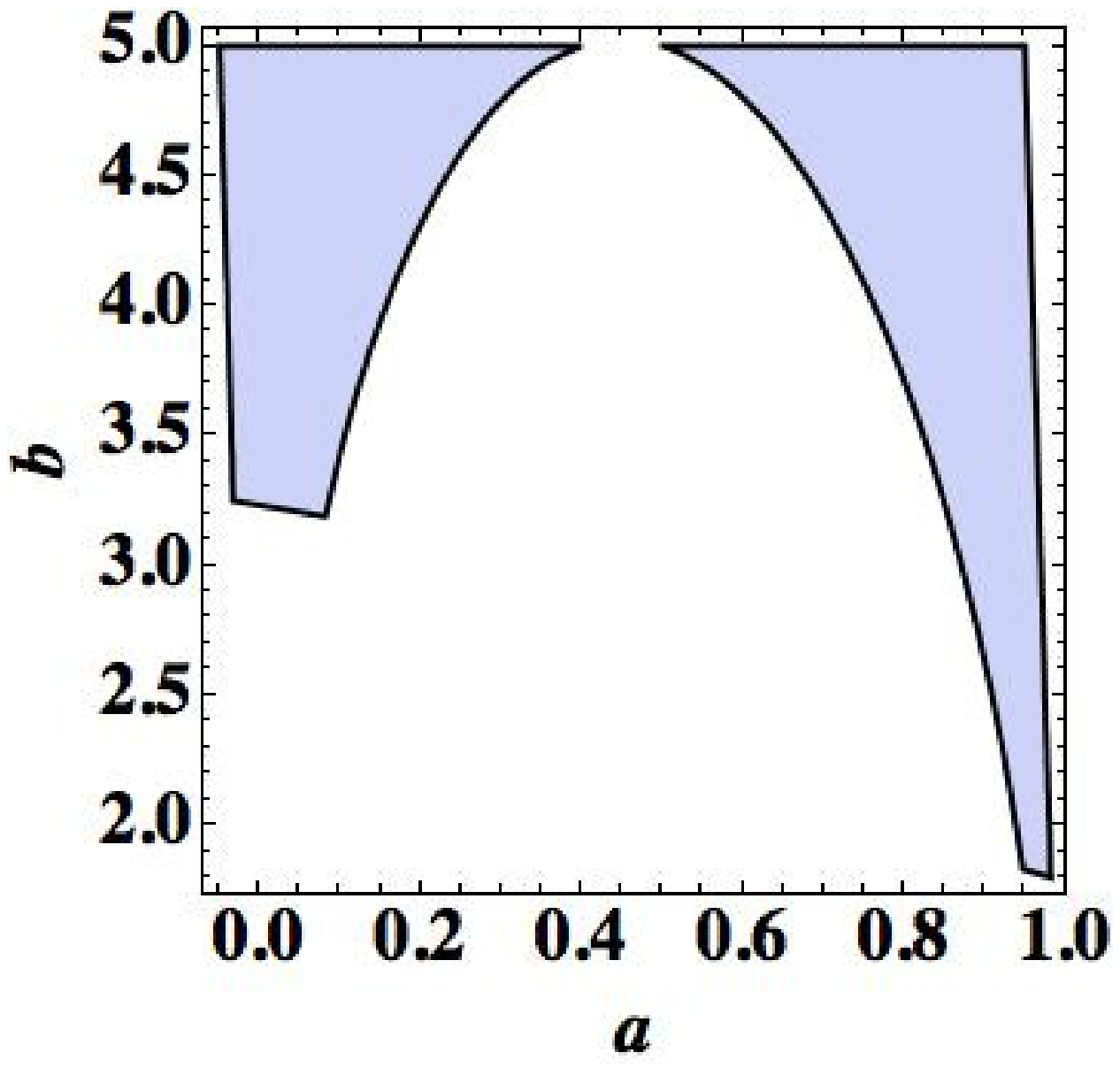}}
  \hspace{0.2in}
  \subfigure[$\epsilon = 200 \quad \mathrm{and} \quad \eta = -1$]
  {\includegraphics[scale=0.425]{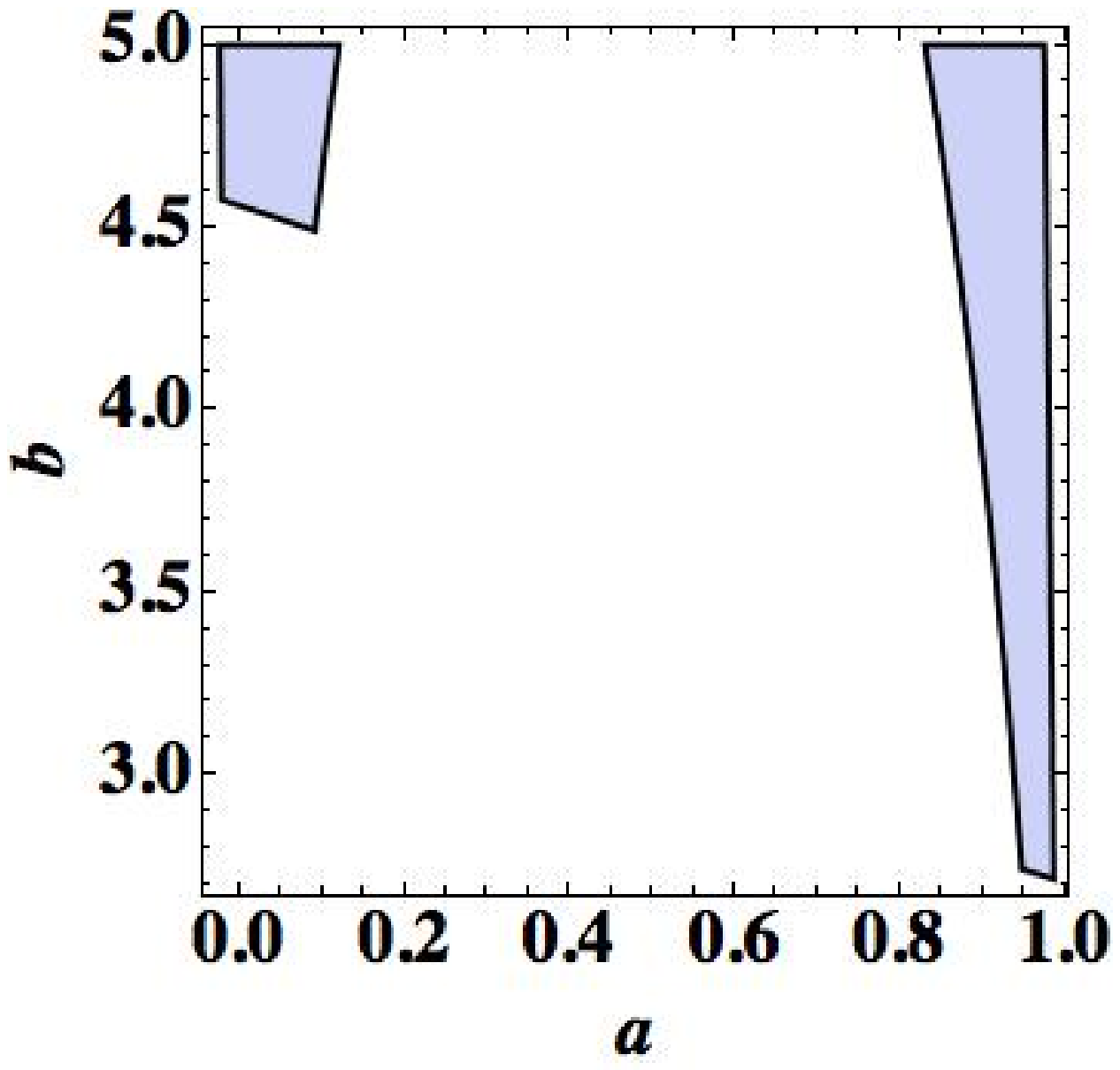}}
  \hspace{0.2in}
  \subfigure[$\epsilon = 500 \quad \mathrm{and} \quad \eta = -1$]
  {\includegraphics[scale=0.425]{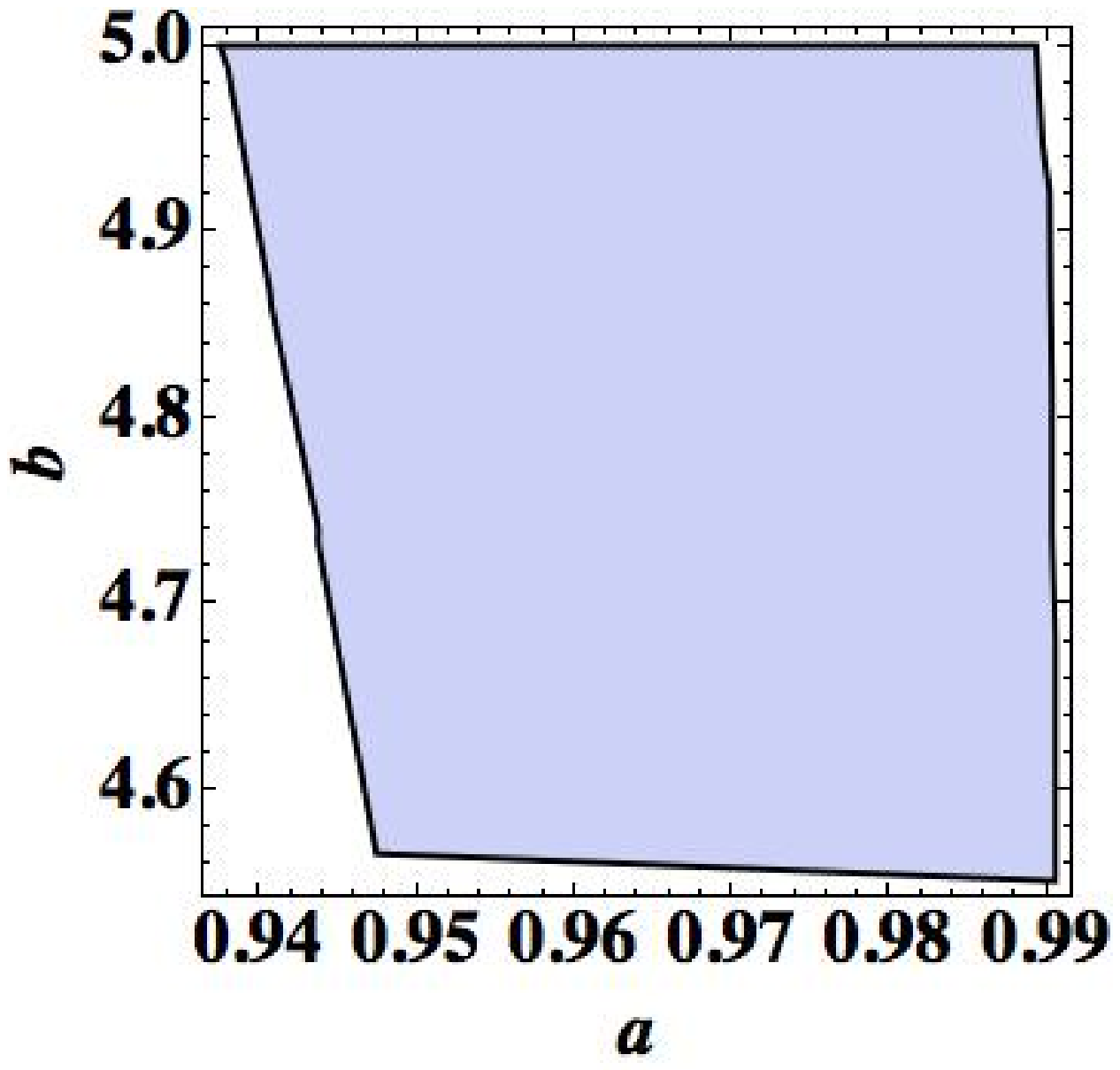}}
  \caption{\textsc{T3 element for fixed $\eta$ and varying $\epsilon$}: 
    A pictorial description of the feasible region (light blue color) 
    for a fixed $\eta$ and varying $\epsilon$. Analysis is performed 
    for $\eta = -1$ and $\epsilon = \{ 2, 10, 50, 100, 200, 500\}$. 
    It is evident there is a drastic variation in the feasible region 
    as $\epsilon$ increases.
    \label{Fig:T3_ConstEta_VaryEpsilon}}
\end{figure}

%---------------------------------------------------------------------------;
%  Figure-15: T3-Element for \epsilon = 100 and \eta = -8, -4, -2, 2, 4, 8  ;
%---------------------------------------------------------------------------;
\begin{figure}
  \subfigure[$\epsilon = 100 \quad \mathrm{and} \quad \eta = -2$]
  {\includegraphics[scale=0.425]{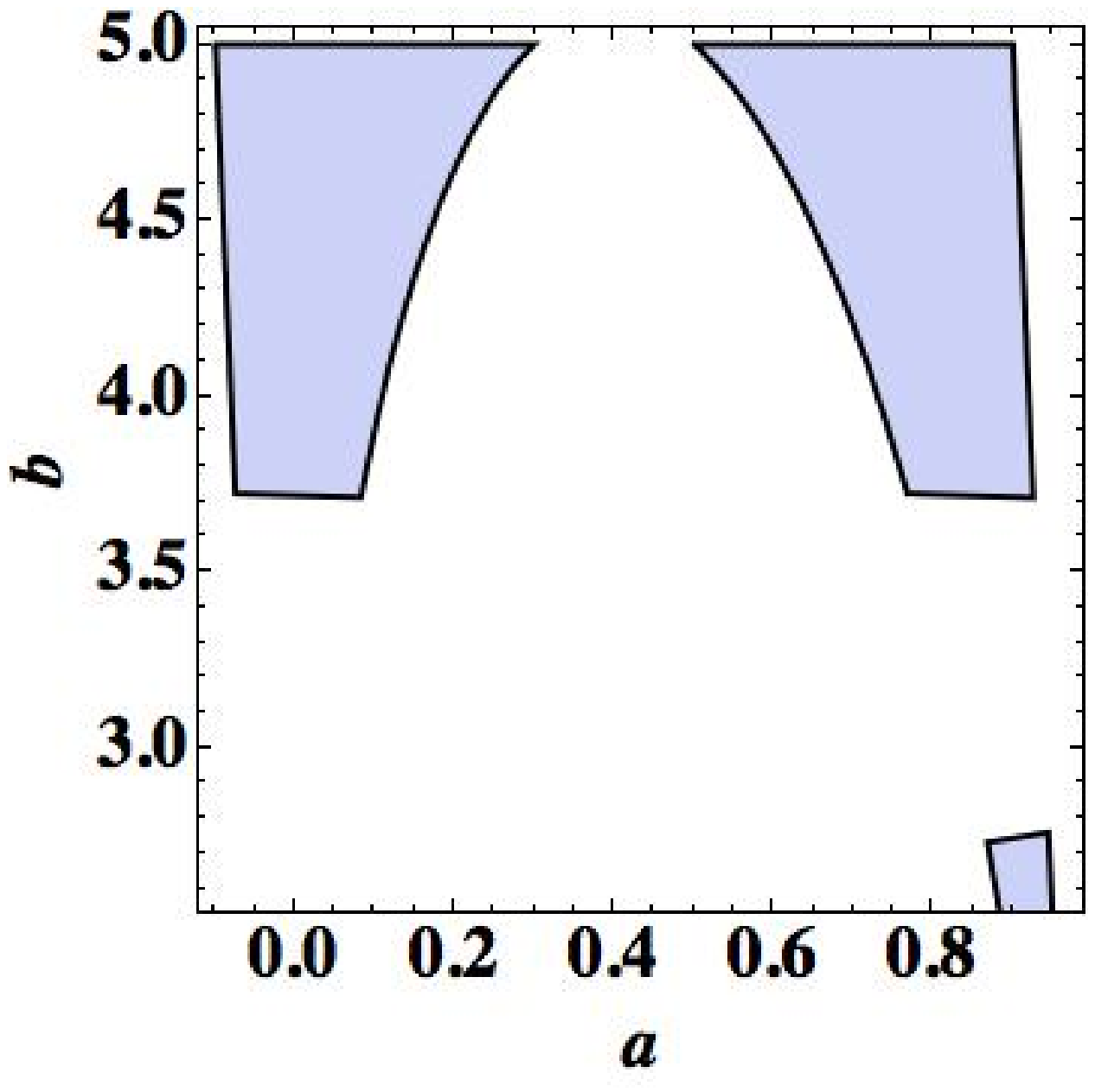}}
  \hspace{0.2in}
  \subfigure[$\epsilon = 100 \quad \mathrm{and} \quad \eta = 2$]
  {\includegraphics[scale=0.425]{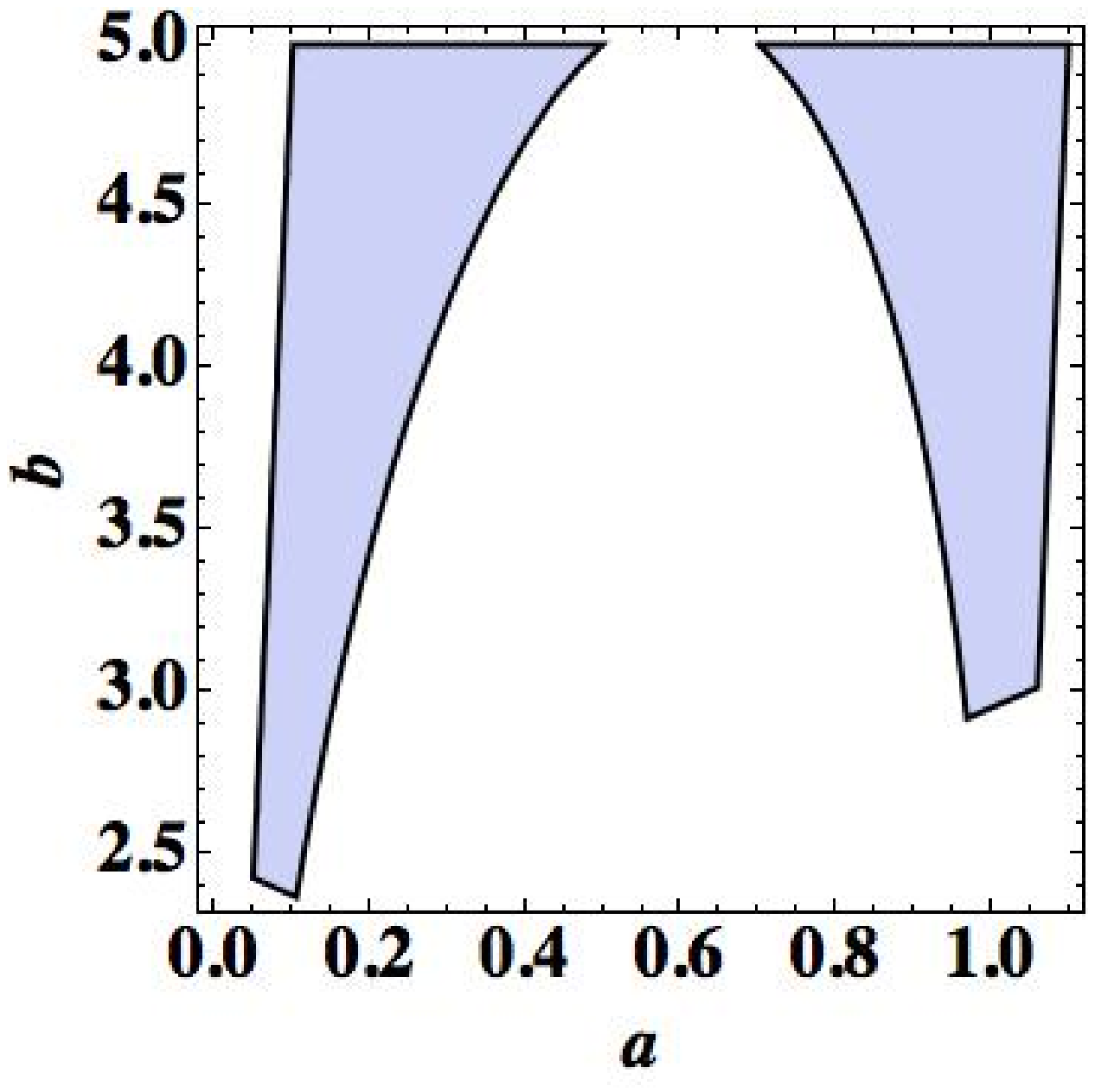}}
  \hspace{0.2in}
  \subfigure[$\epsilon = 100 \quad \mathrm{and} \quad \eta = -4$]
  {\includegraphics[scale=0.425]{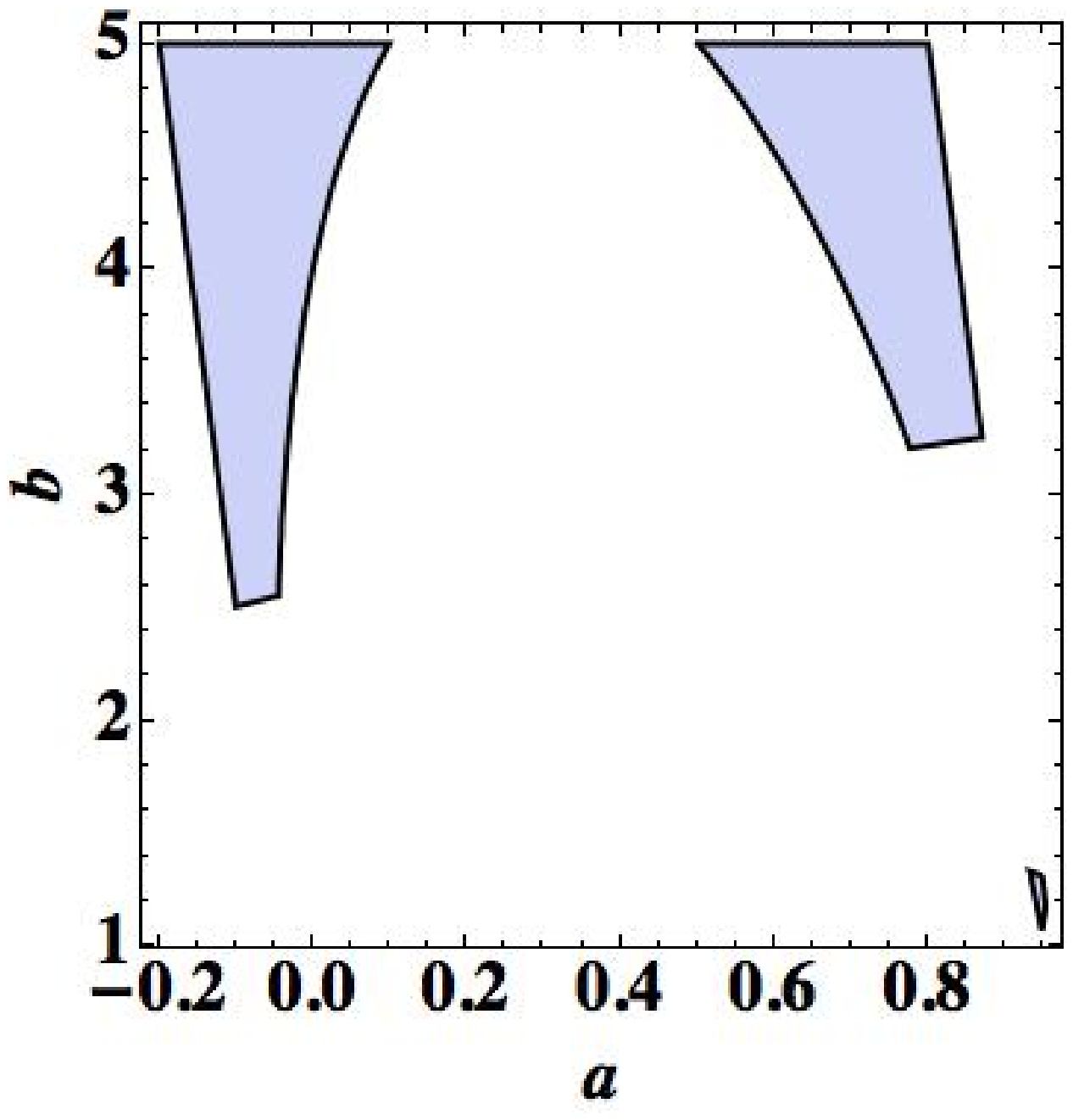}}
  \hspace{0.2in}
  \subfigure[$\epsilon = 100 \quad \mathrm{and} \quad \eta = 4$]
  {\includegraphics[scale=0.425]{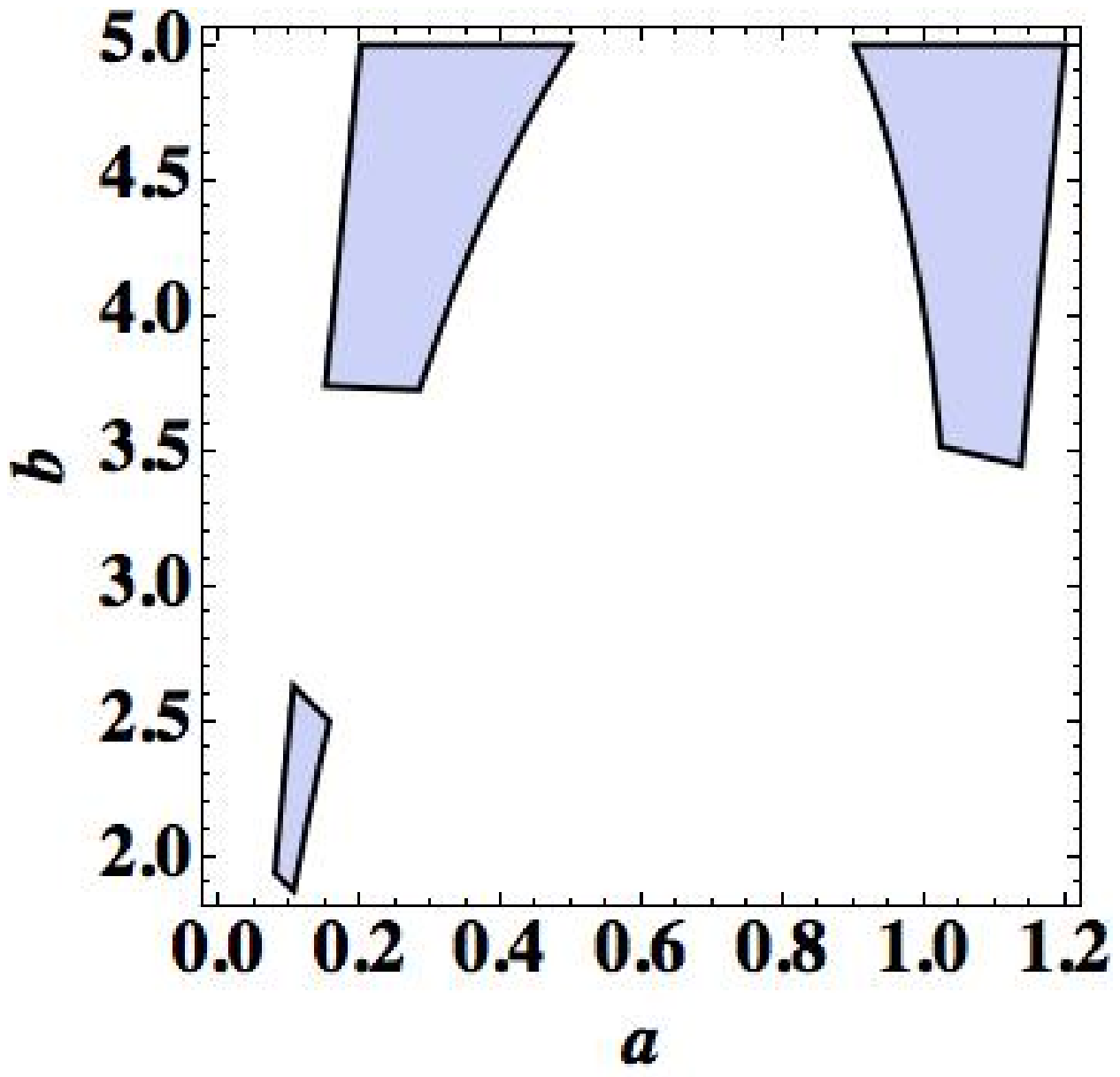}}
  \hspace{0.2in}
  \subfigure[$\epsilon = 100 \quad \mathrm{and} \quad \eta = -8$]
  {\includegraphics[scale=0.425]{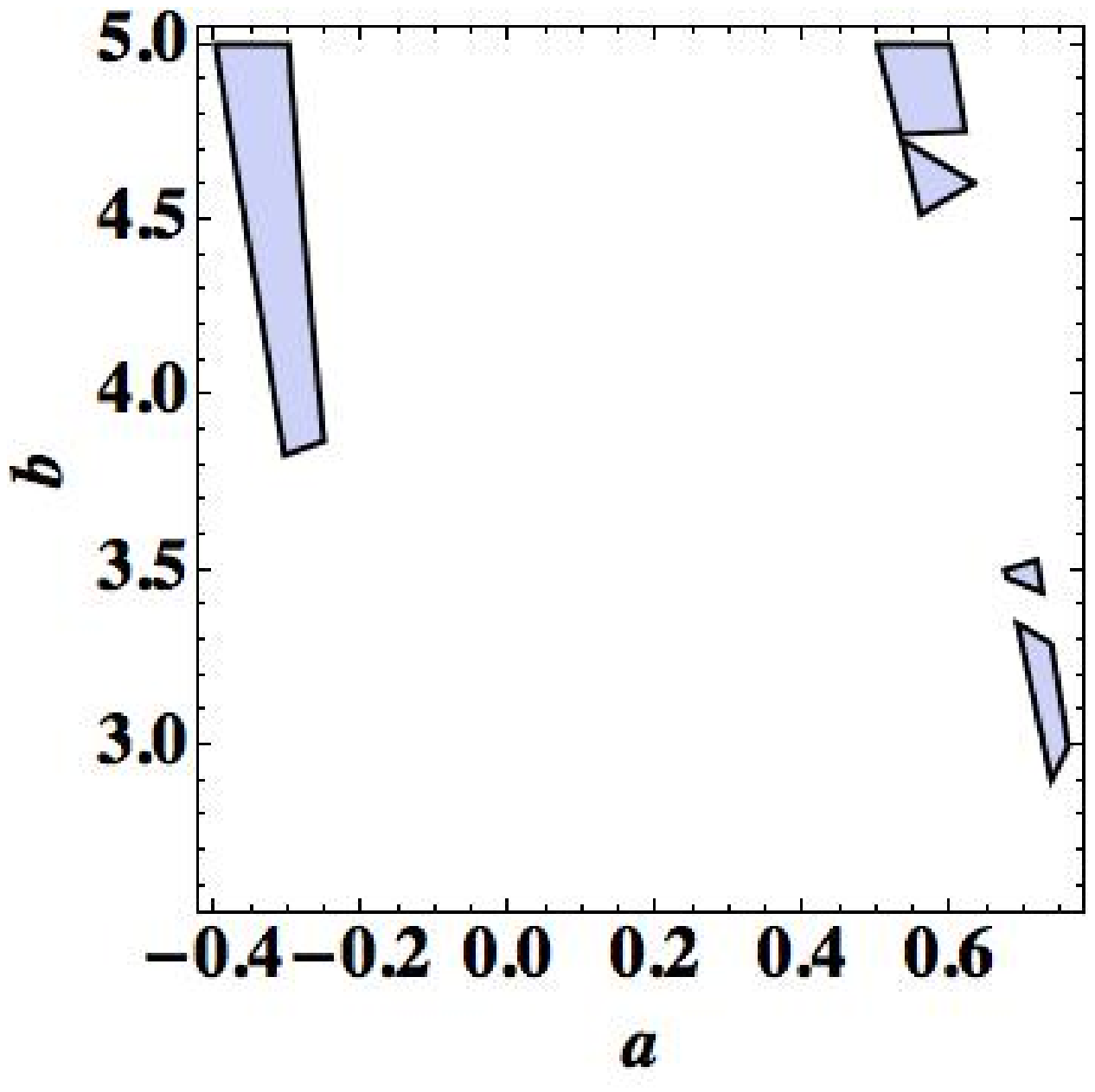}}
  \hspace{0.2in}
  \subfigure[$\epsilon = 100 \quad \mathrm{and} \quad \eta = 8$]
  {\includegraphics[scale=0.425]{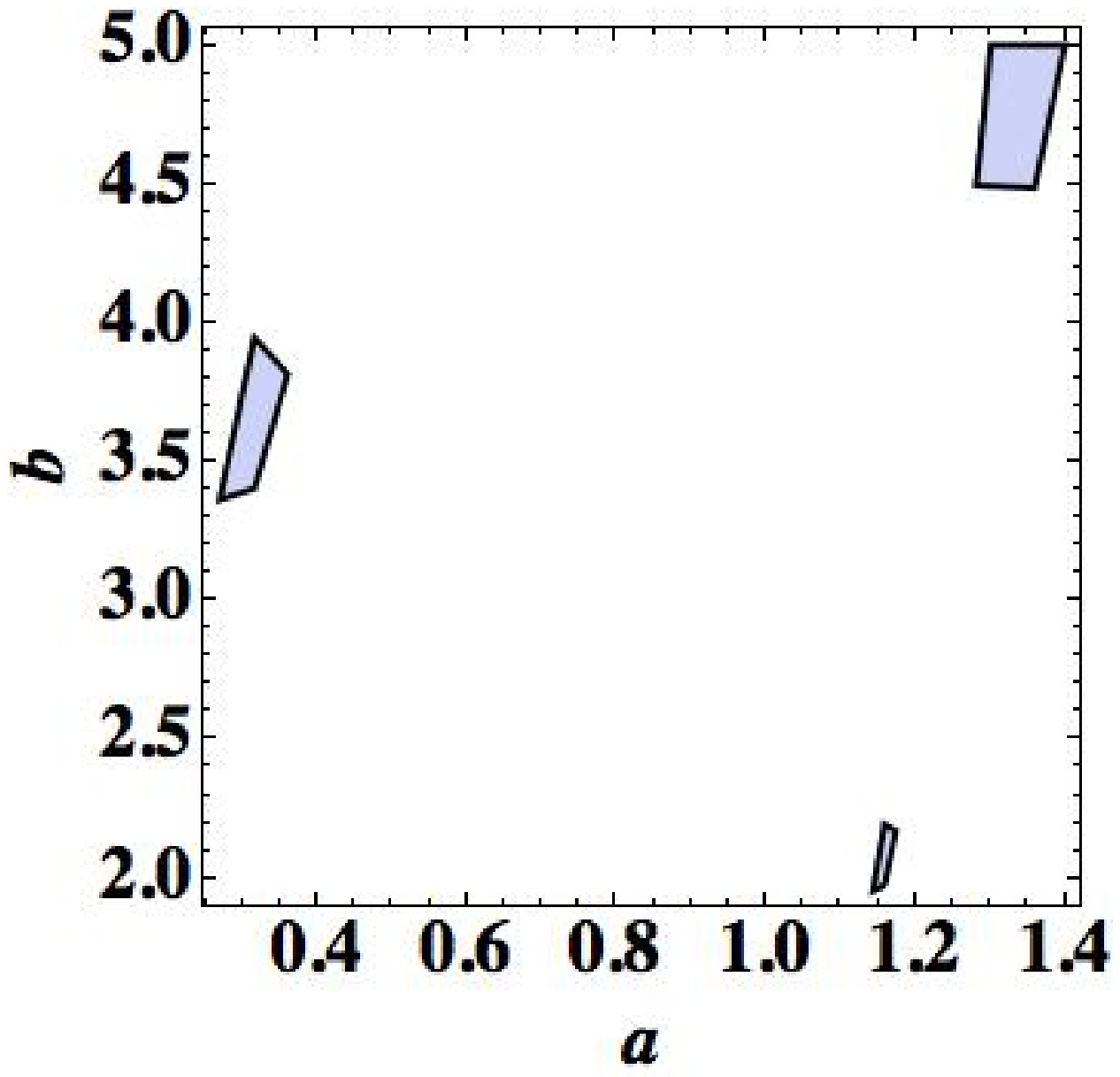}}
  \caption{\textsc{T3 element for fixed $\epsilon$ and varying $\eta$}: 
    A pictorial description of the feasible region (light blue color) 
    for a fixed $\epsilon$ and varying $\eta$. Analysis is performed 
    for $\epsilon = 100$ and $\eta = \{ -8, -4, -2, 2, 4, 8\}$. It is 
    evident there is considerable variation in the feasible region as 
    $\eta$ changes. Also, there is no fixed pattern on this variation 
    about $\eta$.
    \label{Fig:T3_ConstEpsilon_VaryEta}}
\end{figure}

%---------------------------------------------------------------;
%  Pictorial Description: Meshes and concentration profile for  ; 
%                         Problem no # 1 (Fractured domain)     ;
%---------------------------------------------------------------;
\begin{figure}
  \subfigure[Delaunay-Voronoi mesh: $Nv = 539$ and $Nele = 906$]
  {\includegraphics[scale=0.35]
  {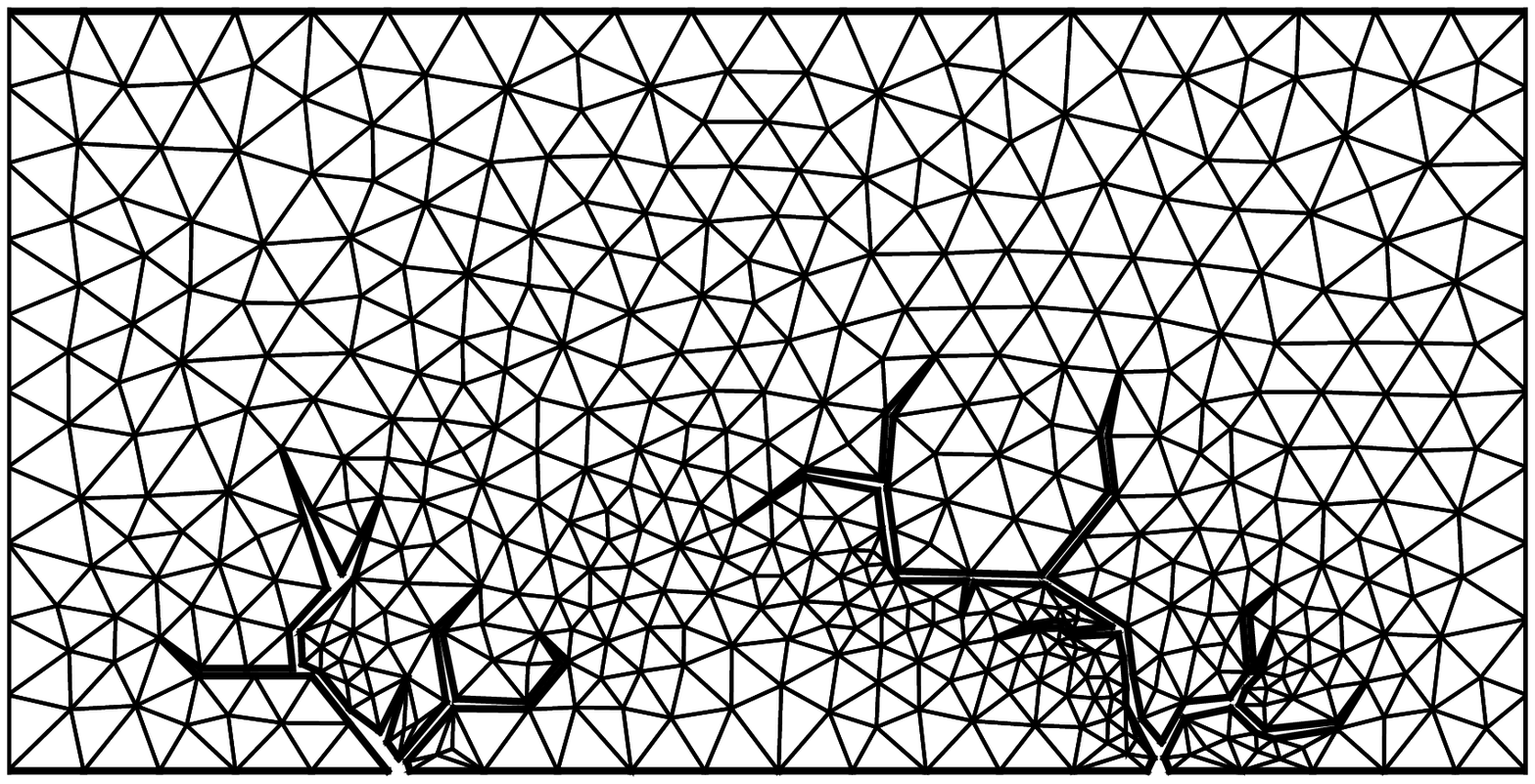}}
  \subfigure[$\mathbf{v} = (0,0)$ and $\alpha = 0$]
  {\includegraphics[scale=0.35]
  {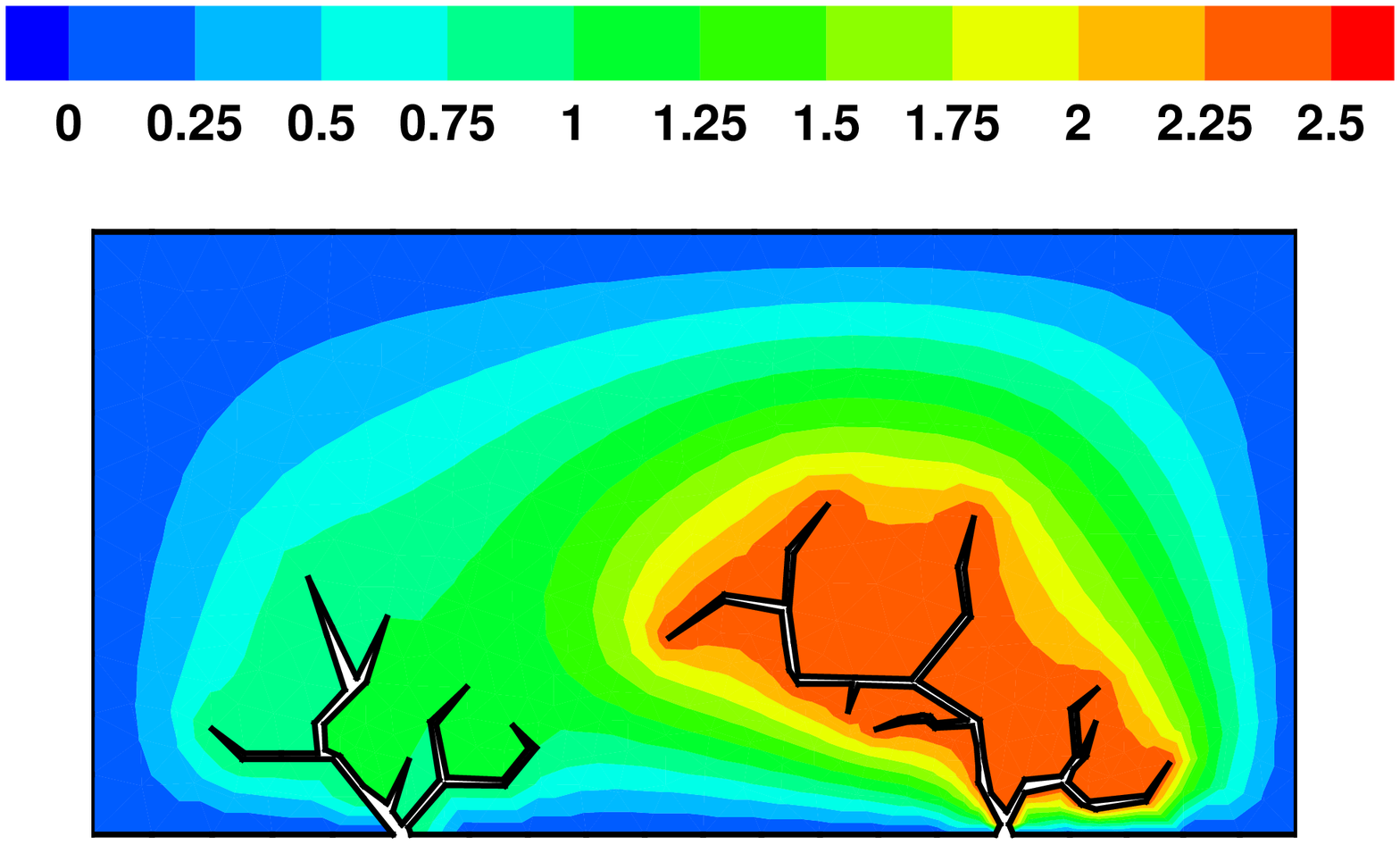}}
  \subfigure[$\mathbf{v} = (0.1,1.0)$ and $\alpha = 0$]
  {\includegraphics[scale=0.35]
  {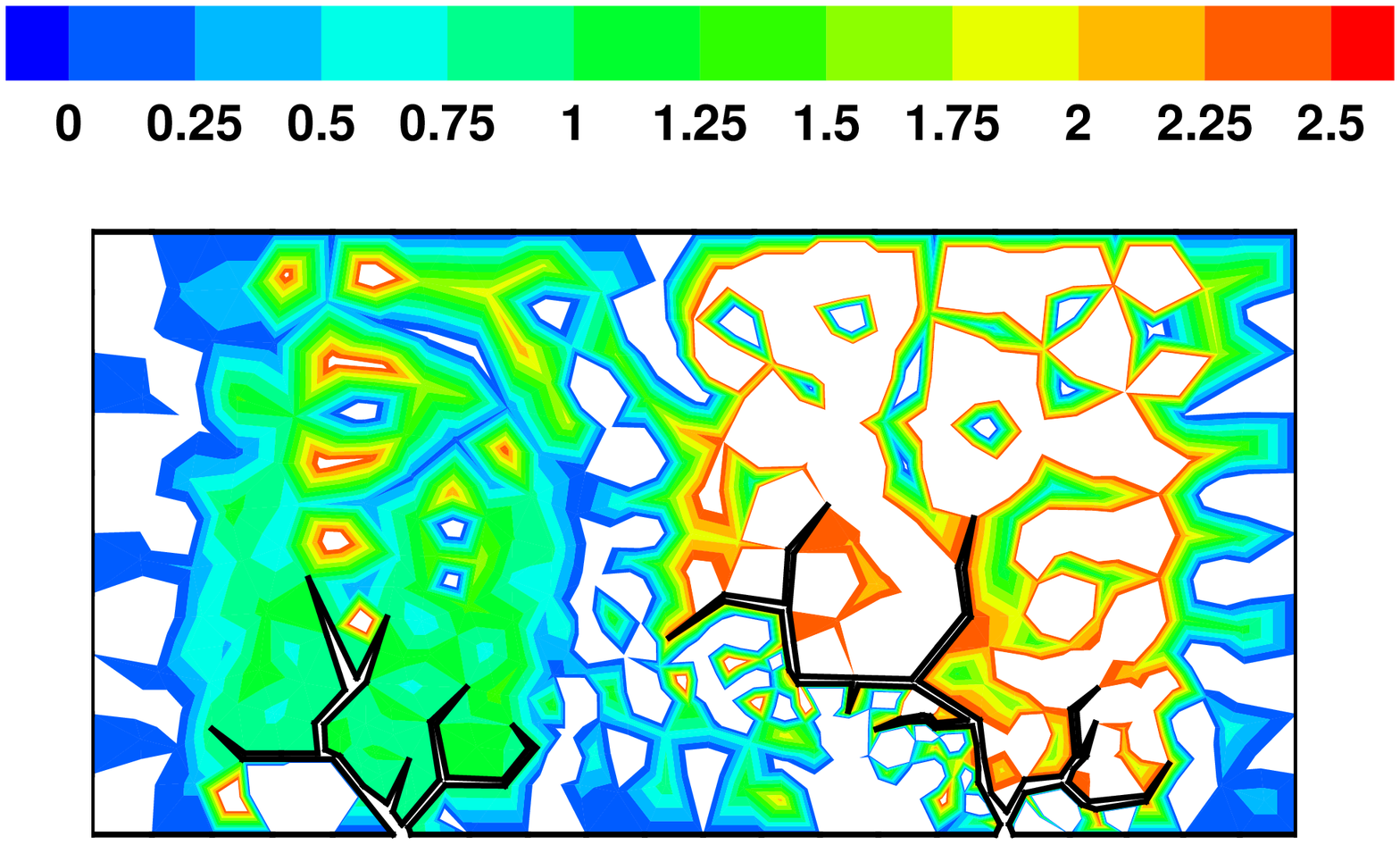}}
  \subfigure[$\mathbf{v} = (0.1,1.0)$ and $\alpha = 1.0$]
  {\includegraphics[scale=0.35]
  {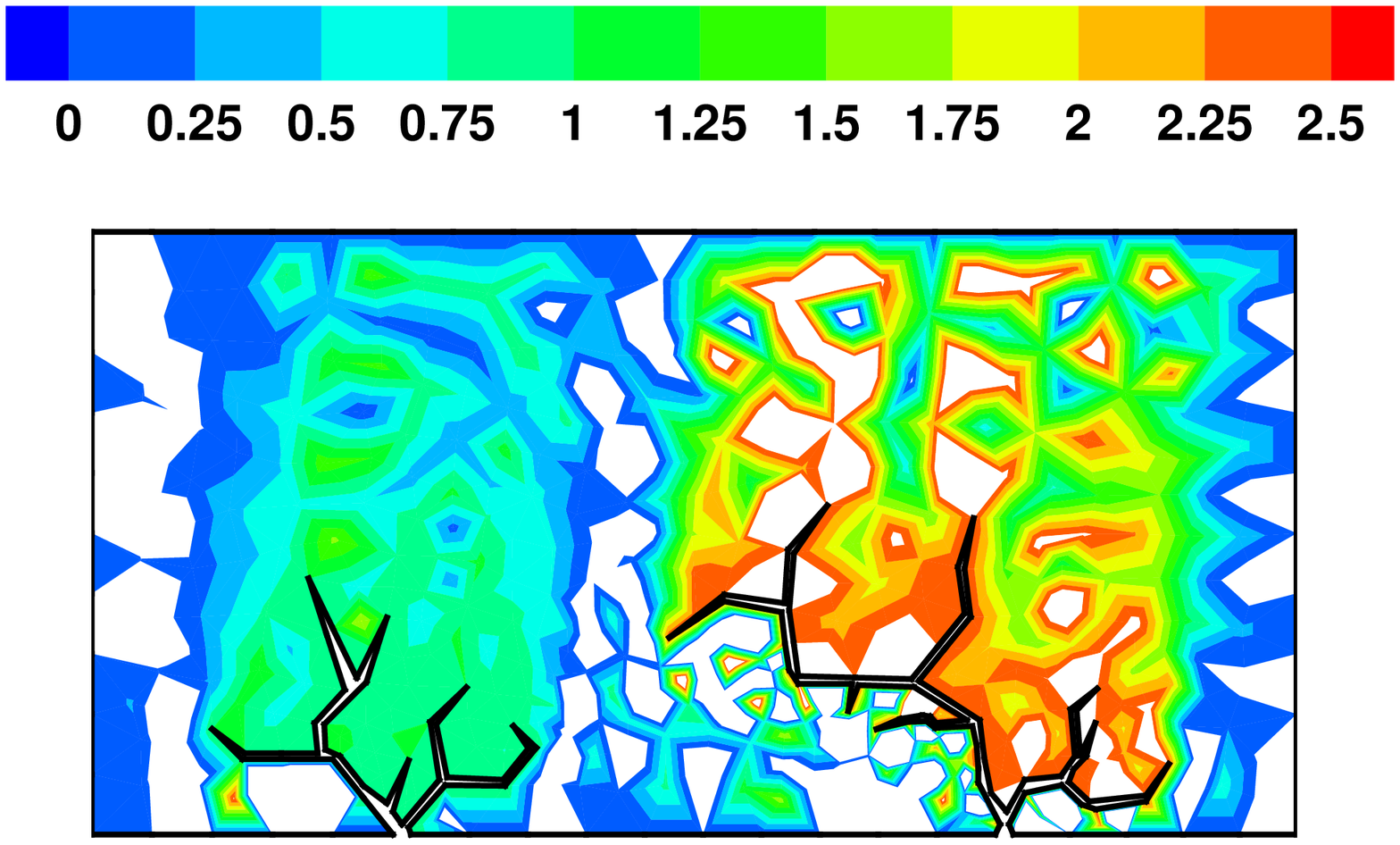}}
  \caption{\textsc{Test problem \#1}: The top left figure shows a coarse triangulation 
    (generated using \textsf{Gmsh} based on Algorithm \ref{Algo:Aniso_Meshes}) employed 
    in the numerical study, \emph{which is to the scale}. The top right figure and the 
    bottom two figures show the concentration profiles obtained for various values of 
    the velocity field and linear reaction coefficient using this mesh.
    \label{Fig:Prob1_Delaunay_DMP_Meshes_Conc}}
\end{figure}

%---------------------------------------------------------------;
%  Pictorial Description: Generalized Delaunay-type condition   ; 
%                         Problem no # 1 (Fractured domain)     ;
%---------------------------------------------------------------;
\begin{figure}
  \subfigure[Element maximum angles: $Nv = 539$ and $Nele = 906$]
  {\includegraphics[scale=0.35]
  {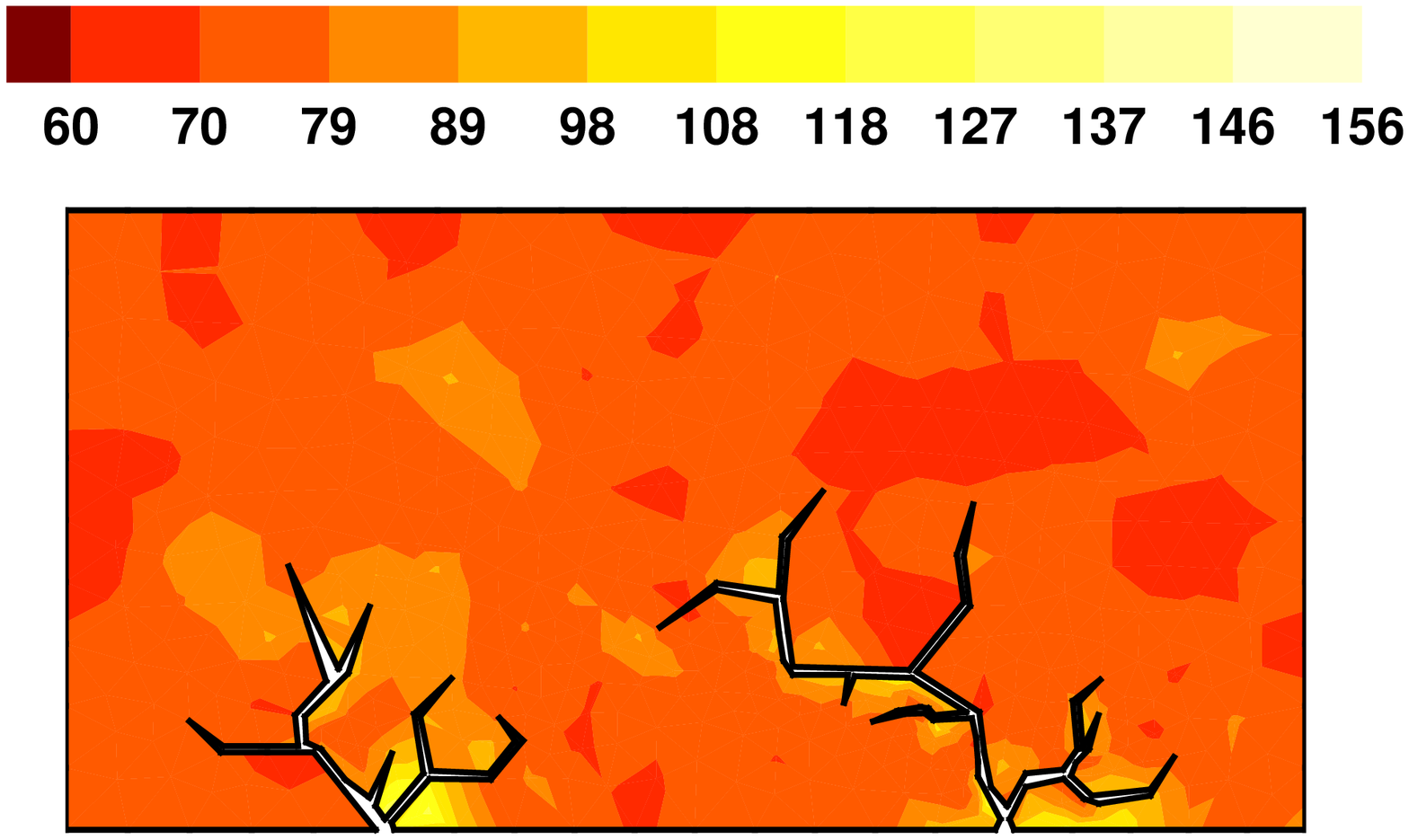}}
  \subfigure[Delaunay-type condition: $\mathbf{v} = (0,0)$ and $\alpha = 0$]
  {\includegraphics[scale=0.35]
  {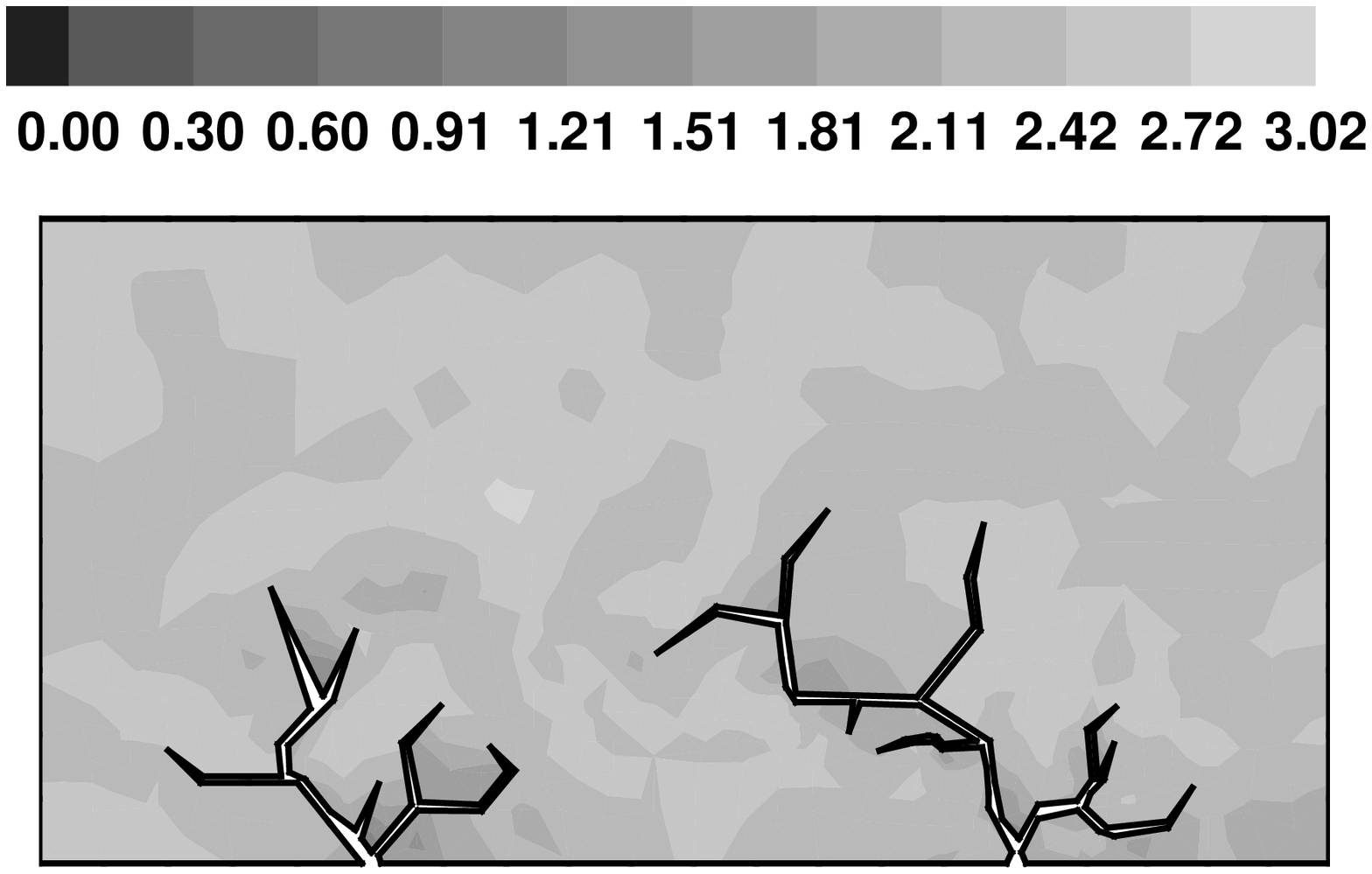}}
  \subfigure[Delaunay-type condition: $\mathbf{v} = (0.1,1.0)$ and $\alpha = 0$]
  {\includegraphics[scale=0.35]
  {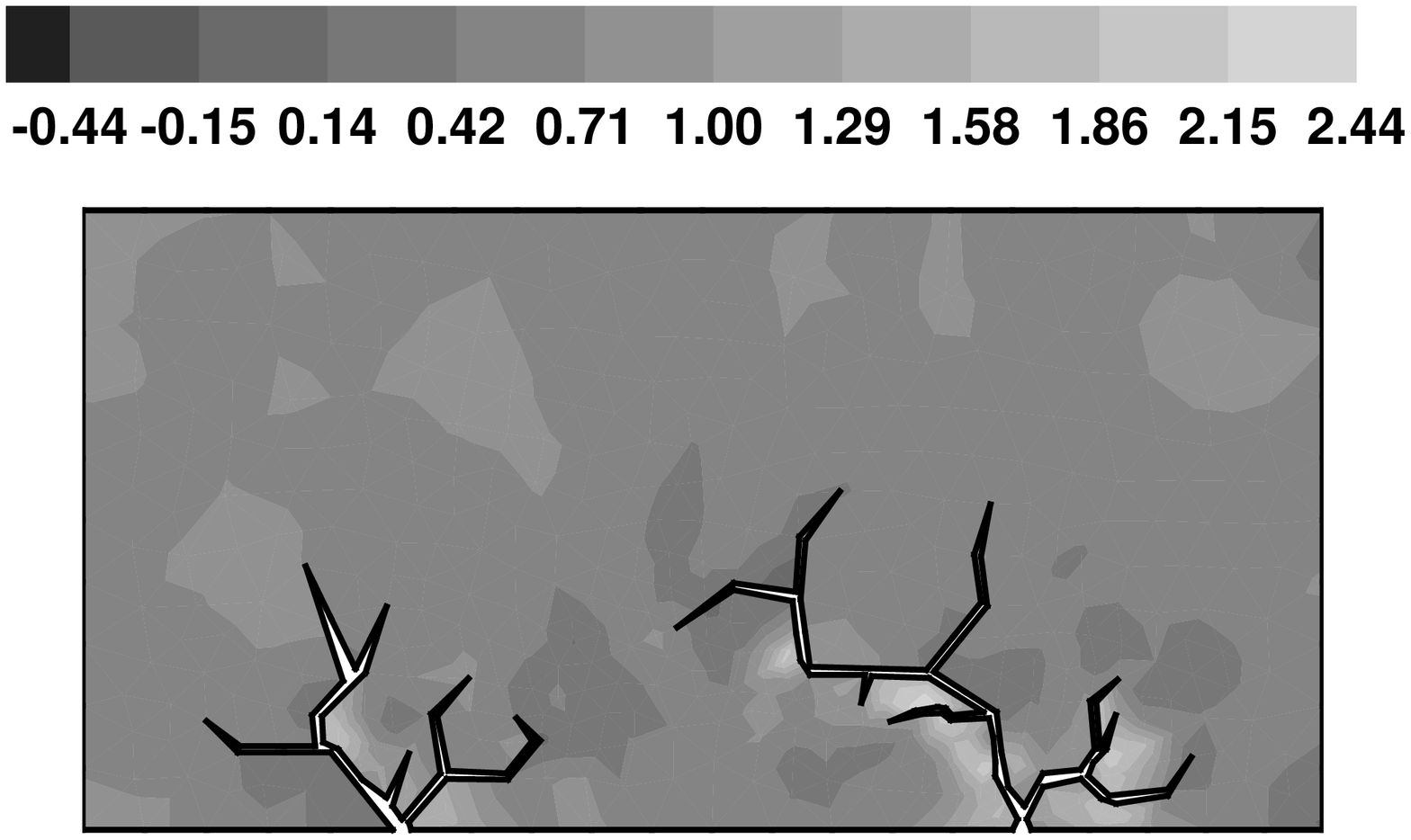}}
  \subfigure[Delaunay-type condition: $\mathbf{v} = (0.1,1.0)$ and $\alpha = 1.0$]
  {\includegraphics[scale=0.35]
  {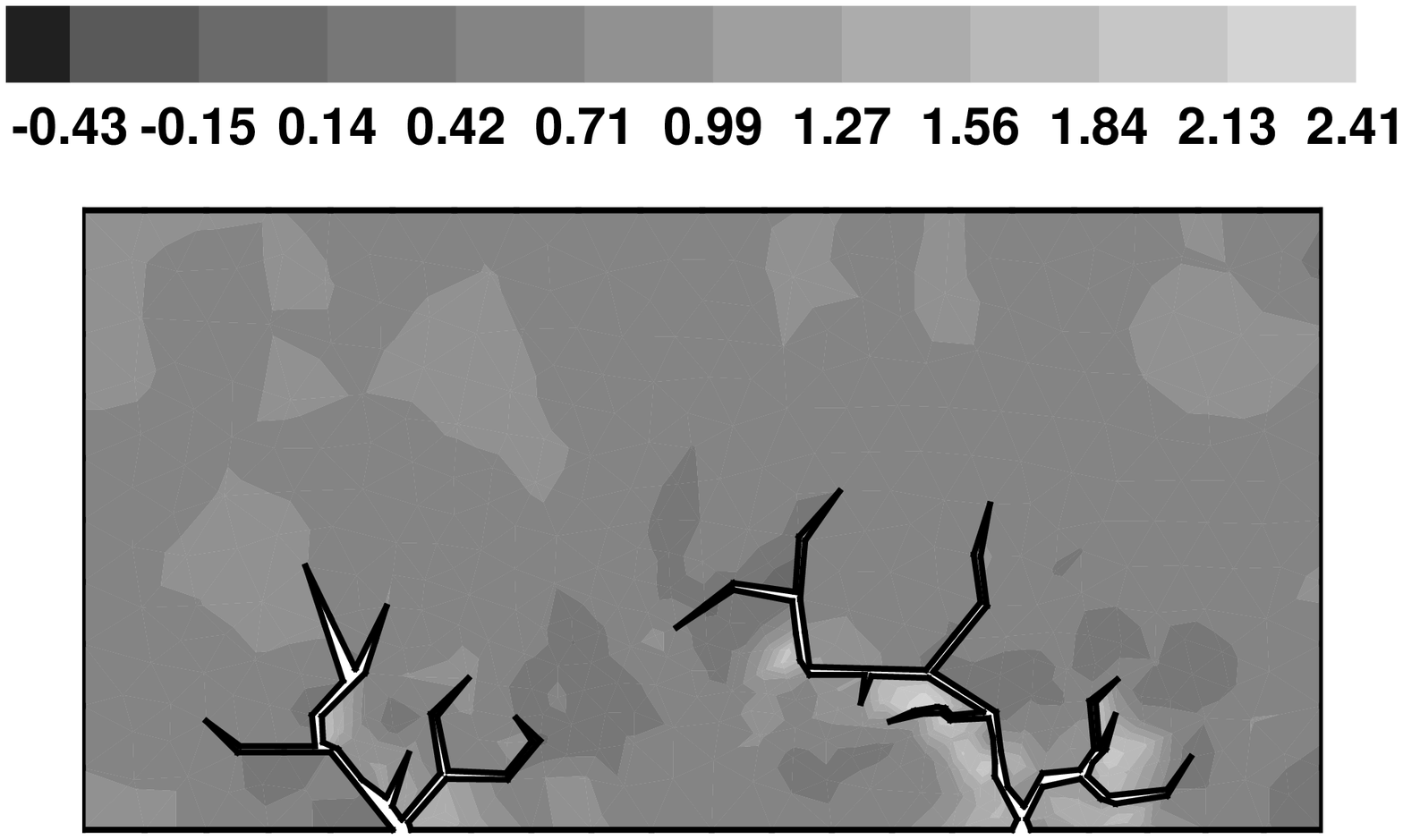}}
  \caption{\textsc{Test problem \#1}: The top left figure shows the maximum angle 
    possible in each element of the mesh. The top right figure and the bottom two 
    figures show the \emph{element maximum} \texttt{generalized Delaunay-type 
    condition}, which is a weaker condition as compared to the \emph{element 
    maximum} \texttt{anisotropic non-obtuse angle condition}.
    \label{Fig:Prob1_Delaunay_Condition}}
\end{figure}

%----------------------------------------------------------;
%  Pictorial Description: Problem no # 2 (Bi unit square)  ;
%----------------------------------------------------------;
\begin{figure}
  \centering 
  \psfrag{O}{$\mathrm{O}$}
  \psfrag{H}{$\mathrm{H}$}
  \psfrag{l1}{$l_x = 0.25$}
  \psfrag{l2}{\rotatebox{-90}{$l_y = 0.25$}}
  \psfrag{f1}{$f(\mathbf{x}) = 0.0$}
  \psfrag{f2}{$f(\mathbf{x}) = 1.0$}
  \psfrag{c1}{$c^{\mathrm{p}}(\mathbf{x}) = 0.0$}
  \psfrag{c2}{\rotatebox{-90}{$c^{\mathrm{p}}(\mathbf{x}) = 0.0$}}
  \psfrag{c3}{$c^{\mathrm{p}}(\mathbf{x}) = 0.0$}
  \psfrag{c4}{\rotatebox{90}{$c^{\mathrm{p}}(\mathbf{x}) = 0.0$}}
  \includegraphics[scale=0.55]{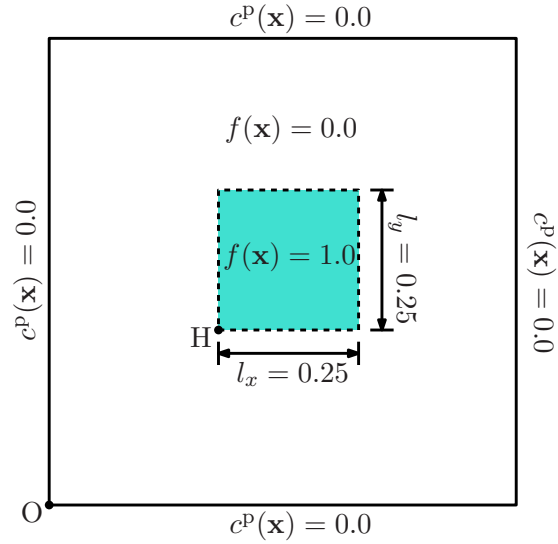}
  \caption{\textsc{Test problem \#2}: The computational domain under consideration 
    is a bi-unit square with one of its vertices at origin $\mathrm{O} = (0,0)$. 
    Homogeneous Dirichlet boundary conditions are prescribed on all sides of the 
    square. The volumetric source $f(\mathbf{x})$ is zero inside the domain, except 
    for the square region (including the boundaries) located at vertex $\mathrm{H} = 
    (0.375,0.375)$. In this region, $f(\mathbf{x})$ is equal to unity. Herein, we 
    assume that the velocity vector field and linear reaction coefficient are equal 
    to zero everywhere in the computational domain.
    \label{Fig:Prob2_Description}}
\end{figure}

%---------------------------------------------------------------------;
%  Pictorial Description: Meshes for Problem no # 2 (Bi unit square)  ;
%---------------------------------------------------------------------;
\begin{figure}
  \subfigure[Background mesh: $Nv = 47$ and $Nele = 68$]
  {\includegraphics[scale=0.35]
  {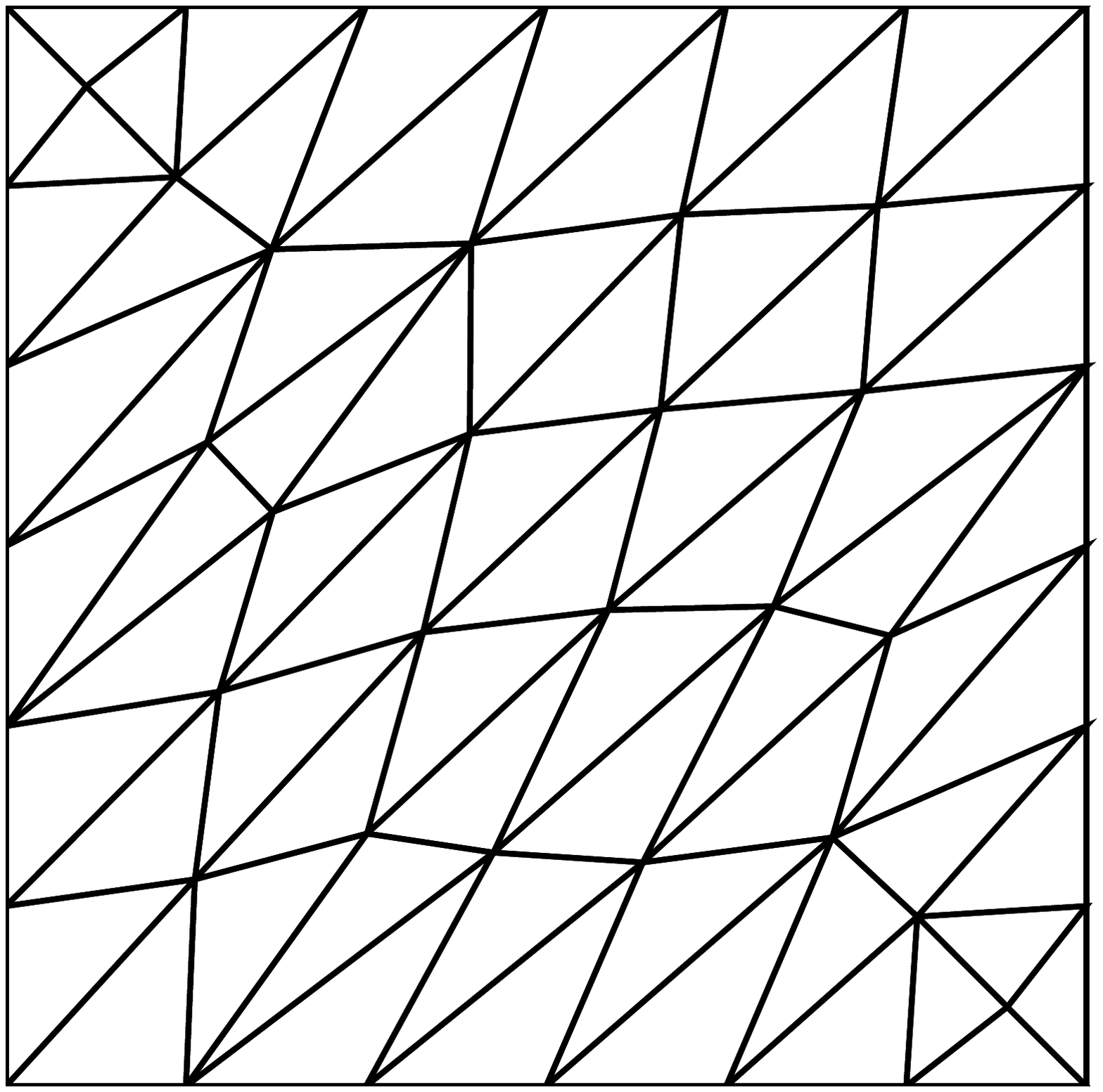}}
  \subfigure[Anisotropic mesh: $Nv = 593$ and $Nele = 1088$]
  {\includegraphics[scale=0.35]
  {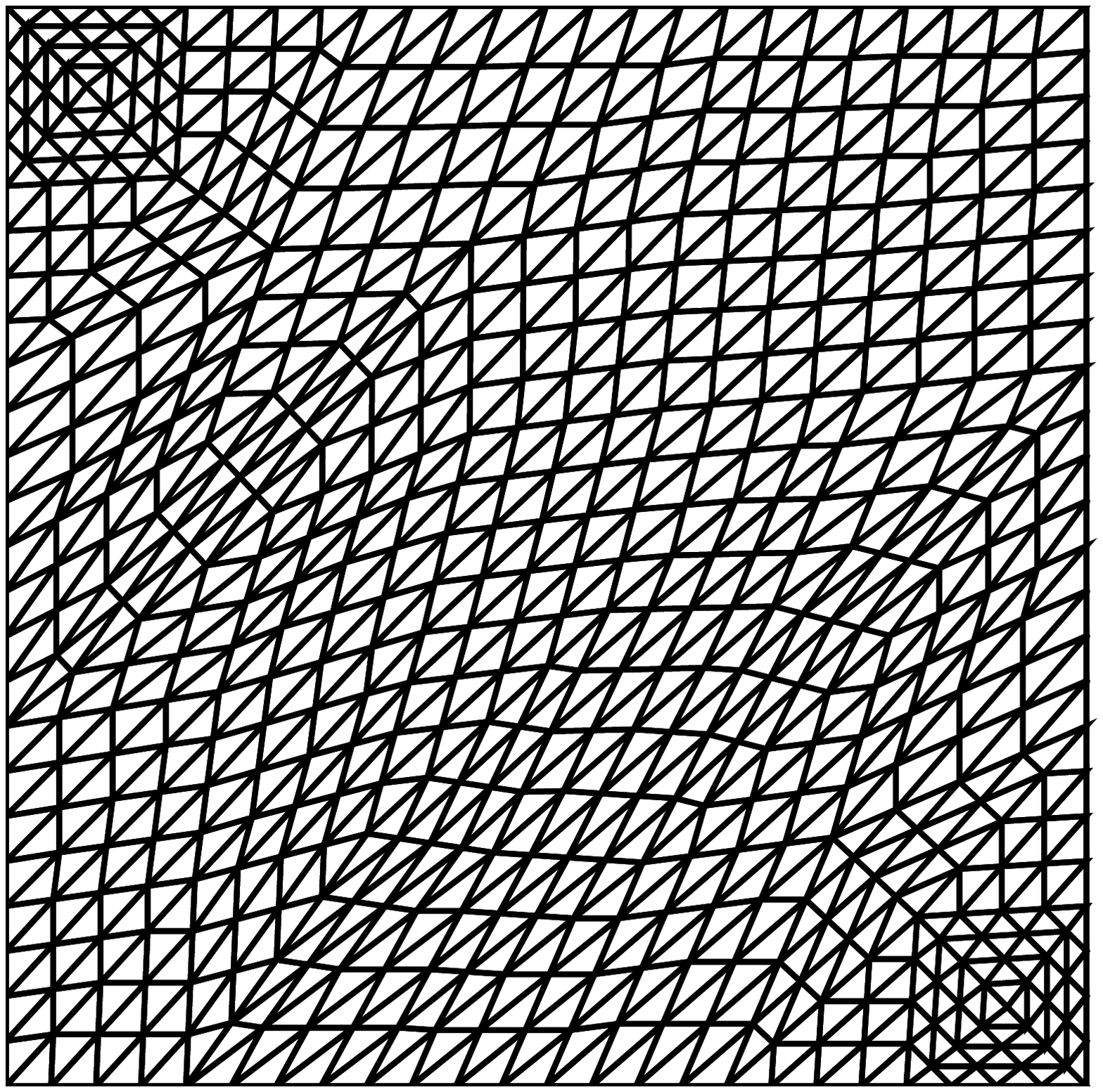}}
  \caption{\textsc{Test problem \#2}: The left figure shows the background mesh 
    on which \textsf{BAMG} operates to give an anisotropic triangulation, which 
    is shown in the right figure. As the ratio of the minimum eigenvalue of anisotropic 
    diffusivity tensor to its maximum is 0.1, which is not very high, so the 
    resulting triangulation consists of a mixture of skinny and normal triangles.
    \label{Fig:Prob2_Background_DMP_Meshes}}
\end{figure}

%-------------------------------------------------------------------------------------;
%  Pictorial Description: Concentration profiles for Problem no # 2 (Bi unit square)  ;
%-------------------------------------------------------------------------------------;
\begin{figure}
  \subfigure[Background mesh: $\mathbf{v} = (0,0)$ and $\alpha = 0$]
  {\includegraphics[scale=0.35]
  {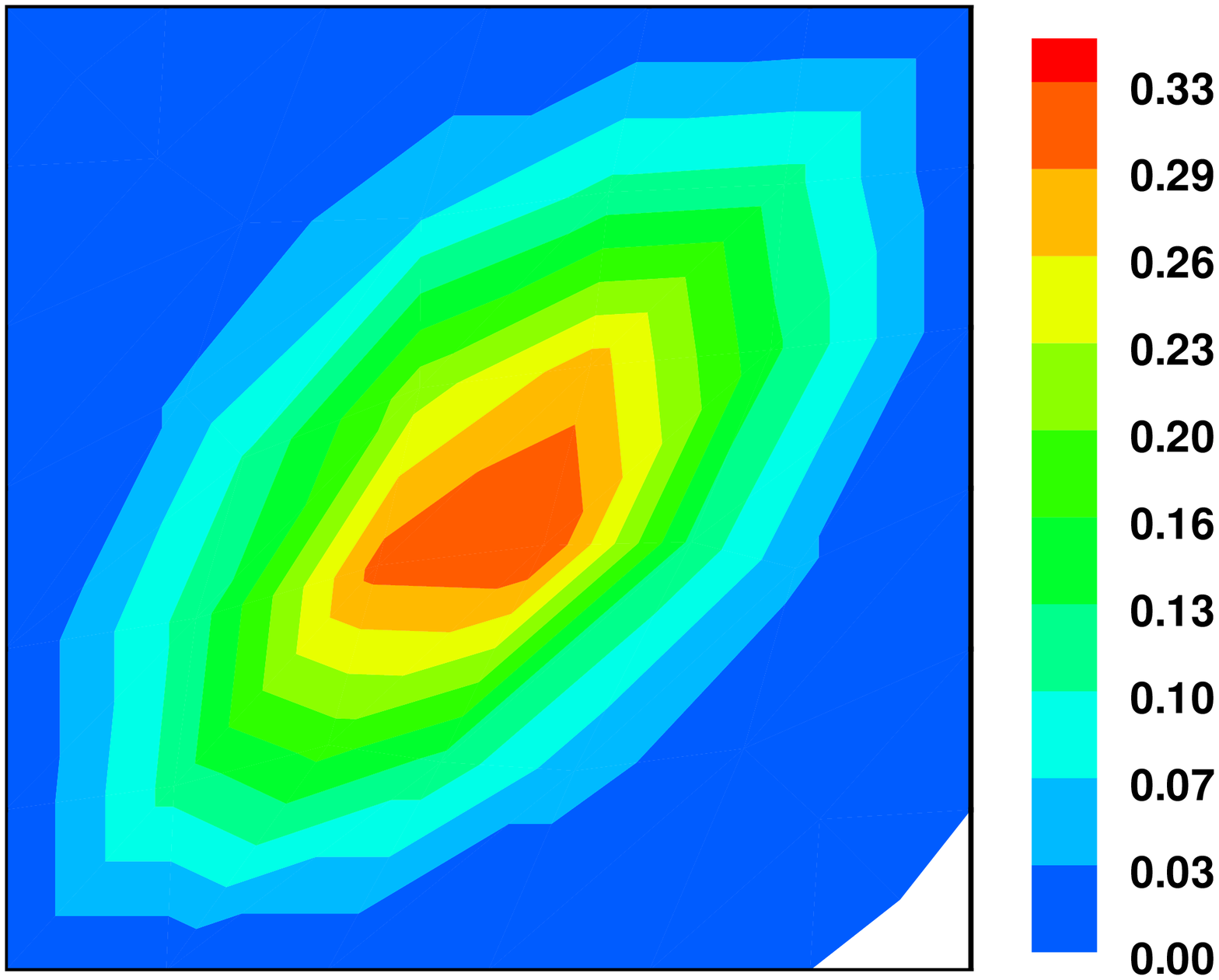}}
  \subfigure[Anisotropic mesh: $\mathbf{v} = (0,0)$ and $\alpha = 0$]
  {\includegraphics[scale=0.35]
  {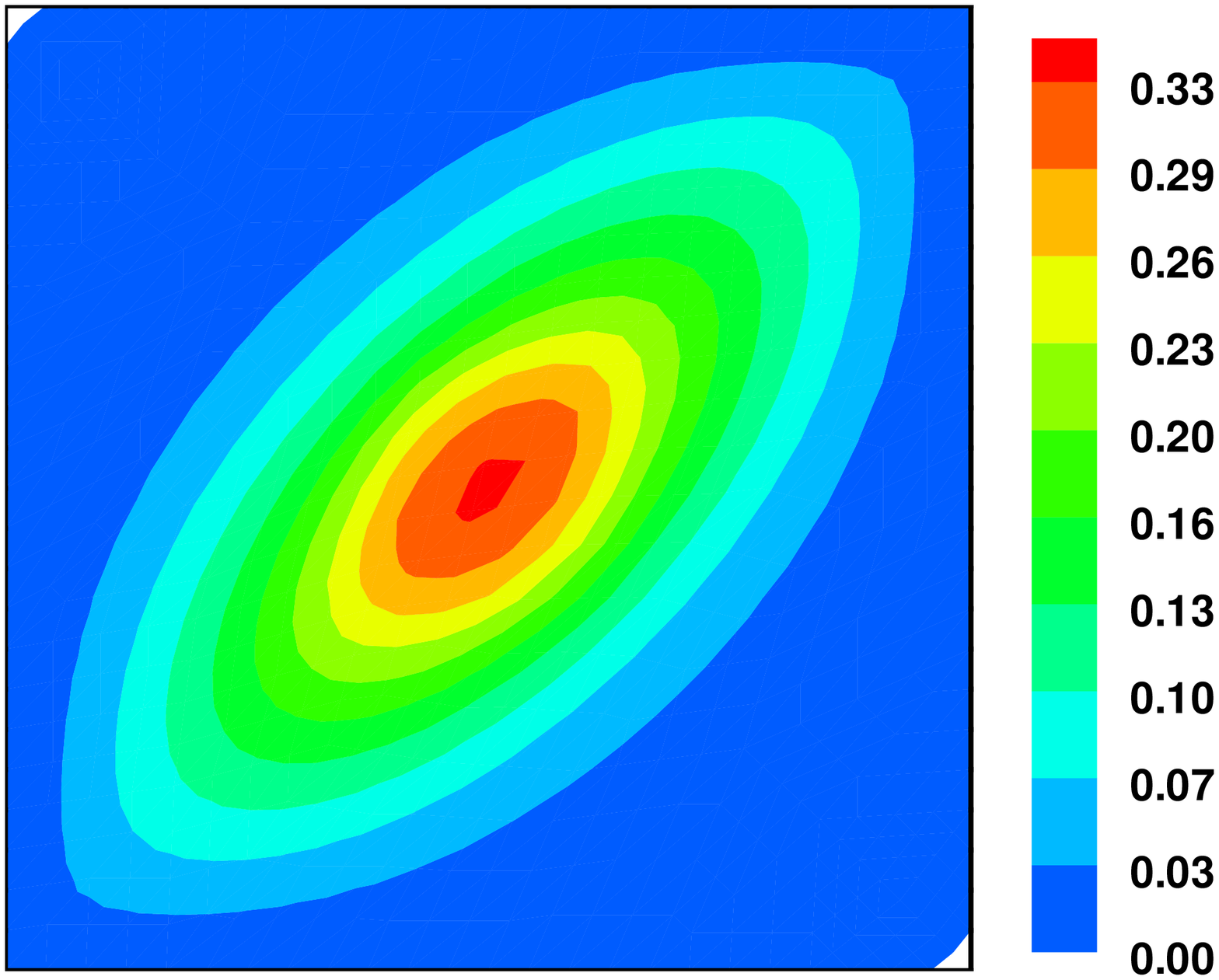}}
  \caption{\textsc{Test problem \#2}: The left figure shows the concentration 
    profile based on the background mesh, while the right figure shows the 
    concentration profile using the anisotropic triangulation.   
    \label{Fig:Prob2_Background_DMP_Meshes_Conc}}
\end{figure}

%-----------------------------------------------------------------------;
%  Pictorial Description: Meshes for Problem no # 3 (circle with hole)  ;
%-----------------------------------------------------------------------;
\begin{figure}
  \subfigure[Background mesh: $Nv = 5079$ and $Nele = 9918$]
  {\includegraphics[scale=0.35]
  {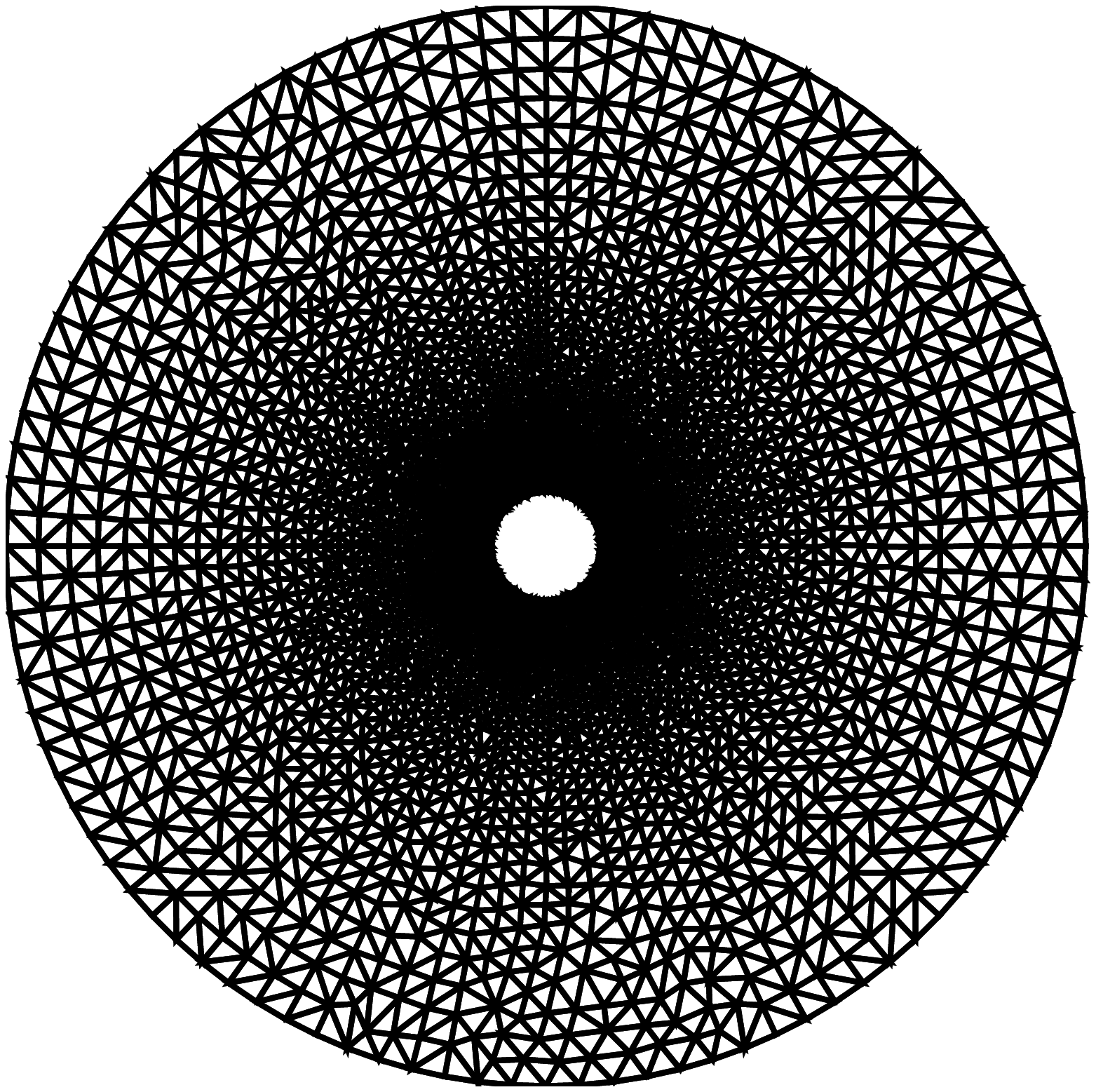}}
  \subfigure[Anisotropic mesh: $Nv = 297$ and $Nele = 436$]
  {\includegraphics[scale=0.35]
  {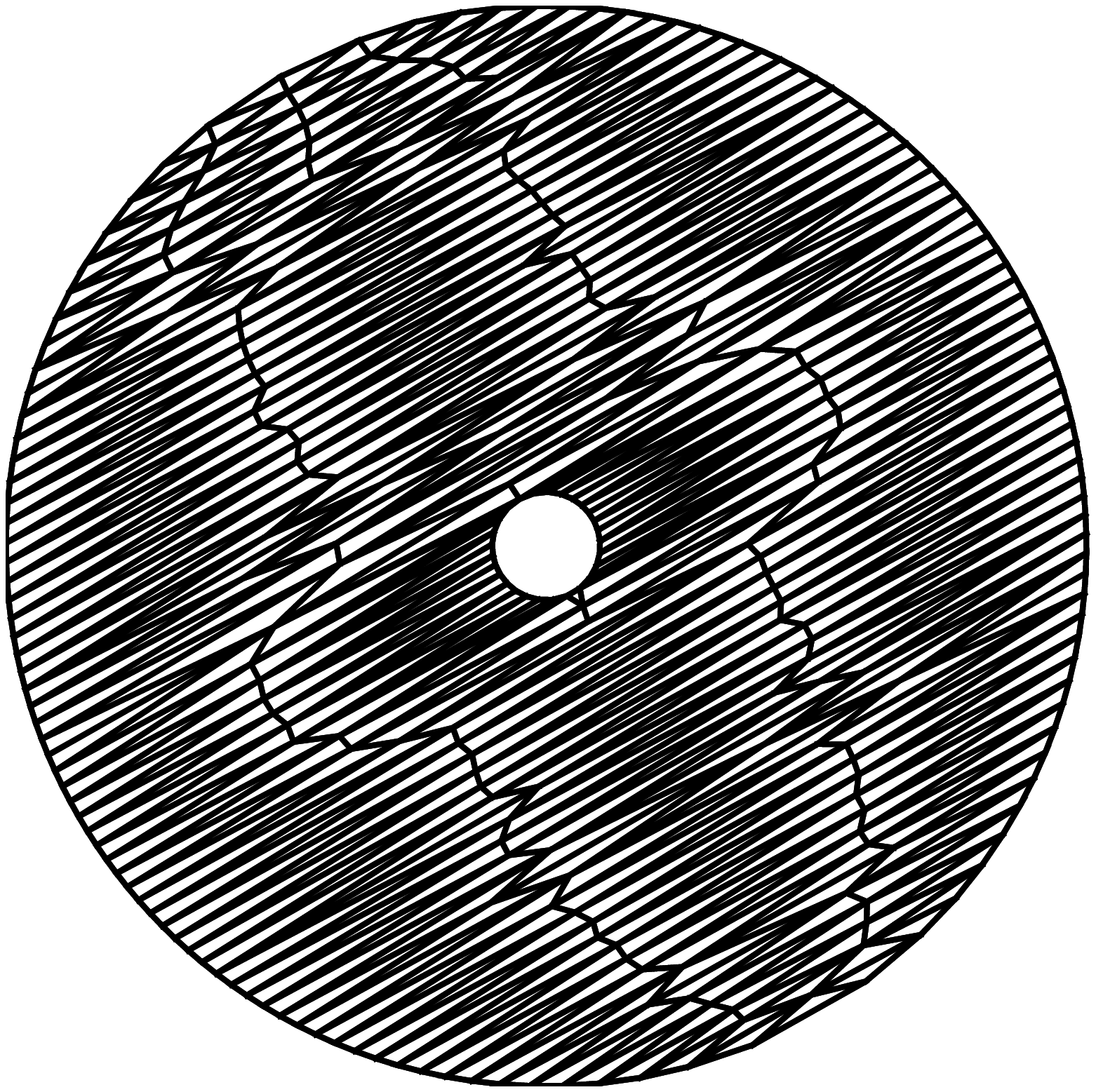}}
  \caption{\textsc{Test problem \#3}: The left figure shows the background 
    mesh and the right figure shows the anisotropic triangulation obtained 
    using \textsf{BAMG} for all the four cases.
    \label{Fig:Prob3_Background_DMP_Meshes}}
\end{figure}

%---------------------------------------------------------------------------------------;
%  Pictorial Description: Concentration profiles for Problem no # 3 (circle with hole)  ;
%---------------------------------------------------------------------------------------;
\begin{figure}
  \subfigure[Background mesh: $\mathbf{v} = (0,0)$ and $\alpha = 0$]
  {\includegraphics[scale=0.3]
  {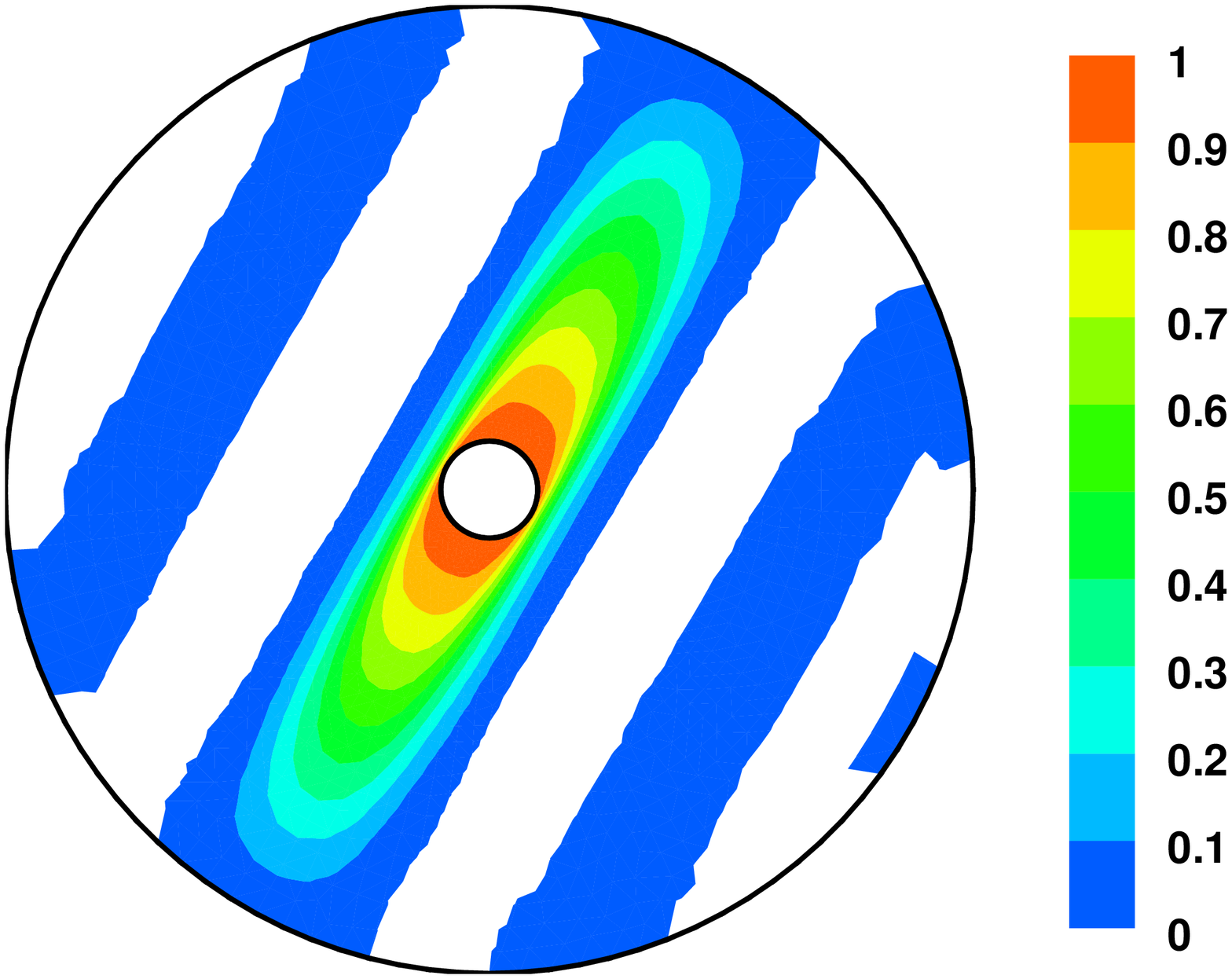}}
  \subfigure[Anisotropic mesh: $\mathbf{v} = (0,0)$ and $\alpha = 0$]
  {\includegraphics[scale=0.3]
  {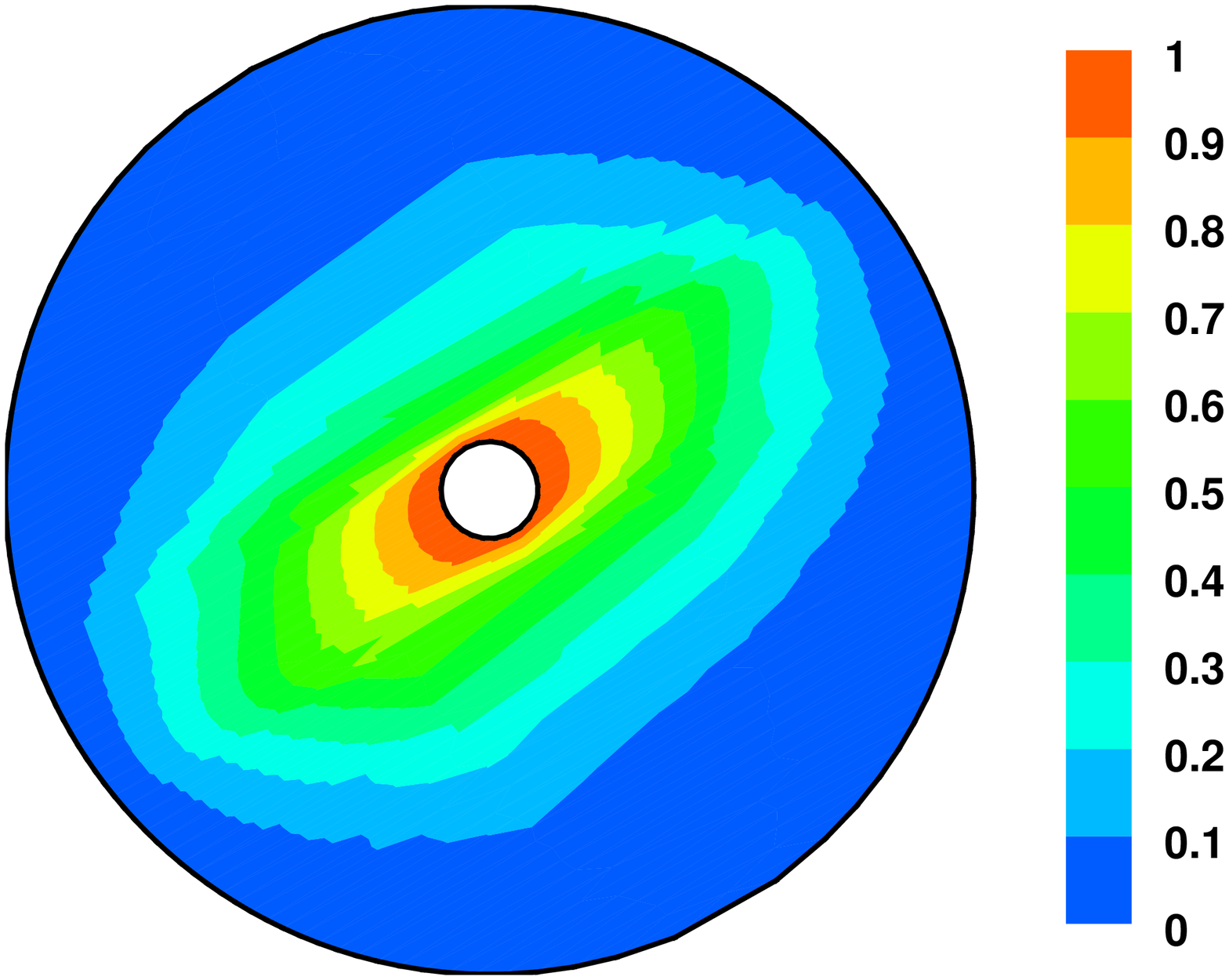}}
  \subfigure[Anisotropic mesh: $\mathbf{v} = (1.5,1.0)$ and $\alpha = 1.0$]
  {\includegraphics[scale=0.3]
  {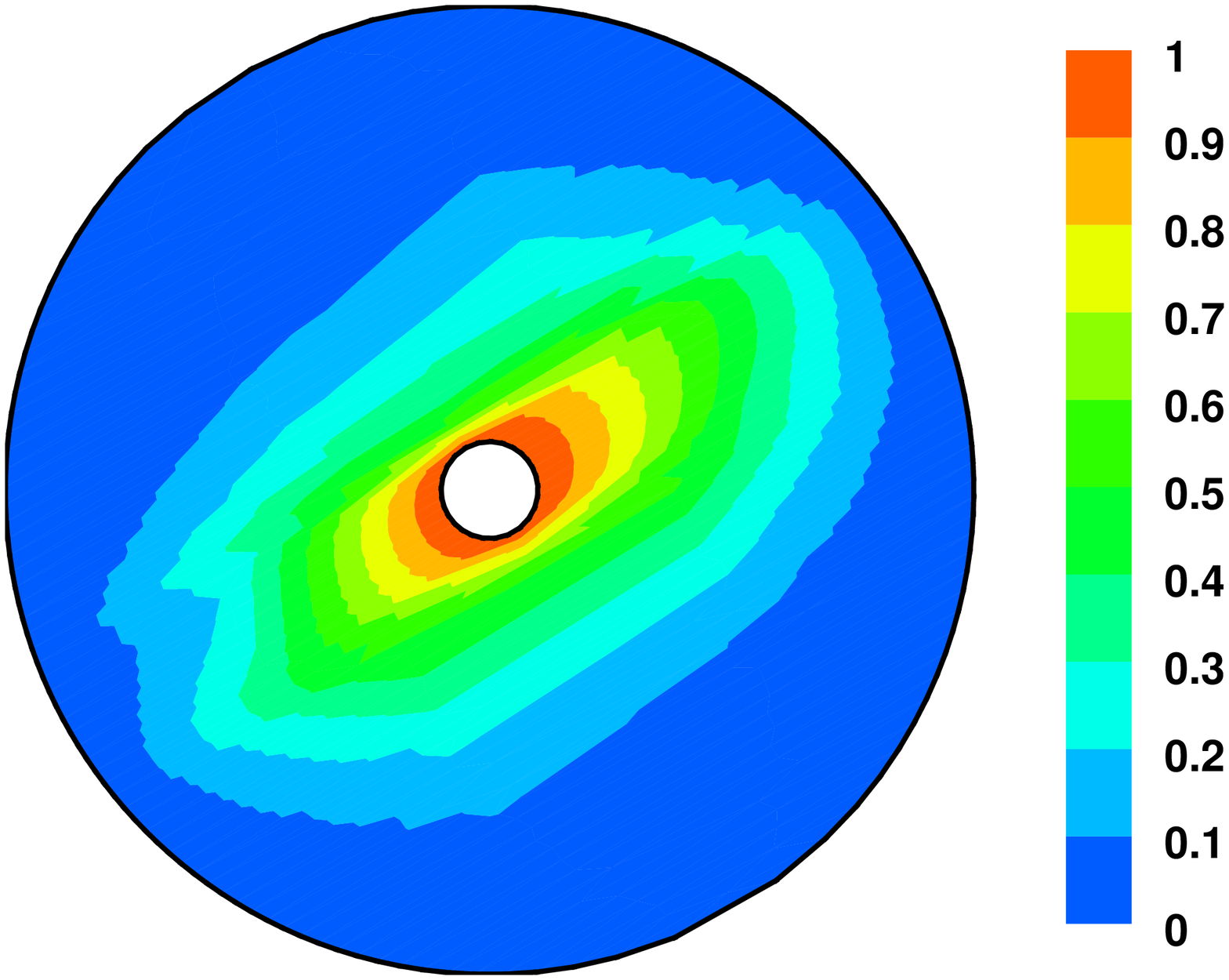}}
  \subfigure[Anisotropic mesh: $\mathbf{v} = (5.0,0.5)$ and $\alpha = 1.0$]
  {\includegraphics[scale=0.3]
  {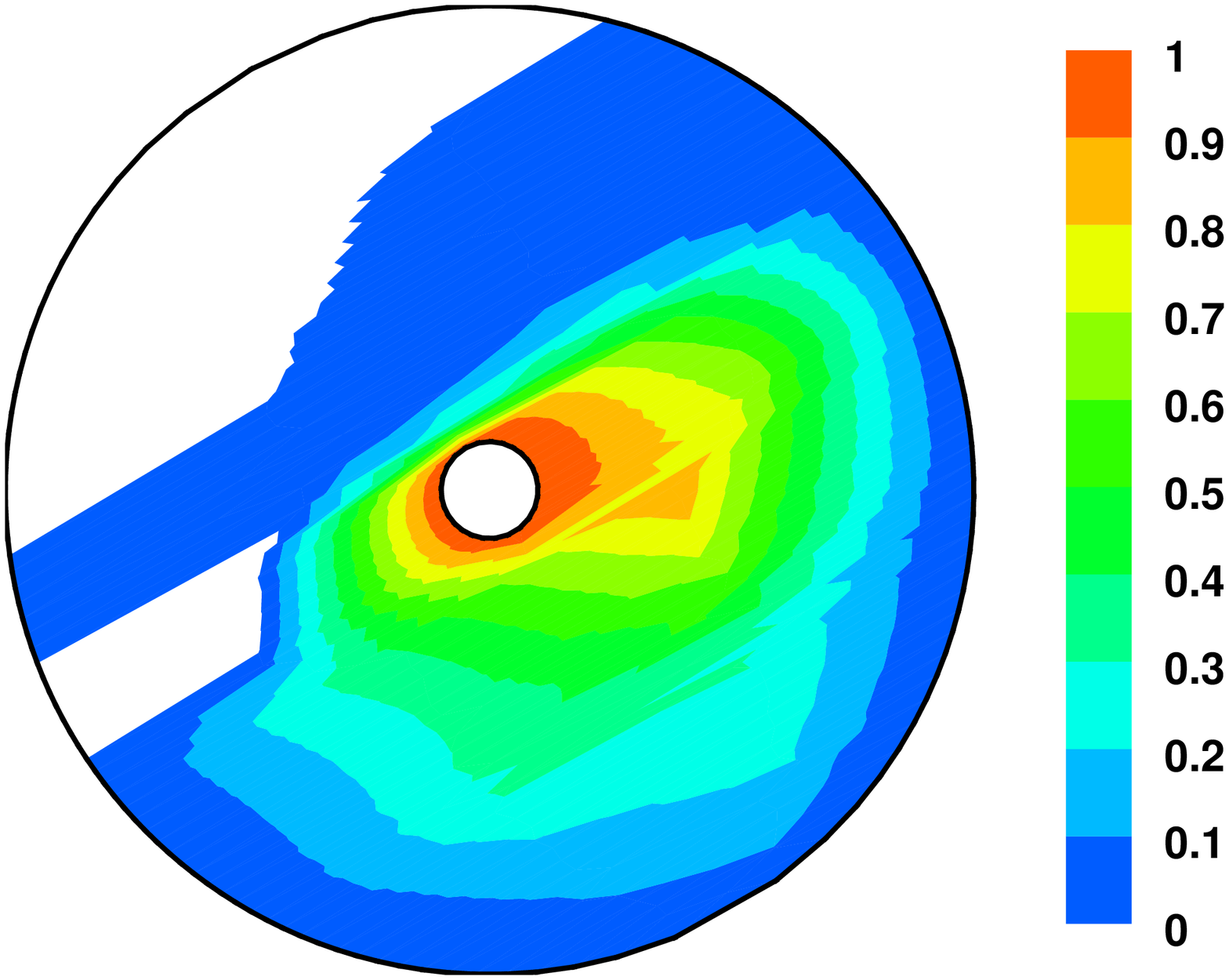}}
  \subfigure[Anisotropic mesh: $\mathbf{v} = (0,0)$ and $\alpha = 1000$]
  {\includegraphics[scale=0.3]
  {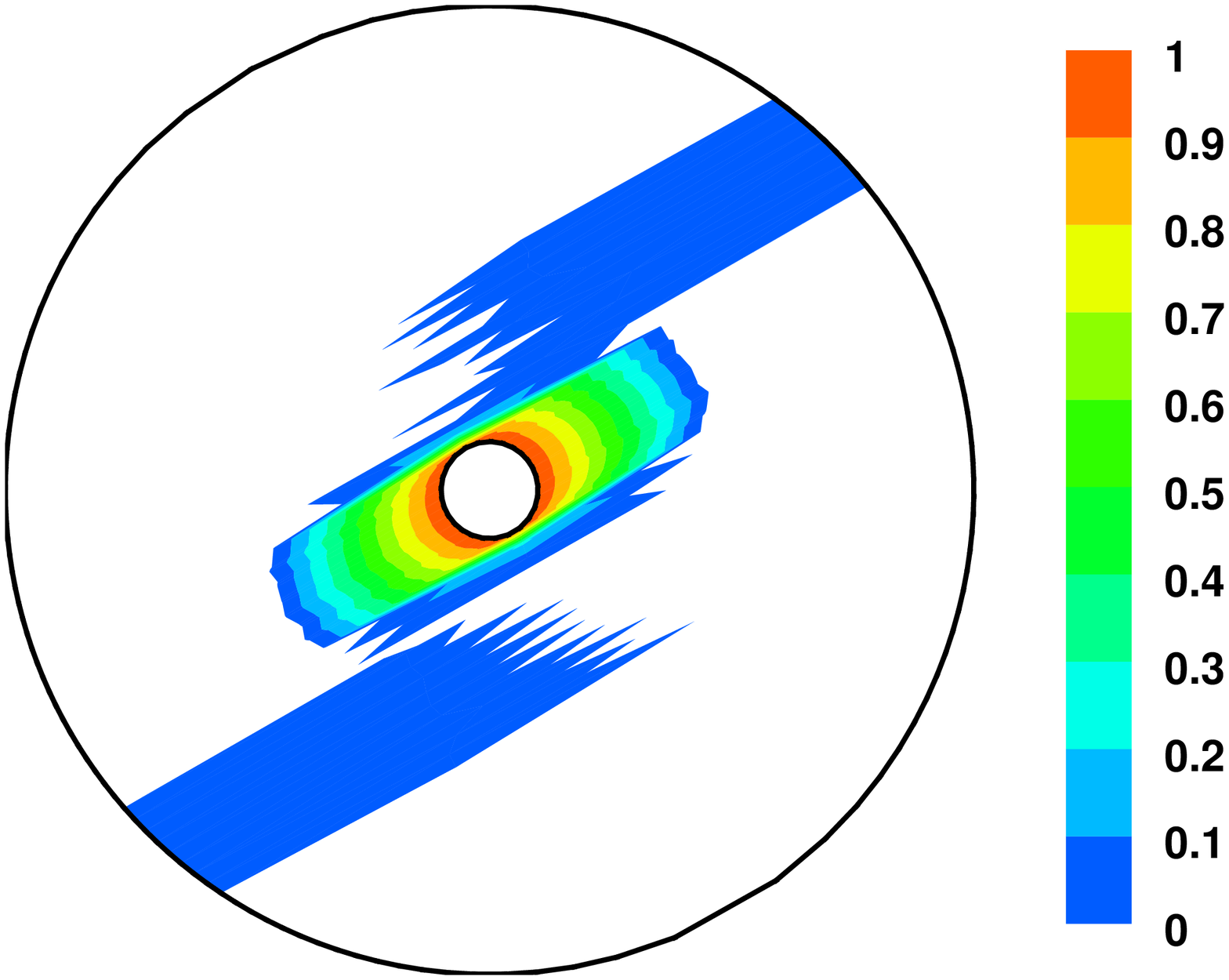}}
  \caption{\textsc{Test problem \#3}: This figure shows the concentration profiles 
    for four different cases based on the background mesh and anisotropic meshes 
    shown in Figure \ref{Fig:Prob3_Background_DMP_Meshes}.
    \label{Fig:Prob3_Background_DMP_Meshes_Conc}}
\end{figure}

%--------------------------------------------------------------------;
%  Pictorial Description: Traditional refinement (Fractured domain)  ;
%--------------------------------------------------------------------;
\begin{figure}
  \subfigure[Delaunay mesh: $Nv = 1564$ and $Nele = 2826$]
  {\includegraphics[scale=0.35]
  {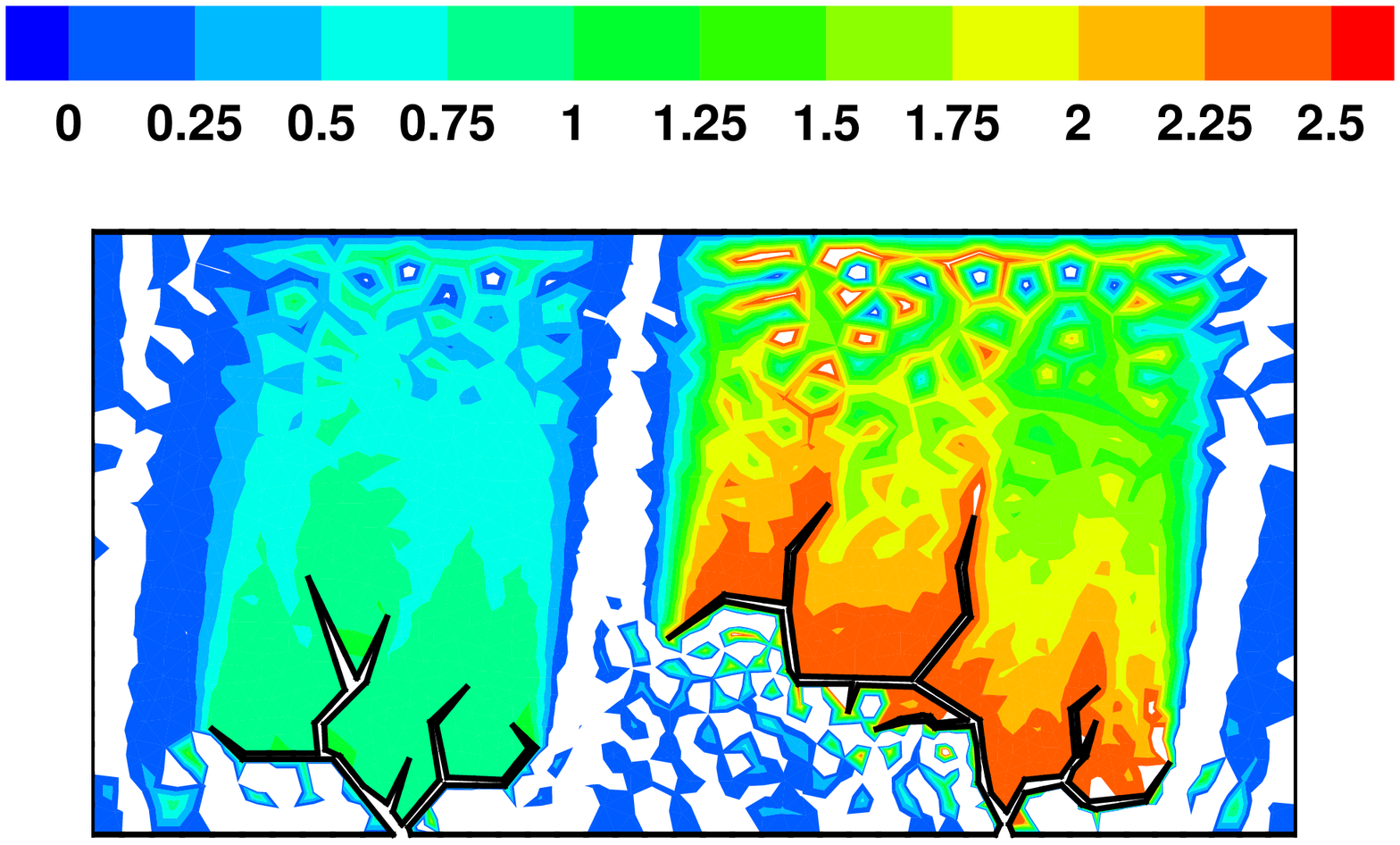}}
  \subfigure[Delaunay mesh: $Nv = 5090$ and $Nele = 9620$]
  {\includegraphics[scale=0.35]
  {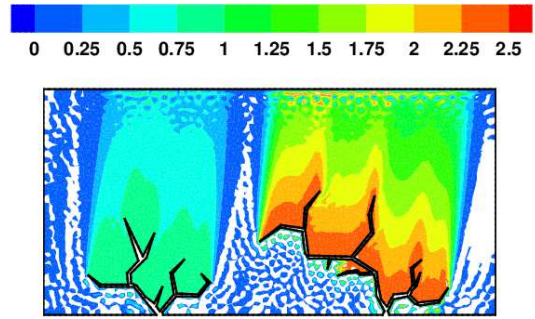}}
  \subfigure[Delaunay mesh: $Nv = 18372$ and $Nele = 35665$]
  {\includegraphics[scale=0.35]
  {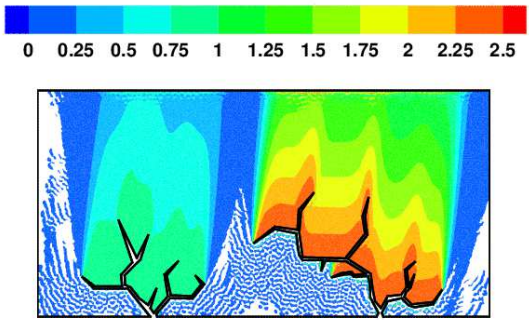}}
  \subfigure[Delaunay mesh: $Nv = 69995$ and $Nele = 137881$]
  {\includegraphics[scale=0.35]
  {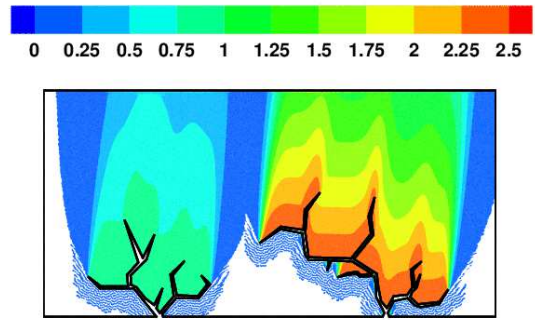}}
  \caption{\textsc{Issues with traditional mesh refinement}: Concentration profiles 
    for the fracture domain when $\mathbf{v} = (0.1,1.0)$ and $\alpha = 1.0$. The 
    white region in the figures shows the area in which the numerical simulation 
    has violated the NC and maximum constraint.
    \label{Fig:Prob1_TradRefine_Meshes}}
\end{figure}

%------------------------------------------------------------------;
%  Pictorial Description: Traditional refinement (Bi unit square)  ;
%------------------------------------------------------------------;
\begin{figure}
  \subfigure[Anisotropic mesh: $Nv = 8897$ and $Nele = 17408$]
  {\includegraphics[scale=0.35]
  {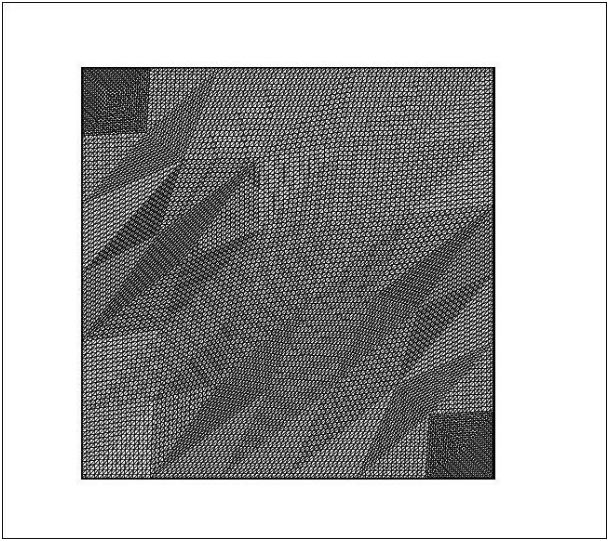}}
  \subfigure[Pure anisotropic diffusion: Concentration profile]
  {\includegraphics[scale=0.35]
  {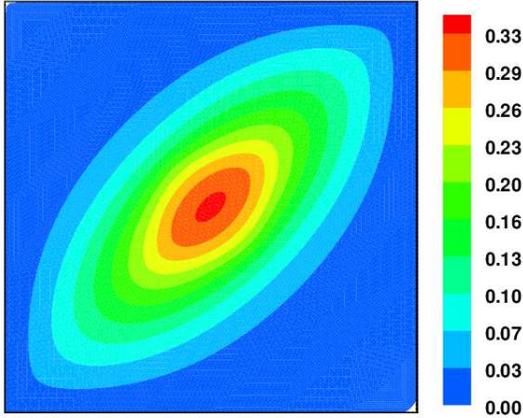}}
  \caption{\textsc{Issues with traditional mesh refinement}: The left figure shows 
    the anisotropic mesh obtained using the traditional mesh refinement procedure
    on the anisotropic triangulation given in Figure \ref{Fig:Prob2_Background_DMP_Meshes}. 
    The right figure shows the concentration profile obtained using this refined mesh. 
    It should be noted that this $h$-refined DMP-based mesh is interiorly connected. 
    Hence, it satisfies $\mathrm{DSMP}_{\boldsymbol{K}}$.
    \label{Fig:Prob2_TradRefine_Meshes}}
\end{figure}

%--------------------------------------------------------------------;
%  Pictorial Description: Traditional refinement (circle with hole)  ;
%--------------------------------------------------------------------;
\begin{figure}
  \subfigure[Anisotropic mesh: $Nv = 2199$ and $Nele = 3924$]
  {\includegraphics[scale=0.35]
  {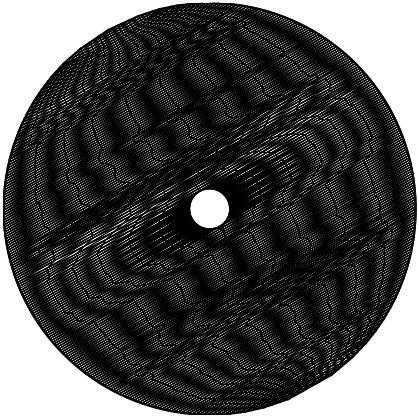}}
  \subfigure[Pure anisotropic diffusion: Concentration profile]
  {\includegraphics[scale=0.35]
  {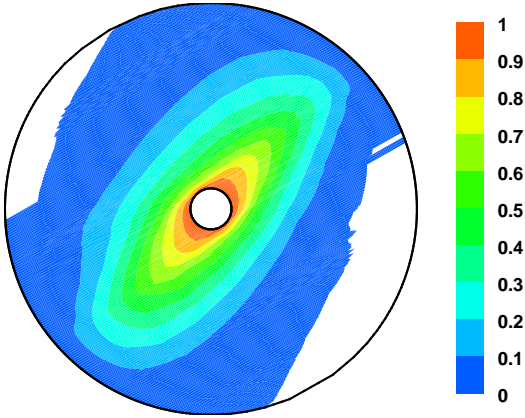}}
  \caption{\textsc{Issues with traditional mesh refinement}: The left figure shows 
    the anisotropic mesh obtained using the traditional mesh refinement procedure
    on the anisotropic triangulation given in Figure \ref{Fig:Prob3_Background_DMP_Meshes}.
    The right figure shows the concentration profile obtained using this refined mesh.
    \label{Fig:Prob3_TradRefine_Meshes}}
\end{figure}

%------------------------------------------------------------------------;
%  Pictorial Description: Non-traditional refinement (circle with hole)  ;
%------------------------------------------------------------------------;
\begin{figure}
  \subfigure[Anisotropic mesh: $Nv = 2647$ and $Nele = 4789$]
  {\includegraphics[scale=0.35]
  {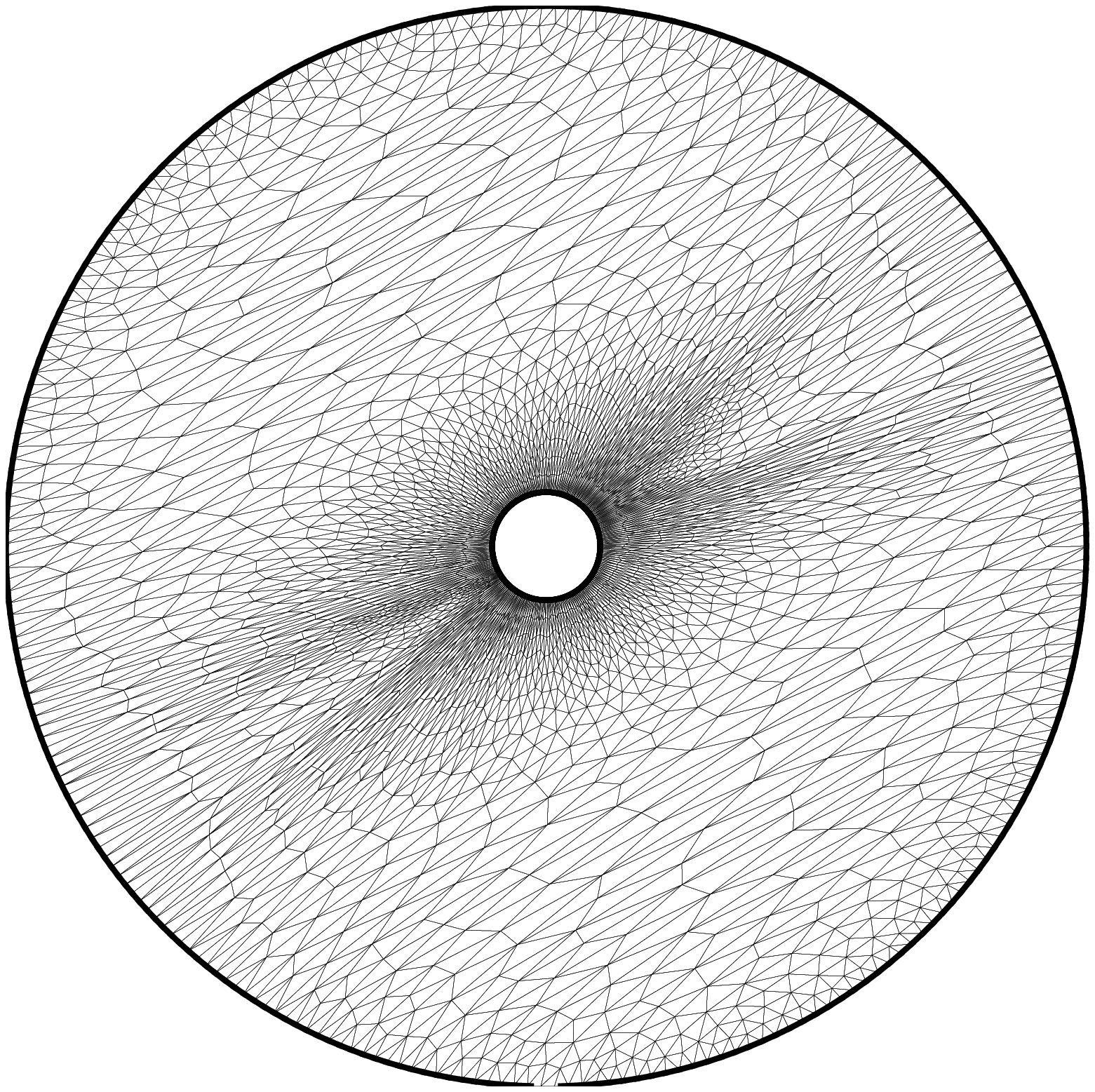}}
  \subfigure[Pure anisotropic diffusion: Concentration profile]
  {\includegraphics[scale=0.35]
  {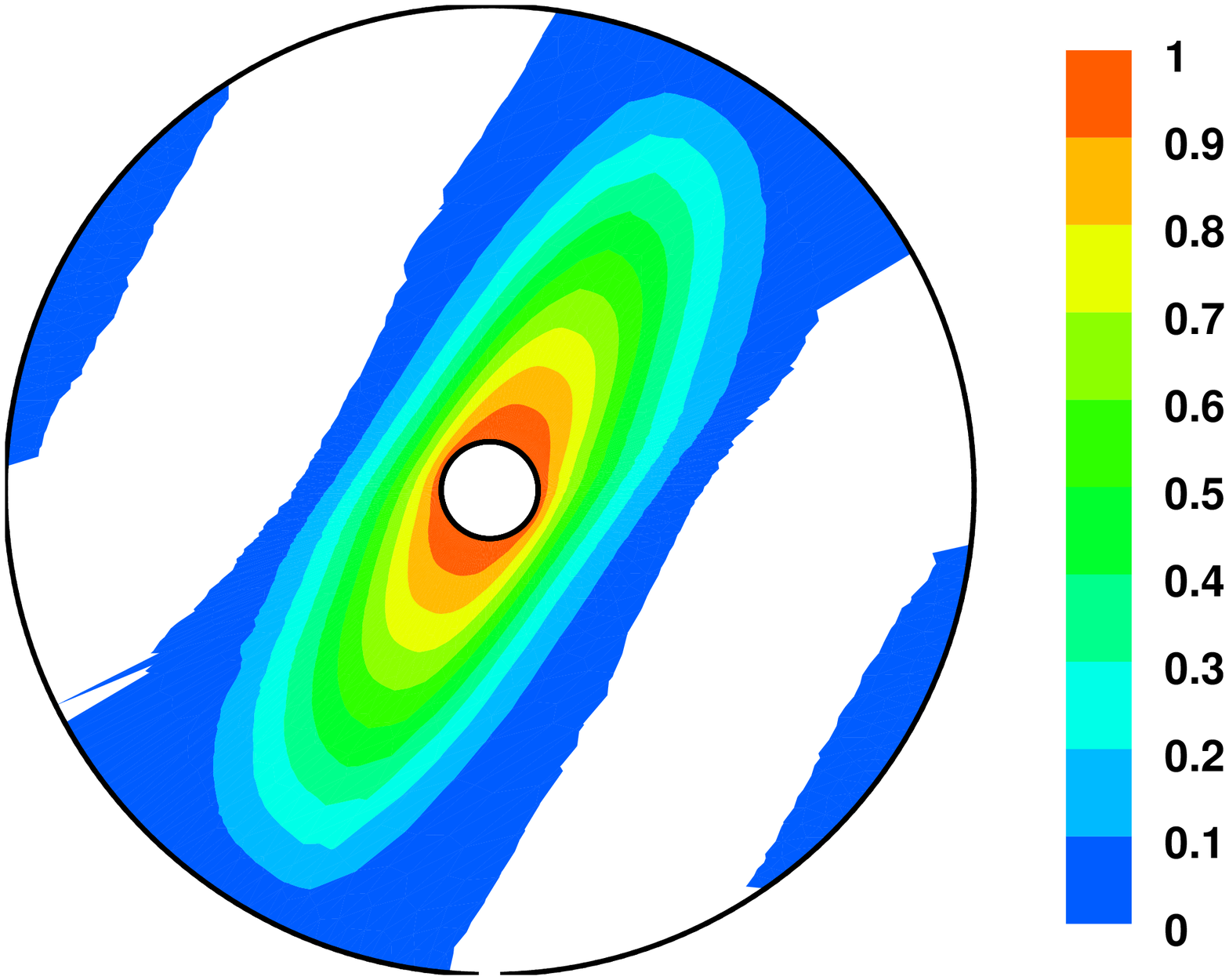}}
  \caption{\textsc{Issues with non-traditional mesh refinement}: The left figure shows 
    a refined anisotropic mesh obtained using the non-traditional approach.The right 
    figure shows the concentration profile obtained using this refined mesh. It should 
    be pointed out that the mesh obtained using this procedure did not converge in 
    \texttt{MaxIters} = 100. Hence, as a result, it violates NC and DMPs.
    \label{Fig:Prob3_NonTradRefine_Meshes}}
\end{figure}

%------------------------------------------------------------------------;
%  Pictorial Description: Errors in local species balance (coarse mesh)  ;
%------------------------------------------------------------------------;
\begin{figure}
  \subfigure[Test problem \#1: $\mathbf{v} = (0,0)$ and $\alpha = 0$]
  {\includegraphics[scale=0.3]
  {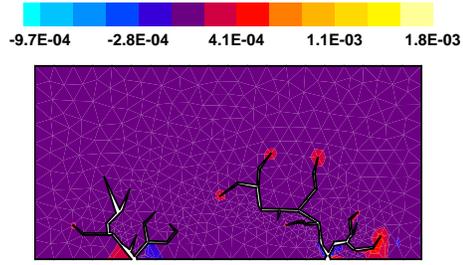}}
  \subfigure[Test problem \#2: $\mathbf{v} = (0,0)$ and $\alpha = 0$]
  {\includegraphics[scale=0.3]
  {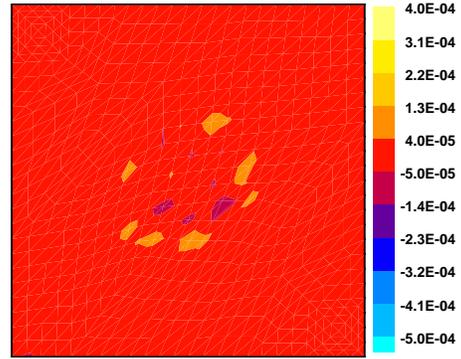}}
  \subfigure[Test problem \#3: $\mathbf{v} = (0,0)$ and $\alpha = 0$]
  {\includegraphics[scale=0.3]
  {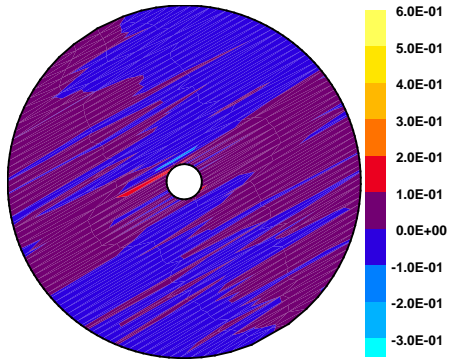}}
  \subfigure[Test problem \#3: $\mathbf{v} = (1.5,1.0)$ and $\alpha = 1.0$]
  {\includegraphics[scale=0.3]
  {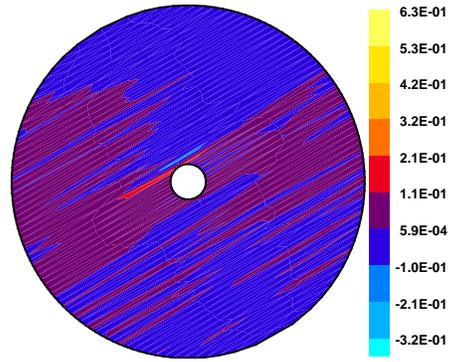}}  
  \caption{\textsc{Local species balance errors}: The figures show the errors incurred 
    in satisfying local species balance for various test problems on coarse 
    meshes.
    \label{Fig:Prob123_Local_Mass_Errors}}
\end{figure}

\end{document}